\documentclass[11pt]{article}
\usepackage[a4paper,top =1.0in,bottom=1.0in,left=1.0in,right=1.0in]{geometry}

\usepackage{tocloft}
\usepackage[nottoc,numbib]{tocbibind} 

\usepackage[dvipsnames]{xcolor}
\usepackage[utf8]{inputenc}
\usepackage{float}
\usepackage[numbers]{natbib}
\usepackage{graphicx}
\graphicspath{{figures/}{pictures/}{figs/}{./}} 
\usepackage{amsfonts}
\usepackage{amsmath}
\usepackage{amssymb}
\usepackage{amsthm} 
\usepackage{mathtools}
\usepackage{amssymb}
\usepackage{enumerate}
\usepackage[hidelinks]{hyperref}
\usepackage{xurl}
\usepackage{bm} 
\usepackage{tikz-cd}
\tikzcdset{every label/.append style = {font = \small}} 
\usepackage{caption}
\usepackage{subcaption}
\usepackage{enumitem} 
\usepackage{faktor} 
\usepackage{longtable}
\usepackage{multirow}

\definecolor{deepmagenta}{rgb}{0.8, 0.0, 0.8}
\definecolor{dartmouthgreen}{rgb}{0.05, 0.5, 0.06}

\usepackage{algorithm}
\usepackage{algpseudocode}

\usepackage{titlesec}
\titleformat{\subsubsection}[runin]
       {\normalfont\bfseries}
       {\thesubsubsection}
       {0.5em}
       {}
       [.]

\newcommand{\isep}{\dots} 
       
\theoremstyle{definition}

\theoremstyle{plain} 
\newtheorem{theorem}{Theorem}[section] 
\newtheorem{proposition}[theorem]{Proposition} 
 
\newtheorem{lemma}[theorem]{Lemma} 
\theoremstyle{definition} 
\newtheorem{example}[theorem]{Example}
\theoremstyle{remark} 
\newtheorem{remark}[theorem]{Remark}

\makeatletter
\newenvironment{breakablealgorithm}
  {
   \vspace*{5pt}
   \begin{center}
     \refstepcounter{algorithm}
     \hrule height.8pt depth0pt \kern2pt
     \renewcommand{\caption}[2][\relax]{
       {\raggedright\textbf{\fname@algorithm~\thealgorithm} ##2\par}%
       \ifx\relax##1\relax 
         \addcontentsline{loa}{algorithm}{\protect\numberline{\thealgorithm}##2}%
       \else 
         \addcontentsline{loa}{algorithm}{\protect\numberline{\thealgorithm}##1}%
       \fi
       \kern2pt\hrule\kern2pt
     }
  }{
     \kern2pt\hrule\relax
   \end{center}
   \vspace*{5pt}
  }
\makeatother

\newcommand{\spn}[1] { \langle #1\rangle }
\newcommand{\projbracket}[1] { \proj\big[#1\big] }
\newcommand{\projbrackethat}[1] { \widehat{\proj}\big[#1\big] }
\newcommand{\HD}[2] { \d_\mathrm{H}\big(#1,#2\big) }    
\newcommand{\HDmid}[2] { \d_\mathrm{H}\big(#1|#2\big) } 

\newcommand{\R}{\mathbb{R}}
\newcommand{\C}{\mathbb{C}}

\newcommand{\V}{\mathcal{V}} 
\newcommand{\G}{\mathcal{G}} 

\newcommand{\T}{\mathbb T}
\newcommand{\Z}{\mathbb Z}

\newcommand{\M}{\mathrm{M}} 
\newcommand{\N}{\mathrm{N}} 
\renewcommand{\T}{\mathrm{T}} 
\newcommand{\W}{\mathrm{W}} 
\newcommand{\reach}{\mathrm{reach}} 
\newcommand{\op}{\mathrm{op}} 
\newcommand{\Sym}{\mathrm{Sym}} 
\newcommand{\sym}{\mathfrak{sym}} 
\newcommand{\man}{\mathcal{M}} 

\newcommand{\E}{\mathbb{E}}
\newcommand{\Var}{\mathbb{V}}

\renewcommand{\O}{\mathcal{O}} 

\newcommand{\GL}{\mathrm{GL}}

\newcommand{\Ort}{\mathrm{O}}
\newcommand{\SO}{\mathrm{SO}}
\newcommand{\U}{\mathrm{U}}
\newcommand{\SU}{\mathrm{SU}}

\newcommand{\Spin}{\mathrm{Spin}}

\newcommand{\h}{\mathfrak{h}}

\newcommand{\so}{\mathfrak{so}}

\newcommand{\su}{\mathfrak{su}}

\newcommand{\g}{\mathfrak{g}}
\renewcommand{\t}{\mathfrak{t}}
\newcommand{\gl}{\mathfrak{gl}}

\newcommand{\proj}{\Pi}

\newcommand{\diag}{\mathrm{diag}}

\DeclareMathOperator{\Tr}{tr}

\newcommand{\iu}{{i\mkern1mu}} 
\renewcommand{\d}{\mathrm{d}} 
\DeclareMathOperator{\Rig}{Rig} 

\begin{document}\sloppy
\begin{center}\begin{minipage}{0.9\linewidth}
\begin{center}
\Large{{LieDetect}: Detection of representation orbits \\ of compact Lie groups from point clouds}
\vspace{.5cm}\\
\large{Henrique Ennes$^1$
~~~~~~~~~~
Raphaël Tinarrage$^2$}
\vspace{.5cm}

\begin{minipage}{0.75\linewidth}
$^{1,2}$~EMAp, Fundação Getulio Vargas, Rio de Janeiro, Brazil\\
$^{1}$~~~INRIA Université Côte d’Azur, Valbonne, France\\
$^{2}$~~~IST Austria, Klosterneuburg, Austria    
\end{minipage}

\end{center}

\paragraph{Abstract.}
We suggest a new algorithm to estimate representations of compact Lie groups from finite samples of their orbits. 
Different from other reported techniques, our method allows the retrieval of the precise representation type as a direct sum of irreducible representations. Moreover, the knowledge of the representation type permits the reconstruction of its orbit, which is useful for identifying the Lie group that generates the action, from a finite list of candidates. 
Our algorithm is general for any compact Lie group, but only instantiations for SO(2), $T^d$, SU(2), and SO(3) are considered. 
Theoretical guarantees of robustness in terms of Hausdorff and Wasserstein distances are derived.
Our tools are drawn from geometric measure theory, computational geometry, and optimization on matrix manifolds. The algorithm is tested for synthetic data up to dimension 32, as well as real-life applications in image analysis, harmonic analysis, density estimation, equivariant neural networks, chemical conformational spaces, and classical mechanics systems, achieving very accurate results.

\paragraph{Keywords.}  
Linear actions of Lie groups $\cdot$
Geometric inference from point clouds $\cdot$
Sensitivity analysis in Wasserstein distance $\cdot$
Machine learning

\paragraph{MSC2020 codes.} 
68U05 $\cdot$ 
49Q12 $\cdot$ 
15B30 $\cdot$ 
49Q22 $\cdot$ 
49Q15 $\cdot$ 
68T07
\vspace{.5cm}

Communicated by Peter Bubenik.
\vspace{.3cm}

\noindent\makebox[\textwidth]{\rule{0.685\paperwidth}{0.4pt}}
\end{minipage}\end{center}
\vspace{-.2cm}

\setcounter{tocdepth}{3}\tableofcontents

%
%

\section{Introduction}\label{sec:introduction}

An especially impactful realization in quantitative sciences in the early last century was the combination of the loose notions of symmetry with the precise mathematical definition of groups. From influential results in cryptography \citep{gallian2021contemporary}, passing through the description of the hydrogen atom \cite{woit2017quantum} and the standard model of particles \cite{griffiths2020introduction}, arriving at the formalization of Felix Klein's \textit{Erlangen} program \cite{fecko2006differential}, the group-theoretic point of view on symmetries allowed for unprecedented developments in many theoretical and practical fields. It is, therefore, unsurprising that the prominent role assumed by symmetries translated itself to the description of learning, be it natural---e.g., the detection of invariances in visual psychology \cite{hoffman1968neuron, hoffman1966lie}---or artificial. 
Indeed, the ability to identify that a rotated `7'-digit is still a 7 tells us much about the important presence of symmetries---and, consequently, of groups---in the process of `understanding'.

Notice that symmetries are not restricted to the cognition part of the process---i.e., they are not only devices created by our brain to simplify its training task, although they can be, sometimes, artificially induced---but they are inherent to data. It is because we have seen many images of the same object rotated around its center or shifted in space and have been told, in whatever form, that these correspond to the same entity, that we acquired the invariance of the digit `7' concerning rotations. In a nutshell, this indicates that there is a genuine interest to identify, directly from data, the existence of symmetries on it, a task which we shall review below. Whether and how the information of existing symmetries is used to improve the training process (which it does, as, again, we will indicate in Section \ref{sec:applications}) is not the focus here; rather, we are primarily interested in the ability to tell, from data only, how likely the symmetries are to be generated by a particular group $G$ and the `symmetry' types involved. 

We attempt to tackle the symmetry identification problem algorithmically. 
That is, in a more mathematical language, our goal is to create a technique to determine the actions of groups that generate the symmetries on a dataset. Of course, as stated, this task is rather open, so we assume some constraints. 
The first of these is related to the groups considered. Instead of an algorithm that works for the whole class of groups, we will focus on detecting symmetries generated by \textit{Lie groups}, those admitting smooth structures within their operations. 

We will also restrict the kinds of symmetry, or in group language, actions, that describe how the Lie groups interact with datasets. 
Instead of the general scenario, we consider only \textit{representations}, which are linear actions of Lie groups in vector spaces. This means that only one particular kind of dataset, point clouds living in some $\R^n$, will interest us.
More precisely, we study point clouds $X$ sampled on or close to a submanifold $\O$ that carries a transitive action of its symmetry group.
Equivalently, we assume $\O$ is the \textit{orbit} of a Lie group representation.
Moreover, because of the linearity constraint, it is always possible to embed such Lie group symmetries as $n\times n$ matrices, which significantly simplifies the computations performed.
Lastly, since we only wish to consider, from a data analysis perspective, samples on compact manifolds---and hence, on compact orbits---we will restrict ourselves to \textit{compact} Lie groups.

Although unapologetically restrictive compared to the general problem of determining arbitrary actions of arbitrary groups on arbitrary datasets, the constrained setting of representations of compact Lie groups in point clouds still sets the stage for many interesting problems in data science, some of which we motivate in the examples below, further explored in Section \ref{sec:applications}.

\begin{example}[Pixel permutations, Section \ref{subsec: ppt}]
\label{ex:pixel_permutations}
Consider a set of $m_x\times m_y$-pixeled images generated by permutations pixels of a fixed initial image. Supposing these permutations form a finite group $\Sigma$, then the embedding of the images in $\R^{m_x\times m_y}$ will necessarily lie on the orbit of a representation of $\Sigma$ in $\R^{m_x\times m_y}$.
As it turns out, when $\Sigma$ is Abelian, the data also lies within an orbit of the compact Lie group $T^d$, the $d$-dimensional torus. We propose using this geometric structure to reconstruct the exact representation type from the embedded vectors only.
\end{example}

\begin{example}[Harmonic analysis, Section \ref{subsec:harmonic_analysis}]
\label{ex:harmonic_analysis}
Abstract harmonic analysis denotes the study of signal decomposition in arbitrary groups. For the case of compact Lie groups, the existence of a discrete orthonormal basis for the set of $L^2$-integrable functions is guaranteed by the Peter-Weyl theorem, which implies that such a basis is parametrized by $\widehat{G}$, the complex irreducible representations of $G$. In particular, ordinary harmonic analysis, taken as the decomposition of periodic functions through sines and cosines, reduces to a special case of this result.
Suppose then $\{(x_i, y_i)\}_{i=1}^N$ is a dataset, where each $x_i$ lies in the orbit $\mathcal{O}$ of a representation of a compact Lie group $G$, $y_i$ are real scalars, and there is an unknown function $f:\mathcal{O}\to \R$ such that $f(x_i)\approx y_i$. If $\{f_{[\phi]}\}_{[\phi]\in \widehat{G}}$ is an abstract harmonic analysis basis for $G$, then this regression can be written as 
\begin{equation*}
    f = \sum_{[\phi]\in \widehat{G}}a_{[\phi]}f_{[\phi]},
\end{equation*}
where $ a_{[\phi]}$ are complex constants. Not only does the summation above converge in the whole function's domain, but the functions $f_{[\phi]}$ can be known for every choice of $G$. Therefore, the knowledge that the input data lies close to an orbit of a certain compact Lie group $G$ suggests a linear regression solution to the machine learning problem.
\end{example}

\begin{example}[Equivariant neural networks, Section \ref{sec: equivariant neural networks}]
Steerable CNNs correspond to a class of equivariant implementations of neural networks, that is, networks in which some symmetry of the data is preserved through all layers. Although there are implementations of steerable CNNs to the compact Lie groups of most interest, in applications, researchers often ensure only equivariance up to some of their finite subgroups. It remains an open question whether the full symmetries are well approximated.
\end{example}

\begin{example}[Conformational spaces, Section \ref{subsec:conformational_space}]
The conformational space of a molecule refers to the space of all possible atomic positions within it. As expected, these spaces naturally exhibit geometric structures related to the molecular constitution \cite{brown2008algorithmic,martin2010topology}.
The existence of Lie group symmetries on these spaces has yet to be explored.
\end{example}

\begin{example}[The three-body problem, Section \ref{subsec:3body}]
The three-body problem has received considerable attention from physicists and mathematicians, especially due to its chaotic nature. This constraint implies that analytic solutions are rather rare and, consequently, there has been much interest in finding them. Solutions that lie close to orbits of group representations are, approximately, analytic, and determining them might propel more theoretical research.
\end{example}

\paragraph{Our contributions and related work.}
In Algorithm \ref{alg: 1}, which we call \texttt{LieDetect}, we propose a framework to transform the problems from the last examples in machine learning tasks. The precise model assumed will be formalized in Section \ref{sec:algorithm_description}, but simply put, we assume the existence of an unknown representation $\phi:G\to \GL_n(\R)$ of some known compact Lie group $G$ in some vector space $\R^n$, for which $\mathcal{O}\subseteq \R^n$ is one of its orbits. We then want to precisely determine the representation type, that is, the group homomorphism $\phi$, using only a point cloud $X$ well enough sampled from $\mathcal{O}$. Moreover, because the knowledge of $\phi$ permits the reconstruction $\mathcal{O}$, it is possible to determine which of the compact Lie groups is most likely to have generated the orbit. Unfortunately, this is an ill-posed problem: two distinct compact Lie groups can generate the same orbit $\mathcal{O}$; moreover, fixed a compact Lie group $G$, two of its distinct representations can also generate $\mathcal{O}$.  Ignoring these subtleties for now, we stress that, as far as we know, ours is the first reported technique in the literature to identify both the precise representation type and the most likely acting Lie group. We believe this information may be fruitful in developing new inference techniques and application settings, some of which we already suggest here. Besides, the theoretical guarantees of Section \ref{sec:algorithm_analysis} imply strong robustness results for our algorithm, which works well even when available data $X$ does not sit perfectly within $\mathcal{O}$ because of noise.

On the other hand, although new in what it accomplishes, our approach is not the first attempt to solve this sort of problem, given the long history of interest in action detection problems. For example, as indicated by \cite{shi2016symmetry}, the related problem of symmetry identification has been tackled since the 1980s, with algorithms for symmetries of $\text{SE}(2)$ representations in planar data \cite{wolter1985optimal, highnam1986optimal}. A similar task, now for 3D objects, has often been addressed in the literature, as well-reviewed by \cite{mitra2013symmetry}. We stress, however, that all these works are limited to the detection of symmetries as subgroups of specific rotation groups, a scope slightly different from ours, which is agnostic to the choice of \textit{any} compact Lie group, provided the list of its real representations, and which outputs the generators of the representation. 

The automatic detection of the representation operators of more general Lie groups from different datasets has also received some attention. Sohl-Dickstein, Wand, and Olshausen, for example, use a modified version of the MAP algorithm to estimate the operators $T$ such that $\sum_{t}|x^{(t)}-Tx^{(t-1)}|$ is minimized \cite{sohldickstein2017unsupervised}. Although having the obvious advantage of dealing with naturally sequential data---the article considers applications of this method to videos, for example---the technique is not general enough to allow arbitrary commutation of group elements, being perhaps better suited to Abelian Lie groups. Cohen and Welling, once more, also address the problem of determining the exact representation operators from the orbit points using a Bayesian approach; however, besides being restricted only to Abelian Lie groups, they do not try to project the operators into the closest Lie algebra, meaning that no information about the exact representation type is retrieved \cite{cohen2014learning}. A more recent approach, now useful even for orbits of non-Abelian Lie groups, was introduced by \cite{pfau2020disentangling} in response to the problem raised by \cite{higgins2018definition} which, using methods from Riemannian geometry, estimates a local basis for orbits of direct products of Lie groups.

In \cite{dehmamy2021automatic}, on the other hand, the authors use a generalization of equivariant neural network architectures to Lie algebras convolutions to estimate derived representation operators, denoted $X$, in datasets for supervised tasks that are equivariant under some unknown Lie group $G$. Although similar to our work, this approach only gives the Lie algebra operators as vectors---i.e., there is no attempt to project the estimated $X$ onto the closest derived representation of $G$---and the output may also be an over-determined set of vectors, meaning that it may not form a basis for the Lie algebra of $G$. The over-determination issue is resolved by Cahill \emph{et al.} in \cite{DBLP:journals/corr/abs-2008-04278}, where the exact Lie algebra of the symmetry group of the orbit, denoted by $\sym(\O)$, seen as a vector space, together with its dimension, is estimated. Our algorithm is, in fact, an extension of this last work, where instead of only finding the basis for the space $\sym(\O)$, we project it onto the closest representation of the associated Lie group, i.e., we can determine exactly the representation type. We stress, however, that the application of the projection step to the result of \texttt{LiePCA} goes beyond a (partial) solution to the open problem left by the authors of ``rounding this estimate [result of \texttt{LiePCA}] to the nearest [Lie algebra]", but allows for many new applications not possible with the results of \cite{DBLP:journals/corr/abs-2008-04278} only. First, because the exact representation type is determined, calculating the Hausdorff distance between the estimated orbit and the sample point cloud becomes simpler, as we know exactly what values to exponentiate the derived representation basis. Second, in Section \ref{ex:sampling}, we empirically demonstrate that our full \texttt{LieDetect} algorithm generally outperforms \texttt{LiePCA} in the task that it was originally envisioned, that is, in density estimation on orbits of representations of Lie groups. Lastly, given a finite list of (compact) Lie groups as possible candidates of generators of some orbit, we can determine which of its entries is more likely to do so, as explained in Section \ref{sec: list}. Moreover, while the approach by Cahill \emph{et al.} permits an estimation of the symmetry group as a whole, our new approach to the \texttt{LiePCA} operator allows for estimating lower-dimensional representations, potentially non-transitive, as it will be further detailed in Section \ref{sec: implementation LiePCA} and illustrated in Example \ref{ex: Mobius} and Section \ref{subsec:conformational_space}. These improvements come, nonetheless, at the cost of generality since we are forced to work with compact groups only. However, once the information regarding the exact representation types is determined, it may be used in other inference and machine learning tasks, as we shall further explore through the examples of Section \ref{sec:applications}.

\paragraph{Technical aspects.}
In addition to applications, this article places particular emphasis on theoretical justification.
As far as geometry is concerned, we follow standard ideas of the literature regarding manifold and tangent space estimation, involving notions of geometric measure theory \cite{10.1214/18-AOS1685,boissonnat2019reach}.
In particular, we adopt a measure-theoretic point of view of the problem.
This is part of a movement in recent research to express stability results with the Wasserstein distance, instead of the traditional Hausdorff distance, less suited to the case of noisy datasets \cite{buet2017varifold,lim2021tangent,tinarrage2023recovering}.
On the algebraic side, at the core of our algorithm lies \texttt{LiePCA}, a recent method for estimating the Lie algebra of symmetry groups \cite{DBLP:journals/corr/abs-2008-04278}.
We analyze it by invoking traditional tools of representation theory.
Last, regarding implementations, we raise concrete matrix computation problems, such as simultaneous reduction \cite{Schnemann1968OnTO,umeyama1988eigendecomposition,nolte2006identifying,meinecke2012simultaneous} and local PCA \cite{tyagi2013tangent,kaslovsky2014non,10.1214/18-AOS1685,singer2012vector,DBLP:journals/corr/abs-2008-04278}.
Throughout this article, and although we do not enter into further details, we draw connections with problems of other areas of mathematics, such as varieties of Lie algebras \cite{friedland1983simultaneous,richardson1988conjugacy,lopatin2011orthogonal,baird2007cohomology,adem2007commuting,torres2008fundamental}, Linnik type problems \cite{linnik2ergodic,duke1988hyperbolic,duke2007introduction,schmidt1998distribution,horesh2023equidistribution}, and rigidity of Lie subalgebras \cite{agrachev2001lie,agrachev2010towards,crainic2014survey,barrionuevo2023deformations}.

\paragraph{Code availability.}
We developed a Python library, {\texttt{LieDetect}}, which fully implements the algorithm described in this article. This can be found at \url{https://github.com/HLovisiEnnes/LieDetect}.
It contains several \texttt{Jupyter} notebooks, implementing all the examples presented in the article, as well as our experiments from Section \ref{sec:applications}.

\paragraph{Outline.}
The rest of the article is as follows. 
In Section \ref{sec:preliminaries}, we give a self-contained introduction to the theory of representations of Lie groups.
We define the main algorithm of this article in Section \ref{sec:algorithm_description} and give additional comments regarding its implementation in Section \ref{sec:algorithm_additional}.
We prove consistency and stability results for the algorithm in Section \ref{sec:algorithm_analysis}.
In Section \ref{sec:applications} we apply the algorithm to several concrete data analysis problems.
Notations are listed in Appendix \ref{sec:notations}, and a few supplementary results are gathered in Appendix \ref{sec:supplementary_results}.

\section{Preliminaries}\label{sec:preliminaries}

We now lay the theoretical background needed to read this article.
The first two sections introduce basic notions regarding Lie groups and their representations, following \cite{brocker2013representations}.
The third section is dedicated to defining and analyzing the structure of the Stiefel and Grassmann varieties of Lie algebras, which will play a crucial role in our algorithm.

\subsection{Lie groups and Lie algebras}\label{subsec:lie_groups}

\subsubsection{Lie groups}\label{subsubsec:lie_groups_definition}
A \textit{Lie group} is a group $G$ that is also a smooth manifold, and such that the multiplication map $(g,h)\mapsto gh$ and the inverse map $g\mapsto g^{-1}$ are smooth.
An example is given by the general linear group $\GL_n(\R)$ of real invertible $n\times n$ matrices, and its complex equivalent $\GL_n(\C)$.
Endowed with the subspace topology inherited from the vector space of $n\times n$ matrices $\M_n(\R)$ and $\M_n(\C)$, both $\GL_n(\R)$ and $\GL_n(\C)$ are smooth submanifolds, and also are Lie groups.
They have real dimensions $n^2$ and $2n^2$, respectively, and are not compact.
In what follows, we will focus on compact Lie groups, and more particularly on the orthogonal group $\Ort(n)$, the special orthogonal group $\SO(n)$, the unitary group $\U(n)$, the special unitary group $\SU(n)$, and the torus $T^n$, respectively defined for any integer $n\geq1$ as:
\begin{align*}
    \Ort(n) &= \big\{ A \in \GL_n(\R) \mid A^\top = A^{-1}  \big\}, \\
    \SO(n) &= \big\{ A \in \GL_n(\R) \mid A^\top = A^{-1}, \det(A) = 1  \big\} ,\\
    \U(n) &= \big\{ A \in \GL_n(\C) \mid A^* = A^{-1}  \big\}, \\
    \SU(n) &= \big\{ A \in \GL_n(\C) \mid A^* = A^{-1}, \det(A) = 1  \big\}, \\
    T^n &= \SO(2) \times \cdots \times \SO(2) ~~~~\text{(}n\text{ copies)},
\end{align*}
where $A^\top$ denote the transpose and $A^*$ the conjugate transpose.
These groups have respective dimensions $n(n-1)/2$, $n(n-1)/2$, $n^2$, $n^2 -1$ and $n$, and are all connected, expect for $\Ort(n)$ which consists of two connected components.
Moreover, they are all \textit{matrix groups}, that is, they can be described as subgroups of a general linear group $\GL_n(\R)$ or $\GL_n(\C)$.
In general, it is a consequence of Peter–Weyl theorem that any compact Lie group is a matrix group. 

A \textit{Lie group homomorphism} is a smooth group homomorphism between two Lie groups. For instance, embedding $\M_n(\R)$ into $\M_n(\C)$ and $\M_n(\C)$ into $\M_{2n}(\R)$ yield Lie group homomorphisms $\SO(n) \hookrightarrow \SU(n) \hookrightarrow \SO(2n)$.
If a Lie group homomorphism is bijective with smooth inverse, it is called a \textit{Lie group isomorphism}, and the Lie groups are said isomorphic.
In the list above, all the groups are non-isomorphic, except for $\SO(1) \simeq \SU(1)$, both singletons, and $\SO(2) \simeq \U(1) \simeq T^1$.
Moreover, any compact Abelian Lie group is isomorphic to a torus.

It is worth mentioning that being a Lie group imposes strong regularities on the underlying manifold. 
For instance, they are all orientable.
Moreover, the fundamental group of a connected Lie group is discrete and Abelian.
As another example, the $n$-sphere $S^n$ admits a Lie group structure only when $n=0$, $1$ or $3$, these cases corresponding respectively to the singleton, $\SO(2)$ and $\SU(2)$.
Besides, being a Lie group implies the existence of a Lie algebra structure on the tangent spaces, as we review now.

\subsubsection{Lie algebras}
A \textit{Lie algebra} is a vector space $V$, together with an operation $[\cdot, \cdot]\colon V\times V\to V$, called \textit{Lie bracket}, that is bilinear, anti-symmetric, and obeys Jacobi's identity, that is, $[x,[y,z]]+[y,[z,x]]+[z,[x,y]]=0$ for all $x,y,z\in V$.
If $[A,B]=0$ for all $A,B\in V$, then the Lie algebra is said to be \textit{Abelian}.
By fixing a basis $(A_i)_{i=1}^m$ of the vector space $V$, the \textit{structure constants} of the Lie algebra relative to this basis are defined as the scalars $(c_{i,j}^k)_{i,j,k=1}^m$ satisfying $[A_i,A_j] = \sum_{k=1}^m c_{i,j}^k A_k$ for all $i,j \in [1\isep m]$.
As a particular case, the structure constants of an Abelian Lie algebra are all zero and hence do not depend on the basis.
As another example, one defines a Lie bracket on the space of $n\times n$ matrices $\M_n(\R)$ via the \textit{commutator} $[A,B] = AB - BA$.
This Lie algebra is non-Abelian whenever $n \geq 2$.

Given a Lie group $G$, its tangent space $\T_I G$ at the origin (identity matrix $I$) can be canonically endowed with a Lie bracket, yielding a Lie algebra, denoted $\g$. This construction is commonly based on left-invariant vector fields on $G$, which we do not describe here.
When $G\subset \M_n(\R)$ is a matrix group, a more direct definition exists: the tangent space $\T_I G$, seen as a linear subspace of $\M_n(\R)$, is endowed with Lie bracket given by the commutator of matrices.
By convention, we write in gothic font the Lie algebra.
For instance, the tangent spaces of the general linear group, the special orthogonal group, the special unitary group, and the torus are
\begin{align*}
    \gl(n) &= \M_n(\R),\\
    \so(n) &= \{A \in \M_n(\R)\mid A^\top = -A\} ,\\
    \su(n) &= \{A \in \M_n(\C)\mid A^* = -A, ~\mathrm{tr}(A) = 0 \}, \\
    \mathfrak{t}^n &= \so(2) \times \cdots \times \so(2) ~~~\text{(}n\text{ copies)},
\end{align*}
i.e., $\so(n)$ is the set of skew-symmetric matrices and $\su(n)$ of zero-trace skew-Hermitian matrices.

A \textit{Lie algebra homomorphism} between two Lie algebras $V$ and $W$ is a linear map $\phi\colon V \rightarrow W$ such that $[\phi(A),\phi(B)] = \phi([A,B])$ for all $A,B\in V$. If $\phi$ is bijective, it is called an \textit{isomorphism of Lie algebras}, and the Lie algebras are said isomorphic.
Equivalently, two Lie algebras are isomorphic if they admit the same structure constants in a certain pair of bases.
For instance, as shown in the example below, $\su(2)$ and $\so(3)$ are isomorphic, although $\SU(2)$ and $\SO(3)$ are not.
This shows that the Lie algebra does not totally determine a Lie group.
Despite this fact, it contains valuable information about the group, as we explain in the next section.

\begin{example}
\label{ex:isom_su(2)_so(3)}
A usual choice of basis for $\so(3)$, as presented in \cite{fecko2006differential}, is
\begin{equation*}
     X_1 = \begin{pmatrix}
            0 & 0 & 0\\
            0 & 0 & -1\\
            0 & 1 & 0
        \end{pmatrix},~~
        X_2 = \begin{pmatrix}
            0 & 0 & 1\\
            0 & 0 & 0 \\
            -1 & 0 & 0 
        \end{pmatrix},~~
       X_3= \begin{pmatrix}
            0 & -1 & 0\\
            1 & 0 & 0 \\
            0 & 0 & 0 
        \end{pmatrix}, 
\end{equation*}
for which we have the brackets $[X_1,X_2]=X_3$, $[X_2,X_3]=X_1$ and $[X_1,X_3]=-X_2$.
Besides, a basis of $\su(2)$ is given by the Pauli matrices, which yield the same Lie bracket:
\begin{equation*}
        Y_1=\frac{1}{2}\begin{pmatrix}
            0 & i\\
            i & 0
        \end{pmatrix},~~
        Y_2=\frac{1}{2}\begin{pmatrix}
            0 & -1\\1 & 0
        \end{pmatrix},~~
        Y_3 = \frac{1}{2}\begin{pmatrix}
            i & 0\\ 0 & -i
        \end{pmatrix}.
\end{equation*}
\end{example}

\begin{example}
\label{ex:normal_form_SO(2)}
It is known that any matrix in $\SO(2)$ can be written as $R(\theta)$ for some $\theta \in [0,2\pi)$, and any matrix in $\so(2)$ as $L(a)$ for some $a\in\R$, where 
\begin{equation*}
R(\theta) =
\begin{pmatrix}
    \cos \theta  & - \sin \theta\\
     \sin \theta& \cos \theta
\end{pmatrix}
~~~~~\mathrm{and}~~~~~
L(a) =
\begin{pmatrix}
     0 & -a\\
     a & 0
\end{pmatrix}.
\end{equation*}
Consequently, matrices of $T^d$ can be written as $\diag(R(\theta_1),\dots,R(\theta_d))$ for some $(\theta_1,\dots,\theta_d)\in[0,2\pi)^d$, and matrices of $\t^d$ as $\diag(L(a_1),\dots,L(a_d))$ for some $(a_1,\dots,a_d)\in\R^d$.
\end{example}

\subsubsection{The exponential map}\label{subsubsec:exponential_map_geometric}
Given a Lie group $G$ with Lie algebra $\g$, there exists a canonical smooth map $\exp\colon\g\rightarrow G$, called the \textit{exponential map}.
In the general construction, this exponential is defined using one-parameter subgroups generated by left-invariant vector fields.
In the particular case of matrix groups $G\subset \GL_n(\R)$, the exponential map is simply the restriction to $\g$ of the matrix exponential, $\exp\colon \M_n(\R)\rightarrow\GL_n(\R)$, given by $\exp(A) = \sum_{i=0}^{+\infty} A^n/n!$.
For $A,B\in \g$, it holds
\[
\exp(A+B) = \exp(A)\exp(B) \iff [A,B] = 0.
\]
Endowing $\g$ with the group structure of the addition, we deduce that $\exp\colon\g\rightarrow G$ is a group homomorphism only when $\g$ is an Abelian Lie algebra, that is, when $G$ is a torus, in the compact case. 
Moreover, if $G$ is connected and compact, or is the general linear group, then the exponential map is surjective.
For instance, elements of $\SO(2)$ take the form $R(\theta)$ for some $\theta \in \R$, which in turn is equal to the exponential of $L(\theta)$, defined in Example \ref{ex:normal_form_SO(2)}.

The exponential map is a fundamental tool in the study of Lie groups, allowing a dialogue between the group and its algebra.
To illustrate this phenomenon, one can show that any Lie group homomorphism $\phi\colon G\to H$ induces a unique Lie algebra homomorphism $\d\phi\colon\g \to \h$ with commutative diagram
\begin{equation}\label{eq:derived_homomorphism}
\begin{tikzcd}[column sep=2cm,row sep=1cm]
G  \arrow[r,"\phi"] & H \\
\g \arrow[u,"\exp"] \arrow[r,"\mathrm{d}\phi"] & \h \arrow[u,"\exp",swap]
\end{tikzcd}
\end{equation}
The map $\d\phi$ is called the \textit{derived homomorphism} of $\phi$, and $\d\phi(\g) \subset \h$ the \textit{pushforward Lie algebra}.
If $\phi$ is an isomorphism, or surjective with discrete kernel, then $\d\phi$ is an isomorphism.
For instance, the well-known double cover $\SU(2)\rightarrow\SO(3)$ yields an isomorphism $\su(2)\rightarrow\so(3)$.

In Riemannian geometry, there is a concurrent notion of an exponential map.
Without entering into the details, we mention that any compact Lie group admits a bi-invariant Riemannian metric. In this case, the Riemannian exponential coincides with the (Lie) exponential \cite{MILNOR1976293}.
On simple Lie groups, such as $\SO(n)$ with $n\neq 4$, such a metric is unique up to a constant.

The Frobenius inner product $\langle A, B \rangle = \Tr(A B^\top)$ on the Lie algebra $\M_n(\R)$ also defines a metric on $\GL_n(\R)$.
In the rest of this article, if $G$ is $\GL_n(\R)$ or one of its subgroups, it will be endowed with this metric, making it a Riemannian manifold.
This metric is not bi-invariant on $\GL_n(\R)$, however, it is when restricted to $\SO(n)$.
On $\SO(3)$, the corresponding geodesic distance is equal to $\d(P,Q) = \arccos((\langle P,Q\rangle-1)/2)$, which is also the absolute value of the angle of the rotation $PQ^\top$ \cite{moakher2002means}.
More generally, the geodesic distance on $\SO(n)$ is explicitly given by the quadratic mean of its angles.
Note that $\SO(n)$ admits other non-equivalent bi-invariant \textit{distances} (not metrics), such as the one induced by the Frobenius norm, called in this context the chordal distance.
We draw the reader's attention to the fact that, by endowing the Lie algebra and the group itself with metrics, we can treat the exponential $\exp\colon\g\rightarrow G$ as a map between metric spaces. 
This can be used to quantify explicitly some constructions.
Namely, if $r$ denotes the diameter of $G$, we will use in Section \ref{sec:step4_hasdorff} the fact that the exponential map is surjective when restricted to a ball of radius $r$ centered at the origin of $\g$.

A last point on Riemannian geometry is worth mentioning.
On a compact Lie group, there exists a unique left-invariant probability measure $\mu_G$, i.e., a measure such that $\mu_G(g A) = \mu_G(A)$ for any $g\in G$ and any Borel subset $A\subset G$. It is called the \emph{Haar measure}.
If $G$ is endowed with a bi-invariant Riemannian structure, then the Haar measure is equal to the Riemannian volume form.
This notion proves to be an invaluable tool, first for being a canonical probability measure on $G$, and second for enabling the technique of `averaging' on Lie groups.
In particular, we will use the Haar measure to `orthogonalize' representations in Section \ref{subsubsec:def_representation} and to define the uniform measure on the orbits in Section~\ref{subsec:structure_orbits}.

\subsubsection{Intrinsic and extrinsic symmetries}\label{subsec:symmetry_group}
Lie groups naturally appear when studying symmetries of objects. To illustrate this statement, let us consider a Riemannian manifold $\man$ and its \textit{isometry group} $\mathrm{Isom}(\man)$, defined as the set of diffeomorphisms $\man\rightarrow \man$ that preserves the metric.
The Myers–Steenrod theorem states that $\mathrm{Isom}(\man)$, when given the topology of pointwise convergence, forms a Lie group \cite{myers1939group}. 
Moreover, it is compact when $\man$ is.
For instance, the $n$-sphere endowed with the metric inherited from $\R^{n+1}$ satisfies $\mathrm{Isom}(S^n)\simeq\Ort(n+1)$.
For the flat torus $S^1\times S^1$ in $\R^4$, this group is a semi-direct product $\mathrm{Isom}(S^1\times S^1) \simeq (\SO(2)\times\SO(2))\rtimes D_4$, where $D_4$ is the dihedral group.
As a last example, the isometry group of $\SO(3)$, endowed with its bi-invariant Riemannian structure, takes the form $\SO(3)\rtimes\SO(3)\times\{\pm 1\}$ \cite{gaal2018certain}.

If $\man$ is isometrically embedded in $\R^n$, a closely related concept is that of \textit{symmetry group}, defined as 
$$\Sym(\man) = \{P \in \GL_n(\R) \mid P \man = \man\}.$$
It can be understood as the set of `extrinsic isometries' of $\man$.
It is a closed subgroup of $\GL_n(\R)$, hence a Lie group by Cartan's theorem \cite{cartan1952theorie}.
By restricting our action of the matrices $P$ to $\man$, we obtain a group homomorphism 
$$\Sym(\man) \rightarrow \mathrm{Isom}(\man).$$
We note that this map may not be injective, since certain matrices $P$ may act trivially on $\man$.
For instance, if $\man$ is embedded on the first $n$ coordinates of $\R^{n}\times \R^m$, then any transformation $P \in \Ort(n+m)$ that stabilizes $\R^{n}$ will be an element of $\Sym(\man)$, although being trivial when seen in $\mathrm{Isom}(\man)$. 

In this article, it will be convenient to work with the Lie algebra of $\Sym(\man)$, denoted $\sym(\man)$. 
It consists of the matrices of $\M_n(\R)$ that exponentiate to an element of $\Sym(\man)$. The following convenient formulation of this algebra is given in \cite{DBLP:journals/corr/abs-2008-04278}, for which we give a proof in Section \ref{sec:supplementary_preliminaries_symmetrygroup}:
\begin{equation}\label{eq:sym_algebra_formulation}
\sym(\man) = \big\{A \in \M_n(\R)\mid \forall x \in \man, A x \in \T_x \man\big\}.   
\end{equation}

In what follows, we will frequently refer to the following hypothesis: all the points of $\man$ are at an equal distance from the origin, and their linear span $\spn{\man}$ is equal to the whole ambient space $\R^n$.
In this case, a matrix $P\in\Sym(\man)$ must preserve the Euclidean norm, hence it is an orthogonal matrix.
That is, $\Sym(\man)\subset \Ort(n)$, and consequently $\sym(\man)\subset\so(n)$.

\subsection{Representations of Lie groups}

\subsubsection{Definitions}\label{subsubsec:def_representation}
A (real) \textit{representation} of a Lie group $G$ in a (real) vector space $V$ is a Lie group homomorphism $G\rightarrow\GL(V)$. 
When $V$ is a space of dimension $n$, the representation is said to be $n$-dimensional.
Similarly, a \textit{representation} of a Lie algebra $\g$ in $V$ is a Lie algebra homomorphism $\g\rightarrow\gl(V)$.
Let us suppose that $V = \R^n$, and consider a representation of $G$ in $\R^n$, denoted $\phi\colon G \rightarrow \GL_n(\R)$.
As we have seen in Equation \eqref{eq:derived_homomorphism}, the derived homomorphism $\d\phi$ is a representation of $\g$ in $\R^n$, fitting in the commutative diagram
\begin{equation}\label{eq:derived_homomorphism_representation}
\begin{tikzcd}[column sep=2cm,row sep=1cm]
G  \arrow[r,"\phi"] & \GL_n(\R) \\
\g \arrow[u,"\exp"] \arrow[r,"\mathrm{d}\phi"] & \gl_n(\R) \arrow[u,"\exp",swap]
\end{tikzcd}
\end{equation}
This is to say, there is a map
\begin{align*}
    \d\colon \mathrm{Rep}(G,\R^n) &\longrightarrow \mathrm{Rep}(\g,\R^n)\\
    \phi &\longmapsto \d\phi
\end{align*}
between the set of representations of $G$ in $\R^n$ and the set of representations of $\g$ in $\R^n$.
When the exponential map of $G$ is surjective, for instance when $G$ is connected and compact, the diagram above shows that the map $\d$ is injective.
Moreover, when $G$ is simply connected, $\d$ is also surjective, a phenomenon known as the Lie group–Lie algebra correspondence.
We point out, though, that among the groups of interest in this article, only $\SU(n)$ is simply connected.

A representation of a Lie group $G$ (resp.~of a Lie algebra $\g$) is said to be \textit{faithful} if the homomorphism $G\rightarrow\GL_n(\R)$ (resp.~$\g\rightarrow\gl_n(\R)$) is injective, and \textit{almost-faithful} if its kernel is discrete.
Clearly, a representation $G\rightarrow\GL_n(\R)$ is almost-faithful if and only if the derived representation $\g\rightarrow\gl_n(\R)$ is faithful.
In this case, the pushforward algebra $\d\phi(\g)$ is isomorphic to $\g$ via $\d\phi$.

Two representations $\phi$ and $\psi$ of the same group $G$ in $\R^n$ are said to be \textit{equivalent} if there exists $M\in \GL_n(\R)$ such that $\psi(g) = M\phi(g) M^{-1}$ for all $g\in G$, that is if they are a change of basis apart from each other.
This invites us to, given a representation, look for the `simplest' representation equivalent to it. To this end, we say that a representation is \textit{orthogonal} if it takes values in $\Ort(n)$.
As it turns out, when $G$ is compact, every representation is equivalent to an orthogonal one \cite[Th.~I.7]{brocker2013representations}. As we will explain in Section \ref{subsubsec:analysis_orthonormalization}, this result is proved via averaging with the Haar measure.
Moreover, if two orthogonal representations are equivalent, then they are also orthogonally equivalent, meaning that the matrix $M\in \GL_n(\R)$ can be chosen in $\Ort(n)$.
We will refer to the process of finding the orthogonal representation equivalent to some representation $\phi$ as the  `orthogonalization' of $\phi$.

\begin{example}
\label{ex:correspondance_SO(2)}
For $k\in\Z$, the application $\phi_k\colon \SO(2) \rightarrow \GL_2(\R)$ defined below is a representation of $\SO(2)$ in $\R^2$, whose derived representation $\d\phi_k \colon \so(2) \rightarrow \M_2(\R)$ is denoted by $L_k$:
\[
\phi_k(\theta) =
\begin{pmatrix}
    \cos k\theta  & - \sin k\theta\\
     \sin k\theta& \cos k\theta
\end{pmatrix}
~~~~~~~~\mathrm{and}~~~~~~~~
L_k(\theta) =
\begin{pmatrix}
     0 & -k\theta\\
     k\theta & 0
\end{pmatrix}.
\]
Besides, the representation
$\theta \mapsto
\left(\begin{smallmatrix}
     \theta & 0\\
     0 & 0
\end{smallmatrix}\right)$
of $\so(2)$ does not come from a representation of $\SO(2)$.
\end{example}

\begin{example}
For $\alpha > 1$, the following representation $\phi$ of $\SO(2)$ in $\R^2$ is not orthogonal, but $M\phi M^{-1}$ is equal to the orthogonal representation $\phi_1$ defined in Example \ref{ex:correspondance_SO(2)}, where 
\begin{equation*}
\phi(\theta) =
\begin{pmatrix}
    \cos \theta & -\alpha\sin \theta\\
    \alpha^{-1}\sin \theta & \cos \theta
\end{pmatrix}
~~~~~~\mathrm{and}~~~~~~~~
M = \begin{pmatrix}
    1 & 0\\
    0 & \alpha
\end{pmatrix}.~~~~~~~~~
\end{equation*}
\end{example}

\subsubsection{Irreducible representations}\label{subsubsec:irreps}
Let $\phi$ be a representation of a Lie group $G$ in some vector space $V$. 
It is called \textit{reducible} if there exists a proper linear subspace $W\subset V$ stabilized by $\phi$, that is, such that $\phi(g) W \subset W$ for all $g\in G$. Otherwise, $\phi$ is said to be \textit{irreducible}, and is called an \textit{irrep}.
Any orthogonal representation $\phi\colon G \rightarrow \R^n$ is \textit{completely decomposable}, meaning that there exists an integer $p$, a decomposition $\R^n = \bigoplus_{i=1}^p W_i$ into $p$ pairwise orthogonal subspaces, and a collection of irreducible representations $\{\phi_i\colon G\rightarrow W_i\}_{i=1}^p$ such that $\phi$ is isomorphic to the direct sum $\bigoplus_{i=1}^p\phi_i$.
In other words, by denoting $O\in\Ort(n)$ the change of basis into the decomposition given by the $W_i$'s, the representation $\phi$ can be written, in matrix notation, as
\[
\phi \colon g \mapsto O\diag\big(\phi_1(g),\dots,\phi_p(g)\big)O^\top.
\]
This integer $p$ is uniquely defined, as well as the irreps, up to permutation.

This result shows that one can obtain any representation of $G$ by simply summing irreducible representations. Having access to a list of these representations is thus a crucial problem in working with Lie groups.
Although real irreducible representations of simple real Lie algebras have been fully classified for a long time \cite{iwahori1959real}, it does not mean the representative matrices are explicitly known.
We give below this list of irreps for the groups of interest in this article.

\begin{example}
\label{ex:irreps_so(2)}
The irreducible representations of $\SO(2)$ consist of the trivial representation $\SO(2)\rightarrow\GL(1)$, denoted $\phi_0$, as well as the infinite family $\{\phi_k\mid k\in\Z\setminus\{0\}\}$, where $\phi_k\colon\theta\mapsto R(k\theta)$, and $R$ is the matrix defined in Example \ref{ex:normal_form_SO(2)}. 
\end{example}

\begin{example}
\label{ex:irreps_torus}
As a generalization of the $\SO(2)$ case, the irreducible representations of the torus $T^d$, apart from the trivial representation, form an infinite family of representations of dimension two, denoted $\phi_\omega$, and depending on a parameter $\omega \in\Z^d \setminus\{0\}$, called the \textit{weight}, where $\phi_{\omega}(\theta_1,\dots,\theta_d) = R\big(\sum_{i=1}^d\omega_i \theta_i\big)$, and $R$ is defined in Example \ref{ex:normal_form_SO(2)}.
\end{example}

\begin{example}
\label{ex:irreps_su(2)}
Because $\SU(2)$ is the double-cover of $\SO(3)$, all representations of the latter yield representations of the former.
Given $n\in\mathbb{N}$, both groups admit at most one irrep into $\R^n$, up to equivalence. 
Moreover, $\SO(3)$ admits an irrep when $n$ is odd, and $\SU(2)$ admits an irrep when $n$ is odd or $n \equiv 0 \pmod{4}$.
Explicit formulae are given in Appendix \ref{app: irreducible reps su(2)}.
\end{example}

\subsubsection{Structure of the orbits}\label{subsec:structure_orbits}
Consider a representation $\phi\colon G\rightarrow\GL_n(\R)$ and a point $x\in \R^n$.
The main object of study of this article is the \textit{orbit} of $x$ generated by $G$, defined as the set $\O = \{\phi(g)x\mid g \in G\}$.
In what follows, we will suppose that $x$ is not fixed by the action, otherwise the orbit would be trivial.
Most of what can be said about orbits of representation of Lie groups also holds in the more general setting of orbits of actions of Lie groups. 
Among their properties, we cite that orbits are submanifolds of $\R^n$, that they are homogeneous $G$-spaces, and that they satisfy the orbit-stabilizer theorem \cite{lee2013smooth}.
Moreover, the map $\phi$ induces a homomorphism $G \rightarrow \Sym(\O)$. 
At the level of Lie algebras, and as in Equation \eqref{eq:derived_homomorphism_representation}, we deduce:
$$\d\phi(\g) \subset \sym(\O).$$
This inclusion will have great importance since it indicates that $G$ can be `found' in $\sym(\O)$.

To illustrate the above, note that the group $\SO(n)$ acts on $\R^n$ by left-multiplication, and its orbits are $(n-1)$-spheres.
Similarly, $\SU(n)$ acts on $\R^{2n}$ by left-multiplication, generating $(2n-1)$-spheres.
Besides, the torus $T^d = \SO(2)\times\cdots\times\SO(2)$ acts on $\R^{2d} = \R^2\times\cdots\times\R^2$, and the orbit of a point $x$ is a $d$-torus, provided that $x$ is not zero in one of the planes $\R^2$.
It is worth mentioning that other representations yield a variety of orbits.
For instance, $\SO(n)$ acts on $\M_{n,d}(\R)$, the space of $n\times d$ matrices, by simultaneous left-multiplication of the $d$ columns. 
If $x\in\M_{n,d}(\R)$ has rank $k$, then its orbit is homeomorphic to $\V(k,\R^n)$, the Stiefel manifold of $k$-frames in $\R^n$.

As pointed out in Section \ref{subsubsec:exponential_map_geometric}, every compact Lie group admits a canonical probability measure $\mu_G$, called the Haar measure.
In this article, orbits $\O$ will always be endowed with their \emph{uniform measure}, denoted $\mu_\O$, and defined as the pushforward of $\mu_G$ via the map $g \mapsto \phi(g)x$, which does not depend on $x\in\O$.
One shows the following equivalent definition: with $l$ the dimension of $\O$, it holds that $\mu_\O$ admits a constant density over the $l$-dimensional Haudorff measure $\mathcal{H}^l$ of $\R^n$ restricted to $\O$, this constant being $\mathcal{H}^l(\O)^{-1}$, the inverse of its volume.

As formalized in Section \ref{sec:algorithm_description}, the aim of this article is to recover, given an orbit $\O\subset\R^n$, the representation $\phi\colon G\rightarrow\GL_n(\R)$ that generates it.
This problem is ill-posed, since several representations may generate the same $\O$.
With this limitation in mind, our method will aim at recovering \textit{any} one of those.
To make this issue more rigorous, we say that two representations $\phi,\phi'\colon G\rightarrow \GL_n(\R)$ are \textit{orbit-equivalent} if there exists a matrix $M \in\GL_n(\R)$ such that the pushforward Lie algebras $\d\phi(\g)$ and $M\d\phi'(\g)M^{-1}$, seen as linear subspaces of $\so(n)$, are equal.
In particular, their orbits are conjugate.
Note that when the representations are orthogonal, the matrix $M$ can be chosen in $\Ort(n)$.
We denote by $\mathrm{OrbRep}(G,n)$ the collection of orbit-equivalence classes of representations of $G$ in $\R^n$.
Naturally, equivalent representations are orbit-equivalent, but the converse is not true.
For instance, all the non-trivial irreps of $\SO(2)$ are orbit-equivalent, although they are not equivalent to each other.
More generally, for any surjective homomorphism $f\colon G\rightarrow G$, the composition $f\circ \phi$ is orbit-equivalent to $\phi$.
We will study the cases $\SO(2)$, $T^d$, $\SO(3)$ and $\SU(2)$ in Section \ref{sec:algorithm_additional}.

\subsubsection{Abstract harmonic analysis}
We now introduce a generalization of classical Fourier analysis to signals over arbitrary groups, called abstract harmonic analysis, which has also been used in recent years in the context of machine learning, as reviewed by \cite{chirikjian2000engineering}.
In opposition to the rest of this article, which only works with real representations, we will consider, in this paragraph and Section \ref{subsec:harmonic_analysis}, \textit{complex} representations.
In this context, we define a \textit{matrix coefficient} as a function $f:G\to \C$ of form
\begin{equation*}
f_{u,v}^V(g)= \langle g\cdot u, v \rangle,
\end{equation*}
where $V$ is a unitary representation of $G$ (i.e., which takes values in $\U(n)$), and $u,v$ any vectors in $V$. In particular, if $\{e_i\}_{i=1}^n$ is a basis for $V$, then the matrix coefficient $f^V_{e_i, e_j}(g)$ is equal to $\phi_{ij}(g)$, the $ij^\mathrm{th}$ coordinate of the matrix $\phi(g)$, hence the name. 
The core foundation of abstract harmonic analysis is established by the Peter-Weyl theorem, which allows for decomposing signals over compact Lie groups---or their orbits---using only a subset of matrix coefficients \cite{folland2016course, peter1927vollstandigkeit}.

To state the theorem, let us denote by $\widehat{G}$ the set of unitary irreducible representations of a Lie group $G$, up to equivalence.
Note that to every irrep $[\phi]\in \widehat{G}$, there is a corresponding vector space $V_{\phi}$.
The theorem states that $\widehat{G}$ is countable, and that $L^2(G)$, the set of square-integrable functions on $G\rightarrow\C$, has an orthonormal basis given by 
\begin{equation*}
\big\{\sqrt{\dim V_\phi}\cdot\phi_{ij}\mid[\phi]\in \widehat{G}, ~i,j \in [1\isep \dim V_\phi]\big\}.
\end{equation*}
In particular, if $f\in L^2(G)$, then we can decompose it as
\begin{equation}\label{eq: harmonic transform}
f(g) = \sum_{[\phi]\in\widehat{G}}\dim V_\phi \cdot \Tr\big(\widehat{f}(\phi)\pi(g)\big)
~~~~~\mathrm{where}~~~~~    
\widehat{f}(\phi)_{ij}=\int_G f(g) \phi_{ij}^* \d\mu_\G(g),
\end{equation}
in which we compute the integral over the Haar measure.
Although we do not discuss them here, several important results from ordinary harmonic analysis, such as the Parseval formula, the unitarity of the transform $f\mapsto \widehat{f}$, and the Convolution theorem can also be shown in this more general setting.

\begin{example}
\label{ex: fourier}
It is a well-known fact that the complex irreps of $\SO(2)$ have dimension $1$ and are given by $\widehat{\SO}(2)=\big\{e^{ni\theta}\mid\theta\in [0,2\pi), n\in \mathbb{Z}\big\}$. Therefore, applying Equation \eqref{eq: harmonic transform} yields
\begin{equation*}
f(\theta) = \sum_{n\in \mathbb{Z}} e^{in\theta }\widehat{f}(n) 
~~~~~\mathrm{where}~~~~~    
\widehat{f}(n)=\int_0^{2\pi}e^{-in\theta}f(\theta)\frac{\d\theta}{2\pi},
\end{equation*}
which is exactly the usual Fourier decomposition of signals on the circle.
This can be easily generalized to the torus $T^d$, for which $\widehat{T^d} = \big\{e^{i(n_1\theta_1 + \dots + n_d\theta_d)} \mid n\in \mathbb{Z}^d, \theta\in [0,2\pi)^d\big\}$, giving
\begin{equation*}
\begin{split}
f(\theta_1,\dots, \theta_d) &= \sum_{n_1, \dots, n_d\in \mathbb{Z}} e^{i(n_1\theta_1 + \dots + n_d\theta_d)}\widehat{f}(n_1,\dots, n_d) \\ 
\mathrm{where}~~~~~
\widehat{f}(n_1,\dots, n_d)&=\frac{1}{(2\pi)^d}\int_{[0,2\pi)^d}e^{-i(n_1\theta_1 + \dots + n_d\theta_d)}f(\theta_1,\dots,\theta_d) \d\theta_1\cdots\d\theta_d.
\end{split}
\end{equation*}
\end{example}

\subsection{Moduli spaces of Lie algebras}\label{subsec:grassmann}
\subsubsection{Stiefel and Grassmann varieties of Lie algebras}\label{subsec:grassmann_lie_notgroup}
To design an algorithm to work with Lie algebras, it will be convenient to understand the structure of the spaces they form. First, we will embed them into a matrix space. Let $n > 0$ be an integer, $\SO(n)$ the special orthogonal group, and $\mathfrak{so}(n)$ its Lie algebra, i.e., the set of $n\times n$ skew-symmetric matrices. 
Their dimension is $n(n-1)/2$.
We also endow $\mathfrak{so}(n)$ with the Frobenius inner product $\langle A, B\rangle = \sum_{i= 1}^n\sum_{j = 1}^n a_{i,j}b_{i,j}$, with norm denoted by $\|A\|$.
Given an integer $d\geq 1$, we define
\begin{itemize}
\item
$\G^\mathrm{Lie}(d, \mathfrak{so}(n))$, the \textit{Grassmann variety of $d$-dimensional Lie subalgebras of $\mathfrak{so}(n)$}, as the set of $d$-dimensional linear subspaces of $\so(n)$ that form a Lie algebra, i.e., are stable under Lie bracket.
It is embedded in $\M_{n(n-1)/2}(\R)$, the $n(n-1)/2 \times n(n-1)/2$ matrices, by converting a subspace $\mathcal{A} \subset \so(n)$ into the orthogonal projection on it, denoted $\projbracket{\mathcal{A}}\colon\so(n)\to\so(n)$.
\item
$\V^\mathrm{Lie}(d, \mathfrak{so}(n))$, the \textit{Stiefel variety of $d$-dimensional Lie subalgebras of $\mathfrak{so}(n)$}, as the set of orthonormal $d$-tuples of skew-symmetric matrices $(A_1,\dots,A_d)$ whose linear span is a Lie algebra.
It is naturally a subset of $d$-fold product $\mathfrak{so}(n)^d$, hence also of $\M_n(\R)^d$.
We will also use another embedding, into $\M_{n^2,d}(\R)$, the $n^2 \times d$ matrices, obtained by converting a $d$-frame $(A_1,\dots,A_d)$ onto the matrix whose $i^\mathrm{th}$ column is the flattening of the matrix $A_i$.
\end{itemize}

\noindent
It is worth mentioning that $\G^\mathrm{Lie}(d, \mathfrak{so}(n))$ and $\V^\mathrm{Lie}(d, \mathfrak{so}(n))$ are subsets of the more general Grassmann and Stiefel manifolds of $d$-dimensional subspaces of $\so(n)$, commonly denoted by $\G(d, \mathfrak{so}(n))$ and $\V(d, \mathfrak{so}(n))$, and where the condition of forming a Lie algebra is dropped.
These latter manifolds have dimension $(n(n-1)/2-d)d$ and $(n(n-1)+d(d+1))d/2$, respectively.

These spaces are linked by a surjective map $\pi\colon \V^\mathrm{Lie}(d, \mathfrak{so}(n)) \rightarrow \G^\mathrm{Lie}(d, \mathfrak{so}(n))$, obtained by sending a $d$-frame $(A_1,\dots,A_d)$ onto its span $\spn{A_1,\dots,A_d}$.
Besides, two useful actions can be defined on them. 
First, by embedding $\V^\mathrm{Lie}(d, \mathfrak{so}(n))$ as a subset of $\M_{n^2,d}(\R)$, we see that the orthogonal group $\Ort(d)$ acts on it by right-multiplication, and we denote the action by $(A_1 , \dots , A_d) \cdot O$.
It has the effect of rotating the frame into the space it spans.
Orbits of this action are exactly the fibers of $\pi$, allowing us to identify the quotient 
$$
\faktor{\V^\mathrm{Lie}(d, \mathfrak{so}(n))}{\Ort(d)} \simeq\G^\mathrm{Lie}(d, \mathfrak{so}(n)).
$$
The case $d=1$ is worth mentioning.
Since any 1-dimensional subspace of $\so(n)$ is a Lie subalgebra, $\V^\mathrm{Lie}(1, \mathfrak{so}(n))$ can be identified with the unit sphere of $\so(n)$, that is, $S^{n(n-1)/2-1}$.
The orthogonal group $O(1)$ has only two elements, and we see that $\G^\mathrm{Lie}(1, \mathfrak{so}(n))$ is nothing but the projective space $\R P^{n(n-1)/2-1}$.

Next, seeing $\V^\mathrm{Lie}(d, \mathfrak{so}(n))$ in $\M_{n}(\R)^d$, we have the action of $\Ort(n)$ by simultaneous conjugation: 
\begin{equation*}
O \cdot (A_1, \dots, A_d) = (O A_1 O^\top, \dots, O A_d O^\top).
\end{equation*}
Orbits of this action are the simultaneously similar tuples of matrices, whose structure is known to be intricate \cite{friedland1983simultaneous,richardson1988conjugacy,lopatin2011orthogonal}.
We note that the action of $\Ort(n)$ can also be defined on $\G^\mathrm{Lie}(d, \mathfrak{so}(n))$ by simultaneously conjugating all elements of a subspace.

Of particular interest to us is the double quotient $\V^\mathrm{Lie}(d, \mathfrak{so}(n))/\Ort(d)/\Ort(n)$, the $d$-subalgebras up to right-multiplication and conjugation.
As an illustration, let us describe this set when $d=1$.
Since the eigenvalues of a skew-symmetric matrix come as pairs of complex numbers $\pm i\lambda$, with $\lambda\in\R$, they can be ordered as $0\leq\lambda_1\leq\dots\leq \lambda_m$, with $m=\lfloor n/2\rfloor$ the greatest integer less than or equal to $n/2$.
Using that our matrices have unit norm, we see these tuples as elements of the `increasing quadrant' $H=\{(x_1,\dots,x_{m}) \in S^{m-1}\subset \R^{m}\mid 0\leq x_1\leq\dots\leq x_{n/2}\}$ of the $(m-1)$-sphere.
Besides, the action of $\Ort(n)$ by conjugation preserves the eigenvalues of matrices and acts transitively on matrices with the same eigenvalues.
In other words, the double quotient $\V^\mathrm{Lie}(1, \mathfrak{so}(n))/\Ort(1)/\Ort(n)$ can be identified with $H$.
In general, however, we do not know whether a simple description of it exists.
Instead, in the next section, we will modify the definition of $\V^\mathrm{Lie}(d, \mathfrak{so}(n))$, adapting it to a fixed Lie algebra, and show that its structure becomes tractable.
Namely, we will prove that the quotient is in correspondence with the orbit-equivalence classes of representations of the group.

Lastly, it will be convenient to give the Grassmann and Stiefel varieties a distance. In this article, we choose to work with the following ones.
Let $E$ be a vector space.
If $(A_1,\dots,A_d)$ and $(B_1,\dots,B_d)$ are two orthonormal families of $E$, we simply consider their Frobenius norm $(\sum_{i=1}^d\| A_i - B_i \|^2)^{1/2}$.
If $\mathcal{A}$ and $\mathcal{B}$ are the linear subspaces they span, we consider the distance 
\begin{equation}\label{eq:distance_grassmannian}
\big\|\projbracket{\mathcal{A}}-\projbracket{\mathcal{B}}\big\|,  
\end{equation}
where $\projbracket{\cdot}$ denotes the orthogonal projection on a subspace, seen as a $\dim(E) \times \dim(E)$ matrix.
The following equivalent formulations will turn out handy later in the article.
\begin{lemma}\label{lem:distance_grassmannian}
Given $(A_i)_{i=1}^d$ and $(B_i)_{i=1}^d$ in $\V^\mathrm{Lie}(d, \mathfrak{so}(n))$, with $\mathcal{A}$ and $\mathcal{B}$ their span, it holds
$$
\big\|\projbracket{\mathcal{A}}-\projbracket{\mathcal{B}}\big\|^2
= 2\sum_{i,j=1}^d\bigg( 1 - \langle A_i, B_j \rangle^2 \bigg)
= 2\sum_{i=1}^d\bigg( 1 - \big\|\projbracket{\mathcal{B}}(A_i)\big\|^2 \bigg).
$$
Moreover, when $d=1$, we have
$$
\big\|\projbracket{\mathcal{A}}-\projbracket{\mathcal{B}}\big\|^2
= \min\big\{2\big\| A_1 - B_1 \big\|^2,~2\big\| A_1 + B_1 \big\|^2 \big\}.
$$
\end{lemma}
Let us clarify that, in the equation above, the first Frobenius norm is that of $\M_{n(n-1)/2}(\R)$, while the second one is that of $\M_{n}(\R)$, even though they are similarly denoted.
We also point out that this lemma is not specific to $\so(n)$ and holds for any Euclidean vector space.

\subsubsection{Stiefel and Grassmann varieties of pushforward Lie algebras}\label{subsec:grassmann_lie}
In this article, we will not only need to work with Lie algebra structures, but more specifically, with those structures coming from representations of a compact Lie group. 
To this end, we fix a compact Lie group $G$ of dimension $d$, and introduce the following variations of $\G^\mathrm{Lie}(d, \mathfrak{so}(n))$ and $\V^\mathrm{Lie}(d, \mathfrak{so}(n))$:
\begin{itemize}
    \item $\G(G, \mathfrak{so}(n))$, the set consisting of those elements $\h \in \G^\mathrm{Lie}(d, \mathfrak{so}(n))$ for which there exists an orthogonal and almost-faithful representation $\phi\colon G \rightarrow \SO(n)$ such that $\h = \d\phi(\g)$;
    \item $\V(G, \mathfrak{so}(n))$, the set consisting of the orthonormal bases of elements in $\G(G, \mathfrak{so}(n))$.
\end{itemize}

To illustrate the above definition, and using the notations of Example \ref{ex:normal_form_SO(2)}, we can describe $\V(\SO(2), \mathfrak{so}(2m))$ as the set of unitary matrices of $\mathfrak{so}(2m)$ which, up to multiplication by a constant, are equivalent to $\diag(L(a_1),\dots,L(a_m))$ for some $(a_1,\dots,a_m)\in\mathbb{N}^m\setminus\{0\}$. 
In comparison, in the variety of Lie algebras $\G^\mathrm{Lie}(1, \mathfrak{so}(n))$, the tuple $(a_1,\dots,a_m)$ can be chosen arbitrarily in $\R^m\setminus\{0\}$.
This discrepancy is due to the fact that $\V^\mathrm{Lie}(1, \mathfrak{so}(n))$ contains representations of $\so(2)$ that may not come from a representation of $\SO(2)$.
This illustrates why $\V(G, \mathfrak{so}(n))$ is more suited to describing representations of the group.

It is worth stressing a few facts.
First, given two distinct Lie groups $G$ and $G'$, it is possible that their varieties intersect.
For instance, $\G(\SO(3), \mathfrak{so}(n))\subset\G(\SU(2), \mathfrak{so}(n))$ for all integer $n$, with equality when $n\leq 3$.
Besides, we note that $\G(G, \mathfrak{so}(n))$ can be empty, which happens precisely when $G$ does not admit almost-faithful representations in $\R^n$. 
This is the case for $T^d$ when $n<2d$.
Moreover, we point out that $\V(T^d, \mathfrak{so}(n))$ is made of $d$-frames of commuting skew-symmetric matrices. 
Taking their exponential yields a frame of commuting matrices of $\SO(n)$.
These objects have received certain attention because of their connection to quantum field theories \cite{baird2007cohomology,adem2007commuting,torres2008fundamental}.
Nevertheless, it is not possible to obtain all the commuting frames of $\SO(n)$ via this procedure.
Indeed, following the example of $\SO(2)$ in the previous paragraph, the exponential of $\V(\SO(2), \mathfrak{so}(2n))$ will only contain `rational' orthogonal matrices.

Just as before, $\Ort(d)$ acts by right multiplication on $\V(G, \mathfrak{so}(n))$, and the quotient set is 
$$
\faktor{\V(G, \mathfrak{so}(n))}{\Ort(d)} \simeq \G(G, \mathfrak{so}(n)).
$$
Besides, $\Ort(n)$ also acts on $\V(G, \mathfrak{so}(n))$ and $\G(G, \mathfrak{so}(n))$ by simultaneous conjugation, but now this action carries another interpretation: two Lie algebras are conjugate if and only if they are pushforwards of \textit{orbit-equivalent} representations of $G$, as defined in Section \ref{subsec:structure_orbits}.
Consequently, the quotient set $\G(G, \mathfrak{so}(n))/\Ort(n)$ is in correspondence with the orbit-equivalence classes of pushforward algebras of (orthogonal and almost-faithful) representations of $G$ in $\R^n$.

Again, the example of $\SO(2)$ is instructive: the matrices of $\V(\SO(2), \mathfrak{so}(2m))$ under the conjugation by $\Ort(n)$ are classified by their eigenvalues, which form increasing tuples $\lambda_1\leq\dots\leq\lambda_m$. 
Moreover, such tuples must be, up to multiplication by a constant, a tuple of integers in $\mathbb{N}^m\setminus\{0\}$.
We can understand this set as the `rational lines' on the $m$-torus $T^m$.
More precisely, $\V(\SO(2), \mathfrak{so}(2m))/\Ort(1)/\Ort(n)$ is in correspondence with the increasing primitive integral vectors of $\mathbb{N}^m$, a connection that will be studied further in Section \ref{subsubsec:orbit_equivalence_so(2)}.
In comparison, we have seen that the quotient $\V^\mathrm{Lie}(1, \mathfrak{so}(2m))/\Ort(1)/\Ort(n)$ is in correspondence with the increasing points $(x_1,\dots,x_m)$ of the $(m-1)$-sphere, a much larger space, which also contains the irrational lines.

More generally, to describe these spaces explicitly, we wish to fix a representative for each equivalence class of representation of $G$. To do so, we first fix a basis $(e_i)_{i=1}^d$ of $\g$.
We also consider $\mathrm{Irr}(G)$, an arbitrary choice of an orthogonal representative for each equivalence class of irreducible representation.
For all $\phi\in\mathrm{Irr}(G)$, the pushforward algebra $\d\phi(\g)$ admits the following basis:
$$
B^\phi=\big(\d\phi(e_1),\dots,\d\phi(e_d)\big).
$$
We apply Gram–Schmidt orthogonalization to obtain an element of $\V(G, \mathfrak{so}(n))$.
Next, we consider the $p$-tuples $(\phi_k)_{k=1}^p$ of elements of $\mathrm{Irr}(G)$, for some integer $p\geq 1$, whose direct sum $\phi = \bigoplus_{k=1}^p\phi_k$ has dimension $n$ and is almost-faithful.
We denote by $\mathrm{OrbRep}(G,n)$ the set consisting of an arbitrary choice of representative for each orbit-equivalence class of such representations $\phi$.
Last, we define $\mathfrak{orb}(G,n)$ as the set of their tuples of pushforward Lie algebras:
\begin{equation}\label{eq:orb}
\mathfrak{orb}(G,n) = \big\{ \big(B^{\phi_1},\dots,B^{\phi_p}\big) \mid (\phi_1,\dots,\phi_p) \in \mathrm{OrbRep}(G,n) \big\}.
\end{equation}
Based on $\mathfrak{orb}(G,n)$, all the operations previously mentioned can be made explicit:
\begin{itemize}
    \item \textbf{Direct sum:} If $\{\phi_k\}_{k=1}^p$ are irreducible representations with $\sum_{k=1}^p \dim(\phi_k)=n$, then their direct sum $\phi = \bigoplus_{k=1}^p\phi_k$ is a representation of $G$ in $\R^n$.
    A basis $B^\phi = (A^\phi_1,\dots,A^\phi_d)$ of $\d\phi(\g)$ can be obtained from the $B^{\phi_k}$'s by putting
    $A^\phi_i = \diag\big(B^{\phi_1}_i,\dots,B^{\phi_p}_i\big).$

    \item \textbf{Right-multiplication:} If $P = (P_{i,j})$ is a matrix in $\Ort(d)$, then the $d$-tuple 
    $
    B^\phi \cdot P
    = \Big(\sum_{j=1}^{d} P_{j,1} A^\phi_j,\dots, \sum_{j=1}^{d} P_{j,d} A^\phi_j\Big)
    $ is a new basis for $\d\phi(\g)$.

    \item \textbf{Conjugation:} If $O = (O_{i,j})$ is a matrix in $\Ort(n)$, then the $d$-tuple 
    $
    O \cdot B^\phi
    = \big(OA^\phi_1 O^\top, \dots,  OA^\phi_d O^\top\big)
    $ is a basis for the conjugate Lie algebra $O\d\phi(\g)O^\top$.
\end{itemize}
As a consequence of the aforementioned considerations, any pushforward Lie algebra of an orthogonal and almost-faithful representation of $G$ in $\R^n$ can be obtained from these operations. 

\begin{lemma}\label{lem:structure_lie_stiefel}
Let $G$ be a compact Lie group of dimension $d$ and $n\geq 1$ an integer.
For any orthogonal and almost-faithful representation $\phi\colon G\rightarrow \GL_n(\R)$, and any orthonormal basis $(A_1,\dots,A_d)$ of its pushforward Lie algebra $\d\phi(\g)$, there exists an integer $p\geq 1$, a $p$-tuple $(B^1,\dots,B^p)\in\mathfrak{orb}(G,n)$ and two matrices $O\in\Ort(n)$ and $P\in\Ort(d)$ such that, for all $i\in[1\isep d]$,
$$
A_i = \sum_{j=1}^{d} P_{j,i}O \diag\big(B^k_j\big)_{k=1}^p O^\top.
$$
In other words, we have $\V(G, \mathfrak{so}(n)) = \Ort(n) \cdot \mathfrak{orb}(G,n) \cdot \Ort(d)$.
\end{lemma}

This lemma gives an algorithmic procedure to generate all orthonormal bases of pushforward Lie algebras of $G$ in $\R^n$.
It will be put into practice in Section \ref{subsec:step3} to solve a minimization problem over $\V(G, \mathfrak{so}(n))$.
As a matter of fact, the core structure of our algorithm relies on the explicit description of the set $\mathfrak{orb}(G,n)$ for the groups under consideration. This will be obtained for $\SO(2)$, $T^d$, $\SU(2)$ and $\SO(3)$ in Equations \eqref{eq:orb_SO2}, \eqref{eq:orb_Td}, \eqref{eq:orb_SU2} and \eqref{eq:orb_SO3} respectively.
As a last comment, we draw the reader's attention to the fact that the formulation $\V(G, \mathfrak{so}(n)) = \Ort(n) \cdot \mathfrak{orb}(G,n) \cdot \Ort(d)$ has some degree of duplication. Indeed, given $(B^1,\cdots,B^p)\in\mathfrak{orb}(G,n)$, there may exist several $O\in\Ort(n)$, or $P\in\Ort(d)$, with the same action on them.

\section{Description of the algorithm}\label{sec:algorithm_description}

We now present our algorithm for detecting Lie group representations from orbits, which we call {\texttt{LieDetect}}. 
It takes as an input a finite point cloud $X\subset \R^n$ and a compact Lie group $G$, and returns a representation of $G$ in $\R^n$ for which $X$ is likely to lie on one of its orbits.
Three additional parameters must be provided: an integer $l$ and a real number $r>0$, to estimate tangent spaces, and a real number $\epsilon>0$, to perform dimension reduction.
However, when the group $G$ is Abelian, an additional integer parameter $\omega_\mathrm{max}$ must be given to limit the infinite set of non-equivalent representations of $G$ that are tested by the algorithm.
While our algorithm can be applied to any point cloud, this section is best understood by imagining that it lies on or near the orbit of a representation of $G$. Besides, this assumption is crucial to the theoretical results we will present in Section \ref{sec:algorithm_analysis}. 
Thus, throughout this section, we keep in mind the following model: 

\begin{center}\fbox{\begin{minipage}{0.95\linewidth}
\textbf{Model.}
Let $G$ be a known compact Lie group of dimension $d$, with an unknown representation $\phi$ on $\mathbb{R}^n$, potentially non-orthogonal.
The Lie algebra of $G$ is denoted by $\mathfrak{g}$ and its pushforward by $\mathfrak{h} = \d\phi(\mathfrak{g})$. Let $\mathcal{O}\subset \mathbb{R}^n$ be an orbit of this representation, $l$ its dimension, and $X = \{x_1,\dots,x_N\}$ a finite sample of points close to or included in $\mathcal{O}$. 
\end{minipage}}\end{center}
\noindent
The representation, pushforward Lie algebra, and orbit estimated by our algorithm will be denoted respectively by $\widehat{\phi}$, $\widehat{\h}$, and $\widehat{\O}_x$ for $x\in X$. A more unsupervised model, in which a list of compact Lie groups is given, rather than a single fixed instance, is explored in Section \ref{sec: list}.

\subsection{Overview}
Let us go over the big-picture justification behind the algorithm, whose details will be further explained in the next sections. 
It works in four steps: 

\begin{enumerate}[label=\textbf{Step \arabic*:},ref=Step \arabic*]
    \item\label{item:step1} \textbf{Orthonormalization} (Section \ref{subsec:step1})
    First, we normalize the orbit $\mathcal{O}$ to make $\phi$ an orthogonal representation.
    Following Section \ref{subsubsec:def_representation}, there exists a positive-definite matrix $M$ such that the translated orbit $M \mathcal{O}$ lies in the unit sphere.
    We find it as the square root of the Moore-Penrose pseudo-inverse of the covariance matrix:
    \begin{equation*}
    M = \sqrt{\Sigma[X]^+}
    ~~~~~\mathrm{where}~~~~~
    \Sigma[X] = \frac{1}{N}\sum_{i=1}^N x_ix_i^\top.
    \end{equation*}
    However, when $\Sigma[X]$ is singular or close to singular, this procedure can lead to large numerical errors. 
    This is the case when $\O$ does not span $\R^n$.
    We solve this issue with dimension reduction, via Principal Component Analysis (PCA): let $\proj_{\Sigma[X]}^{>\epsilon}$ be the projection matrix on the eigenspaces of eigenvalue greater than $\epsilon$, where $\epsilon$ is a parameter, we compute
    \begin{equation}\label{eq: program of projection2}
    \widetilde{X} = \sqrt{\Sigma[X]^+} \cdot \proj_{\Sigma[X]}^{>\epsilon} \cdot X.
    \end{equation}
    To simplify the notation for the next parts, we will write $X$ instead of $\widetilde{X}$.
        
    \item\label{item:step2} \texttt{LiePCA} (Section \ref{subsec:step2})
    The second step of our algorithm consists of estimating $\sym(\O)$, the Lie algebra of the symmetry group of $\O$ (introduced in Section \ref{subsec:symmetry_group}), using the \texttt{LiePCA} algorithm \cite{DBLP:journals/corr/abs-2008-04278}.
    In this context, we see $\sym(\O)$ as a linear subspace of the matrices $\M_{n}(\R)$. 
    \texttt{LiePCA} is based on the operator $\Lambda\colon \M_{n}(\R)\rightarrow \M_{n}(\R)$ defined as 
    \begin{equation}\label{eq:operator_sigma}
    \Lambda(A) = \sum_{1\leq i \leq N} \projbrackethat{\N_{x_i} X} \cdot A \cdot \projbracket{\spn{x_i}},        
    \end{equation}
    where the $\projbrackethat{\N_{x_i}X}$'s are estimations of the projection matrices on the normal spaces $\N_{x_i} \O$.
    In practice, we find them as $\Pi_{x_i}^{l,r}[X]$, the projection matrix on the top $n-l$ eigenvectors of the local covariance matrices.
    As we will prove in Section \ref{subsec:analysis_liepca}, the kernel of $\Lambda$ is approximately $\sym(\O)$.
    That is, matrices $A \in \sym(\O)$ can be estimated as those for which $\|\Lambda(A)\|/\|A\|$ is small.   
  
    \item\label{item:step3} \textbf{Closest Lie algebra} (Section \ref{subsec:step3})
    For this third step, we seek the pushforward Lie algebra $\h = \d\phi(\g)$, which is included in $\sym(\O)$.
    Consequently, an orthonormal basis $(A_1,\dots,A_d)$ of $\h$ must approximately belong to the kernel of $\Lambda$, thus we are invited to consider the program
    \begin{equation*}
        \arg \min \sum_{i=1}^d  \| \Lambda(A_i) \|^2
        ~~~~\mathrm{s.t.}~~~~
        (A_1,\dots,A_d) \in \mathcal{V}(\g, \mathfrak{so}(n)),
    \end{equation*}
    where $\mathcal{V}(G, \mathfrak{so}(n))$ is the Stiefel variety of Lie subalgebras of $\so(n)$ pushforward of $G$.
    Concretely, we implement this problem using the decomposition of Lemma \ref{lem:structure_lie_stiefel}: any element of $\mathcal{V}(G, \mathfrak{so}(n))$ can be obtained from a $p$-tuple of irreps of $G$, yielding
    \begin{equation}\label{eq:minimization_stiefel_2}
    \arg \min \sum_{i=1}^d \bigg\| \Lambda \bigg(O \diag\big(B^k_j\big)_{k=1}^p O^\top\bigg) \bigg\|^2
    ~~~~\mathrm{s.t.}~~~~ 
    \begin{cases}
      (B^1,\dots,B^p) \in \mathfrak{orb}(G,n),\\
      O \in \Ort(n).
    \end{cases}    
    \end{equation}
    This program consists in an optimization over $\Ort(n)$ for each $p$-tuple $(B^1,\dots,B^p)$.
    An explicit description of $\mathfrak{orb}(G,n)$ is necessary; that we give for $\SO(2)$, $T^d$, $\SU(2)$ and $\SO(3)$ in Section \ref{sec:algorithm_additional}.
    We denote $\widehat{\h}$ the subalgebra spanned by a minimizer of the program. 
    
    \hspace{-4.75em} \textbf{Step 3':}  
    We also propose a variation of \ref{item:step3}, valid when $\h = \sym(\O)$, e.g., when $G \simeq \sym(\O)$.
    Let $(A_1,\dots,A_d)$ be the bottom $d$ eigenvectors of $\Lambda$. They are close to $\sym(\O)$, but may not form a Lie algebra. 
    \ref{item:step3}' consists of projecting them to the closest Lie algebra.
    Using the distance on the Grassmannian defined in Equation \eqref{eq:distance_grassmannian}, we naturally consider
    \begin{equation*}
        \arg \min\big \| \projbracket{\spn{A_i}_{i=1}^d} - \projbracket{\widehat{\h}}  \big\| ~~~~\text{s.t. }~~~~ \widehat{\h}\in \mathcal{G}(G, \mathfrak{so}(n)),
    \end{equation*}
    where $\mathcal{G}(G, \mathfrak{so}(n))$ is the Grassmann variety of Lie subalgebras of $\so(n)$ that are derived from $G$, embedded in $\M_{n(n-1)/2}(\R)$.
    Using Lemma \ref{lem:structure_lie_stiefel}, this can be split into
    \begin{equation}\label{eq:minimization_grassmann}
        \arg \min \big\| \projbracket{\spn{A_i}_{i=1}^d} - \projbracket{\spn{O \diag(B^k_i)_{k=1}^p O^\top}_{i=1}^d}  \big\|
        ~~~~\mathrm{s.t.}~~~~
        \begin{cases}
      (B^k)_{k=1}^p \in \mathfrak{orb}(G,n),\\
          O \in \Ort(n).
        \end{cases}    
    \end{equation}
    This variation of \ref{item:step3} is closer to the original formulation of \texttt{LiePCA} but relies critically on the assumption $\h = \sym(\O)$, not generally satisfied.
    
    \item\label{item:step4} \textbf{Distance to orbit} (Section \ref{subsec:step4})
    Finally, we check whether the previously estimated Lie subalgebra $\widehat{\h}$ yields indeed an orbit close to the input $X$.
    To do so, we pick an arbitrary point $x\in X$, build its orbit $\widehat{\O}_x$, and compute the non-symmetric Hausdorff distance:
    \begin{equation}\label{eq:minimization_hausdorff}
        \HDmid{X}{\widehat{\O}_x}
        ~~~~~\mathrm{where}~~~~~
        \widehat{\O}_x = \big\{\exp(A) x \mid A \in \widehat{\h}\big\},
    \end{equation}
    This quantity can be approximated by sampling $\widehat{\h}$.
    In the case where $X$ contains anomalous points, the Hausdorff distance may be large, and measure-theoretical distances are better suited. 
    We build instead a measure $\mu_{\widehat{\O}}$ and compute the Wasserstein distance 
    \begin{equation}\label{eq:orbit_measure}
        \W_2\big(\mu_X,\mu_{\widehat{\O}}\big)
        ~~~~~\mathrm{where}~~~~~
        \mu_{\widehat{\O}} = \frac{1}{N}\sum_{i=1}^N \mu_{\widehat{\O}_{x_i}}, 
        ~~~~~~~~~~~~
    \end{equation}
    where $\mu_X$ is the empirical measure on $X$, and each $\mu_{\widehat{\O}_{x_i}}$ the uniform measure on $\widehat{\O}_{x_i}$.
    We stress that the orthonormalized data is considered here, since this allows for comparing the distances to reference values given in Section~\ref{subsubsec:statistics}, and to assess the algorithm's success.
    If one requires to return to the original point cloud, simply reverse \ref{item:step1}.
\end{enumerate}

The undetailed version of the algorithm is exposed in Algorithm \ref{alg: 1}.
In this section, we will focus on describing the precise implementation of each step.
The theoretical analysis of the algorithm will only be given later, in Section \ref{sec:algorithm_analysis}.
To ease the navigation in this article, we indicate in the following table the results that justify each step.
Moreover, throughout this section, we will illustrate the algorithm on a concrete dataset, presented in Example \ref{ex:running_ex}.

\begin{center}
\begin{tabular}{ ||c|c|c|c|c|| }\hline 
\ref{item:step1}&\ref{item:step2} & \ref{item:step3}~\&~\ref{item:step3}'&\ref{item:step4}&Algorithm \ref{alg: 1}\\\hline
Prop.~\ref{prop:stability_step1}&
Prop.~\ref{prop:Lie-PCA}&
Lemma~\ref{lem:stability_minimization_liealgebra}&
Prop.~\ref{prop:stability_step4}&
Th.~\ref{th:robustness_algorithm}\\\hline
\end{tabular}
\end{center}

\begin{breakablealgorithm}
\caption{{\texttt{LieDetect}}: Detection of orbits representations of compact Lie groups}\label{alg: 1}
\begin{algorithmic}
\Require A point cloud $X =\{x_i\}_{i=1}^N$ in $\R^n$, a compact Lie group $G$, integers $l$, $\omega_\mathrm{max}$ and $\epsilon,r>0$.
\Ensure A representation $\widehat{\phi}$ of $G$ in $\R^n$, with pushforward Lie algebra $\widehat{\h}$, and an orbit $\widehat{\O}_x$ of it or a measure $\mu_{\widehat{\O}}$ that most likely generates $X$.
\vspace{.1cm}

\State \textbf{\ref{item:step1}:} Compute the matrices of Equation \eqref{eq: program of projection2} and set $x_i \gets \sqrt{\Sigma[X]^+}\cdot\proj_{\Sigma[X]}^{>\epsilon}\cdot x_i$ for all $i$.
\vspace{.1cm}

\State \textbf{\ref{item:step2}:} Compute the operator $\Lambda$ defined in Equation \eqref{eq:operator_sigma} in the canonical basis of $\M_n(\R)$.
\vspace{.1cm}

\State \textbf{\ref{item:step3}:} Find a Lie algebra $\widehat{\h}$ that minimizes Equation \eqref{eq:minimization_stiefel_2}.
\vspace{.1cm}

\State \textbf{\ref{item:step3}':} Find a Lie algebra $\widehat{\h}$ that minimizes Equation \eqref{eq:minimization_grassmann}.
\vspace{.1cm}

\State \textbf{\ref{item:step4}:} Compute the quantity defined in Equations \eqref{eq:minimization_hausdorff} or \eqref{eq:orbit_measure}.
\end{algorithmic}
\end{breakablealgorithm}

\begin{example}
\label{ex:running_ex}
We consider a set $X\subset\R^4$ of $300$ points sampled close to the curve 
\begin{equation*}
\mathcal{O} = \big\{(\cos t, 2 \sin t, \cos 4t, \sin 4t) \mid t \in [0,2\pi)\big\}.
\end{equation*}
This is generated by sampling $t\in[0,2\pi)$ uniformly and then adding some small Gaussian noise (standard deviation $\sigma=0.01$). The data is visualized in Figure \ref{fig: sample so(2)}.
The curve $\mathcal{O}$ is the orbit of the point $(1,0,1,0)^\top$ for the representation $\phi \colon \SO(2) \rightarrow \M_4(\R)$ defined as
\begin{equation*}
t \mapsto
   \diag \bigg(\begin{pmatrix}
    \cos t & -2\sin t\\
    (1/2)\sin t & \cos t
\end{pmatrix}
, \begin{pmatrix}
    \cos 4t & -\sin 4t\\
    \sin 4t & \cos 4t
\end{pmatrix}\bigg).
\end{equation*}
This representation is not orthogonal.
However, by removing the two $2$'s, we see that it is equivalent to
$\phi_1 \oplus \phi_4$, that is, a sum of rotations with weights $1$ and $4$ (see Section \ref{subsubsec:irreps}).
The corresponding pushforward subalgebra $\h$ is the subspace of $\so(4)$ spanned by the matrix 
\begin{equation*}
A = \diag \bigg(\begin{pmatrix}
    0 & -1\\
    1 & 0
\end{pmatrix}
,\begin{pmatrix}
    0 & -4\\
    4 & 0
\end{pmatrix}\bigg)\bigg).
\end{equation*}
In this example, we expect our algorithm to output $\phi_1 \oplus \phi_4$.
\end{example}

\subsection{\ref*{item:step1}: Orthonormalization}\label{subsec:step1}
Let us start with \ref{item:step1}, whose objective is to project the dataset $X=\{x_i\}_{i=1}^N$ close to the unit sphere of $\R^n$, transforming the representation $\phi\colon G \rightarrow \M_n(\R)$ into an orthogonal one.
This pre-processing step significantly simplifies the rest of the algorithm by restricting the analysis to orthogonal representations only.
It is important to note that this `projection' cannot be simply understood as a metric projection on the unit sphere.
Indeed, generally, such an operation does not yield an orthogonal representation, as we show in Example \ref{ex:running_ex_step1}.
Instead, the transformation should be found using the following fact: there exists a positive-definite matrix $M$ such that the conjugate representation $M \phi M^{-1}$ takes values in $\Ort(n)$, that is, is orthogonal. 
Orbits for the latter representation are obtained by left translation by $M$ of orbits for the former.

Our construction is based on the \textit{covariance matrix} of $X$, defined as 
$$
\Sigma[X] = \frac{1}{N}\sum_{i=1}^N x_ix_i^\top,
$$
where $x_ix_i^\top$ represents the outer product of $x_i$ by itself.
It is a symmetric $n \times n$ matrix.
First of all, we will employ $\Sigma[X]$ to reduce the dimension.
More precisely, we use Principal Components Analysis (PCA): $X$ is projected on the subspace of $\R^n$ spanned by the eigenvectors of $\Sigma[X]$ of eigenvalue greater than $\epsilon$, where $\epsilon$ is a parameter of the algorithm.
The corresponding projection matrix is denoted by $\proj_{\Sigma[X]}^{>\epsilon}$.
The parameter $\epsilon$ is supposedly small, so as to project $X$ into $\spn{\O}$, the subspace of $\R^n$ spanned by $\O$.
We stress that this pre-processing step has two objectives: (1) it allows for avoiding numerical errors, when computing the pseudo-inverse of $\Sigma[X]$, as we will do below; and (2) by working in the subspace spanned by $\O$, as detailed in Section \ref{subsec:symmetry_group}, the homomorphism $\Sym(\O) \rightarrow \mathrm{Isom}(\O)$ becomes injective.
This property will be crucial in \ref{item:step3} for ensuring the detection of non-trivial actions of $G$ in $\O$.

Next, we compute the Moore–Penrose pseudo-inverse $\Sigma[X]^+$, and take its square root $\sqrt{\Sigma[X]^+}$.
Note that this operation is well-defined since a positive semi-definite matrix admits a unique square root.
We eventually define the product $M = \sqrt{\Sigma[X]^+}\cdot\proj_{\Sigma[X]}^{>\epsilon}$ and generate the close-to-the-sphere dataset $\{ M x_i\}_{i=1}^N$, which will be used in the following parts. 

\begin{example}
\label{ex:running_ex_step1}
We consider the set $X$ and the orbit $\mathcal{O}$ defined in Example \ref{ex:running_ex}.
Projecting $\mathcal{O}$ onto the sphere $S^3$ would yield the set
$$
 \bigg\{\frac{1}{\sqrt{2+3\sin^2t}}(\cos t, 2\sin t, \cos 4t, \sin 4t) \mid t \in [0,2\pi)\bigg\},
$$
which is not an orbit of an orthogonal representation of $\SO(2)$.
However, the minimization procedure of \ref{item:step1} yields a matrix and projected set approximately equal to
$$
M = \frac{1}{\sqrt{2}}\diag\big(1,1/2,1,1\big)
 ~~~~\text{and}~~~~
 \bigg\{\frac{1}{\sqrt{2}}(\cos t, \sin t, \cos 4t, \sin 4t) \mid t \in [0,2\pi)\bigg\},
$$
which is indeed an orbit of an orthogonal representation.
The non-orthogonalized and the orthogonalized point clouds can be visualized in Figure \ref{fig: sample so(2)}.

\begin{figure}[ht]\centering
\begin{minipage}{0.49\linewidth}\centering
\includegraphics[width=0.8\textwidth]{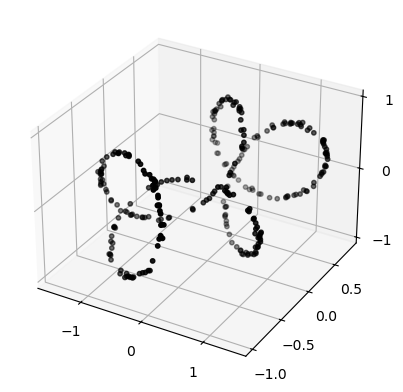}    
\end{minipage}
\begin{minipage}{0.49\linewidth}\centering
\includegraphics[width=0.8\textwidth]{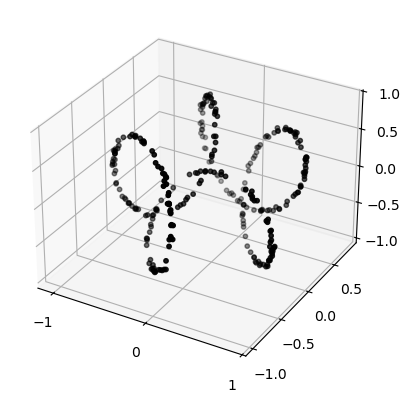}    
\end{minipage}
\caption{Illustrations of Example \ref{ex:running_ex_step1}. The input point cloud $X$, and its transformation after \ref{item:step1} (both projected in dimension $3$ through PCA for visualisation purposes).}
\label{fig: sample so(2)}
\end{figure}
\end{example}

\subsection{\ref*{item:step2}: \texttt{LiePCA}}\label{subsec:step2}
For Step 2, we consider the point cloud $X \subset \R^n$, having applied \ref{item:step1} or not, and want to find $\sym(\O)$. 
We remind the reader that the symmetry group $\Sym(\O)$ and its Lie algebra $\sym(\O)$ have been defined in Section \ref{subsec:symmetry_group}.
This step is based on the work developed in \cite{DBLP:journals/corr/abs-2008-04278}.

\subsubsection{The \texttt{LiePCA} operator}
This step is based on the fact that, for all $x\in \mathcal{O}$, $\sym(\O)$ is contained in the set 
$$
S_x\mathcal{O} = \{A\in \M_n(\R) \mid A x\in \mathrm{T}_x\mathcal{O}\},
$$
where $\mathrm{T}_x\mathcal{O}$ denotes the tangent space of the manifold $\mathcal{O}$ at $x$.
Indeed, as we have seen in Equation \eqref{eq:sym_algebra_formulation}, $\sym(\O)$ is equal to the intersection $\cap_{x\in\mathcal{O}}S_x \mathcal{O}$.
Now, having access only to the point cloud $X$, the authors of \cite{DBLP:journals/corr/abs-2008-04278} propose to estimate $\sym(\O)$ via 
$$
\bigcap_{i=1}^N S_{x_i} \mathcal{O} = \ker \bigg( \sum_{i=1}^N \projbracket{{(S_{x_i} \mathcal{O})^\bot}}\bigg),
$$
where the projections $\projbracket{(S_{x_i}\O)^\bot}$ are seen as operators on $\M_n(\R)$.
In practice, the $S_{x_i} \mathcal{O}$ are unknown, and a robust estimation is to be found. 
It is shown that $\proj_{(S_{x_i}\O)^\bot}$ is equal to 
$$A \mapsto \projbracket{\N_{x_i} \mathcal{O}} \cdot A \cdot \projbracket{\spn{x_i}},$$
where $\N_{x_i} \mathcal{O}$ is the normal space of $\O$ at $x_i$, $\spn{x_i}$ is the line spanned by $x_i$, and $\projbracket{\N_{x_i}\O}$ and $\projbracket{\spn{x_i}}$ denotes the projection matrices on these subspaces.
By denoting with $\projbrackethat{\N_{x_i} X}$ an estimation of $\proj_{\N_{x_i} \mathcal{O}}$ computed from the observation $X$, we replace the previous operator with
$$A \mapsto \projbrackethat{\N_{x_i} X} \cdot A\cdot \projbracket{\spn{x_i}}.$$ 
Further details on the estimation of normal spaces are provided below.
Combining all the aforementioned aspects, a reliable estimation of $\sym(\O)$ can be obtained by computing the kernel of the operator $\Lambda\colon \M_{n}(\R)\rightarrow \M_{n}(\R)$ defined as
\begin{equation*}
       \Lambda(A) = \sum_{i = 1}^N \projbrackethat{\N_{x_i} X}
    \cdot A \cdot \projbracket{\spn{x_i}}. 
\end{equation*}
More accurately, $\sym(\O)$ should be sought as the linear subspace of $\M_n(\R)$ spanned by the eigenvectors associated with `almost-zero' eigenvalues of $\Lambda$.
We stress that $\Lambda$ is a symmetric operator and, hence is diagonalizable when seen as an $n^2 \times n^2$ matrix.

\subsubsection{Estimation of normal spaces}\label{subsubsec:step2_estimation_normal_spaces}
We now turn to the estimations $\projbrackethat{\N_{x_i} X}$ of the projections onto normal spaces $\projbracket{\N_{x_i} \mathcal{O}}$.
First, we note that the problem of estimation of normal spaces is equivalent to the problem of estimating tangent spaces, since we have the relation 
$$
\projbracket{\mathrm{T}_{x_i} \mathcal{O}} + \projbracket{\N_{x_i} \mathcal{O}} = I
$$ 
with $I$ the identity matrix.
Indeed, $\mathrm{T}_{x_i} \mathcal{O}$ and $\N_{x_i} \mathcal{O}$ are complementary orthogonal subspaces. 
Now, estimating tangent spaces can be done through the common technique of \textit{local PCA}.
Given a real number $r>0$, called the scale parameter, this technique first computes the \emph{local covariance matrix} of $X$ at scale $r$ and centered at $x_i$, defined as
\begin{equation}\label{eq:local_cov_matrix}
\Sigma_{x_i}^r[X] = \frac{1}{|Y|} \sum_{y \in Y} (y-x_i) (y-x_i)^\top,    
\end{equation}
where $Y = \{y \in X \mid \|y-x_i\|\leq r\}$ is the set of input points at distance at most $r$ from $x_i$.
Now, given the dimension $l$ of $\O$, either known in advance or estimated, we estimate $\T_{x_i}\O$ via the space spanned by the $l$ bottom eigenvectors of $\Sigma_{x_i}^r[X]$. 
In what follows, we denote by $\Pi_{x_i}^{l,r}[X]$ the projection on this space.
We eventually consider the estimator 
$$
\projbrackethat{\N_{x_i} X} = I - \Pi_{x_i}^{l,r}[X].
$$

Due to the popularity of local PCA, probabilistic guarantees for this procedure can be found in several works, such as \cite{tyagi2013tangent,kaslovsky2014non,10.1214/18-AOS1685,lim2021tangent}.
Moreover, closely related variations have been studied, obtained by weighting the matrix $\Sigma_{x_i}^r[X]$, or by using $k$-neighbors \cite{singer2012vector,DBLP:journals/corr/abs-2008-04278}.
All of the works cited so far, however, state their results in probability, and not as deterministic inequalities, the kind we will need when analyzing our algorithm.
This will be remedied in Section \ref{subusbsec:stability_tangent_space_estimation}, which contains deterministic stability results for $\Pi_{x_i}^{r,l}[X]$.
In particular, based on \cite{tinarrage2023recovering}, we obtain in Lemma \ref{lem:consistency_loccovnorm} an explicit inequality as a function of the Wasserstein distance, which, as far as we know, is its first appearance in the literature.
We note, however, that somewhat comparable results have been obtained in recent years, such as in \cite{arias2017spectral,buet2017varifold,lim2021tangent}.

We stress that the accuracy of local PCA, in addition to its dependence on the geometry of $\O$ and the density of $X$, is highly affected by the parameter $r$, since too small values imply large variance, whereas bigger ones give biased estimations. 
We will derive in Lemmas \ref{lem:consistency_loccovnorm} and \ref{lem:wasserstein_stability_tangent_space} a theoretical range in which $r$ should be chosen, although, in practice, there is little \textit{ad hoc} guidance in this choice.

\subsubsection{Implementation}\label{sec: implementation LiePCA}
As presented in the original article, the last step of \texttt{LiePCA} consists of computing the linear subspace of $\M_n(\R)$ spanned by the $l$ bottom eigenvectors of $\Lambda$, offering an estimation of $\sym(\O)$.
This procedure, however, is not suited to our context, since we seek $\h$, which is only a subalgebra of $\sym(\O)$.
Again, this may be illustrated by the unit 3-sphere $S^{3}\subset\R^4$, which is an orbit of a representation of $\SU(2)$, but whose symmetry group is $\SO(4)$, of higher dimension. 
In this case, the inclusion of $\h = \su(2)$ in $\sym(S^3) = \so(4)$ is not an equality.
Consequently, instead of computing the eigenvectors of $\Lambda$, we simply compute the values it takes on the canonical basis of $\M_n(\R)$, and save them in a $n^2\times n^2$ matrix.
This will allow us, in \ref{item:step3}, to propose an estimation for $\h$.
Our implementation of \texttt{LiePCA} is given in Algorithm \ref{alg: 2} below.
In addition, we illustrate some applications of \texttt{LiePCA} in Example \ref{ex:lie_pca}.

\begin{breakablealgorithm}
\caption{\ref{item:step2} of Algorithm \ref{alg: 1} (\texttt{LiePCA})}\label{alg: 2}
\begin{algorithmic}[1]
\Require A point cloud $X = \{x_i\}_{i=1}^N$ in $\R^n$, a real number $r>0$ and an integer $l$.
\Ensure The \texttt{LiePCA} operator $\Lambda$ computed in the canonical basis of $\M_n(\R)$.
\For{$i=1,\dots, N$}
    \State let $\overline{\Sigma}_{x_i}^r[X]$ be the local covariance matrix at scale $r$, as in Equation \eqref{eq:local_cov_matrix}
    \State let $\Pi_{x_i}^{l,r}[X]$ be the projection on its $l$ bottom eigenvectors
    \State $\projbrackethat{\N_{x_i} X} \gets I - \Pi_{x_i}^{l,r}[X]$
\EndFor
\State $\Lambda \gets \bigg[\frac{1}{N}\sum_{i=1}^N \projbrackethat{\N_{x_i} X} \cdot A_{v,w} \cdot  \projbracket{\spn{x_i}}\bigg]_{1\leq i,j \leq n}$ for $A_{v,w}$ the $(v,w)$-th basis vector of $\M_n(\R)$
\end{algorithmic}
\end{breakablealgorithm}

\begin{example}\label{ex:lie_pca}
Figure \ref{fig: LiePCA} illustrates applications of \texttt{LiePCA} to point clouds sampled on a circle in $\mathbb{R}^2$, a flat torus in $\mathbb{R}^4$, and a sphere in $\mathbb{R}^3$.
They consist respectively of $100$, $500$ and $500$ points.
These spaces correspond to orbits of representations of the Lie groups $\mathrm{SO}(2)$, $T^2$, and $\mathrm{SO}(3)$, respectively, which are also their symmetry groups. 
The algorithm shows excellent performance, as evidenced by the fact that the operator $\Lambda$ exhibits a few significantly small eigenvalues, the number of which corresponds to the dimension of the corresponding Lie group.

\begin{figure}[ht]
\centering
\begin{minipage}{0.32\linewidth}\center
\includegraphics[width=0.8\textwidth]{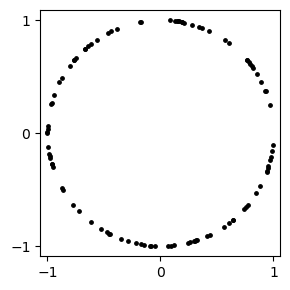}
\end{minipage}
\begin{minipage}{0.32\linewidth}\center
\includegraphics[width=0.95\textwidth]{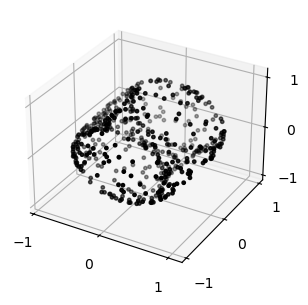}
\end{minipage}
\begin{minipage}{0.32\linewidth}\center
\includegraphics[width=0.95\textwidth]{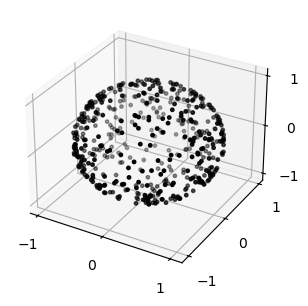}
\end{minipage}
\begin{minipage}{0.32\linewidth}\center
\includegraphics[width=0.99\textwidth]{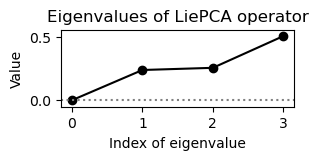}
\end{minipage}
\begin{minipage}{0.32\linewidth}\center
\includegraphics[width=0.99\textwidth]{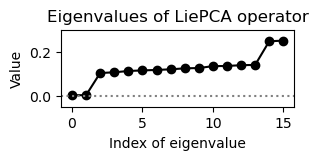}
\end{minipage}
\begin{minipage}{0.32\linewidth}\center
\includegraphics[width=0.99\textwidth]{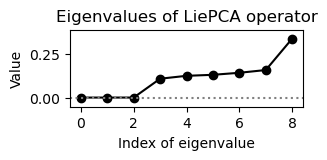}
\end{minipage}
\caption{Illustrations of Example \ref{ex:lie_pca}. Results of \texttt{LiePCA} to points sampled around the unit circle in $\R^2$ (\textbf{left}), the unit torus in $\R^4$ (\textbf{middle}, with dataset reduced to dimension 3 through PCA for visualization purposes), and the unit sphere in $\R^3$ (\textbf{right}). 
One can infer the respective Lie algebra dimensions to be 1, 2, and 3 based on the graphs of eigenvalues of the matrix $\Lambda$, shown at the bottom.}
\label{fig: LiePCA}
\end{figure}
\end{example}

\begin{example}
\label{ex:running_ex_step2}
We apply \ref{item:step2} to the point cloud $X$ presented in Example \ref{ex:running_ex}.
The \texttt{LiePCA} operator has the following eigenvalues, given within 3 decimal places: 
\begin{align*}
0.001, ~0.102, ~0.109, ~0.112, ~0.135, ~0.145, ~0.156, ~0.212, \\
0.212, ~0.233, ~0.236, ~0.247, ~0.249, ~0.259, ~0.296, ~0.296.
\end{align*}
One is significantly small, suggesting that the symmetry group has dimension $1$, as expected.
\end{example}

\subsection{\ref*{item:step3}: Closest Lie algebra}\label{subsec:step3}
For the third step, we seek a Lie subalgebra $\widehat{\h}$ of $\sym(\O)$ isomorphic to $\h$.
We propose two variations of this step: a general algorithm, and another one, valid only when $\h = \sym(\O)$. We stress that this step was left as an open problem by \cite{DBLP:journals/corr/abs-2008-04278}, and both solutions consist of our main algorithmic contribution at this stage.

\subsubsection{Formulation of the problem via \ref{item:step3}}\label{subsubsec:step3_formulation}
As mentioned previously, $\sym(\O)$ can be estimated as the kernel of the \texttt{LiePCA} operator $\Lambda\colon \M_{n}(\R)\rightarrow \M_{n}(\R)$, or more exactly, as the set of matrices $A\in\M_n(\R)$ for which $\|\Lambda(A)\|/\|A\|$ is small.
However, \texttt{LiePCA} does not involve any information regarding the commutators, i.e., it only allows to estimate $\sym(\O)$ as if it were a linear subspace.
To find $\h$, one has to force a Lie algebra structure.
This is done through $\mathcal{V}(G, \mathfrak{so}(n))$, the Stiefel variety of Lie subalgebras of $\so(n)$ that are pushforwards by almost-faithful representations of $G$ (see Section \ref{subsec:grassmann_lie}).
Its elements are the $d$-frames $(A_1,\dots,A_d)$ of $\so(n)$ spanning a subalgebra isomorphic to $\g$.
We are invited to consider the program
\begin{equation}\label{eq:minimization_stiefel}
    \arg \min \sum_{i=1}^d  \| \Lambda(A_i) \|^2
    ~~~~\mathrm{s.t.}~~~~
    (A_1,\dots,A_d) \in \mathcal{V}(G, \mathfrak{so}(n)).
\end{equation}
In implementing this problem, the structure of $\mathcal{V}(G, \mathfrak{so}(n))$ must be known.
Such an analysis has already been given in Section \ref{subsec:grassmann}, and we summarize it now.
The set $\mathfrak{orb}(G,n)$, defined in Equation \eqref{eq:orb}, consists of a choice of tuples $\big(B^{\phi_1},\dots,B^{\phi_p}\big)$, where $(\phi_1,\dots,\phi_p)$ is a collection of irreducible representations of $G$ such that their sum $\bigoplus_{k=1}^p\phi_k$ has dimension $n$ and is orthogonal and almost-faithful, and where $B^\phi$ is defined as $$B^\phi=\big(\d\phi(e_1),\dots,\d\phi(e_d)\big),$$
for a fixed basis $(e_i)_{i=1}^d$ of $\g$.
Moreover, the elements of $\mathfrak{orb}(G,n)$ have been chosen to generate non-orbit-equivalent representations (defined in Section \eqref{subsec:structure_orbits}).
Based on this initial set of tuples of irreducible representations, we have seen a few operations to generate new Lie algebras: their direct sum, the right-multiplication by a matrix $P\in\Ort(d)$, and the conjugation by a matrix $O\in\Ort(n)$.
Moreover, as stated in Lemma \ref{lem:structure_lie_stiefel}, any element $(A_1,\dots,A_d)$ of $\mathcal{V}(G, \mathfrak{so}(n))$ can be obtained via this process: there exists a collection $(B^{1},\dots,B^{p})\in \mathfrak{orb}(G,n)$ and two matrices $P\in\Ort(d)$ and $O \in \Ort(n)$ such that, for all $i\in [1\isep d]$, we have
$$
A_i = \sum_{j=1}^{d} P_{j,i}C_j
~~~~~\mathrm{where}~~~~~
C_j = O\diag\big( B^k_j \big)_{k=1}^p O^\top.
$$
By injecting this expression in Equation \eqref{eq:minimization_stiefel}, we observe a simplification:
\begin{align*}
\sum_{i=1}^d \bigg\| \Lambda \bigg(\sum_{j=1}^{d} P_{j,i} C_j\bigg) \bigg\|^2
&=\sum_{j=1}^{d}\big\| \Lambda  \big( C_j \big) \big\|^2.
\end{align*}
Indeed, the sum of squares of images of an orthonormal frame under any linear operator only depends on the space they span.
That is, the minimization problem is independent of the matrix $P$.
We deduce the equivalent formulation of Equation \eqref{eq:minimization_stiefel}, already stated in Equation~\eqref{eq:minimization_stiefel_2}:
\begin{equation*}
    \arg \min \sum_{j=1}^d \bigg\| \Lambda \bigg( O \diag\big(B^k_j\big)_{k=1}^p O^\top\bigg) \bigg\|^2
    ~~~~\mathrm{s.t.}~~~~ 
    \begin{cases}
      (B^1,\dots,B^p) \in \mathfrak{orb}(G,n),\\
      O \in \Ort(n).
    \end{cases}    
\end{equation*}

By definition of $\mathfrak{orb}(G,n)$, the tuples $(B^1,\dots,B^p)$ come from irreducible representations $(\phi_1,\dots,\phi_p)$, such that the corresponding Lie algebra is a $d$-dimensional subset of $\so(n)$.
Moreover, only one representative for each orbit-equivalence class is chosen.
Let us denote by $|\mathfrak{orb}(G,n)|$ its cardinality.
The program above naturally splits into $|\mathfrak{orb}(G,n)|$ minimization problems on $\Ort(n)$.
In practice, these minimizations can be performed by gradient descent, as described in \cite{absil2009optimization}. 
For the actual implementation, we used the Python package \texttt{Pymanopt} \cite{JMLR:v17:16-177}, which uses as a retraction map the so-called QR retraction \cite{absil2009optimization}.
Since the manifold $\Ort(n)$ is not connected, the minimization must be run twice: on $\SO(n)$ and on its complement. We then save the smallest cost.
Algorithm \ref{alg: 3} summarizes the steps necessary for solving this problem. 

It is worth noting that further complications arrive when the Lie group $G$ is Abelian: the set of non-orbit-equivalent representations of $G$ in $\R^n$ is infinite as soon as $n\geq 2d$.
To circumvent this issue, we simply consider a finite number of them.
In practice, we fix an integer $\omega_\mathrm{max}$, considered as a hyperparameter, and work with the irreducible representations whose weights are lower than $\omega_\mathrm{max}$ in absolute value.
In Section \ref{sec:algorithm_additional}, we will describe explicitly the set $\mathfrak{orb}(G,n)$ in the case of $\SO(2)$, $T^d$, $\SU(2)$ and $\SO(3)$, allowing us to put the algorithm in practice.

\begin{breakablealgorithm}
\caption{Pushforward Lie algebra of $G$ minimizing $\Lambda$ (\ref{item:step3} of Algorithm \ref{alg: 1})}\label{alg: 3}
\begin{algorithmic}[1]
\Require The \texttt{LiePCA} operator $\Lambda$ computed in the canonical basis of $\M_n(\R)$, a compact Lie group $G$ of dimension $d$, and a list $\mathfrak{orb}(G,n)$ of pushforward algebras of tuples of irreps.
\Ensure A basis $A^* = (A_1^*,\dots,A_d^*)$ formed by direct sum, right-multiplication, and conjugation of elements in $\mathfrak{orb}(G,n)$, that minimizes Equation \eqref{eq:minimization_stiefel_2}.
\vspace{.1cm}

\State Let L be the list of tuples $(A_1^\phi,\dots,A_d^\phi)$ obtained from the $(B^{\phi_1},\dots,B^{\phi_p})$'s in $\mathfrak{orb}(G,n)$ by putting $A_j^\phi = \diag\big(B^{\phi_1}_j,\dots,B^{\phi_p}_j\big)$, and where $\phi = \bigoplus_{k=1}^p\phi_k$.

\For{$(A_1^\phi,\dots,A_d^\phi)$ in L}
    \State solve
    $~~~\mathrm{cost(\phi)}\gets\arg \min \sum_{j=1}^d \big\| \Lambda \big( O A_j^\phi O^\top\big) \big\|^2~~~$
    constrained to $O\in \Ort(n)$. Denote the minimizer as $O_\phi$.
\EndFor
\State $\phi^* \gets$ minimizer of cost.
\State $A^* \gets (O_{\phi^*} A^{\phi^*}_1 O_{\phi^*}^\top,\dots, O_{\phi^*}A^{\phi^*}_d O_{\phi^*}^\top)$.
\end{algorithmic}
\end{breakablealgorithm}

\begin{example}
\label{ex:running_ex_step3}
Still considering the point cloud $X$ presented in Example \ref{ex:running_ex}, we apply \ref{item:step3} with $G = \SO(2)$, and with weights at most $\omega_\mathrm{max}=4$.
There are 6 classes of orbit-equivalent representations of $\SO(2)$ in $\R^4$ with such weights.
For each of them, we compute the minimum of Equation \eqref{eq:minimization_stiefel_2}, and write it in Table \ref{table:running_ex_step3}.
We check that the algorithm correctly points the representation $\phi_1\oplus \phi_4$ as the most likely to generate the orbit underlying the points. 

\begin{table}[ht]\center
\begin{tabular}{||r | c c c c c c ||} 
\hline
Weights &  $(0,1)$ & $(1,2)$ &  $(1,3)$ &$(1,4)$ & $(2,3)$&$(3,4)$ \\ 
\hline
Costs & $0.003$&$0.002$&$0.0003$ & \bm{$6.635\times 10^{-6}$}&$0.006$ & $0.009$ \\ 
\hline
\end{tabular}
\caption{Results of \ref{item:step3} in Example \ref{ex:running_ex_step3}. The best score is shown in bold.}
\label{table:running_ex_step3}
\end{table}
\end{example}

\subsubsection{Formulation of the problem via \ref{item:step3}'}\label{subsubsec:step3_variation}
As mentioned earlier, in the original article of \texttt{LiePCA}, the authors propose to estimate $\sym(\O)$ as $\spn{A_1,\dots,A_d}$, the linear subspace of $\M_n(\R)$ spanned by the $d$ bottom eigenvectors of $\Lambda$, where $d$ is a parameter, supposedly equal to the dimension of $\sym(\O)$.
In certain cases, such as those of Example \ref{ex:lie_pca}, $\sym(\O)$ has the dimension of $G$, hence $\spn{A_1,\dots,A_d}$ can serve as an estimator of $\h$.
In this paragraph, we will suppose that this is the case.
As estimated by the authors, this linear subspace may not form a Lie algebra, in the sense that it may not be closed under Lie bracket. 
Hence, a natural step would be projecting $\spn{A_1,\dots,A_d}$ onto the `closest Lie algebra'.
As it turns out, a variation of Algorithm \ref{alg: 3} gives a partial answer to this problem.
Namely, by denoting $\mathcal{G}^\text{Lie}(d, \mathfrak{so}(n))$ the set of $d$-dimensional Lie subalgebras of $\so(n)$ (introduced in Section \ref{subsec:grassmann_lie_notgroup}), we are compelled to solve the program 
\begin{equation*}
    \arg \min \big\| \projbracket{\spn{A_i}_{i=1}^d} - \projbracket{\h}  \big\|~~~~\text{s.t. }~~~~ \h \in \mathcal{G}^\text{Lie}(d, \mathfrak{so}(n)),
\end{equation*}
where we recall that, as is common in the literature, we compute the distance between spaces as the Frobenius distance between their projection matrices.
Here, we consider subspaces of $\so(n)$, hence the corresponding projection matrices have size $n(n-1)/2 \times n(n-1)/2$.

Unfortunately, due to the intricate nature of the space $\mathcal{G}^\text{Lie}(d, \mathfrak{so}(n))$, we were not able to find a way to minimize this problem directly. 
To circumvent this issue, we restrict $\mathcal{G}^\text{Lie}(d, \mathfrak{so}(n))$ to $\mathcal{G}(G, \mathfrak{so}(n))$, the set of Lie subalgebras of $\so(n)$ coming from an almost-faithful representation of some known $G$ in $\R^n$, defined in Section \ref{subsec:grassmann_lie}.
We are invited to consider the second problem 
\begin{equation*}
    \arg \min \big\| \projbracket{\spn{A_i}_{i=1}^d} - \projbracket{\h}  \big\| ~~~~\text{s.t. }~~~~ \h\in \mathcal{G}(G, \mathfrak{so}(n)).
\end{equation*}
which, this time, can be solved.
Indeed, as a consequence of Lemma \ref{lem:structure_lie_stiefel}, any element of $\h \in \mathcal{G}(G, \mathfrak{so}(n))$ can be obtained via a process similar to what has been described previously.
Explicitly, there exists a $p$-tuple $(B^1,\dots,B^p) \in \mathfrak{orb}(G,n)$ and a matrix $O\in\Ort(n)$ such that the $d$-tuple $(O \diag(B^{\phi_k}_1)_{k=1}^p O^\top,\dots,O \diag(B^{\phi_k}_d)_{k=1}^p O^\top)$ forms a basis of $\h$.
Consequently, the minimization problem above can be decomposed into the already mentioned Equation \eqref{eq:minimization_grassmann}:
\begin{equation*}
        \arg \min \big\| \projbracket{\spn{A_i}_{i=1}^d} - \projbracket{\spn{O \diag(B^k_i)_{k=1}^p O^\top}_{i=1}^d}  \big\|
        ~~~~\mathrm{s.t.}~~~~
        \begin{cases}
          (B^1,\dots,B^p) \in \mathfrak{orb}(G,n),\\
          O \in \Ort(n).
        \end{cases}    
\end{equation*}
This program naturally splits into $|\mathfrak{orb}(G,n)|$ minimization problems over $\Ort(n)$, just as it was the case for the minimization of \ref{item:step3}, given in Equation \eqref{eq:minimization_stiefel_2}.

We see that, formulated like this, the computational complexity of \ref{item:step3} and \ref{item:step3}' are comparable.
In practice, we observed that performing a minimization over $\Ort(n)$ for the latter problem is slightly faster, this being because the term in Equation \eqref{eq:minimization_grassmann} is less costly to calculate than in \eqref{eq:minimization_stiefel_2}.
Being slightly faster is, however, not the main interest of \ref{item:step3}'.
As we will develop in Sections \ref{sec: algorithm so(2) simp} and \ref{sec: algorithm torus simp}, Equation \eqref{eq:minimization_grassmann} admits a convenient formulation in the case where $G$ is the torus, enabling us to derive a more reliable and significantly faster algorithm.
We stress, nonetheless, that it comes at the cost of a theoretical limitation: when the Lie algebra $\h$ is a strict subset of $\sym(\O)$, that is, when $d<\dim(\sym(\O))$, then the $d$ bottom eigenvectors of $\Lambda$ may not cover $\h$, hence the minimizer of Equation \eqref{eq:minimization_grassmann} may not be close.

We provide a high-level implementation of this program in Algorithm \ref{alg: 3_variation}. 
As a technical detail, we point out that the bottom eigenvectors of the \texttt{LiePCA} operator may not be skew-symmetric matrices. This is corrected by the first two lines of the algorithm.
In practice, as for all optimization problems in this article, we use gradient descent, implemented in the Python package \texttt{Pymanopt} \cite{JMLR:v17:16-177}.
We emphasize that neither the convergence nor the landscape of this optimization problem are studied in this article. Our theoretical analysis of \ref{item:step3}, in Section \ref{subsec:stability-minimization}, assumes that a global minimum is exactly obtained. Nonetheless, in Section \ref{subsubsec:running_times}, we perform an empirical analysis of the algorithm, for a large collection of input data, and quantify its success rate.

\begin{breakablealgorithm}
\caption{Projection on closest pushforward Lie algebra of $G$ (\ref{item:step3}' of Algorithm \ref{alg: 1})}\label{alg: 3_variation}
\begin{algorithmic}[1]
\Require An orthonormal basis $\{A_i\}_{i=1}^d$ for the estimated pushforward algebra in \ref{item:step2},  a compact Lie group $G$ of dimension $d$, and a list $\mathfrak{orb}(G,n)$ of pushforward algebras of tuples of irreps.
\Ensure A basis $A^* = (A_1^*,\dots,A_d^*)$ formed by direct sum, right-multiplication, and conjugation of elements in $\mathfrak{orb}(G,n)$, that minimizes Equation \eqref{eq:minimization_grassmann}.
\State Make the basis $\{A_i\}_{i=1}^d$ skew-symmetric via $A_i \gets \frac{1}{2}(A_i-A_i^\top)$
\State $\{A_i\}_{i=1}^d\gets \text{Gram-Schmidt process applied to }\{A_i\}_{i=1}^d$;
\State Let L be the list of tuples $(A_1^\phi,\dots,A_d^\phi)$ obtained from the $(B^{\phi_1},\dots,B^{\phi_p})$'s in $\mathfrak{orb}(G,n)$ by putting $A_j^\phi = \diag\big(B^{\phi_1}_j,\dots,B^{\phi_p}_j\big)$, and where $\phi = \bigoplus_{k=1}^p\phi_k$.

\For{$(A_1^\phi,\dots,A_d^\phi)$ in L}
    \State solve
    $~~~\mathrm{cost(\phi)}\gets\arg \min \big\| \projbracket{\spn{A_i}_{i=1}^d} - \projbracket{\spn{O A_i^\phi O^\top}_{i=1}^d}  \big\|~~~$
    constrained to $O\in \Ort(n)$. Denote the minimizer as $O_\phi$.
\EndFor
\State $\phi^* \gets$ minimizer of cost.
\State $A^* \gets (O_{\phi^*} A^{\phi^*}_1 O_{\phi^*}^\top,\dots, O_{\phi^*}A^{\phi^*}_d O_{\phi^*}^\top)$.
\end{algorithmic}
\end{breakablealgorithm}

\begin{example}
\label{ex:running_ex_step3_variation}
We reproduce the same experiment as Example \ref{ex:running_ex_step3}, now with \ref{item:step3}' instead of \ref{item:step3}.
For each of the $6$ non-orbit-equivalent representations of $\SO(2)$ in $\R^4$ with weights at most $4$, we compute the minimum of Equation \eqref{eq:minimization_grassmann}, and write it in the table.
As expected, the representation $\phi_1\oplus \phi_4$ yields the minimal cost.

\begin{table}[ht]\center
\begin{tabular}{||r | c c c c c c ||} 
\hline
Weights &  $(0,1)$ & $(1,2)$ &  $(1,3)$ &$(1,4)$ & $(2,3)$&$(3,4)$ \\ 
\hline
Costs & $0.004$&$0.002$&$0.002$ & \bm{$4.29\times 10^{-5}$}&$0.006$ & $0.008$ \\ 
\hline
\end{tabular}
\caption{Results of \ref{item:step3}' in Example \ref{ex:running_ex_step3_variation}. The best score is shown in bold.}
\label{table:running_ex_step3_variation}
\end{table}
\end{example}

\subsection{\ref*{item:step4}: Distance to orbit}\label{subsec:step4}
In the previous step, we have calculated the representation $\widehat{\phi}\colon G \rightarrow \SO(n)$ whose pushforward Lie algebra $\widehat{\h}$ is closest to creating an orbit underlying the point cloud $X$, with closeness being computed at the level of Lie algebras.
We now close our algorithm with a verification step, by checking whether $\widehat{\phi}$ does indeed generate an orbit containing $X$. 
To this end, we define the estimated orbit as
$$\widehat{\O}_x 
= \widehat{\phi}(G) \cdot x
= \big\{\exp(A) x \mid A \in \widehat{\h}\big\},$$
where $\exp$ denotes the matrix exponential, and where $x$ is taken as any point of $X$.
We aim to assess whether $X$ is close to $\widehat{\O}_x$. We propose two alternatives for computing relevant distances between them: either using the Hausdorff or the Wasserstein distance, both popular in literature.

We stress that this step is more than just a safeguard of the algorithm's performance: if the task consists of determining which compact Lie group $G$ from a finite list of Lie groups $\{G_1,\dots, G_k\}$ is the most likely to generate a particular orbit, we tackle it by comparing the distances of the point cloud $X$ to the orbits, $\widehat{\O}_{x,i}$, estimated from each of the resulting representations of the previous steps, $\widehat{\phi}_i$, when applied to every Lie group on the list, $G_i$. This point is discussed further in Section~\ref{sec: list}.

\subsubsection{Hausdorff distance}\label{sec:step4_hasdorff}
To quantify the proximity between the input and output of our algorithm, we will use the \textit{non-symmetric Hausdorff distance} $\HDmid{X}{\widehat{\O}_x}$, defined as
$$
\HDmid{A}{B} = \adjustlimits \sup_{a \in A} \inf_{b \in B} \| a - b\|.
$$
Note that we recover the (usual) Hausdorff distance as 
$$\HD{A}{B} = \max\left\{\HDmid{A}{B},\HDmid{B}{A}\right\}.$$
We stress that the (usual) Hausdorff distance is not the correct notion of proximity here: the point cloud $X$ is thought as only a subset of the orbit, hence $\HD{X}{\widehat{\O}_x}$ would also reflect `gaps' in the sampling.
It is more appropriate to consider the non-symmetric distance $\HDmid{X}{\widehat{\O}_x}$, which captures the extent to which the first is included in the other. In the case where $X$ lives exactly on the estimated orbit, this distance is zero, even for a finite point cloud.

In practice, the distance between the subsets $\widehat{\mathcal{O}}_x$ and $X=\{x_i\}_{i=1}^N$ cannot be calculated directly, since $\widehat{\mathcal{O}}_x$ is a continuous manifold.
This can be remedied through an approximation by sampling a large but finite set of points on $\widehat{\mathcal{O}}_x$ and computing the distance between the finite sets.  
In this regard, efficient methods for computing the distance are described in \cite{taha2015efficient}, and implemented in the package \texttt{SciPy} in Python, which we use.

To generate a finite sample of $\widehat{\mathcal{O}}_x$, we use the surjectivity of the exponential map $\exp\colon \g \rightarrow G$. As presented in Section \ref{subsubsec:lie_groups_definition}, there is a commutative diagram
\begin{equation*}
\begin{tikzcd}[column sep=2cm,row sep=1cm]
G  \arrow[r,"\widehat{\phi}"] & \SO(n) \\
\g \arrow[u,"\exp"] \arrow[r,"\mathrm{d}\widehat{\phi}"] & \h \arrow[u,"\exp",swap]
\end{tikzcd}
\end{equation*}
Moreover, let $\mathrm{diam}(G)$ denote the diameter of $G$, when endowed with a bi-invariant Riemannian structure, as discussed in Section \ref{subsubsec:exponential_map_geometric}.
By definition of the diameter, the image of a ball of radius $\mathrm{diam}(G)$ around the origin of $\g$ covers the whole group $G$.
If $\lambda$ is a Lispchitz constant for $\widehat{\phi}$, we deduce that the matrix exponential of a ball of radius $\lambda\times\mathrm{diam}(G)$ around the origin of $\h$ will also cover the whole image $\phi(G)$. 
In other words, we have a parametrization
$$
\widehat{\O}_x = \big\{\exp (A) x \mid A \in \h, \|A\|\leq \lambda\times\mathrm{diam}(G)\big\}.
$$
Because the Lie algebra $\h$ is a $d$-dimensional vector subspace of $\so(n)$, we can easily sample points on it and deduce a sample of $\widehat{\O}_x$.
More precisely, we consider a basis $\{A_i\}_{i=1}^d$ of $\h$, choose a large and regularly-spaced finite subset of the unit ball $B^{d}$ of $\R^d$, and build
\begin{equation}\label{eq:orbit_approximation_sample}
\bigg\{ \exp \bigg( \lambda\times\mathrm{diam}(G) \sum_{i=1}^d  t_i A_i \bigg) \cdot x \mid (t_1\dots,t_d)\in B^d \bigg\} \subset \widehat{\O}_x.
\end{equation}
We illustrate this procedure in Example \ref{ex:running_ex_step4}.

Let us now suppose that the Hausdorff distance $\HDmid{X}{ \widehat{\mathcal{O}}_x}$ has been computed.
Because a `close enough' distance grows with dimension, it is not trivial to establish fixed thresholds to separate what are data close enough to an orbit to be considered a sample from it. 
In practice, one may vary the number of samples used in the estimation of the orbit and observe whether or not the distance converges to a small error value.
As another answer to this problem, we could compare $\HDmid{X}{ \widehat{\mathcal{O}}_x}$ to the theoretical upper bound given in Theorem \ref{th:robustness_algorithm}. 
This bound, however, depends on certain unknown quantities that would need to be estimated.
A last solution, for which we will give some hints in Section \ref{subsubsec:statistics}, is to compute empirically what a `small distance' should be, in the manner of a table of quantiles in statistical inference.
In this regard, we stress that the Hausdorff distance is not calculated from the point cloud $X$ given as an input to our algorithm, but from its orthonormalized version, after \ref{item:step1}.
This point is crucial in assessing the quality {\texttt{LieDetect}}'s output. Indeed, by considering only orthonormalized point clouds, we ensure that the results are comparable from one experiment to another.

\begin{remark}
In practice, if the point cloud $X$ is noisy, there is a risk that the initial arbitrary point $x$ will be an outlier, resulting in an orbit $\widehat{\mathcal{O}}_x$ that badly fits the data.
A simple way to circumvent this problem is to test the results for the whole collection of points $x\in X$, and choose the one that yields the smallest distance $\HDmid{X}{ \widehat{\mathcal{O}}_x}$.
This will be put into practice in Section \ref{subsec:conformational_space}, for a dataset representing chemical configurations, and which has a certain degree of noise.
We stress that cleverer ways could be developed, but we have not taken this question further.
Instead, we propose below a measure-theoretic point of view, which helps to process point clouds with a high level of noise.
\end{remark}

\subsubsection{Wasserstein distance}\label{subsubsec:step4_wasserstein}
Comparing the input data $X$ and the output orbit $\widehat{\O}_x$ can be realized, as an alternative, using the popular notion of distance offered by optimal transport.
We recall that, given two probability measures $\mu$ and $\nu$ over $\R^n$, a \textit{transport plan} between $\mu$ and $\nu$ is a probability measure $\pi$ over $\R^n \times \R^n$ whose marginals are $\mu$ and $\nu$. 
For any real number $p\geq 1$, the $p$-\textit{Wasserstein distance} between $\mu$ and $\nu$ is defined as 
\[\W_p(\mu,\nu) = \left(\inf_\pi \int \|x-y\|^p \d\pi(x,y) \right)^\frac{1}{p},\]
where the infimum is taken over all the transport plans $\pi$.
We denote by \textit{optimal} transport plan the $\pi$ that attains this infimum.
Notice that if $p\leq q$, then $\W_p(\mu,\nu) \leq \W_q(\mu,\nu)$ by Jensen's inequality.
In this article, we choose to work with the parameter $p=2$, which turned out to be the most convenient for deriving theoretical results in Section \ref{sec:algorithm_analysis}.

In this framework, we must associate probability measures to the subsets $X$ and $\widehat{\O}_x$. We choose for the former the \textit{empirical measure}, defined as $\mu_X = \frac{1}{N}\sum_{i=1}^N \delta_{x_i}$, where $\delta$ denotes the Dirac mass.
Regarding the latter, which is an orbit of $G$, a natural choice is the \textit{uniform measure} $\mu_{\widehat{\O}_x}$.
As defined in Section \ref{subsec:structure_orbits}, it is the pushforward of the Haar measure on $G$:
$$
\mu_{\widehat{\O}_x} = \big(\widehat{\phi}\cdot x\big)\mu_G.
$$
In this context, the Wasserstein distance $\W_2\big(\mu_X, \mu_{\widehat{\O}_x} \big)$ is a relevant measure of proximity, computing how close the supports $X$ and $\widehat{\O}_x$ are and how uniform the distribution of $X$ is.

However, in the context of data analysis with the Wasserstein distance, it is common to allow $X$ to have a few anomalous points, that is, points far away from the underlying orbit $\O$.
If $x$ is such a point, the estimated orbit $\widehat{\O}_x$ would not be a reliable approximation of $\O$.
We circumvent this issue by considering all the points of $X$ and defining the \textit{average measure}
$$
\mu_{\widehat{\O}} = \frac{1}{N}\sum_{i=1}^N \mu_{\widehat{\O}_{x_i}}.
$$
Eventually, the last step of our algorithm consists of computing $\W_2\big(\mu_X, \mu_{\widehat{\O}} \big)$.

In practice, we face a problem similar to that of the Hausdorff distance: the support of $\mu_{\widehat{\O}}$ is infinite, and we are not aware of algorithms that would compute it exactly.
Instead, we propose to use the approximation of Equation \eqref{eq:orbit_approximation_sample} to compute a finite sample on each orbit $\widehat{\O}_x$.
We then gather all the samples in a set $\widehat{\O}'$ and consider its empirical measure $\mu_{\widehat{\O}'}$.
The Wasserstein distance $\W_2\big(\mu_X, \mu_{\widehat{\O}'} \big)$, only involving discrete measures, can be computed using the package \texttt{POT} in Python \cite{flamary2021pot}, based on efficient implementations of the Sinkhorn algorithm, such as \cite{NIPS2013_af21d0c9}.
This is shown in Examples \ref{ex:running_ex_step4} and \ref{ex:running_ex_step4_noise}.

We stress that, compared to the Hausdorff distance, the Wasserstein distance has the advantage of being robust to the presence of outliers in $X$. 
Indeed, from its definition, we see that adding a few anomalous points to $X$ cannot change too much the value of $\W_2\big(\mu_X, \mu_{\O} \big)$, hence implying robustness against outliers.
In Theorem \ref{th:robustness_algorithm}, we will show that $\W_2\big(\mu_X, \mu_{\widehat{\O}} \big)$ is small as long as the initial distance $\W_2\big(\mu_X, \mu_{\O} \big)$ is, proving the robustness of our algorithm against anomalous points.
This property is illustrated in Example \ref{ex:running_ex_step4_noise}, where we successfully apply the algorithm to a dataset corrupted with noise. 
On the downside, we draw the reader's attention to the fact that, in this framework, the point cloud $X$ is seen as a sample of the \textit{uniform} distribution on $\O$ or some close distribution:
it does not include the more general assumption of non-uniform distributions.
In contrast, the Hausdorff distance is only sensitive to the support and may be small even when non-uniformly distributed.

\begin{example}
\label{ex:running_ex_step4}
We apply \ref{item:step4} on $X$, based on the optimal Lie algebra returned by \ref{item:step3} in Example \ref{ex:running_ex_step3}. 
We select an arbitrary $x\in X$, and approximate the Hausdorff distance of Equation \eqref{eq:minimization_hausdorff} by sampling $500$ points on the estimated orbit $\widehat{\O}_x$, yielding $\HDmid{X}{\widehat{\O}_x} \approx 0.018$.
As observed in Figure \ref{fig:step4_orbit}, $\widehat{\O}_x$ appears to fit $X$ correctly.
Besides, to approximate the Wasserstein distance of Equation \eqref{eq:orbit_measure}, we sample $50$ points on $\widehat{\O}_x$ for each $x\in X$, and obtain $\W_2\big(\mu_X, \mu_{\widehat{\O}} \big) \approx 0.335$.
To visualize the measure $\mu_{\widehat{\O}}$, we consider its kernel density estimator $f\colon \R^4 \rightarrow [0,+\infty)$ with Gaussian kernel of bandwidth $0.1$, and represent its sublevel set $f^{-1}([0.5,+\infty))$.
We observe from the figure its adequacy with $X$.
\begin{figure}[H]
\centering
\begin{minipage}{0.49\linewidth}\center
\includegraphics[width=0.8\textwidth]{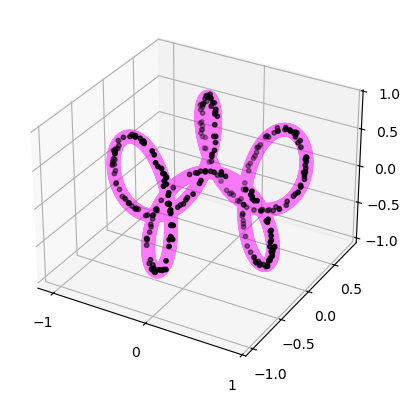}
\end{minipage}
\begin{minipage}{0.49\linewidth}\center
\includegraphics[width=0.8\textwidth]{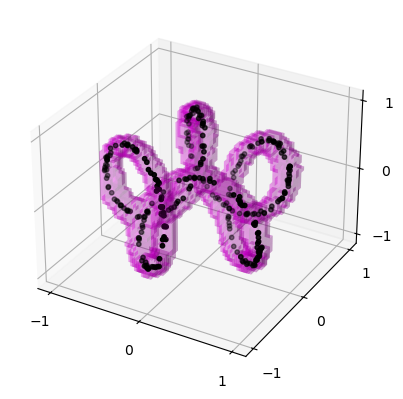}
\end{minipage}
\caption{Illustrations of Example \ref{ex:running_ex_step4}. The point cloud $X$ in black, together with the estimated orbit $\widehat{\O}_x$ in magenta (\textbf{left}) and a sublevel set of a kernel density estimator on $\mu_{\widehat{\O}}$ (\textbf{right}).}\label{fig:step4_orbit}
\end{figure}
\end{example}

\begin{example}
\label{ex:running_ex_step4_noise}
We reproduce the running example of this section, starting from the same point cloud $X$ of cardinality $300$, presented in Example \ref{ex:running_ex}, to which we add some additive Gaussian noise of standard deviation $\sigma = 0.03$, as well as $30$ points drawn uniformly in the cube $[-1,1]^4$.
Algorithm \ref{alg: 1} retrieves successfully the representation $\phi_1 \oplus \phi_4$.
However, by selecting an arbitrary $x\in X$, we obtain the Hausdorff distance $\HDmid{X}{\widehat{\O}_x} \approx 1.128$, which we consider as large.
It can be seen in Figure \ref{fig:step4_orbit_corrupted} that $\widehat{\O}_x$ does not fit $X$ correctly.
On the other hand, the Wasserstein distance is $\W_2\big(\mu_X, \mu_{\widehat{\O}} \big) \approx 0.392$, significantly better.
We represent a sublevel set of a kernel density estimation on $\mu_{\widehat{\O}}$, similar to that of Example \ref{ex:running_ex_step4}, which appears to match $X$.
\begin{figure}[H]
\centering
\begin{minipage}{0.49\linewidth}\center
\includegraphics[width=0.8\textwidth]{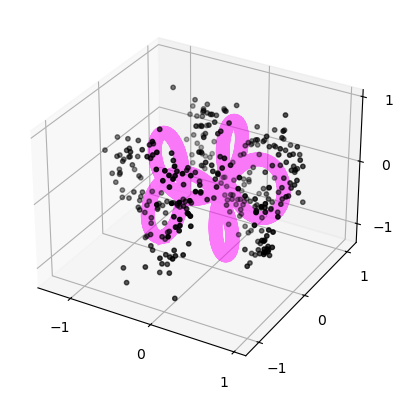}
\end{minipage}
\begin{minipage}{0.49\linewidth}\center
\includegraphics[width=0.8\textwidth]{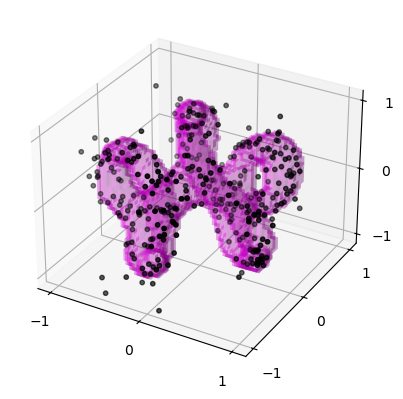}
\end{minipage}
\caption{Illustrations of Example \ref{ex:running_ex_step4_noise}. The corrupted version of $X$ in black, together with the estimated orbit $\widehat{\O}_x$ in magenta (\textbf{left}) and sublevel set of kernel density estimator on $\mu_{\widehat{\O}}$ (\textbf{right}).}\label{fig:step4_orbit_corrupted}
\end{figure}
\end{example}

\section{Concrete implementation}\label{sec:algorithm_additional}

A downside of Algorithm \ref{alg: 1}, presented in the previous section, is that it has to be adapted to the particular Lie group $G$ considered. 
More precisely, in \ref{item:step3}, a list of irreps of $G$ must be provided, as well as an explicit description of its orbit-equivalence classes, $\mathrm{OrbRep}(G,\R^{n})$ (defined in Section~\ref{subsec:structure_orbits}), and a choice of pushforward Lie algebras, $\mathfrak{orb}(G,n)$ (defined in Equation~\eqref{eq:orb}).
In this section, we provide these lists in the case of $\SO(2)$, $T^d$, $\SU(2)$, and $\SO(3)$.
Given any other group of interest, similar computations must be performed.
However, it must be emphasized that, even if irreducible representations of compact Lie groups are well understood, it may be complicated to find explicit descriptions of them in the literature. 
Finally, we collect in Section \ref{subsec:additional} additional comments. This includes a new context in which \texttt{LieDetect} becomes useful: when the precise Lie group $G$ is unknown but may be determined from a list of candidates $\{G_1,\dots, G_k\}$. We also propose a rule of thumb to determine whether the algorithm succeeded, based on the output Hausdorff distance  $\HDmid{X}{\widehat{\O}_x}$.

\subsection{The algorithm for \texorpdfstring{$\SO(2)$}{SO(2)}}\label{sec: algorithm so(2)}

We will observe that the set $\mathfrak{orb}(\SO(2),n)$ is closely linked to the primitive integral vectors, allowing an explicit description.
We then propose a simplification of \ref{item:step3}'.

\subsubsection{Orbit-equivalence for \texorpdfstring{$\SO(2)$}{SO(2)}}\label{subsubsec:orbit_equivalence_so(2)}
As a consequence of the classification of irreps of $\SO(2)$, its representations in $\R^{n}$ are, up to equivalence, classified by a multiset $\{\{k_1,\dots,k_p\}\}$ of weights in $\Z$.
As shown in the next lemma, up to orbit-equivalence, representations are also classified by a certain tuple of integers.
For the sake of notation simplicity, we will suppose that the dimension of the ambient space is even, equal to $n=2m$.
Note that the odd-dimensional case is similar, since representations of $\SO(2)$ in $\R^{2m+1}$ are obtained from representations in $\R^{2m}$, to which we add the trivial representation.

\begin{lemma}\label{lem:orb_eq_so(2)}
The set of orbit-equivalence classes of representations of $\SO(2)$ in $\R^{2m}$ is in correspondence with the non-negative and non-decreasing primitive $m$-tuples of integers (i.e., with trivial greatest common divisor), denoted $\mathbb{N}^{m}_\mathrm{prim}$. 
That is to say, one has a bijection
\begin{equation*}
\mathrm{OrbRep}(\SO(2),\R^{2m}) \simeq \mathbb{N}^{m}_\mathrm{prim}.
\end{equation*}
Explicitly, a choice of representatives for their pushforward Lie algebras is given by
\begin{equation}\label{eq:orb_SO2}
\mathfrak{orb}(\SO(2),2m) = \big\{B(k_1,\dots,k_{m}) \mid (k_1,\dots,k_{m})\in \mathbb{N}^{m}_\mathrm{prim}\big\},
\end{equation}
where
\begin{equation}\label{eq:skewsymmetric_B}
B(k_1,\dots,k_{m}) = \frac{\mathrm{diag}\big(L(k_1), \dots, L(k_{m})\big)}{\sqrt{2}\big\|(k_1,\dots,k_{m})\big\|} 
~~~~~\mathrm{and}~~~~~
L(k) =
\begin{pmatrix}
     0 & -k\\
     k & 0
\end{pmatrix}.
\end{equation}
When the representation is trivial, we define $B(0,\dots,0)$ as the zero matrix.
\end{lemma}

\begin{proof}
Let us recall the classification of irreps of $\SO(2)$, already given in Example \ref{ex:irreps_so(2)}: except the trivial representation $\phi_0\colon\SO(2)\rightarrow\GL_1(\R)$, they are all of dimension $2$, and form the family $\{\phi_k\mid k \in \Z\setminus\{0\}\}$.
In other words, irreps of $\SO(2)$ are parametrized by $\Z$ (the weight).
As a consequence of the decomposition into irreps, any representation $\phi\colon\SO(2)\rightarrow\GL_n(\R)$ corresponds to a unique multiset $\{\{k_1,\dots,k_p\}\}$ of integers such that $\phi$ is equivalent to the direct sum $\bigoplus_{i=1}^p \phi_{k_i}$.
Since $n$ is assumed even, $\{\{k_1,\dots,k_p\}\}$ must contain an even number of zero values.
In what follows, it will be convenient to denote by $\Omega(\phi)$ the $m$-tuple containing all the nonzero integers in $\{\{k_1,\dots,k_p\}\}$, and one zero for each pair of zeros.
After potentially renaming these integers, let us write $\Omega(\phi) = (k_1,\dots,k_m)$, whose ordering is, for the moment, arbitrary.
Note that there exists $A\in\GL_{2m}(\R)$ such the pushforward algebra $\h=\d\phi(\g)$ is generated by the skew-symmetric matrix $A B(k_1,\dots,k_{m}) A^{-1}$, where $B$ is defined in Equation~\eqref{eq:skewsymmetric_B}.

We remind the reader that, as defined in Section \ref{subsec:structure_orbits}, two representations $\phi$ and $\phi'$ are said orbit-equivalent if there exists a matrix $M \in\GL_n(\R)$ such that $\d\phi(\g) = M\d\phi'(\g)M^{-1}$.
In the particular case of $\SO(2)$, one obtains a convenient equivalent characterization that can be deduced from the fact that conjugate matrices share the same eigenvalues.
First of all, given a representation $\phi$, we can suppose that its weights $\Lambda(\phi)=(k_1,\dots,k_{m})$ are non-negative and in non-decreasing order, since this operation would yield an equivalent representation, hence orbit-equivalent.
Moreover, two representations $\phi$ and $\phi'$ are orbit-equivalent when there exists an $\alpha \in \mathbb{N}\setminus\{0\}$ such that $\Lambda(\phi)$ can be obtained from $\Lambda(\phi')$ by multiplying each element by $\alpha$, or dividing each element by $\alpha$.
That is to say, $\mathrm{OrbRep}(\SO(2),\R^{2m})$ is in correspondence with the set $\mathbb{N}^{m}_\mathrm{prim}$, and the result follows.
\end{proof}

\begin{remark}
We point out that $\mathbb{N}^{m}_\mathrm{prim}$ has been extensively studied in the context of Linnik's problem in analytic number theory \cite{linnik2ergodic,duke1988hyperbolic,duke2007introduction}.
It is known that, as $r \to +\infty$, the sets 
$$
\big\{x/\|x\| \mid x \in \mathbb{N}^{m}_\mathrm{prim}, ~ \|x\|\leq r\big\}
$$ 
are \textit{equidistributed} on the intersection of the sphere $S^{m-1}$ and the `increasing quadrant' $H=\{(x_1,\dots,x_{m}) \in \R^{m}\mid 0\leq x_1\leq\dots\leq x_{m}\}$.
That is to say, the sequence of empirical measures on these sets converges in distribution to the Lebesgue measure on $S^{m-1}\cap H$.
We stress that the restriction to $H$ comes from the fact that, in our context, only non-negative \textit{ordered} primitive vectors are considered.
In other words, the primitive vectors, in addition to being dense on the sphere, tend to be uniformly distributed.
Although we will not use the equidistribution property of $\mathbb{N}^{m}_\mathrm{prim}$ directly in this article, it explains our intuition when defining, for a parameter $\omega_\mathrm{max} \in \mathbb{N}$, the set of orbit-equivalence classes of representations of weights at most $\omega_\mathrm{max}$:
\begin{equation}\label{eq:def_orb_so(2)_maximalweight}
\mathfrak{orb}(\SO(2),2m,\omega_\mathrm{max}) 
= 
\big\{B(k_1,\dots,k_{m}) \mid (k_1,\dots,k_{m})\in \mathbb{N}^{m}_\mathrm{prim},~ \max\{|k_i|\}_{k=1}^{m}\leq \omega_\mathrm{max}\big\}.    
\end{equation}
Indeed, when running our algorithm, and as pointed out in Section \ref{subsubsec:step3_formulation}, the infinite set $\mathfrak{orb}(\SO(2),n)$ must be restricted to a finite subset.
Choosing the restriction above will tend to be `uniformly distributed' on $\mathfrak{orb}(\SO(2),n)$.
This will be studied further in Section \ref{subsubsec:rigidity_lie_subalgebras}, in the context of \textit{rigidity} of Lie subalgebras. 
\end{remark}

Based on this explicit description of $\mathfrak{orb}(\SO(2),n)$, \ref{item:step3} of Algorithm \ref{alg: 1} can be performed.
An entire execution of the algorithm has already been given in Examples \ref{ex:running_ex}, \ref{ex:running_ex_step1}, \ref{ex:running_ex_step2}, \ref{ex:running_ex_step3} and \ref{ex:running_ex_step4}, and we provide below the Examples \ref{ex:algo_so(2)_dim4} and \ref{ex:algo_so(2)_dim6}.  

\subsubsection{Reformulation of \ref{item:step3}'}\label{sec: algorithm so(2) simp}
The fact that $\so(2)$ is one-dimensional allows for a significant simplification of the algorithm.
To see so, we suppose that the first two steps have been performed.
Instead of going through \ref{item:step3}, we consider the variation \ref{item:step3}', presented in Section \ref{subsubsec:step3_variation}. 
Namely, we let $A$ be a unit eigenvector of the \texttt{LiePCA} operator $\Lambda$ associated with the smallest eigenvalue, skew-symmetrize it if it is not, and consider Equation \eqref{eq:minimization_grassmann}, which reads 
\begin{equation*}
    \arg \min \big\| \projbracket{\spn{A}} -  \projbracket{\spn{O B(k_1,\dots,k_{m}) O^\top}} \big\|^2
    ~~~~\mathrm{s.t.}~~~~
    \begin{cases}
      (k_1,\dots,k_{m}) \in \mathbb{N}^{m}_\mathrm{prim},\\
      O \in \Ort(2m).
    \end{cases}    
\end{equation*}
To continue, let us consider the normal form of $A$.
Namely, there exist a matrix $P\in\Ort(2m)$ and a unique $m$-tuple of non-negative real numbers $\alpha_1\leq\dots\leq\alpha_{m}$ such that
$$
A = P B(\alpha_1,\dots,\alpha_{m}) P^\top,
$$
where $B$ is defined in Equation \eqref{eq:skewsymmetric_B}.
The problem can now be solved exactly.

\begin{lemma}\label{lem:step3'_simplification_so(2)}
For $G=\SO(2)$, Equation \eqref{eq:minimization_grassmann} is equivalent to
\begin{equation}\label{eq:lemstep3'_simplification_so(2)}
    \arg \min f\big((\alpha_i)_{i=1}^{m},~(k_i)_{i=1}^{m}\big) ~~~~\mathrm{s.t.}~~~~ 
      (k_i)_{i=1}^{m} \in \mathbb{N}^{m}_\mathrm{prim},
\end{equation}
where $f(x,y) = \big\|x/\|x\|-y/\|y\|\big\|^2$.
\end{lemma}

\begin{proof}
According to Lemma \ref{lem:distance_grassmannian}, the distance on the Grassmannian can be reformulated as
\begin{equation}\label{eq:minimization_grassmann_so(2)}
2\big\| A \pm O B(k_1,\dots,k_{m}) O^\top \big\|^2,    
\end{equation}
where the sign is chosen to minimize the norm.
By Hoffman-Wielandt inequality \cite{hoffman2003variation}, the value 
$$
\big\| P B(\alpha_1,\dots,\alpha_{m}) P^\top - O B(k_1,\dots,k_{m}) O^\top \big\|^2
$$
is lower bounded by $2 \big\|x/\|x\|-y/\|y\|\big\|^2$, where $x=(\alpha_i)_{i=1}^{n/2}$ and $y=(k_i)_{i=1}^{n/2}$.
Moreover, this bound is attained when $O=P$.
\end{proof}

\begin{remark}
A tuple $(k_1,\dots,k_{m})$ being fixed, we recognize in Equation \eqref{eq:minimization_grassmann_so(2)} what is known as the \textit{two-sided orthogonal Procrustes problem with one transformation}: given two matrices $A$ and $B$, find an orthogonal matrix $O$ that minimizes $\|A-OB O^\top\|$.
When both matrices $A$ and $B$ are symmetric, it has been shown that the problem admits an explicit solution, based on the best pairing between their eigenvalues \citep{Schnemann1968OnTO,umeyama1988eigendecomposition}.
In our context, the matrices are skew-symmetric, but the problem is solved similarly, as shown in the proof of Lemma \ref{lem:step3'_simplification_so(2)}.
\end{remark}

The lemma makes explicit a specific issue, occurring when the vector $x=(\alpha_1,\dots,\alpha_{m})$ does not come from an integral vector, i.e., when it does not span a rational line.
In this case, the map $y \in \mathbb{N}^{m}_\mathrm{prim} \mapsto f(x,y)$ admits $0$ as an infimum, but does not admit a minimizer.
Consequently, \ref{item:step3}' is ill-defined.
Nevertheless, this detail is fixed by restricting $\mathbb{N}^{m}_\mathrm{prim}$ to the tuples of coordinates lower or equal to the parameter $\omega_\mathrm{max}$.
In other words, we restrict the set of representations to $\mathfrak{orb}(\SO(2),2m,\omega_\mathrm{max})$ defined in Equation \eqref{eq:def_orb_so(2)_maximalweight}.
Under this assumption, we will derive in Lemma \ref{lem:lower_bound_gamma_abelian}, in the next section, a lower bound on $f$.

From a computational point of view, we obtain $(\alpha_1,\dots,\alpha_{m})$ by reducing $A$ in normal form via Schur decomposition. 
The classical algorithm for such a reduction is based on Householder reflections, and implemented in the library \texttt{numpy} for Python \cite{harris2020array}, which we use.
As a last remark, we point out that, instead of considering all non-negative ordered tuples of $\mathbb{N}^{m}_\mathrm{prim}$, we can restrict our list to those which do not contain zero nor repeated elements.
Indeed, after \ref{item:step1}, the point cloud $X$ is supposedly included in an orbit $\O$ that spans the whole ambient space, which is possible only when the representation has positive and unique weights.

\begin{example}
\label{ex:algo_so(2)_dim4}
We reproduce the experiment of Example \ref{ex:running_ex_step3_variation}, now using the minimization of Equation \eqref{eq:lemstep3'_simplification_so(2)} instead of Equation \eqref{eq:minimization_grassmann}.
As expected, the weights $(1,4)$ are found to be optimal, with a cost of $0.003$.
Additionally, \ref{item:step4} yields $\HDmid{X}{\widehat{\O}_x} \approx 0.019$, a value considered as small.    
\end{example}

\begin{example}\label{ex:algo_so(2)_dim6}
We consider a representation of $\SO(2)$ in $\R^{10}$ with weights $(2, 4, 5, 7, 8)$, sample uniformly $600$ points on one of its orbits, that we corrupt with a Gaussian additive noise of standard deviation $\sigma = 0.03$, and apply Algorithm \ref{alg: 1} on it.
Among the $251$ positive and strictly increasing primitive integral vectors of $\mathbb{N}^5$ with coordinates at most $\omega_\mathrm{max}=10$, \ref{item:step3}' recovers successfully the weights of the representation. 
We indicate in Table \ref{table:algo_so(2)_dim6} the values of Equation \eqref{eq:lemstep3'_simplification_so(2)} for the top $12$ tuples.
Last, we get from \ref{item:step4} the approximations $\HD{X}{ \widehat{\O}_x} \approx 0.231$ and $\W_2\big(\mu_X, \mu_{\widehat{\O}} \big) \approx 0.389$.
As shown on Figure \ref{fig:alg1_R10}, both the estimated orbit $\widehat{\O}_x$ and the estimated measure $\mu_{\widehat{\O}}$, represented via a sublevel set of a kernel density estimator, appear to fit correctly the input point cloud.

\setlength\tabcolsep{3pt}
\begin{table}[ht]\center    
\begin{tabular}{||r | c c c c c c ||} 
\hline
Weights &  $(2, 4, 5, 7, 8)$ & $(2, 5, 6, 8, 9)$ &  $(3, 5, 7, 9, 10)$ &$(3, 6, 7, 9, 10)$ & $(3, 5, 6, 8, 9)$&$(2, 4, 5, 6, 7)$ \\ 
\hline
Costs & \bm{$0.028$}&$0.032$&$0.037$ & $0.037$&$0.038$ & $0.044$ \\ 
\hline
Weights &  $(3, 5, 6, 9, 10)$ & $(2, 5, 7, 9, 10)$ &  $(2, 3, 4, 5, 6)$ &$(2, 5, 6, 9, 10)$ & $(2, 6, 7, 9, 10)$&$(3, 5, 6, 8, 10)$ \\ 
\hline
Costs & $0.046$&$0.055$&$0.057$ & $0.058$&$0.058$ & $0.058$ \\ 
\hline
\end{tabular}
\caption{Results of \ref{item:step3}' in Example \ref{ex:algo_so(2)_dim6}. The best score is shown in bold.}
\label{table:algo_so(2)_dim6}
\end{table}

\begin{figure}[ht]
\centering
\begin{minipage}{0.49\linewidth}\center
\includegraphics[width=0.8\textwidth]{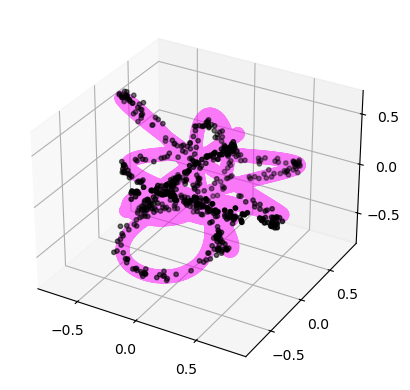}
\end{minipage}
\begin{minipage}{0.49\linewidth}\center
\includegraphics[width=0.8\textwidth]{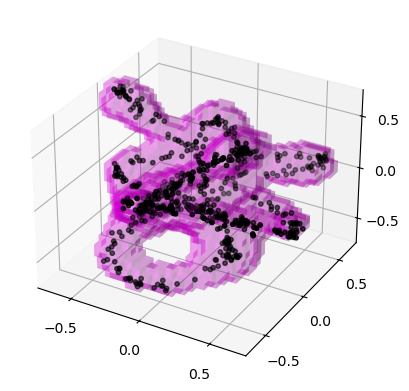}
\end{minipage}
\caption{Illustrations of Example \ref{ex:algo_so(2)_dim6}. The input point cloud in black, the estimated orbit $\widehat{\O}_x$ in magenta (\textbf{left}), and a sublevel set of a kernel density estimator on $\mu_{\widehat{\O}}$ (\textbf{right}).}
\label{fig:alg1_R10}
\end{figure}
\end{example}

\subsection{The algorithm for \texorpdfstring{$T^d$}{Td}}\label{sec: algorithm torus}

We still work in even dimension $n=2m$.
Employing standard ideas, we describe the set $\mathfrak{orb}(T^d,2m)$ as a direct generalization of $\mathfrak{orb}(\SO(2),2m)$, in which the role of primitive integral vectors is played by primitive lattices of $\Z^{m}$.
We also derive a simplification of \ref{item:step3}'.

\subsubsection{Orbit-equivalence for \texorpdfstring{$T^d$}{Td}}
As we have seen in Example \ref{ex:irreps_torus}, irreps of the $d$-dimensional torus are parametrized by a weight $\omega$ in $\Z^d$ (the \textit{weight lattice} of the torus), where the zero weight corresponds to the trivial representation and the other irreps have dimension $2$.
Similar to $\SO(2)$, the decomposition into irreps of a representation $\phi$ of $T^d$ in $\R^{2m}$ forms a unique multiset $\{\{\omega_1,\dots,\omega_m\}\}$ of elements of $\Z^d$.
We will restrict our discussion to almost-faithful representations, implying that the weights are nonzero and that $m\geq d$.
We will denote by $\Omega(\phi)$ the $d\times m$ matrix of integers formed by the weights, where the order is chosen arbitrarily.

Next, we denote by $(\pi_1,\dots,\pi_{d})$ the rows of $\Omega(\phi)$, alternatively defined as $\pi_i = (\omega_1^i,\dots,\omega_{m}^i)$ for $i\in [1\isep d]$.
They span an \textit{integral lattice} $\spn{\Omega(\phi)}$, defined as the discrete subgroup of $\Z^{m}$:
$$
\spn{\Omega(\phi)} = \bigg\{ \sum_{i=1}^d t_i \pi_i \mid \Omega(\phi)=(\pi_1,\dots,\pi_{d})^\top, ~(t_1,\dots,t_d) \in \Z^d \bigg\}.
$$
Since $\phi$ is almost-faithful, the lattice has rank $d$.
The collection of rank-$d$ lattices will be denoted by $\mathcal{R}^{m,d}$.
In addition, we denote by $\mathcal{R}^{m,d}_\mathrm{prim}$ the rank-$d$ \textit{primitive} lattices, i.e., those sublattices of $\Z^{m}$ that are not contained in any larger integral lattice of the same rank.
The group of \textit{signed permutations} $\{\pm1\}^m\rtimes\mathfrak{S}_{m}$ acts on both $\mathcal{R}^{m,d}$ and $\mathcal{R}^{m,d}_\mathrm{prim}$ by permutation of the axes and reflection along them.
Combinatorially, this corresponds to permuting of the columns of $\Omega(\phi) = (\omega_1,\dots,\omega_{m})$, or multiplying some of them by $-1$.
The following lemma states that the quotient set of primitive lattices classifies the orbit-equivalence classes of $T^d$.

\begin{lemma}
Let $d\in\mathbb{N}$ positive. The set of orbit-equivalence classes of almost-faithful representations of $T^d$ in $\R^{2m}$ is in correspondence with the rank-$d$ primitive sublattices of $\Z^{m}$, up to the action of the signed permutations on the axes.
That is to say, one has a bijection
\begin{equation*}
\mathrm{OrbRep}(T^d,\R^{2m}) \simeq \faktor{\mathcal{R}^{m,d}_\mathrm{prim}}{\{\pm1\}^m\rtimes\mathfrak{S}_{m}}.
\end{equation*}
Once a choice of representatives in this quotient has been chosen, one obtains an explicit choice of representatives for their pushforward Lie algebras:
\begin{equation}\label{eq:orb_Td}
\mathfrak{orb}(T^d,2m) = \big\{\big(B(\omega_1^i,\dots,\omega_{m}^i)\big)_{i=1}^d \mid \spn{\omega_1,\dots,\omega_m} \in \mathrm{OrbRep}(T^d,\R^{2m})\big\},
\end{equation}
where $B$ is defined in Equation \eqref{eq:skewsymmetric_B}.
\end{lemma}

In particular, when $d=1$, we recognize the correspondence of Lemma \ref{lem:orb_eq_so(2)}:
$$
\mathrm{OrbRep}(\SO(2),2m)
\simeq \faktor{\mathcal{R}^{m,1}_\mathrm{prim}}{\{\pm1\}^m\rtimes\mathfrak{S}_{m}},
$$ 
that is, the primitive elements of $\Z^m$ up to signed permutations, which is represented by the non-negative non-decreasing primitive tuples $\mathbb{N}^{m}_\mathrm{prim}$.
Besides, we point out that when $m=d$, there is only one primitive lattice: $\Z^d$ itself.

\begin{proof}
Determining whether two representations $\phi$ and $\phi'$ are equivalent, based on matrix encodings $\Omega(\phi)$ and $\Omega(\phi')$, is a classical question of representation theory, involving the notions of root systems and Weyl groups.
In the case of the torus, the problem is easily solved.
Indeed, the representations are equivalent if and only if $\Omega(\phi)$ can be obtained from $\Omega(\phi')$ via the action of the signed permutations on $\Z^m$.
That is to say, one has a bijection
\begin{equation*}
\mathrm{Rep}(T^d,\R^{2m}) \simeq \faktor{\Z^{m\times d}}{\{\pm1\}^m\rtimes\mathfrak{S}_{m}}.
\end{equation*}
In the case of orbit-equivalence, however, one is allowed to permute the rows $(\pi_1,\dots,\pi_{d})$ and $(\pi_1',\dots,\pi_{d}')$, or take linear combinations; in onther words, one must consider the lattice encodings $\spn{\Omega(\phi)}$ and $\spn{\Omega(\phi')}$.
Indeed, we see that the representations $\phi$ and $\phi'$ are orbit-equivalent if and only if one of the lattices $\spn{\Omega(\phi)}$ and $\spn{\Omega(\phi')}$ is \textit{included} in the other, after a series of permutation of the axes and reflection along them.
This proves the result for $\mathrm{OrbRep}(T^d,\R^{2m})$.
The second point of the lemma follows from the fact that, by decomposition into irreps, there exists $M\in\GL_{2m}(\R)$ such that the pushforward algebra $\h=\d\phi(\g)$ admits the following basis:
\begin{equation*}
\big( M B(\pi_1) M^{-1}, \dots, M B(\pi_d) M^{-1} \big).\qedhere
\end{equation*}
\end{proof}

\begin{remark}
We mention that the set of primitive lattices $\mathcal{R}^{m,d}_\mathrm{prim}$ enjoys an equidistribution property similar to the primitive vectors. 
To state this result, let us denote, for any integral lattice $L \subset \R^{m,d}$, the space it spans as $V_L$.
This is a linear subspace of $\R^{m}$, that we see as an element of $\G(d,m)$, the Grassmannian of $d$-dimensional planes of $\R^{m}$. 
We also denote by $d(L)$ the \textit{covolume} of $L$, that is, the $d$-dimensional volume of a fundamental region, seen in the subspace $V_L$.
As it has been shown, for instance in \cite{schmidt1998distribution,horesh2023equidistribution}, the sequence of sets
$$
\big\{ V_L \mid L \in \mathcal{R}^{m,d}_\mathrm{prim}, ~ d(L) \leq r\big\}
$$ 
equidistributes, when $r\to \infty$, to the uniform measure on $\G(d,m)$.
Besides, we remind the reader that, to apply our algorithm, one must input a finite subset of $\mathfrak{orb}(T^d,2m)$, the Lie algebras of a set of representatives of orbit-equivalence classes.
Therefore we choose a positive integer $\omega_\mathrm{max}$ and define $\mathfrak{orb}(T^d,2m,\omega_\mathrm{max})$ as those coming from representations with weights at most $\omega_\mathrm{max}$.
The equidistribution property implies that it is an `accurate' sample of $\mathfrak{orb}(T^d,2m)$.
\end{remark}

In practice, the set $\mathcal{R}^{m,d}_\mathrm{prim}$ can be generated as follows: we first consider the set of all $d\times m$ matrices with integer coefficients of absolute value at most $\omega_\mathrm{max}$, then discard those that do not span a primitive rank lattice of rank $d$, and finally discard matrices when they span the same linear subspace.
These conditions are easily verified via algebraic manipulations of their Smith normal form.
In our case, however, lattices must be considered up to the action of $\{\pm1\}^m\rtimes\mathfrak{S}_{m}$.
To do so, we simply associate to every lattice $L\in\mathcal{R}^{m,d}_\mathrm{prim}$ the collection of the $m!\times 2^m$ spaces $V_L$ it spans (that we encode as $m\times m$ projection matrices), after permutation and reflection of the axes.
Another lattice $L'$ will define an orbit-equivalent representation if these collections agree.
Ultimately, we store the smallest projection matrix (for the lexicographic order on the entries)---this is a complete invariant of orbit-equivalence.
We point out that, although yielding accurately the set $\mathfrak{orb}(T^d,2m,\omega_\mathrm{max})$, this procedure is arguably time and memory-consuming, and we expect that improvements could be defined.

\subsubsection{Reformulation of \ref{item:step3}'}\label{sec: algorithm torus simp}
We now inspect \ref{item:step3}' in the context of $G = T^d$.
Let $A_1,\dots,A_d$ be eigenvectors associated with the bottom $d$ eigenvalues of the \texttt{LiePCA} operator, as given by \ref{item:step2}, that we skew-symmetrize if they are already not, and orthonormalize.
We suppose that, after this operation, they still form a free family, i.e., $\spn{A_1,\dots,A_d}$ has dimension $d$.
We see this family as a point of the Grassmannian $\G(d,\so(2m))$ of $d$-planes of $\so(2m)$.
Notice that  the projection matrix
$$\projbracket{\spn{A_1 ,\dots,A_d }}$$ 
has dimension $m(2m-1)\times m(2m-1)$.
Besides, for any lattice $L \in \mathcal{R}^{m,d}_\mathrm{prim}$, we consider an orthonormal basis $\pi_1,\dots,\pi_d$ of its span and build the skew-symmetric matrices $B(\pi_i) = \diag(B(\pi_i^k))_{k=1}^{m}$ for all $i\in[1\isep m]$, where $B$ as been defined in Equation \eqref{eq:skewsymmetric_B}.
We denote their span by $\mathcal{B} = \spn{B(\pi_1),\dots,B(\pi_d)}$.
Given a matrix $O\in\Ort(2m)$, we consider the conjugate space $O \mathcal{B} O^\top$.
With these notations, Equation \eqref{eq:minimization_grassmann} can be rewritten as
\begin{equation}\label{eq:minimization_grassmann_td}
    \arg \min \big\| \projbracket{\spn{ A_1 ,\dots,A_d }} -  \projbracket{\spn{OB(\pi_1)O^\top,\dots,OB(\pi_d)O^\top}}\big\|^2
    ~~~~\mathrm{s.t.}~~~~ 
    \begin{cases}
      L \in \mathcal{R}^{m,d}_\mathrm{prim},\\
      O \in \Ort(2m).
    \end{cases}    
\end{equation}
Using Lemma \ref{lem:distance_grassmannian}, we can reformulate it as
\begin{align*}
2 \sum_{i=1}^d \bigg( 1 - \big\| \projbracket{\mathcal{B}}(OA_iO^\top)\big\|^2\bigg).
\end{align*}
To further study this equation, let us denote by $\mathcal{C}$ the subspace of $\so(2m)$ spanned by the $(2\times2)$-block-diagonal skew-symmetric matrices.
In particular, $\mathcal{B}$ is a linear subspace of $\mathcal{C}$.
Consequently, to minimize the previous equation, the matrices $O$ must also minimize 
\begin{equation}\label{eq:simultaneous_reduction}
O \mapsto 2 \sum_{i=1}^d \big\| \projbracket{\mathcal{C}^\bot}(OA_iO^\top)\big\|^2.
\end{equation}
Let $O_*$ be a matrix that minimizes the equation.
For each $i\in[1\isep d]$, the matrices $\projbracket{\mathcal{C}^\bot}(O_*A_iO^\top_*)$ take the form $B(\rho_i)$ for some $\rho_i \in \R^{m}$.
The collection of $\rho_i$'s spans a rank-$d$ lattice in $\R^{m}$.
As a consequence, we see that Equation \eqref{eq:minimization_grassmann_td} is nothing but a distance between the lattices $\spn{\rho_i}_{i=1}^d$ and $\spn{\pi_i}_{i=1}^d$. We obtain a direct generalization of Lemma \ref{lem:step3'_simplification_so(2)}.

\begin{lemma}\label{lem:step3'_simplification_torus}
For $G=T^d$, Equation \eqref{eq:minimization_grassmann} is equivalent to
\begin{equation}\label{eq:lemstep3'_simplification_torus}
     \underset{L \in \mathcal{R}_d^{m}}{\mathrm{arg~min}} \sum_{k=1}^d f\bigg((\rho_i^k)_{i=1}^{m},~(\pi_i^k)_{i=1}^{m}\bigg), 
\end{equation}
where $f(x,y) = \big\|x/\|x\|-y/\|y\|\big\|^2$ and where $\spn{\pi_i}_{i=1}^d$ denotes any basis of the lattice $L$.
\end{lemma}

\begin{remark}
We recognize in Equation \eqref{eq:simultaneous_reduction} the problem of \emph{simultaneously reducing} a $d$-tuple of skew-symmetric matrices.
When the $A_i$'s commute, this problem is solved by simply reducing each matrix to its normal form.
When they do not, the usual closed-form reduction algorithms cannot be used, and optimization-based methods have been proposed \cite{nolte2006identifying,meinecke2012simultaneous}.
We will employ the methods of the first of these articles: Equation \eqref{eq:simultaneous_reduction} is minimized by a direct optimization on $\Ort(n)$ via projected gradient descent.
\end{remark}

As was the case for $\SO(2)$, we see that the minimizer of Equation \eqref{eq:lemstep3'_simplification_torus} is ill-defined when $(\rho_i)_{i=1}^d$ is not a rational point of the Grassmannian.
However, this will make no difference in the application of our algorithm, since we restrict the lattices to those admitting a basis with coordinates upper bounded by the parameter $\omega_\mathrm{max}$.
In practice, we observe that this new formulation of \ref{item:step3}' allows for a significant reduction in computation time.
We provide two examples below, consisting of a representation of $T^2$ in $\R^6$ and of $T^3$ in $\R^{8}$.

\begin{example}\label{ex:algo_T2_R6}
Figure \ref{fig:alg1_t2_r6} shows a uniform $750$-sample from an orbit of the representation $\phi_{(1,1)}\oplus\phi_{(1,2)}\oplus\phi_{(2,1)}$ of $T^2$ in $\R^6$, to which we apply Algorithm \ref{alg: 1}.
Besides, we represent the eigenvalues of the \texttt{LiePCA} operator.
We identify on this graph two significantly small eigenvalues, in accordance with the dimension of the symmetry group.
We apply \ref{item:step3}', with the formulation of Equation \eqref{eq:lemstep3'_simplification_torus}, to the $18$ almost-faithful representations of $T^2$ in $\R^6$ up to orbit-equivalence, which we encode as a $2\times 3$ matrix. We indicate the best twelve values in the Table \ref{table:algo_T2_R6}.
The algorithm's output is $\left(\begin{smallmatrix}0&1&1\\2&-2&1\end{smallmatrix}\right)$, that is, the representation $\phi_{(0,2)}\oplus\phi_{(1,-2)}\oplus\phi_{(1,1)}$, which is indeed orbit-equivalent to the original one.
\ref{item:step4} yields the approximation $\HDmid{X}{\widehat{\O}_x} \approx 0.071$, indicating a strong accordance between the input point cloud and the output orbit.
As pointed out in Section \ref{sec:step4_hasdorff}, this approximation is obtained by sampling $K$ points on $\widehat{\O}_x$, with $K$ large, and calculating the Hausdorff distance between the finite point clouds. 
The bottom right graph of Figure \ref{fig:alg1_t2_r6} indicates the values of this approximation, for $K$ varying between $100$ and $90,000$.

\begin{table}[ht]\center
\begin{tabular}{||r | c c c c c c ||} 
\hline
Type &  $\left(\begin{smallmatrix}0&1&1\\2&-2&1\end{smallmatrix}\right)$ & $\left(\begin{smallmatrix}1&1&2\\-2&2&-1\end{smallmatrix}\right)$ &  $\left(\begin{smallmatrix}0&1&2\\2&-2&-1\end{smallmatrix}\right)$&$\left(\begin{smallmatrix}0&1&1\\1&-2&0\end{smallmatrix}\right)$  & $\left(\begin{smallmatrix}0&1&1\\1&-2&-1\end{smallmatrix}\right)$&$\left(\begin{smallmatrix}0&1&2\\2&-2&1\end{smallmatrix}\right)$ \\ 
\hline
Costs & \bm{$0.036$}&$0.136$&$0.198$ & $0.233$&$0.244$ & $0.312$ \\ 
\hline
Type &$\left(\begin{smallmatrix}0&1&2\\1&-2&-2\end{smallmatrix}\right)$ & $\left(\begin{smallmatrix}0&1&2\\1&-2&-1\end{smallmatrix}\right)$ &  $\left(\begin{smallmatrix}1&2&2\\-2&-2&1\end{smallmatrix}\right)$ &$\left(\begin{smallmatrix}1&1&1\\-2&-1&2\end{smallmatrix}\right)$ & $\left(\begin{smallmatrix}0&1&2\\1&-2&0\end{smallmatrix}\right)$&$\left(\begin{smallmatrix}0&1&1\\1&-2&1\end{smallmatrix}\right)$ \\ 
\hline
Costs & $0.331$&$0.348$&$0.388$ & $0.447$&$0.457$ & $0.472$ \\ 
\hline
\end{tabular}
\caption{Results of \ref{item:step3}' in Example \ref{ex:algo_T2_R6}. The best score is shown in bold.}
\label{table:algo_T2_R6}
\end{table}

\begin{figure}[ht]
\centering
\begin{minipage}{0.54\linewidth}\center
\includegraphics[width=0.99\textwidth]{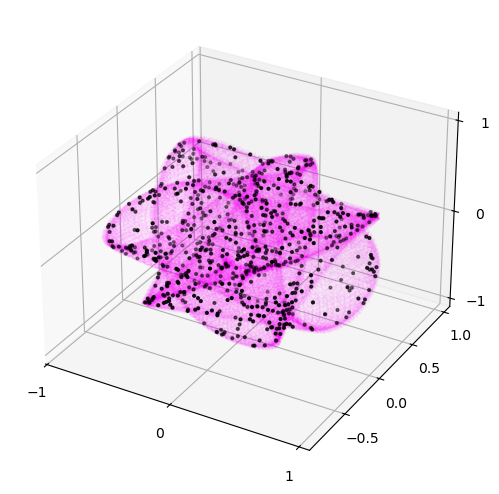}
\end{minipage}
\begin{minipage}{0.44\linewidth}\flushright
\includegraphics[width=0.9\textwidth]{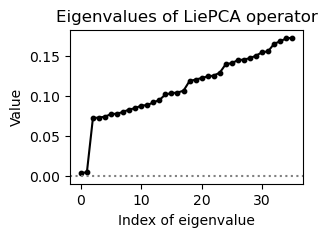}\\
\includegraphics[width=0.9\textwidth]{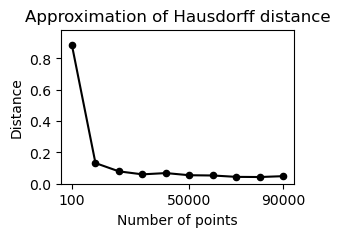}
\end{minipage}
\caption{Illustrations of Example \ref{ex:algo_T2_R6}. \textbf{Left:} the original point cloud (black) and the orbit estimated (magenta) by Algorithm~\ref{alg: 1}. The tight fit suggests that the algorithm converged to the correct representation.
\textbf{Top right:} eigenvalues of the \texttt{LiePCA} operator, where two small values are identified, corresponding to the dimension of the symmetry group, $T^2$. 
\textbf{Bottom right:} approximation of the Hausdorff distance between the point cloud and the orbit, as the approximation parameter $K$ increases (number of points sampled on the estimated orbit).}
\label{fig:alg1_t2_r6}
\end{figure}
\end{example}

\begin{example}\label{ex:algo_T3_R8}
We apply Algorithm \ref{alg: 1} to a $1500$-samples of an orbit of the representation $\phi_{(-1,0,-1)}\oplus\phi_{(1,0,0)}\oplus\phi_{(1,1,0)}\oplus\phi_{(-1,-1,-1)}$ of $T^3$ in $\R^8$.
There are $10$ non-orbit-equivalent and almost-faithful representations of $T^3$ in $\R^8$ with weights in $\{-1,0,1\}$.
We apply \ref{item:step3}' with the formulation of Equation \eqref{eq:lemstep3'_simplification_torus} to these representations and obtain as output $\phi_{(0,0,1)}\oplus\phi_{(1,0,-1)}\oplus\phi_{(1,-1,-1)}\oplus\phi_{(0,-1,1)}$, which is indeed orbit-equivalent to the initial representation.
\ref{item:step4} gives the Hausdorff distance $\HDmid{X}{\widehat{\O}_x} \approx 0.058$, a value which we consider as a good fit.
\end{example}

\subsection{The algorithm for \texorpdfstring{$\SU(2)$}{SU(2)} and \texorpdfstring{$\SO(3)$}{SO(3)}}\label{sec: algorithm so(3)}

We now turn to the two non-Abelian groups of interest in this article: $\SU(2)$ and $\SO(3)$.
For such, we start showing that the notion of orbit-equivalence coincides with that of equivalence for representations of these groups.
In opposition to $T^d$, we did not obtain a simplification of \ref{item:step3}'.
We give several applications of Algorithm \ref{alg: 1} on synthetic datasets.

\subsubsection{Orbit-equivalence for \texorpdfstring{$\SU(2)$}{SU(2)} and \texorpdfstring{$\SO(3)$}{SO(3)}}
As pointed out in Section \ref{subsec:structure_orbits}, if $G$ denotes a compact Lie group, then for any representation $\phi\colon G\rightarrow \GL_n(\R)$ and any surjective homomorphism $f\colon G\rightarrow G$, the representation $f\circ \phi$ is orbit-equivalent to $\phi$.
In the case of the torus $T^d$, these homomorphisms can be identified with $\GL_d(\R)\cap \M_d(\Z)$, the $d\times d$ integer matrices with nonzero determinant. 
This was the idea behind Sections \ref{sec: algorithm so(2)} and \ref{sec: algorithm torus}, where we eventually described $\mathrm{OrbRep}(T^d,n)$ as (a quotient of) the set of lattices.
In the case of $\SU(2)$ or $\SO(3)$, however, this analysis is greatly simplified because their surjective homomorphisms are automorphisms, which are in turn inner automorphisms.

\begin{lemma}
Two orbit-equivalent representations of $G = \SU(2)$ or $\SO(3)$ are equivalent.
\end{lemma}

\begin{proof}
Consider two representations $\phi,\phi'\colon G \rightarrow \GL_n(\R)$, with $\h=\d\phi(\g)$ and $\h'=\d\phi'(\g)$ their pushforward Lie algebras.
We suppose that $\phi$ and $\phi'$ are orbit-equivalent, meaning that there exists a matrix $M\in\GL_n(\R)$ such that $h' = M \h M^{-1}$.
In particular, the conjugation by $M$ induces an automorphism of $\h$, that we denote by $u$. 
It is a property of the groups $\SU(2)$ and $\SO(3)$, as a consequence of Ado's theorem \cite[Th.~7.4.1]{hilgert2011structure}, that any automorphism of their Lie algebra comes from an automorphism of the group. 
That is, there exists an automorphism $f\colon G \rightarrow G$ such that $\d f = u$.
Moreover, because these groups are of type $B_n$ in Dynkin diagram's classification of Lie algebras, we know that their automorphisms are all inner automorphisms.
In other words, there exists $g\in G$ such that $\forall h\in G, f(h) = ghg^{-1}$.
The matrix $M' = \phi(g) \in \GL_n(\R)$ provides an equivalence between the representations $\phi$ and $\phi'$.
\end{proof}

We are left with describing explicitly the set of equivalence classes of representations.
Let us start with the group $\SU(2)$.
Referring to Example \ref{ex:irreps_su(2)}, it admits one irrep in $\R^k$ for each $k\geq 1$ such that $k\equiv 1 \pmod{2}$ or $k\equiv 0 \pmod{4}$. 
We denote the corresponding representation by $\phi_k$.
We deduce that the equivalence classes of representations of $\SU(2)$ in $\R^n$ are given by the \textit{partitions} of $n$ into integers in $(2\mathbb{N}+1)\cup 4\mathbb{N}$.
Let us denote this set as $\mathrm{Part}\big(n,(2\mathbb{N}+1)\cup 4\mathbb{N}\big)$.
By letting $B(k) = (B(k)_1, B(k)_2, B(k)_3)$ be the basis of the pushforward Lie algebra of $\phi_k$ given in Appendix \ref{app: irreducible reps su(2)}, we obtain the explicit description
\begin{equation}\label{eq:orb_SU2}
    \mathfrak{orb}(\SU(2),n)
= \big\{\big(B(k_1), \dots, B(k_p)\big) \mid (k_1,\dots,k_p) \in \mathrm{Part}\big(n, (2\mathbb{N}+1)\cup 4\mathbb{N}\big)\big\}.
\end{equation}
In the case of $\SO(3)$, only the irreps with $k=1$ or $k\equiv 0 \pmod{4}$ are considered, yielding
\begin{equation}\label{eq:orb_SO3}
\mathfrak{orb}(\SO(3),n)
= \big\{\big(B(k_1), \dots, B(k_p)\big) \mid (k_1,\dots,k_p) \in \mathrm{Part}\big(n, 2\mathbb{N}+1\big)\big\}.~~~~~~~~~~
\end{equation}
We note that the cardinality of both these sets is lower than the number of partitions of the integer $n$, which we know is equivalent to $\exp(\pi\sqrt{2n/3})/(4\sqrt{3}n)$.
Moreover, these sets being finite, we do not need to restrict them to $\omega_\mathrm{max}$, as we did for $\mathfrak{orb}(\SO(2),n,\omega_\mathrm{max})$ and $\mathfrak{orb}(T^d,n,\omega_\mathrm{max})$.

\subsubsection{Applications}
As described above, the classes of orbit-equivalent representations of $\SU(2)$ and $\SO(3)$ are given by the partitions of the integer $n$, respectively with odd terms and multiples of four, and only odd terms.
In practice, these partitions are easily computed, for instance, using the algorithm of \cite{kelleher2009generating}.
\ref{item:step3} of Algorithm \ref{alg: 1} is put into practice by parsing this list of partitions.
This is illustrated in the seven examples below, for representations in the Euclidean space $\R^n$ with $n=5,7,11,6,8,9$ and $16$.

\begin{example}\label{ex:SO3_R5}
Using the explicit formulae for irreps of $\so(3)$ in Appendix \ref{app: irreducible reps su(2)}, we apply the algorithm for the data $X$ of $1500$ points created from an orbit of the irreducible representation of $\SO(3)$ in $\R^5$. Namely, the pushforward Lie algebra is spanned by
\begin{equation*}
    \left(\begin{matrix}
        0 & 0 & 0 & -1 & 0\\
        0 & 0 & 1 & 0 & 0 \\
        0 & -1 & 0 & 0 &0\\
        1 & 0 & 0 & 0 &-\sqrt{3}\\
        0 & 0 & 0 & \sqrt{3} & 0
    \end{matrix}\right),~~~
    \left(\begin{matrix}
        0 & 0 & 1 & 0 & 0\\
        0 & 0 & 0 & 1 & 0 \\
        -1 & 0 & 0 & 0 &-\sqrt{3}\\
        0 & -1 & 0 & 0 &0\\
        0 & 0 & \sqrt{3} & 0 & 0
    \end{matrix}\right),~~~\mathrm{and}~~~
    \left(\begin{matrix}
        0 & -2 & 0 & 0 & 0\\
        2 & 0 & 0 & 0 & 0 \\
        0 & 0 & 0 & -1 &0\\
        0 & 0 & 1 & 0 &0\\
        0 & 0 & 0 & 0 & 0
    \end{matrix}\right).
\end{equation*}
The left side of Figure \ref{fig: type j2} indicates the result of spectral analysis applied to this data, which makes it clear that a three-dimensional Lie algebra generates it, making $\su(2)$ a good candidate. Through \ref{item:step3}, we find that the three non-trivial representations of $\SU(2)$ in $\R^5$, associated with the partitions $(5)$, $(1,1,3)$ and $(3,4)$, have respective cost $0.001$, $0.011$ and $0.015$, confirming that the point cloud has been sampled on the irrep of $\SO(3)$.
Through \ref{item:step4}, we obtain the distance $\HDmid{X}{\widehat{\O}_x}\approx 0.088$, evidence of a successful run of the algorithm.
We give the graph of the estimated Hausdorff distances for different sample sizes on the right side of Figure \ref{fig: type j2}.

\begin{figure}[ht]
\centering
\begin{minipage}{0.49\linewidth}\center
\includegraphics[width=0.75\textwidth]{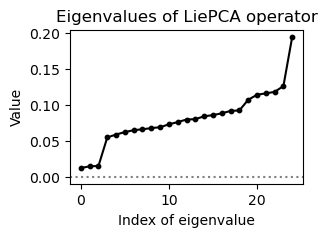}
\end{minipage}
\begin{minipage}{0.49\linewidth}\center
\includegraphics[width=0.75\textwidth]{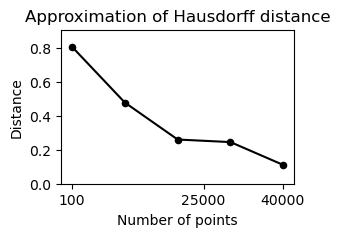}
\end{minipage}
\caption{Illustrations of Example \ref{ex:SO3_R5}. \textbf{Left:} eigenvalues of the \texttt{LiePCA} operator. \textbf{Right:} estimated Hausdorff distance $\HDmid{X}{\widehat{\O}_x}$ for varying numbers of samples.}\label{fig: type j2}
\end{figure}
\end{example}

\begin{example}\label{ex:rep_SU2_R7}
The non-trivial representations of $\SU(2)$ in $\R^{7}$ are given by the partitions 
$$
(1, 1, 1, 1, 3),~
(1, 1, 1, 4),~
(1, 1, 5),~
(1, 3, 3),~ 
(3, 4),
~\mathrm{and}~
(7).
$$
Only the last two representations yield orbits that span the ambient space $\R^{7}$.
We sample $3000$ points on them, denote the point clouds $X_{(3,4)}$, and $X_{(7)}$, and run Algorithm \ref{alg: 1}.
The \texttt{LiePCA} operator, in Figure \ref{fig:repsR7_LiePCA}, exhibits respectively four and three small eigenvalues.
According to the results of \ref{item:step3} in Table \ref{table:rep_SU2_R7}, the algorithm has succeeded in identifying the representations.
\ref{item:step4} yields estimated orbits for $X_{(3,4)}$ and $X_{(7)}$ with Hausdorff distances approximately $0.096$ and $0.168$.
These values can be considered small, as seen in Figure \ref{fig:repsR7_orbits}.

\begin{table}[ht]\center
\begin{tabular}{||r | c | c | c| c| c| c||} 
 \hline
Representation &  $(1, 1, 1, 1, 3)$ &$(1, 1, 1, 4)$ &$(1, 1, 5)$ & $(1, 3, 3)$ & $(3, 4)$ &$(7)$
\\\hline Cost for $X_{(3,4)}$ &  $0.008$ &$0.0.013$ &$0.003$ &$0.003$ &\bm{$9\times10^{-6}$} & $0.005$
\\\hline Cost for $X_{(7)}$ & $0.013$ & $0.014$& $0.010$ & $0.012$& $0.012$& \bm{$0.008$} 
\\\hline
\end{tabular}
\caption{Results of \ref{item:step3} in Example \ref{ex:rep_SU2_R7}. The best score in each raw is shown in bold.}
\label{table:rep_SU2_R7}
\end{table}

\begin{figure}[ht]
\centering
\begin{minipage}{0.49\linewidth}\center
\includegraphics[width=0.75\textwidth]{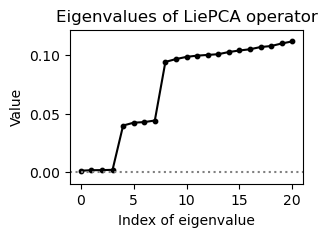}
\end{minipage}
\begin{minipage}{0.49\linewidth}\center
\includegraphics[width=0.8\textwidth]{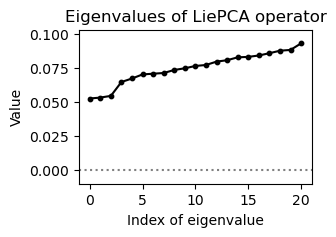}
\end{minipage}
\caption{Illustrations of Example \ref{ex:rep_SU2_R7}. Eigenvalues of \texttt{LiePCA} on orbits of the representations $(3,4)$ and $(7)$ of $\SU(2)$ in $\R^7$ (restricted to skew-symmetric matrices for visualization ease).}\label{fig:repsR7_LiePCA}
\end{figure}

\begin{figure}[ht]
\centering
\begin{minipage}{0.49\linewidth}\center
\includegraphics[width=0.8\textwidth]{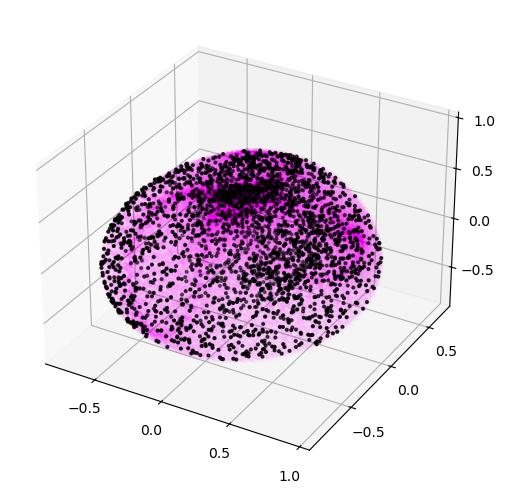}
\end{minipage}
\begin{minipage}{0.49\linewidth}\center
\includegraphics[width=0.8\textwidth]{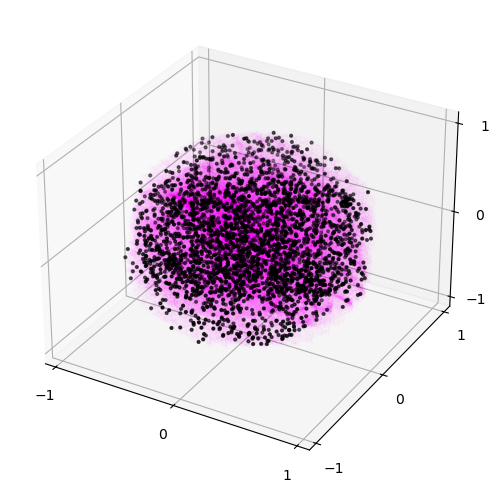}
\end{minipage}
\caption{Illustrations of Example \ref{ex:rep_SU2_R7}. \textbf{Left:} a sample of an orbit of the representation $(3,4)$ of $\SU(2)$ in $\R^7$ (black) and the estimated orbit by Algorithm \ref{alg: 1} (magenta).  \textbf{Right:} same for the representation of partition $(7)$.}\label{fig:repsR7_orbits}
\end{figure}
\end{example}

\begin{example}
We reproduce Example \ref{ex:rep_SU2_R7}, now in $\R^{11}$. The representations of $\SU(2)$ are given by the following admissible partitions of $11$, where we omit the $1$'s:
\begin{align*}
&(3),~
(4),~
(5),~ 
(3, 3),~
(3, 4),~ 
(7),~ 
(3, 5),~
(4, 4),~
(8),~
(3, 3, 3),~
(4, 5),~\\
&(9),~
(3, 3, 4),~
(3, 7),~
(5, 5),~
(3, 3, 5),~
(3, 4, 4),~
(3, 8),~
(4, 7),~
~\mathrm{and}~
(11).
\end{align*}
We sample $5000$ points on orbits of the last three representations.
In each case, Algorithm~\ref{alg: 1} outputs the correct representation, with Hausdorff distances respectively $0.077$, $0.107$, and $0.184$.
Repeating the same experiment with drawing only $3000$ points yields $0.122$, $0.182$ and $0.239$.
\end{example}

\begin{example}\label{ex:rep_SU2_R6}
Let $\O$ be the set of $2 \times 3$ matrices with orthonormal rows, embedded in $\R^6$ by flattening.
We create a set $X$ of $3000$ points uniformly sampled on $\O$, on which we apply Algorithm \ref{alg: 1}.
Through \ref{item:step1}, we observe that the covariance matrix of $X$ has eigenvalues between $0.159$ and $0.174$, hence, we do not perform dimension reduction. 
The \texttt{LiePCA} operator, computed in \ref{item:step2}, presents three significantly small eigenvalues.
Among the four non-trivial representations of $\SU(2)$ in $\R^6$, \ref{item:step3} detects the partition $(3,3)$ as that yielding the minimal cost.
It corresponds to a representation of $\SO(3)$.
Finally, \ref{item:step4} outputs the distance $\HDmid{X}{\widehat{\O}_x} \approx 0.091$, witnessing a good fit between the point cloud and the estimated orbit.
These results are coherent with the fact that the initial set $\O$ is indeed an orbit of a representation of $\SO(3)$ in $\R^6$, namely, that it comes from the action of $\SO(3)$ on $\M_{2,3}(\R)$ by simultaneous multiplication of the rows.
The orbit is the Stiefel manifold $\V(2,\R^3)$, homeomorphic to $\SO(3)$.
\end{example}

\begin{example}\label{ex:rep_SU2_R8}
We reproduce Example \ref{ex:rep_SU2_R6} in higher dimension: $\O$ now is the Stiefel manifold $\V(2,\R^4)$ seen in $\M_{2,4}(\R)$, i.e., it is the $2\times 4$ matrices with orthonormal rows. It carries a transitive action of $\SO(4)$.
We generate a $10,000$-sample $X$.
By applying Algorithm \ref{alg: 1} with $G = \SO(3)$, we expect to identify a sub-action induced by an inclusion $\SO(3)\hookrightarrow\SO(4)$.
Up to orbit-equivalence, this group admits five distinct non-trivial representations in $\R^8$, indexed by
\[
(1, 1, 1, 1, 1, 3), ~(1, 1, 1, 5), ~(1, 7), ~(1, 1, 3, 3), ~~\mathrm{and}~~ (3, 5).
\]
\ref{item:step3} recognizes $(1, 1, 3, 3)$ as the optimal representation.
By choosing an arbitrary point $x\in X$, \ref{item:step4} yields the distance $\HDmid{X}{\widehat{\O}_x}\approx 1.838$, showing that the orbit does not approximate $X$ correctly. 
This is expected, since $\SO(3)$ does not act transitively on $\V(2,\R^4)$ (it has dimension 5).
The opposite Hausdorff distance is smaller: $\HDmid{\widehat{\O}_x}{X}\approx0.571$.
This suggests that we have generated a \textit{subset} of $\O$.
To verify this statement, we generate the set $\widehat{\O} = \bigcup_{x\in X}\widehat{\O}_x$.
The fitting is now correct: $\HDmid{X}{\widehat{\O}} = 0$ and $\HDmid{\widehat{\O}}{X}\approx 0.642$.
Note that this latter value is not expected to be significantly small, since, differently from $\HDmid{X}{\widehat{\O}}$, it also reflects the sampling error in $X$.
We emphasize that this and Example \ref{ex:rep_SU2_R16} have a major distinction from the other examples: the underlying actions are non-transitive. 
In particular, the Hausdorff distances obtained are not comparable.
While we will propose, in Section \ref{subsubsec:statistics}, a systematic analysis of the values of Hausdorff distances in the case of transitive actions, we did not push the study further for non-transitive ones.
\end{example}

\begin{example}\label{ex:rep_SU2_R9}
In the same vein as Examples \ref{ex:rep_SU2_R6} and \ref{ex:rep_SU2_R8}, let $\O$ be the $3\times3$ special orthogonal matrices, embedded in $\R^9$, and $X$ a $3000$-sample set.
Since $\SO(3)$ acts transitively on itself, $\O$ can be seen as an orbit of it.
While running Algorithm \ref{alg: 1}, \ref{item:step2} shows a \texttt{LiePCA} operator with six almost-zero eigenvalues.
This coincides with the dimension of the isometry group of $\SO(3)$ given in Section \ref{subsec:symmetry_group}, equal to $\SO(3)\rtimes\SO(3)\times\{\pm 1\}$.
Among the twelve non-trivial representations of $\SU(2)$ in $\R^9$, \ref{item:step3} yields the costs given in Table \ref{table:rep_SU2_R9}, where we omit the $1$'s in the partitions:
in particular, the optimum is given by the partition $(3, 5)$.
However, this Lie algebra does not seem to generate $X$: the Hausdorff distance, equal to $\HDmid{X}{\widehat{\O}_x} \approx 2.658$, is large. 
In comparison, the distance from the orbit to $X$ is small: $\HDmid{\widehat{\O}_x}{X} \approx 0.543$. 
This indicates that the representation $(3, 5)$ yields an orbit that is only a subset of $\O$.
As seen in the previous table, another good candidate is the representation $(3,3,3)$.
We run \ref{item:step4} again, now using the optimal Lie algebra computed for this latter representation.
The set $X$ is now well approached, since $\HDmid{X}{\widehat{\O}_x} \approx 0.061$.
These results suggest that we have detected two distinct linear actions of $\SO(3)$ on itself: first, the action by conjugation, which is not transitive; and second, the action by left or right multiplication, which is.

\begin{table}[ht]\center
\begin{tabular}{||r | c | c | c| c| c| c||} 
 \hline
Representation &  $(3, 5)$ &$(3,3,3)$ &$(4, 5)$ & $(8)$ & $(5)$ &$(7)$
\\\hline Cost &  \bm{$2\times 10^{-5}$} &\bm{$4\times 10^{-5}$} &$0.001$ &$0.001$ &$0.03$ & $0.004$
\\\hline Representation &  $(9)$ &$(3,3)$ &$(3,4)$ & $(4,4)$ & $(3)$ &$(4)$
\\\hline Cost & $0.004$ & $0.006$& $0.007$ & $0.009$& $0.011$& $0.013$
\\\hline
\end{tabular}
\caption{Results of \ref{item:step3} in Example \ref{ex:rep_SU2_R9}. The two best scores are shown in bold.}
\label{table:rep_SU2_R9}
\end{table}
\end{example}

\begin{example}\label{ex:rep_SU2_R16}
For our last example, we consider the embedding $\G(2,\R^4) \hookrightarrow \R^{4\times 4}$ of the Grassmann manifold into the Euclidean space via projection matrices, and sample $5000$ points on it.
Just as it was the case for the Stiefel manifold $\V(2,\R^4)$ in Example \ref{ex:rep_SU2_R8}, the group $\SO(4)$ acts transitively on $\G(2,\R^4)$.
By applying Algorithm \ref{alg: 1} with $G=\SO(3)$, we identify $(1,3,5)$ as the optimal representation. 
However, this action cannot be transitive, since $\G(2,\R^4)$ has dimension 4.
Instead of computing $\widehat{\O}_x$ for a certain data point $x$, we consider their union $\widehat{\O}$.
We eventually get $\HDmid{\widehat{\O}}{X} \approx 0.453$, a value we consider small, and indicating that we have found, indeed, a non-transitive action.
\end{example}

\subsection{Additional algorithmic considerations}\label{subsec:additional}

\subsubsection{Application of the algorithm for a list of groups}\label{sec: list}
Up to now, we have considered a model in which the compact Lie group $G$ with a representation that generates the data is known beforehand. As we will illustrate in some of the examples of Section \ref{sec:applications}, it is not impossible to conceive applications in which, although the precise representation is unknown, the compact Lie group responsible for it might be correctly guessed through the very statement of the task.

However, there are situations, especially in more exploratory data science problems, in which we might suspect the data to lie within a representation of a compact Lie group, but where the exact group is unknown. In these cases, we might be interested in determining not only the representation $\phi$ but $G$ itself, from a list of potential candidates $\{G_1,\dots,G_k\}$. \texttt{LiePCA}, as established in \cite{DBLP:journals/corr/abs-2008-04278}, is unable to solve this task because it fails to estimate $\mathfrak{h}$ as a Lie algebra. On the other hand, the last steps of Algorithm \ref{alg: 1} can be used to tackle this problem. 

In this context, successive applications of {\texttt{LieDetect}} can be used to determine which group of the list is isomorphic to $G$: for each $G_i$, $1\leq i\leq k$, run Algorithm \ref{alg: 1} and save the Hausdorff and/or Wasserstein distance estimated at the end of \ref{item:step4}. Then, the most likely group of $\{G_1,\dots,G_k\}$ to generate the orbit $\mathcal{O}$ is arguably the one that attains the smallest distance. 

Two remarks are worth mentioning here. Even though the above rationale is valid, we might still worry about cases in which two non-isomorphic compact Lie groups share some representations in $\R^n$. For example, any compact Lie group $G$ of more than one connected component will have the same representations (up to orbit-equivalence) of its principal connected component. Moreover, even if we restrict our attention to connected compact Lie groups, this sort of ill-posedness might still happen. This is the case, for example, of the groups $\SU(2)$ and $\SO(3)$ which, although non-isomorphic, have the same irreducible representation on spaces $\R^n$ for odd $n$. In this case, notice that the two Lie groups involved have the same Lie algebra, so the set of representations of one needs to be contained in the set of representations of the other. Another example is when the two groups, although having different representations in $\R^n$, have representations that are orbit-equivalent. That happens, for instance, with any two representations of $\SO(2)$ and $T^2$ in $\R^2$. Unfortunately, situations like these are of no rescue: no algorithm will ever be able to tell which of the two groups is more likely to generate the dataset from the point cloud only. Therefore, when considering such a list of candidates, or the output distances $d_i$, such subtleties must be kept in mind.

The second remark has to do with the choice of the list $\{G_1,\dots,G_k\}$ itself. Again, we are here assuming that the precise $G$ is unknown, however, as reviewed in Section \ref{subsec:step2}, an upper bound of its \textit{dimension} $d$ can be estimated through the \texttt{LiePCA} operator $\Sigma$ by identifying the number of very small eigenvalues it has. Let us denote this estimation by $\widehat{d}$. Then, because we are restricting the analysis to compact Lie groups only, the list of possible candidates to generate the orbit, that is, the list of Lie groups of dimension at most $\widehat{d}$, becomes finite. In fact, as it is well-known in literature (see Theorem 4 in Section 10.7 of \cite{procesi2007lie}), the full list of all connected compact Lie groups is given by the three classes of classical Lie groups (the fundamental groups of $\SO(n)$, the groups $\SU(n)$ and the compact symplectic groups), the five `exceptional Lie groups', their finite direct products, their finite covers, and their quotients by finite central subgroups. Because of the redundancy of representations of finite covers, quotients, and connected components (up to orbit-equivalence), we have that, for every $\widehat{d}$, there is but a finite list of compact Lie groups of dimension at most $\widehat{d}$. Table \ref{table:groups_per_dimension} shows all compact Lie groups of dimension up to $5$ capable of generating representations with different orbits, that is, this is the list of all possible groups up to dimension 5 that can be distinguished from one another with $\texttt{LieDetect}$.

\begin{table}[ht]\center
\begin{tabular}{||c | c ||} 
\hline
Dimension &  Compact Lie groups
\\\hline 1 & $\SO(2)$
\\\hline 2 & $T^2$
\\\hline 3 & $\SU(2),T^3$
\\\hline 4 & $\SO(2)\times \SU(2),T^4$
\\\hline 5 & $T^2\times \SU(2),T^5$
\\\hline
\end{tabular}
\caption{List of compact Lie groups of small dimension.}
\label{table:groups_per_dimension}
\end{table}

We notice that, except for the products, the set of real irreps of all compact Lie groups in Table \ref{table:groups_per_dimension} were worked out in this section. One might imagine that the classifications of irreps for the direct products could be achieved from the irreps of the terms, as was the case for $T^d$. That is, one might guess that if $G_1$ and $G_2$ are compact Lie groups then any irrep of $G_1\times G_2$ equals $\phi_1\otimes \phi_2$, where $\phi_1$ and $\phi_2$ are irreps of $G_1$ and $G_2$, respectively. 
Unfortunately, although the above result is true, it is only for representations of Lie groups over closed fields, as it depends on Schur's lemma. Consequently, in our scenario of real representations, one may only assume that $\text{irrep}(G_1)\otimes \text{irrep}(G_2)\subset \text{irep}(G_1\times G_2)$, meaning that, if Algorithm \ref{alg: 1} is to be applied to products of Lie groups that have already their representations worked out---such as $\SO(2)$, $\SU(2)$, and $T^d$---we need to first derive their sets of irreps using other techniques, a task that we do not further pursue here.

With the caveat regarding the direct products of compact Lie groups in mind, the model discussed in this section can still be applied to exploratory machine learning, as Examples \ref{ex: distinguish SU(2) and T^3} and \ref{ex: Mobius} illustrate on synthetic data. 
In addition, it will be applied to real data in Section~\ref{subsec:conformational_space}.

\begin{example}\label{ex: distinguish SU(2) and T^3} 
We generate a point cloud $X$ of 1500 points on an orbit of the representation $(1,5)$ of $\SU(2)$ in $\R^6$. 
As shown in Figure \ref{subfig:distinguish_groups_1}, the \texttt{LiePCA} operator exhibits three significantly small eigenvalues.
Supposing we do not know a priori that the acting group is $\SU(2)$, but only that its dimension is $3$, we read on Table \ref{table:groups_per_dimension} two candidates: $\SU(2)$ and $T^3$.
We first apply our algorithm to $T^3$. 
Up to orbit-equivalence, only one almost-faithful representation on $\R^3$ exists.
After application of \ref{item:step4}, we obtain an orbit $\widehat{\O}_x$ of $T^3$, which is not close to $X$: the Hausdorff distance is $\HDmid{X}{\widehat{\O}_x} \approx 0.7513$ (see Figure \ref{subfig:distinguish_groups_2}).
On the other hand, when applied with the group $\SU(2)$, the algorithm detects, among five representations, $(1,5)$ as the optimal one and generates a close orbit: $\HDmid{X}{\widehat{\O}_x} \approx 0.0628$ (see Figure \ref{subfig:distinguish_groups_3}).

This example can be repeated in higher dimensions: let us now take a sample of 5000 points on the orbit $(3,5)$ of $\SU(2)$ in $\R^8$. 
We observe three small eigenvalues of the \texttt{LiePCA} operator and apply the algorithm with $T^3$ and $\SU(2)$. 
For the former group, we obtain a Hausdorff distance of approximately $1.0554$ (we test representations up to frequency $\omega_\mathrm{max}=2$), while for the latter, only $0.1084$ (nine non-orbit-equivalent representations are tested).

\begin{figure}[ht]\center
\begin{subfigure}{.31\textwidth}\center
\includegraphics[width=.99\linewidth]{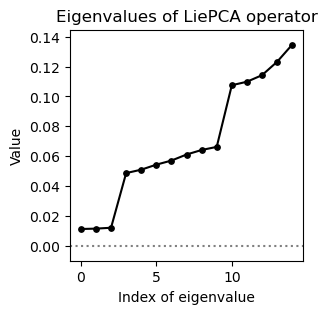}
\caption{Eigenvalues of \texttt{LiePCA} (restricted to skew-sym.~matrices).}
\label{subfig:distinguish_groups_1} 
\end{subfigure}
~
\begin{subfigure}{.31\textwidth}\center
\includegraphics[width=.99\linewidth]{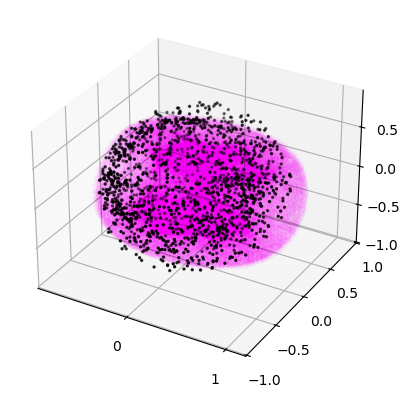}
\caption{Orbit generated by the algorithm for the group $T^3$.}
\label{subfig:distinguish_groups_2} 
\end{subfigure}
~
\begin{subfigure}{.31\textwidth}\center
\includegraphics[width=.99\linewidth]{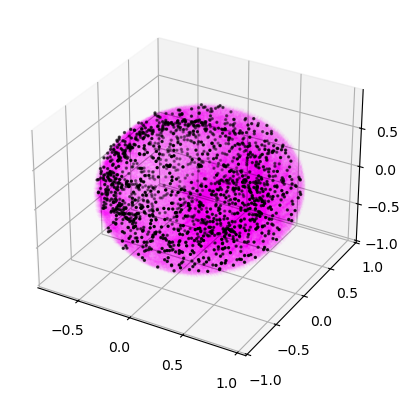}
\caption{Orbit generated by the algorithm for the group $\SU(2)$.}
\label{subfig:distinguish_groups_3} 
\end{subfigure}
\caption{Illustrations of Example \ref{ex: distinguish SU(2) and T^3}. The input point cloud is sampled on an orbit of $\SU(2)$ in $\R^8$. By inspecting the number of small eigenvalues of \texttt{LiePCA}, we deduce two potential acting groups: $\SU(2)$ and $T^3$.}
\end{figure}
\end{example}

\begin{example}\label{ex: Mobius}
Let us sample $500$ points on the Möbius strip in $\R^4$, defined as
$$
\left\{\left(\cos\theta,\sin\theta,r\cos\frac{\theta}{2},r\sin\frac{\theta}{2}\right) \mid \theta\in[0,2\pi),
~r \in [-1,1]
\right\}.
$$
The \texttt{LiePCA} operator, visualized in Figure \ref{subfig:Mobius_1}, clearly shows two eigenvalues close to zero.
However, the Möbius strip does not support any almost-faithful action of $T^2$. 
As a consequence, applying our algorithm with the group $T^2$ generates an orbit that does not fit the point cloud correctly (see Figure \ref{subfig:Mobius_2}).
It admits, however, an action of $\SO(2)$, which we estimate with our algorithm.
The obtained action is not transitive, since the orbits are one-dimensional.
However, their collection accurately covers the Möbius strip.
To quantify this, we compute, for each point $x\in X$ of the point cloud, the orbit $\widehat{\O}_x$ generated by this point, and the Hausdorff distance $\HDmid{\widehat{\O}_x}{X}$ from the orbit to $X$.
The largest value obtained this way is relatively small: $0.1586$.

\begin{figure}[ht]\center
\begin{subfigure}{.31\textwidth}\center
\includegraphics[width=.99\linewidth]{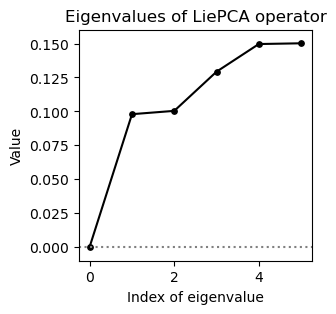}
\caption{Eigenvalues of \texttt{LiePCA} (restricted to skew-sym.~matrices).}
\label{subfig:Mobius_1} 
\end{subfigure}
~
\begin{subfigure}{.31\textwidth}\center
\includegraphics[width=.99\linewidth]{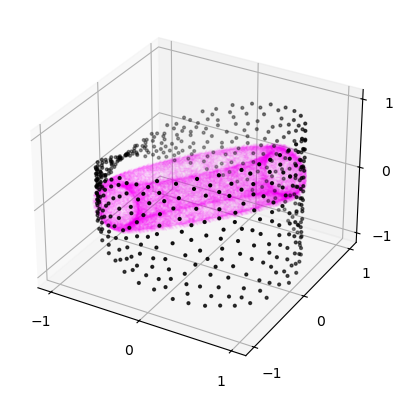}
\caption{Orbit generated by the algorithm for the group $T^2$.}
\label{subfig:Mobius_2} 
\end{subfigure}
~
\begin{subfigure}{.31\textwidth}\center
\includegraphics[width=.99\linewidth]{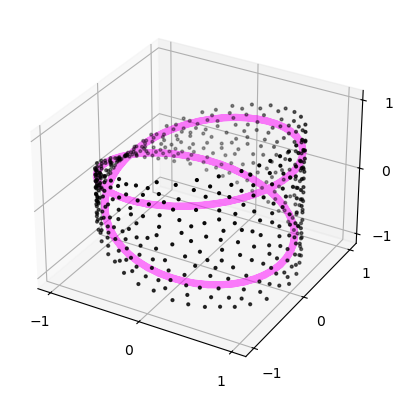}
\caption{Orbit generated by the algorithm for the group $\SO(2)$.}
\label{subfig:Mobius_3} 
\end{subfigure}
\caption{Illustrations of Example \ref{ex: Mobius}. The input is sampled on the Möbius strip in $\R^4$. 
While \texttt{LiePCA} detects an action of $T^2$, only $\SO(2)$ acts almost-faithfully on it.}
\end{figure}
\end{example}

\begin{remark}
At this point in the text, it is important to stress that {\texttt{LieDetect}} can return actions that are not transitive. In addition to the dataset studied above, this has already been encountered in Examples \ref{ex:rep_SU2_R8} and \ref{ex:rep_SU2_R9}, and will also occur in Section \ref{subsec:conformational_space} when we study chemical data.
Even though we have chosen the theoretical framework of transitive actions to study our algorithm, non-transitive ones are equally detectable.
This comes from the fact the \texttt{LiePCA} operator (that we will study theoretically in Section \ref{subsubsec:lie-pca_consistency}) is a consistent estimator of the symmetry algebra $\sym(\O)$ and that this object, as one sees from its definition in Section \ref{subsec:symmetry_group}, is blind to the fact the group acts transitively or not.
However, in all these cases, transitivity is checked through \ref{item:step4} of our algorithm. If this item fails, then one can, as done above, generate the orbits from various points $x\in X$, and check whether they are close to the point cloud, through the Hausdorff distance $\HDmid{\widehat{\O}_x}{X}$. If these distances are small, this indicates that a correct action, although non-transitive, has been estimated.
\end{remark}

\subsubsection{Typical distance between orbits}\label{subsubsec:statistics}

For most experiments so far, the result of our algorithm, measured by the Hausdorff distance $\HDmid{X}{\widehat{\O}_x}$ in \ref{item:step4}, has been judged rather arbitrarily.
In this section, we would like to sketch out a comprehensive analysis of this value.
Our model, which has been briefly given at the beginning of Section~\ref{sec:algorithm_description} and will be formalized further in Section~\ref{sec:algorithm_analysis}, is that of a point cloud $X$, coming as an $N$-sample of the uniform measure $\mu_\O$---or a measure close to it---on an orbit $\O$ of $G$.
Supposing that $\O$ is a $l$-dimensional submanifold, it is a well-known fact of probability theory that the symmetric Hausdorff distance $\HD{X}{\O}$ tends in probability to 0, with speed $(\log(N)/N)^l$ \cite{niyogi2008finding}. 
In particular, $\widehat{\O}_x$, the output orbit of the algorithm, cannot be closer than that, when measured with the distance $\HDmid{\widehat{\O}_x}{X}$.
This result, however, does not help us to understand the other distance, $\HDmid{X}{\widehat{\O}_x}$.

We, therefore, seek to provide some empirical clues for the determination of explicit thresholds on $\HDmid{X}{\widehat{\O}_x}$, at which the algorithm can be regarded as having succeeded or failed.
To simplify the situation, we can focus on the distance between two orbits originating from the same point $x\in\R^n$, for two representations $\phi_1,\phi_2$ of the same group $G$ in $\R^n$.
Since, in \texttt{LieDetect}, the data goes through an orthonormalization process, we can restrict our study to that of orthogonal representations and points of norm 1.
Note that this reduction is quite crude since it omits the influence of noise in the data. However, it would give us an initial insight.

Accordingly, denote $\widehat{\O}_x^1$ and $\widehat{\O}_x^2$ the orbits of $x$ under $\phi_1$ and $\phi_2$.
Their difference, measured by the symmetric Hausdorff distance, $\HD{\widehat{\O}_x^1}{\widehat{\O}_x^2}$, or the non-symmetric one, $\HDmid{\widehat{\O}_x^1}{\widehat{\O}_x^2}$, can be caused by two effects: either the representations are non-equivalent, and then a large distance is expected, or they are, and the Hausdorff distance depends continuously on a change of basis that sends the first to the second.
Since we are primarily interested here in estimating the representation types, we will focus on the first effect.
In particular, if we knew the minimal value reached by $\HDmid{\widehat{\O}_x^1}{\widehat{\O}_x^2}$ when $(\phi_1,\phi_2)$ covers all pairs of non-orbit-equivalent representations of $G$ in $\R^n$, then we would have a relevant reference for evaluating the result of our algorithm. 

As a last simplification, rather than considering all pairs in this way, we will choose a single representation $\phi$ for each class of orbit equivalence.
That is, we chose a set of representatives $\mathrm{OrbRep}(G,n)$, as defined in Section \ref{subsec:grassmann_lie}.
If $k$ denotes the number of classes, then we have $k(k-1)/2$ pairs, and we compute the Hausdorff distance to the orbits they generate.
On the other hand, when the group is Abelian, this number is infinite, and we must restrict ourselves to a maximum frequency. For representations of $\SO(2)$ in $\R^{2m}$, we choose to go up to frequency $2m$, and for the tori $T^2$ and $T^3$, up to frequency 2 and 1, respectively, since they cover all the cases encountered in this article.
The results of this operation are shown in Figure \ref{fig:statistics}.
We stress that, for $n\leq 2d+1$, the torus $T^d$ admits at most $1$ almost-faithful representation in $\R^n$, up to orbit-equivalence, thus making trivial the problem of identifying this representation.
Consequently, these dimensions are omitted. 
Similarly, $\SU(2)$ is considered only in dimensions for which it admits several orbit-equivalence classes of representations.
Lastly, we point out that $\SO(3)$ is not represented, since all its representations are already featured by $\SU(2)$. 

\begin{figure}[ht]\center
\includegraphics[width=.49\linewidth]{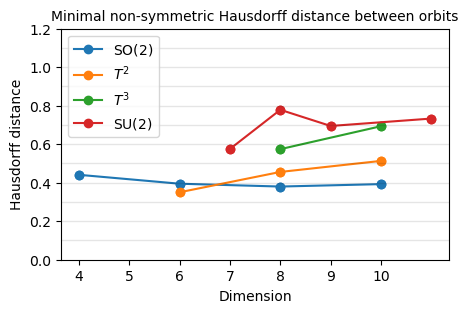}
\includegraphics[width=.49\linewidth]{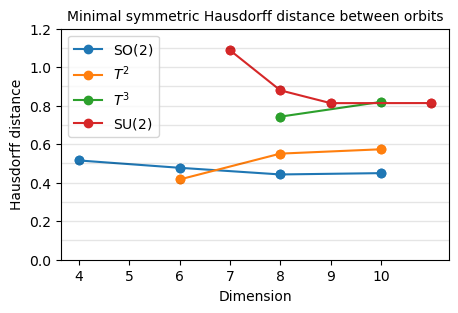}
\caption{\textbf{Left:} empirical estimation of the minimal non-symmetric Hausdorff distance $\HDmid{\widehat{\O}_x^1}{\widehat{\O}_x^2}$ between two orbits of a same initial point $x$ for two non-orbit equivalent representations $\phi_1,\phi_2$ of a compact Lie group $G$ in $\R^n$. The minimal value is approximately $0.35$.
\textbf{Right:} same for the symmetric Hausdorff distance $\HD{\widehat{\O}_x^1}{\widehat{\O}_x^2}$. The minimal value is $0.42$.}
\label{fig:statistics}
\end{figure}

As seen on the figure, all the non-symmetric distances $\HDmid{\widehat{\O}_x^1}{\widehat{\O}_x^2}$ computed during our experiment were greater than $0.35$, and the symmetric distances $\HD{\widehat{\O}_x^1}{\widehat{\O}_x^2}$ greater than $0.42$ regardless of group or dimension.
As a conclusion, we draw the following interpretation: if \texttt{LieDetect}, applied to a point cloud $X$, outputs an orbit that verifies $\HDmid{X}{\widehat{\O}_x}<0.35$, then it is reasonable to assume that the algorithm has succeeded.
Besides, if the symmetric distance satisfies $\HD{X}{\widehat{\O}_x}<0.42$, then we can draw the same conclusion (this can be useful when $X$ has been well sampled).
Of course, in cases where we know that $X$ has been drawn from a specific representation, we must check that the output representation is orbit-equivalent to it.
This rule of thumb comes in handy precisely when no additional information on $X$ is available, such as in the datasets we will explore in Section \ref{sec:applications}.

\subsubsection{Running time and convergence}\label{subsubsec:running_times}
To demonstrate the technical performance of our algorithm, Table \ref{tab:running_times} shows the running times for complete executions, from \ref{item:step1} to \ref{item:step4}.
More precisely, we apply the algorithm for the groups $\SO(2)$, $T^2$, $T^3$, and $\SU(2)$, from point clouds of size 250, 500, 1000, and 1000, respectively.
These point clouds are sampled randomly from a representation orbit of the group in $\R^n$, with $n$ up to 10 for the Abelian group, and up to 12 for $\SU(2)$.
For the Abelian groups, we considered representations with frequencies up to $n$, 2, and 1, respectively, while for $\SU(2)$, all the orbit-equivalence classes of representations in $\R^n$ are tested.
We stress that, since no reductions have been obtained for $\SU(2)$, we need to perform gradient descent optimization on $\Ort(n)$, as in \ref{item:step3}, while the $\SO(2)$ group is treated with the reformulation of Section \ref{sec: algorithm so(2) simp}, and the higher-dimensional tori with that of Section \ref{sec: algorithm torus simp}.
The experiment was repeated several times (100 for Abelian groups and 10 for $\SU(2)$), and the average execution times are reported. In addition, we indicate the proportion of successful runs of the algorithm, i.e., runs where the representation type was exactly identified.
Note that for dimensions 4, 5, and 10, $\SU(2)$ only admits one representation, up to orbit-equivalence, hence the algorithm automatically succeeds.
These dimensions are still given in the table: although the type of representation is fixed, the algorithm has yet to find its orientation.

As expected, the table reveals a significant discrepancy between the running times of the Abelian groups and of $\SU(2)$.
The algorithm for $\SO(2)$ is the fastest since it is implemented as a simple Schur decomposition, in Lemma \ref{lem:step3'_simplification_so(2)}.
Similarly, $T^2$ and $T^3$ are performed through a simultaneous reduction of two or three matrices, as in Lemma \ref{lem:step3'_simplification_torus}.
Although implemented as an optimization over $\Ort(n)$, via projected gradient descent, the problem is still quickly solved.
Similarly, $\SU(2)$ is implemented as a series of optimizations over $\Ort(n)$, but in this case, the function, given in Equation \eqref{eq:minimization_stiefel_2}, is much more complicated to minimize, resulting in longer running times.
We point out that the default parameters of \texttt{Pymanopt} have been used, namely, the initial point of $\Ort(n)$ for the optimization is chosen randomly, and the stopping criterion is to reach a small gradient norm.
The problem of reducing this computation time is reserved for future work, and we will use, in Section \ref{sec:applications}, the algorithm as it is.

\begin{table}[ht]\centering
\subfloat[$\SO(2)$]{
\begin{tabular}{||r|cccc||} 
\hline
Dimension &  4 & 6 & 8 & 10 
\\\hline 
Running time & 0.04s & 0.05s & 0.08s & 0.14s
\\\hline
Success & 100.0\% & 100.0\% & 100.0\% & 100.0\%
\\\hline
\end{tabular}}
\vspace{.33cm}\\
\subfloat[$T^2$]{
\begin{tabular}{||r|ccc||} 
\hline
Dimension &  6 & 8 & 10 
\\\hline 
Running time & 0.24s & 0.63s & 4.03s
\\\hline
Success & 82.0\% & 100.0\% & 98.0\%
\\\hline
\end{tabular}}
\\
\vspace{.5cm}
\subfloat[$T^3$]{
\begin{tabular}{||r|cc||} 
\hline
Dimension &  8 & 10 
\\\hline 
Running time & 1.44s & 5.98s
\\\hline
Success & 100.0\% & 100.0\%
\\\hline
\end{tabular}}
\vspace{.33cm}\\
\subfloat[$\SU(2)$]{
\begin{tabular}{||r|cccccc||} 
\hline
Dimension & 4 & 5 & 7 & 8 & 9 & 10 
\\\hline 
Running time & 0.6s & 5.04s & 4min 21s & 13min 7s & 16min 9s & 10min 53s
\\\hline
Success &  100.0\% & 100.0\% & 90.0\% & 100.0\% & 100.0\% & 100.0\%
\\\hline
\end{tabular}}
\caption{Running time (in seconds or minutes) and success rate (percentage) of full execution of \texttt{LieDetect}, as a function of the input group, and the dimension of the ambient Euclidean space. 
The input of the algorithm is a point cloud sampled from the uniform measure on an orbit chosen randomly.
For the Abelian groups $\SO(2)$, $T^2$, and $T^3$, the representations are considered up to a maximal frequency (described in the main text), 100 runs of the algorithm are performed, and the results are averaged. For $\SU(2)$, 10 runs have been performed.}
\label{tab:running_times}
\end{table}

\section{Theoretical guarantees}\label{sec:algorithm_analysis}

In this section, we derive theoretical results regarding Algorithm \ref{alg: 1}, guaranteeing that the correct representation will be found (up to orbit equivalence), and that the output $\widehat{\O}_x$ will be close to the underlying orbit $\O$.
The quality of these results depends on the proximity of the input $X$ to the orbit $\O$.
In the literature, two common measures of proximity are the Hausdorff distance $\HD{X}{\O}$ and the Wasserstein distance $\W_2(\mu_X, \mu_\mathcal{O})$ (defined in Sections \ref{sec:step4_hasdorff} and \ref{subsubsec:step4_wasserstein}).
Here, $\mu_X$ and $\mu_\mathcal{O}$ are measures associated with $X$ and $\mathcal{O}$, typically the empirical and uniform measures.
We found that the Wasserstein distance is better suited to analyze our algorithm.
Indeed, the whole algorithm depends on an estimation of normal spaces via local PCA, an operation that is stable for the Wasserstein, but not the Hausdorff distance.

Certain results of this section are stated for arbitrary probability measures $\mu$ and $\nu$ on $\R^n$, and others hold only with additional assumptions.
To ensure clarity, the notation $\mu_X$ will denote exclusively the empirical measure on $X$, and $\mu_\O$ the uniform measure on $\O$.
As recalled in Section \ref{subsubsec:exponential_map_geometric}, $\mu_\O$ can be seen as a pushforward of the Haar measure $\mu_G$, or, equivalently, as the restriction of the $l$-dimensional Hausdorff measure to it, where $l$ is its dimension.

Our investigation is organized as follows.
In Section \ref{subsec:PCA_preprocessing}, we raise the issue of dimensionality reduction with PCA in the context of Lie group orbits.
Next, we analyze \texttt{LiePCA} in Section~\ref{subsec:analysis_liepca}.
We gather, in Section \ref{subsec:stability-minimization}, guarantees regarding the minimization problems involved in the algorithm.
These results are brought together in Section \ref{subsec:robustness_algorithm}, giving Theorem \ref{th:robustness_algorithm}.

As a guideline, we can already formulate the result of the theorem (presented in a simplified form in Remark~\ref{rem:reformulation_theorem}): if the distance $\W_2\big(\mu_{X},\mu_\O\big)$ between the input point cloud and the underlying orbit (unknown) is small enough, then the algorithm returns a representation that is orbit-equivalent to that generating $\O$.
In addition, under a certain choice of parameters, the output orbit $\widehat{\O}_x$ generated from a point $x$ (defined in Section \ref{sec:step4_hasdorff}) and the output measure $\mu_{\widehat{\O}}$ (defined in Section \ref{subsubsec:step4_wasserstein}) satisfy the inequalities
\begin{align*}
\HD{\widehat{\O}_x}{\O}
&\leq \mathrm{constant}\cdot\left(\HDmid{X}{\O}+ \W_2\big(\mu_{X},\mu_\O\big)^{1/4(l+3)}\right),\\
\W_2\big(\mu_{\widehat{\O}}, \mu_{\O} \big)
&\leq \mathrm{constant}\cdot \W_2\big(\mu_{X},\mu_\O\big)^{1/4(l+3)},
\end{align*}
where the leftmost distances refer to the orthonormalized versions of the orbits.

\subsection{PCA and orthonormalization}\label{subsec:PCA_preprocessing}

As a general practice in data analysis, dimensionality reduction helps lower the computational cost of algorithms that are subsequently applied.
In the current context, however, applying dimension reduction not only allows us to speed up our algorithm, but is also a crucial procedure for ensuring its consistency, as we will explain in Remark \ref{rem:projection_mindimspace}.
Nevertheless, we shall not apply any general reduction algorithm: supposing that the dataset $X$ lies on or close to an orbit $\O$ of a representation of a Lie group $G$, we must ensure that the projected dataset still does.
In this section, we will show how to project $X$ into $\spn{\O}$, the linear subspace spanned by $\O$, via the common method of PCA.
More generally, one may want to reduce the dimension further.
The particular case of $G = \SO(2)$ and $T^d$ will be treated in Section \ref{subsec:pixel_permutations}.

Such a pre-processing step intervenes a second time in \ref{item:step1}, where the points are normalized to transform the representation into an orthogonal one.
This step, similarly to PCA, involves the computation of the covariance matrix of $X$.
To prove stability results for the Wasserstein distance, we should use a definition of PCA adapted to measures.
In this regard, we define the covariance matrix of a square-integrable measure $\mu$ on $\R^n$ as 
\begin{equation}\label{eq:covariance_matrix}
\Sigma[\mu] = \int x x^\top \d \mu(x) = \int \|x\|^2 \projbracket{\spn{x}} \d \mu(x),
\end{equation}
where $x x^\top$ represents the tensor product of the vector $x$ by itself, which is equal to the projection $\projbracket{\spn{x}}$ when $x$ has norm $1$.
This integral is a symmetric $n\times n$ matrix, hence diagonalizable.
Given an integer $m$, PCA refers to the projection operator on any $m$ top eigenvectors of $\Sigma[\mu]$, that is, eigenvectors associated with the $m$ largest eigenvalues; it is denoted by $\Pi^m_{\Sigma(\mu)}$.
When $\mu$ is $\mu_X$, the empirical measure on a finite point cloud, we recover the usual definition of PCA.

We divide our analysis into two. 
Section \ref{subsubsec:analysis_pca_orbit} gathers results regarding covariance matrices and applies them in the context of PCA. 
The same tools are employed in Section \ref{subsubsec:analysis_orthonormalization} to the problem of orthonormalization.
We stress that the ideas developed in this section are not new; the originality lies in the effort to obtain explicit bounds in terms of the Wasserstein distance.

\subsubsection{Analysis of PCA}\label{subsubsec:analysis_pca_orbit}
Our strategy to study the covariance matrix of $X$ consists of comparing it to $\Sigma[\mu_\O]$, the covariance of the orbit $\O$, which we consider as the `ideal' case.
Let $\phi\colon G\rightarrow\M_n(\R)$ be a representation that generates the orbit.
Let also $x_0 \in \O$ be an arbitrary point that is fixed throughout this section.
Note that any other point $x\in\O$ can be written $x=\phi(g)x_0$ for some $g\in G$.
Using that $\mu_\O$ is the pushforward of the Haar measure $\mu_G$ on $G$, Equation \eqref{eq:covariance_matrix} reads
\begin{equation*}
\Sigma[\mu_\O] = \int \big( \phi(g)  x_0\big)\cdot \big( \phi(g)  x_0\big)^\top \d \mu_G(g).
\end{equation*}
We observe that its trace is equal to the moment of order two of $\|\mu_\O\|$:
\begin{align}\label{eq:ideal_cov_variance_trace}
\Tr\big( \Sigma[\mu_\O] \big)
= \int \Tr\big( \big( \phi(g)  x_0\big) \cdot\big( \phi(g)  x_0\big)^\top \big) \d \mu_G(g)    
= \int \big\|\phi(g)  x_0\big\|^2 \d \mu_G(g)    
= \E[\|\mu_\O\|^2].
\end{align}
Now, let $\R^n = \bigoplus_{i=1}^m V_i$ be the decomposition of $\phi$ into irreps, and denote by $(\projbracket{V_i})_{i=1}^m$ the projection matrices on these subspaces.
We can decompose the previous equality in 
$$
\Sigma[\mu_\O] = 
\sum_{i=1}^m \underbrace{\int \phi_i(g) \bigg(\projbracket{V_i} (x_0) \cdot\projbracket{V_i} (x_0)^\top\bigg) \phi_i(g)^\top \d \mu_G(g)}_{C_i}.
$$
Each $C_i$ is a symmetric matrix and is zero on the $V_j$'s for which $j\neq i$.
Moreover, $C_i$ is zero if and only if $\projbracket{V_i} (x_0)=0$. We deduce that the kernel of $\Sigma[\mu_\O]$ is $\spn{\O}^\bot$.
We now suppose that $\phi$ is orthogonal. 
In particular, $g\mapsto\|\phi(g)  x_0\|$ is constant. 
We deduce that the variance $\Var[\|\mu_\O\|]$ is zero and that $\Tr( \Sigma[\mu_\O] ) = \|x_0\|^2$.
Besides, by Schur's lemma, the $C_i$ are homotheties:
$$
\Sigma[\mu_\O] = \sum_{i=1}^m \sigma_i^2 \projbracket{V_i}
~~~~~~\mathrm{where}~~~~~
\sigma_i^2 
= \frac{1}{\dim(V_i)} \int \big\|\projbracket{V_i}(x)\big\|^2 \d\mu_\O(x)
= \frac{\big\|\projbracket{V_i}(x_0)\big\|^2}{\dim(V_i)}.
$$
Note that, without further assumptions, the eigenvalues $(\sigma_i^2)_{i=1}^m$ can be any vector of positive coordinates whose sum is $\|x_0\|^2$.
In what follows, we will call by \textit{homogeneous} an orbit of an orthogonal representation whose eigenvalues are all equal.
We see it as the most `regular' orbit.

In this section, we do not assume the hypothesis that the representation is orthogonal, but we deduce from this discussion that two quantities are naturally involved in the problem:
\begin{itemize}
\item The variance $\Var[\|\mu_\O\|]$, which we see as a measure of \textit{deviation from orthogonality} of $\O$. When the representation is orthogonal, it is zero.
\item The ratio $\sigma_\mathrm{max}^2/\sigma_\mathrm{min}^2$ between the top and bottom nonzero eigenvalue of $\Sigma[\mu_\O]$, understood as a measure of \textit{homogeneity} of $\O$.
When the orbit is homogeneous, it is one.
\end{itemize}

The ideal case having been treated, we now turn to $X$.
To study how accurately $\Sigma[\mu_\O]$ can be approximated by $\Sigma[\mu_X]$, we give a general stability result for covariance matrices.
Combined with the Davis-Kahan theorem (Lemma \ref{lem:davis_kahan}), we deduce Proposition \ref{prop:stability_PCA_orbit}.
We stress that the following results do not involve the hypothesis that the measure $\mu_X$ comes from a point cloud. Instead, we work in the broader setting of an arbitrary probability measure $\nu$.

\begin{lemma}
\label{lem:wasserstein_stability_covariance_matrix}
For any two square-integrable measures $\mu$ and $\nu$ on $\R^n$, we have
    $$
    \| \Sigma[\mu] - \Sigma[\nu] \|
    \leq \big(2\Var[\|\mu\|]^{1/2} + \W_2(\mu,\nu)\big)\W_2(\mu,\nu).
    $$
\end{lemma}

\begin{proof}
Let $\pi$ be an optimal transport plan for the Wasserstein distance $\W_2(\mu,\nu)$.
We write 
\begin{align*}
    \| \Sigma[\mu] - \Sigma[\nu] \|
    &\leq \int \| x  x^\top - y  y^\top \| \d\pi(x,y)\nonumber\\
    &\leq \int (\|x\|+\|y\|)\|x - y\| \d\pi (x,y)\nonumber\\
    &\leq \sqrt{\int (\|x\|+\|y\|)^2 \d\pi (x,y)} \cdot \sqrt{\int \|x - y\|^2 \d\pi (x,y)},
\end{align*}
where the second line follows from Lemma \ref{lem:stability_outer_product}, stated below, and the last line from Hölder's inequality.
The second term of the product is equal to $\W_2(\mu,\nu)$.
Concerning the first term, we introduce the quantity $\E[\|\mu\|]$ and use the subadditivity of the $L^2$-norm:
\begin{align*}\label{eq:stability_pca_lemma_2}
    \sqrt{\int (\|x\|+\|y\|)^2 \d\pi (x,y) }
    &\leq \sqrt{\int (\|x\|-\E[\|\mu\|])^2 \d\pi (x,y) } + \sqrt{\int (\|y\|-\E[\|\mu\|])^2 \d\pi (x,y) }.
\end{align*}
The first term of the sum is equal to $\Var[\|\mu\|]^{1/2}$.
We bound the second one using the triangle inequality $\|y\|-\E[\|\mu\|] \leq (\|x\|-\E[\|\mu\|])+\|y-x\|$ and the subadditivity:
\begin{align*}
\sqrt{\int (\|y\|-\E[\|\mu\|])^2 \d\pi (x,y)}
\leq 
\sqrt{\int (\|x\|-\E[\|\mu\|])+\|y-x\|])^2 \d\pi (x,y)}
\leq \Var[\|\mu\|]^\frac{1}{2} + \W_2(\mu,\nu).
\end{align*}
Adding these terms yields the inequality of the lemma.
\end{proof}

\begin{lemma}\label{lem:stability_outer_product}
For every $x,y \in \R^n$, we have $\|x x^\top - y y^\top\| \leq (\|x\|+\|y\|)\|x-y\|$. 
\end{lemma}

\begin{proof}
The triangle inequality on $x x^\top - y  y^\top = (x-y)  x^\top + y  (x-y)^\top$ yields 
\begin{align*}
\|x  x^\top - y  y^\top\|
&\leq \|x-y\|\|x\| + \|y\|\|x-y\|=(\|x\|+\|y\|)\|x-y\|.\qedhere
\end{align*} 
\end{proof}

Lemma \ref{lem:wasserstein_stability_covariance_matrix} shows that the distance between the covariance matrices of two measures is small, provided that their Wasserstein distance is.
We now wish to obtain a similar result for the PCA, that is, the projection matrices on their top eigenspaces.
If $U$ and $V$ are linear subspaces of $\R^n$, we will quantify their proximity via $\|\projbracket{U}-\projbracket{V}\|$, the Frobenius norm between their projections, just as we defined in Section \ref{subsec:grassmann_lie_notgroup} for subspaces of $\so(n)$.
We mention the relation
\begin{equation}\label{eq:norm_proj_principalangles}
\|\projbracket{U}-\projbracket{V}\| = \sqrt{2} |\sin\Theta(U,V)|,
\end{equation}
where $\Theta(U,V) \in \R^n$ denotes the \emph{principal angles} between $U$ and $V$. 
These are defined recursively as the minimal angles between two unit vectors of the subspaces (see \cite[Sec.~2]{davis1970rotation}).

Now, let $A,B\in\M_n(\R)$ be two symmetric matrices, $m\in[1\isep n]$ an integer, and $\Pi_A^m,\Pi_B^m$ the projection matrices on the subspaces of $\R^n$ spanned by any top $m$ eigenvectors of $A$ and $B$.
The problem of bounding $\|\Pi_A^m-\Pi_B^m\|$ in terms of $\|A-B\|$ is a classical question of matrix perturbation theory.
The \emph{stability of eigenvalues} is given by Weyl's theorem: for any eigenvalue $\lambda$ of $A$, there exists an eigenvalue $\mu$ of $B$ such that $|\lambda-\mu|\leq \|A-B\|$ \cite[Cor.~VI.4.10]{stewart1990matrix}.
On the other hand, the \emph{stability of eigenspaces} is given by Davis-Kahan theorem \cite{davis1970rotation}, which we reformulate in the lemma below, and that will also be used on several occasions throughout the rest of this article.
Given a symmetric matrix in $\M_n(\R)$ with sorted eigenvalues $\lambda_1\geq\dots\geq\lambda_n$ and an integer $m\in[1\isep n-1]$, we call \textit{$m^\mathrm{th}$ eigengap} the difference $\lambda_m-\lambda_{m+1}$.
As a consequence of Weyl's theorem, if $\delta$ is the $m^\mathrm{th}$ eigengap of $A$, then that of $B$ is at least $\delta-2\|A-B\|$.

\begin{lemma}(Davis-Kahan theorem)\label{lem:davis_kahan}
Let $A,B\in\M_n(\R)$ be symmetric matrices, $m\in [1\isep n-1]$ an integer, and $\delta$ the $m^\mathrm{th}$ eigengap of $A$.
If $\|A-B\|< \delta/2$, then
\[
\|\Pi_A^m-\Pi^m_B\|\leq \sqrt{2}/\delta\|A-B\|.
\] 
The result also holds if we consider the bottom eigenvectors instead of the top or if we consider the operator norm instead of the Frobenius norm.
\end{lemma}

Equipped with these results, we can derive our main proposition concerning PCA.

\begin{proposition}
\label{prop:stability_PCA_orbit}
Let $\mathcal{O}\subset\R^n$ be the orbit of a representation, potentially non-orthogonal, $\mu_\O$ its uniform measure, $\projbracket{\spn{\mathcal{O}}}$ the projection on its span, and $\sigma_\mathrm{max}^2,\sigma_\mathrm{min}^2$ the top and bottom nonzero eigenvalues of $\Sigma[\mu_\O]$.
Also let $\nu$ be a measure, $\Sigma[\nu]$ its covariance matrix, $\epsilon>0$ and $\Pi_{\Sigma[\nu]}^{>\epsilon}$ the projection on the subspace spanned by eigenvectors of eigenvalue at least $\epsilon$. 
We suppose
\begin{align*}
&\frac{\W_2(\mu_\O,\nu)}{\sigma_\mathrm{min}} < \bigg(\frac{\Var\big[\|\mu_\O\|\big]}{\sigma_\mathrm{min}^2}+\frac{1}{2}
\bigg)^{1/2}-\bigg(\frac{\Var\big[\|\mu_\O\|\big]}{\sigma_\mathrm{min}^2}\bigg)^{1/2}
\\\mathrm{and}~~~~~
&\bigg(2\bigg(\frac{\Var\big[\|\mu_\O\|\big]}{\sigma_\mathrm{min}^2}\bigg)^{1/2}  + \frac{\W_2(\mu_\O,\nu)}{\sigma_\mathrm{min}}\bigg)
\bigg(\frac{\W_2(\mu_\O,\nu)}{\sigma_\mathrm{min}}\bigg) < \frac{\epsilon}{\sigma_\mathrm{min}^2} \leq \frac{1}{2}.
\end{align*}
Then we have the following bound between the pushforward measures after PCA:
\[
\W_2\bigg(\projbracket{\spn{\mathcal{O}}}\mu_\O,~\Pi_{\Sigma[\nu]}^{>\epsilon}\nu\bigg)\bigg/\sigma_\mathrm{min}
\leq
3(n+1) \bigg( \frac{\sigma_\mathrm{max}^2}{\sigma_\mathrm{min}^2} \bigg)\bigg(\frac{\W_2(\mu_\O,\nu)}{\sigma_\mathrm{min}}\bigg).
\]
\end{proposition}

\begin{proof}
We will first derive the following bound:
\begin{align}\label{eq:majoration_frobenius_pca}
\big\|\projbracket{\spn{\mathcal{O}}} - \Pi_{\Sigma[\nu]}^{>\epsilon} \big\|
&\leq 
\frac{\sqrt{2}}{\sigma_\mathrm{min}^2}\bigg(2\Var\big[\|\mu_\O\|\big]^{1/2}  + \W_2(\mu_\O,\nu)\bigg)\W_2(\mu_\O,\nu).
\end{align}
Observe that the hypothesis can be written in the form of a quadratic polynomial in $\W_2(\mu_\O,\nu)$:
\begin{align}
&\frac{\W_2(\mu_\O,\nu)}{\sigma_\mathrm{min}} < \bigg(\frac{\Var\big[\|\mu_\O\|\big]}{\sigma_\mathrm{min}^2}+\frac{1}{2}
\bigg)^{1/2}-\bigg(\frac{\Var\big[\|\mu_\O\|\big]}{\sigma_\mathrm{min}^2}\bigg)^{1/2}
\nonumber\\&
~~~~~~~~~~~~~~~~~~~~~~~~~~~~~\iff 
\bigg(2\bigg(\frac{\Var\big[\|\mu_\O\|\big]}{\sigma_\mathrm{min}^2}\bigg)^{1/2}  + \frac{\W_2(\mu_\O,\nu)}{\sigma_\mathrm{min}}\bigg)
\bigg(\frac{\W_2(\mu_\O,\nu)}{\sigma_\mathrm{min}}\bigg)<\frac{1}{2}.
\label{eq:bound_W2_with_sigma}
\end{align}
By combining Lemma \ref{lem:wasserstein_stability_covariance_matrix} with this second inequality, we obtain
\begin{align}
\|\Sigma[\mu_\O] - \Sigma[\nu] \| 
&\leq \bigg(2\Var\big[\|\mu_\O\|\big]^{1/2}  + \W_2(\mu_\O,\nu)\bigg)\W_2(\mu_\O,\nu)
<\frac{\sigma_\mathrm{min}^2}{2}.
\label{eq:proof_consistencyPCA_covariance}
\end{align}
Denote by $m$ the dimension of $\spn{\O}$, and consider the projections $\Pi_{\Sigma[\mu_\O]}^m$ and $\Pi^m_{\Sigma[\nu]}$ on the top $m$ eigenvectors of $\Sigma[\mu_\O]$ and $\Sigma[\nu]$.
Note that the kernel of $\Sigma[\mu_\O]$ is $\spn{\O}^\bot$ and its $m^\mathrm{th}$ eigengap is $\sigma_\mathrm{min}^2$. 
With Equation \eqref{eq:proof_consistencyPCA_covariance}, we deduce that the hypothesis of Lemma \ref{lem:davis_kahan} is satisfied.
It yields
\begin{align*}
\big\|\Pi_{\Sigma[\mu_\O]}^m-\Pi^m_{\Sigma[\nu]}\big\|
&\leq 
\frac{\sqrt{2}}{\sigma_\mathrm{min}^2}\bigg(2\Var\big[\|\mu_\O\|\big]^{1/2}  + \W_2(\mu_\O,\nu)\bigg)\W_2(\mu_\O,\nu).
\end{align*}
We will now identify the matrices $\Pi_{\Sigma[\mu_\O]}^m$ and $\Pi^m_{\Sigma[\nu]}$.
Since the kernel of $\Sigma[\mu_\O]$ is $\spn{\O}^\bot$, we have  $\Pi_{\Sigma[\mu_\O]}^m=\projbracket{\spn{\mathcal{O}}}$.
Besides, using Equation \eqref{eq:proof_consistencyPCA_covariance} and the assumption $\epsilon/\sigma_\mathrm{min}^2\leq1/2$, we have that
\begin{align*}
\sigma_\mathrm{min}^2 -\big\|\Sigma[\mu_\O] - \Sigma[\nu] \big\| 
> \sigma_\mathrm{min}^2-\sigma_\mathrm{min}^2/2 \geq \epsilon.
\end{align*}
Similarly, Equation \eqref{eq:proof_consistencyPCA_covariance} and the assumption $\big(2\Var\big[\|\mu_\O\|\big]^{1/2}+\W_2(\mu_\O,\nu)\big) \cdot\W_2(\mu_\O,\nu) < \epsilon$ yields
\begin{align*}
0 + \big\|\Sigma[\mu_\O] - \Sigma[\nu] \big\| 
&\leq \big(2\Var\big[\|\mu_\O\|\big]^{1/2}+\W_2(\mu_\O,\nu)\big) \W_2(\mu_\O,\nu) < \epsilon.
\end{align*}
Consequently, by Weyl's theorem for the stability of eigenvalues, $\Sigma[\nu]$ admits exactly $m$ eigenvalues greater than or equal to $\epsilon$.
Hence, we have $\Pi^m_{\Sigma[\nu]} = \Pi_{\Sigma[\nu]}^{>\epsilon}$.
As a consequence of these equalities, we obtain Equation \eqref{eq:majoration_frobenius_pca}.
Now, to prove the statement of the proposition, we apply the triangle inequality of the Wasserstein distance on the pushforward measures $\projbracket{\spn{\mathcal{O}}}\mu_\O$ and $\Pi_{\Sigma[\nu]}^{>\epsilon}\nu$:
\begin{align*}
\W_2\left(\projbracket{\spn{\mathcal{O}}}\mu_\O,\Pi_{\Sigma[\nu]}^{>\epsilon}\nu\right)
&\leq \W_2\left(\projbracket{\spn{\mathcal{O}}}\mu_\O,\Pi_{\Sigma[\nu]}^{>\epsilon}\mu_\O\right)
+
\W_2\left(\Pi_{\Sigma[\nu]}^{>\epsilon}\mu_\O,\Pi_{\Sigma[\nu]}^{>\epsilon}\nu\right)\\
&\leq \big\|\projbracket{\spn{\mathcal{O}}}-\Pi_{\Sigma[\nu]}^{> \epsilon}\big\|_\mathrm{op} \cdot {\E[\|\mu_\O\|^2]}^{1/2}
+\big\|\Pi_{\Sigma[\nu]}^{>\epsilon}\big\|_\mathrm{op}\cdot \W_2(\mu_\O,\nu).
\end{align*}
The second term is equal to $\W_2(\mu_\O,\nu)$.
To bound the first one, we use the loose bound $\E[\|\mu_\O\|^2] = \Tr\big( \Sigma[\mu_\O]\big) \leq n\sigma_\mathrm{max}^2$, seen in Equation \eqref{eq:ideal_cov_variance_trace}.
Combined with Equation \eqref{eq:majoration_frobenius_pca}, we get
\[
\big\|\projbracket{\spn{\mathcal{O}}}-\Pi_{\Sigma[\nu]}^{> \epsilon}\big\|_\mathrm{op} \cdot {\E[\|\mu_\O\|^2]}^{1/2}
\leq
\frac{\sqrt{2n}}{\sigma_\mathrm{min}} \bigg( \frac{\sigma_\mathrm{max}}{\sigma_\mathrm{min}} \bigg) \bigg(2\Var\big[\|\mu_\O\|\big]^{1/2}  + \W_2(\mu_\O,\nu)\bigg)\W_2(\mu_\O,\nu).
\]
We simplify it with $\Var\big[\|\mu_\O\|\big]\leq \E[\|\mu_\O\|^2]\leq n\sigma_\mathrm{max}^2$ and $\W_2(\mu_\O,\nu)\leq\sigma_\mathrm{min}/\sqrt{2}\leq\sigma_\mathrm{max}/\sqrt{2}$:
\[
\big\|\projbracket{\spn{\mathcal{O}}}-\Pi_{\Sigma[\nu]}^{> \epsilon}\big\|_\mathrm{op} \cdot {\E[\|\mu_\O\|^2]}^{1/2}
\leq
\frac{\sqrt{2n}}{\sigma_\mathrm{min}} \bigg( \frac{\sigma_\mathrm{max}}{\sigma_\mathrm{min}} \bigg) \bigg( 2\sqrt{n}+\frac{1}{\sqrt{2}} \bigg) \sigma_\mathrm{max}  \W_2(\mu_\O,\nu).
\]
To conclude, we gather the terms and use the inequality $\sqrt{2n} \big( 2\sqrt{n}+1/\sqrt{2} \big)\leq 3n+2$:
\begin{align*}
\W_2\left(\projbracket{\spn{\mathcal{O}}}\mu_\O,\Pi_{\Sigma[\nu]}^{>\epsilon}\nu\right)
&\leq
\sqrt{2n} \bigg( 2\sqrt{n}+\frac{1}{\sqrt{2}} \bigg) \bigg( \frac{\sigma_\mathrm{max}^2}{\sigma_\mathrm{min}^2} \bigg) \W_2(\mu_\O,\nu)
+ \W_2(\mu_\O,\nu)\\
&\leq 3(n+1) \bigg( \frac{\sigma_\mathrm{max}^2}{\sigma_\mathrm{min}^2} \bigg) \W_2(\mu_\O,\nu).\qedhere
\end{align*}
\end{proof}

We point out that Proposition \ref{prop:stability_PCA_orbit} has an important consequence in practice: performing PCA on $X$ will give a new point cloud that is again close to the orbit of a representation.
Indeed, if $\phi$ denotes the representation that generates $\O$, then the span of $\O$ is an invariant subspace of $\phi$, hence the projection $\projbracket{\spn{\mathcal{O}}}(\O)$ is still an orbit of a representation.

Besides, we note that the assumptions of the proposition are valid, provided that $\W_2(\mu_\O,\nu)$ is small enough and that the parameter $\epsilon$ is chosen in accordance with $\W_2(\mu_\O,\nu)$.
In practice, however, these conditions are hardly verifiable, since the quantities $\sigma_\mathrm{max}^2$, $\sigma_\mathrm{min}^2$ and $\Var\big[\|\mu_\O\|\big]$ are unknown.
To address this issue, we could try to estimate these quantities using the observation $X$, in a statistical fashion, although we did not go further in this direction.

\begin{remark}\label{rem:invariant_under_rescaling}
The stability of PCA, expressed in Proposition \ref{prop:stability_PCA_orbit} above, involves the terms
$$
\frac{\sigma_\mathrm{max}^2}{\sigma_\mathrm{min}^2},~~~~
\frac{\Var\big[\|\mu_\O\|\big]}{\sigma_\mathrm{min}^2},
~~~~\mathrm{and}~~~~~
\frac{\W_2(\mu_\O,\nu)}{\sigma_\mathrm{min}}.
$$
They correspond respectively to the homogeneity of $\O$, its deviation from orthogonality, and the quality of the sample $\nu$.
Moreover, they are all \emph{invariant under rescaling}.
That is, multiplying the input $\mu_\O$ and $\nu$ by a fixed positive constant will not modify their value.
This is an important property, showing that the algorithm is not affected by the scale and, hence, works with the intrinsic geometric features of the input.
We will ensure this property in all the results to come.
\end{remark}

\subsubsection{Orthonormalization}\label{subsubsec:analysis_orthonormalization}
The second part of \ref{item:step1} consists of computing $M = \sqrt{\Sigma[X]^+}$, the square root of the Moore–Penrose pseudo-inverse of the covariance matrix, and in translating $X$ by $M$.
This step actually implements the orthogonalization of representations, a standard practice when working with Lie groups.
In fact, along the way, we will implement a stronger result: $M X$ will be close to a homogeneous orbit, as defined in Section \ref{subsubsec:analysis_pca_orbit}.

We recall that a representation $\phi\colon G\rightarrow\M_n(\R)$ is orthogonal if it takes values in $\Ort(n)$, i.e., if the relation $\big\langle \phi(g)x,\phi(g)y \big\rangle = \langle x,y \rangle$ holds for all $g\in G $ and $x,y\in \R^n$, where we consider the usual Euclidean inner product.
As recalled in Section \ref{subsubsec:def_representation}, for any representation $\phi$ of a compact Lie group, there exists a positive-definite matrix $M\in\GL_n(\R)$ such that the conjugate representation $M\phi M^{-1}$ is orthogonal.
Let us write down the classic construction of such a matrix.
We define an inner product on $\R^n$ via
\[
\langle x, y \rangle_G = \int \big\langle \phi(g)x, \phi(g)y \big\rangle\d\mu_G(g),
\]
with $\mu_G$ the Haar measure.
By construction, the representation is orthogonal with respect to this inner product. 
Let $K$ be the matrix of scalar products of the canonical basis vectors $(e_i)_{i=1}^n$ of $\R^n$.
It is a symmetric positive-definite matrix and satisfies the relation $\langle x, y \rangle_G = \langle K x, y \rangle$.
By denoting $M = K^{1/2}$ a square root, the relation becomes $\langle x, y \rangle_G=\langle M x, M y \rangle$ and we see that the conjugate representation is orthogonal with respect to the usual inner product:
\begin{align*}
\big\langle M\phi(g)M^{-1}x, M\phi(g)M^{-1} y \big\rangle 
= \big\langle \phi(g)M^{-1}x, \phi(g)M^{-1} y \big\rangle_G
= \big\langle M^{-1}x, M^{-1} y \big\rangle_G
= \big\langle x, y \big\rangle.
\end{align*}

In the context of this work, this construction of the matrix $M$ cannot be used, since we do not have access to the representation $\phi$, but only to a sample $X$ of an orbit $\O$ of it.
However, it can be easily adapted.
Indeed, with $\Sigma[\mu_\O]$ the covariance matrix of $\O$ and any matrix $M \in \M_n(\R)$, we observe that Equation \eqref{eq:covariance_matrix} yields
\begin{align*}
\Sigma[M \mu_\O] = \int \big(M x) \big(M x)^\top \d \mu_\O(x)
= M \Sigma[\mu_\O] M^\top.
\end{align*}
We remark that the choice of $M = \Sigma[ \mu_\O]^{-1/2}$, the square root of the inverse of $\Sigma[ \mu_\O]$, gives the particularly simple covariance matrix $\Sigma[M \mu_\O] = I$.
Moreover, in this case, the conjugate representation $M\phi M^{-1}$ is orthogonal. 
Indeed, by definition of $\Sigma[ \mu_\O]$, we have, for all $g\in G$,
$$\phi(g)^{-1} \Sigma[ \mu_\O] = \Sigma[ \mu_\O] \phi(g)^\top,$$ 
We deduce that $M\phi(g) M^{-1}$ is an orthogonal matrix.
This justifies our choice of $M$.
In addition, since $\Sigma[M \mu_\O]=I$, the ratio of its top and bottom eigenvalues is $1$, hence $M\O$ is a homogeneous orbit, as defined in Section \ref{subsubsec:analysis_pca_orbit}.
We could say that the orbit has been `orthonormalized'.

We now discuss the concrete implementation of this procedure when the input is a point cloud $X$ or, more generally, a measure $\nu$ on $\R^n$.
We suppose that PCA pre-processing has been applied successfully to $\O$ and $X$, hence that $\spn{\O} = \R^n$.
In particular, $\Sigma[\mu_\O]$ is non-singular.
We aim to show that the square root $M = \Sigma[\nu]^{-1/2}$ allows us to make the underlying representation orthogonal.
To do so, we will follow the same strategy as we did when studying PCA: deriving an upper bound on the distance between $\Sigma[\nu]$ and $\Sigma[\mu_\O]$, expressed with the Wasserstein distance, and deducing that the output of our algorithm is consistent.
In opposition to PCA, the following result does not involve projection on eigenspaces and is more straightforward to prove.

\begin{proposition}\label{prop:stability_orthonormalization}
Let $\mathcal{O}\subset\R^n$ be the orbit of a representation, potentially non-orthogonal, that spans $\R^n$.
Let $\mu_\O$ be its uniform measure and $\sigma_\mathrm{max}^2,\sigma_\mathrm{min}^2$ the top and bottom eigenvalues of $\Sigma[\mu_\O]$.
Let $\nu$ be a measure and $\Sigma[\nu]$ its covariance matrix.
We suppose that $\W_2(\mu_\O,\nu)$ satisfies the first hypothesis of Proposition \ref{prop:stability_PCA_orbit}.
Then $\Sigma[\nu]$ is positive-definite, and we have a bound
\begin{align*}
&\W_2\bigg(\Sigma[\mu_\O]^{-1/2}\mu_\O, ~\Sigma[\nu]^{-1/2}\nu\bigg)\\
&~~~~~~~~~~~~~~~~~~~~~
\leq  \big(\sqrt{2n}+2\big)
\bigg(\frac{\sigma_\mathrm{max}}{\sigma_\mathrm{min}}\bigg)
\bigg(\frac{\W_2(\mu_\O,\nu)}{\sigma_\mathrm{min}}\bigg)^{1/2}
\bigg(2\bigg(\frac{\Var\big[\|\mu_\O\|\big]}{\sigma_\mathrm{min}^2}\bigg)^{1/2}  + \frac{\W_2(\mu_\O,\nu)}{\sigma_\mathrm{min}}\bigg)^{1/2}.
\end{align*}
\end{proposition}

\begin{proof}
The first step of the proof consists of showing the inequality
\begin{align}\label{eq:proof_orthonormalization_3}    
\big\|\Sigma[\mu_\O]^{-1/2} - \Sigma[\nu]^{-1/2}\big\|_\mathrm{op}
\leq
\frac{\sqrt{2}}{\sigma_\mathrm{min}^2} \bigg(2\Var\big[\|\mu_\O\|\big]^{1/2}  + \W_2(\mu_\O,\nu)\bigg)^{1/2}\W_2(\mu_\O,\nu)^{1/2}.
\end{align}
Let us recall a few relations regarding the operator norm of matrices.
First, the inequality 
$$
\|A^{-1}-B^{-1}\|_\mathrm{op} \leq \|A-B\|_\mathrm{op}\cdot \|B^{-1}\|_\mathrm{op}\cdot\|A^{-1}\|_\mathrm{op}
$$ 
is valid for any pair of invertible matrices.
This relation comes from applying the submultiplicativity of the operator norm to the equality $A^{-1}-B^{-1} = A^{-1}(B-A)B^{-1}$.
Now, if $A$ is a positive definite matrix, then $\|A\|_\mathrm{op}$ is equal to its top eigenvalue.
Consequently, $\|A^{1/2}\|_\mathrm{op} = \|A\|_\mathrm{op}^{1/2}$ and $\|A^{-1}\|_\mathrm{op} = \lambda^{-1}$, with $\lambda$ its bottom eigenvalue.
Moreover, it is known that
$$\|A^{1/2}-B^{1/2}\|_\mathrm{op}\leq \|A-B\|_\mathrm{op}^{1/2}.$$
Putting these relations together, we deduce
\begin{align}
\|A^{-1/2}-B^{-1/2}\|_\mathrm{op} 
&\leq 
\|A^{1/2}-B^{1/2}\|_\mathrm{op}\cdot\|A^{-{1/2}}\|_\mathrm{op}\cdot \|B^{-{1/2}}\|_\mathrm{op}\nonumber\\
&\leq \|A-B\|_\mathrm{op}^{1/2}\cdot \lambda^{-1/2} (\lambda - \|A-B\|_\mathrm{op} )^{1/2},\label{eq:inequality_op_norm_inversesquareroot}
\end{align}
where we used Weyl's eigenvalues stability for the last term. 
We now apply these results to our context.
In Equation \eqref{eq:proof_consistencyPCA_covariance} of the proof of Proposition \ref{prop:stability_PCA_orbit}, we have observed that
\begin{align*}
\| \Sigma[\mu_\O] - \Sigma[\nu] \| \leq \bigg(2\Var\big[\|\mu_\O\|\big]^{1/2}  + \W_2(\mu_\O,\nu)\bigg)\W_2(\mu_\O,\nu)
< \frac{\sigma_\mathrm{min}^2}{2}.
\end{align*}
This second inequality, combined with the fact that the eigenvalues of $\Sigma[\mu_\O]$ are at least $\sigma_\mathrm{min}^2$, shows that $\Sigma[\nu]$ is non-singular.
Moreover, the first and second inequalities respectively yield
\begin{align*}
\big\|\Sigma[\mu_\O]-\Sigma[\nu]\big\|^{1/2}
&\leq 
\bigg(2\Var\big[\|\mu_\O\|\big]^{1/2}  + \W_2(\mu_\O,\nu)\bigg)^{1/2}\cdot\W_2(\mu_\O,\nu)^{1/2}
,\\
\sigma_\mathrm{min}^2 - \big\|\Sigma[\mu_\O]-\Sigma[\nu]\big\|&\geq \sigma_\mathrm{min}^2-\sigma_\mathrm{min}^2/2=\sigma_\mathrm{min}^2/2.
\end{align*}
We eventually bound $\big\|\Sigma[\mu_\O]^{-1/2} - \Sigma[\nu]^{-1/2}\big\|_\mathrm{op}$ through the inequality of Equation \eqref{eq:inequality_op_norm_inversesquareroot}:
\begin{align*}
\big\|\Sigma[\mu_\O]^{-1/2} - \Sigma[\nu]^{-1/2}\big\|_\mathrm{op}
&\leq  
\big\|\Sigma[\mu_\O]-\Sigma[\nu]\big\|^{1/2}\sigma_\mathrm{min}^{-1} 
\big(\sigma_\mathrm{min}^2 - \big\|\Sigma[\mu_\O]-\Sigma[\nu]\big\| \big)^{-1/2}\\
&\leq
\bigg(2\Var\big[\|\mu_\O\|\big]^{1/2}  + \W_2(\mu_\O,\nu)\bigg)^{1/2}\W_2(\mu_\O,\nu)^{1/2}
 \frac{\sqrt{2}}{\sigma_\mathrm{min}^2}.
\end{align*}
This is Equation \eqref{eq:proof_orthonormalization_3} previously announced.
Next, to prove the inequality of the proposition, we consider the pushforward measures and apply the triangle inequality:
\begin{align*}
&\W_2\left(\Sigma[\mu_\O]^{-1/2}\mu_\O,\Sigma[\nu]^{-1/2}\nu\right)\\
&\leq \W_2\left(\Sigma[\mu_\O]^{-1/2}\mu_\O,\Sigma[\nu]^{-1/2}\mu_\O\right)
+
\W_2\left(\Sigma[\nu]^{-1/2}\mu_\O,\Sigma[\nu]^{-1/2}\nu\right)\\
&\leq \big\|\Sigma[\mu_\O]^{-1/2}-\Sigma[\nu]^{-1/2}\big\|_\mathrm{op} \cdot {\E[\|\mu_\O\|^2]}^{1/2}
+\big\|\Sigma[\nu]^{-1/2}\|_\mathrm{op}\cdot \W_2(\mu_\O,\nu)\\
&\leq \big\|\Sigma[\mu_\O]^{-1/2}-\Sigma[\nu]^{-1/2}\big\|_\mathrm{op} \left(  {\E[\|\mu_\O\|^2]}^{1/2}+\W_2(\mu_\O,\nu)\right)
+\big\|\Sigma[\mu_\O]^{-1/2}\|_\mathrm{op}\cdot \W_2(\mu_\O,\nu),
\end{align*}
where we used $\big\|\Sigma[\nu]^{-1/2}\|_\mathrm{op}\leq\big\|\Sigma[\mu_\O]^{-1/2}\|_\mathrm{op}+\W_2(\mu_\O,\nu)$ on the last line.
The second term is equal to $\W_2(\mu_\O,\nu)/\sigma_\mathrm{min}$.
For the first one, we use the inequality $\W_2(\mu_\O,\nu)\leq  \sigma_\mathrm{min}/\sqrt{2}$, coming from the hypothesis, and $\E\big[\|\mu_\O\|^2\big]\leq \Tr{\Sigma[\mu_\O]} \leq n\sigma_\mathrm{max}^2$, coming from Equation \eqref{eq:ideal_cov_variance_trace}, to obtain
\begin{align*}
\E\big[\|\mu_\O\|^2\big]^{1/2}+\W_2(\mu_\O,\nu)
\leq \sqrt{n}\sigma_\mathrm{max} + \frac{\sigma_\mathrm{min}}{\sqrt{2}} 
\leq \frac{1}{\sqrt{2}}\big(\sqrt{2n}+1\big)\sigma_\mathrm{max}.
\end{align*}
Combined with Equation \eqref{eq:proof_orthonormalization_3}, we deduce that
\begin{align*}
&\big\|\Sigma[\mu_\O]^{-\frac{1}{2}} - \Sigma[\nu]^{-\frac{1}{2}}\big\|_\mathrm{op} \big(\E\big[\|\mu_\O\|^2\big]^{1/2}+\W_2(\mu_\O,\nu)\big)\\
&\leq
\bigg(2\frac{\Var\big[\|\mu_\O\|\big]^{1/2}}{\sigma_\mathrm{min}}  + \frac{\W_2(\mu_\O,\nu)}{\sigma_\mathrm{min}}\bigg)^{1/2}\bigg(\frac{\W_2(\mu_\O,\nu)}{\sigma_\mathrm{min}}\bigg)^{1/2} \bigg(\sqrt{2n}+1\bigg)\frac{\sigma_\mathrm{max}}{\sigma_\mathrm{min}}.
\end{align*}
We collect the terms to conclude:
\begin{align*}
&\W_2\bigg(\Sigma[\mu_\O]^{-1/2}\mu_\O, ~\Sigma[\nu]^{-1/2}\nu\bigg)\leq
\\ &
\bigg(\frac{\W_2(\mu_\O,\nu)}{\sigma_\mathrm{min}}\bigg)
+ \big(\sqrt{2n}+1\big)\bigg(\frac{\sigma_\mathrm{max}}{\sigma_\mathrm{min}}\bigg)
\bigg(\frac{\W_2(\mu_\O,\nu)}{\sigma_\mathrm{min}}\bigg)^{1/2}
\bigg(2\bigg(\frac{\Var\big[\|\mu_\O\|\big]}{\sigma_\mathrm{min}^2}\bigg)^{1/2}  + \frac{\W_2(\mu_\O,\nu)}{\sigma_\mathrm{min}}\bigg)^{1/2}\bigg).
\end{align*}
Simplifying this equation gives the result of the proposition.
\end{proof}

We observe that in Proposition \ref{prop:stability_orthonormalization}, the distance $\W_2\big(\Sigma[\mu_\O]^{-1/2}\mu_\O, ~\Sigma[\nu]^{-1/2}\nu\big)$ does not appear normalized by a constant, as it is the case for the other terms.
This, however, does not contradict the invariance under rescaling, raised in Remark \ref{rem:invariant_under_rescaling}. 
Indeed, this distance is `dimensionless', i.e., already invariant, since $\Sigma[\mu_\O]^{-1/2}\mu_\O$ and $\Sigma[\nu]^{-1/2}\nu$ are `normalized' measures.

We now gather Propositions \ref{prop:stability_PCA_orbit} and \ref{prop:stability_orthonormalization} to obtain our main result for \ref{item:step1}.
It shows that the orthonormalized measure $\sqrt{\Sigma[\nu]^+}\Pi_{\Sigma[\nu]}^{>\epsilon}\nu$, output of \ref{item:step1} defined in Equation \eqref{eq: program of projection2}, is close to $\sqrt{\Sigma[\mu_\O]^+}\projbracket{\spn{\mathcal{O}}}\mu_\O$, which is the uniform measure on an orbit of an orthogonal representation that spans the ambient space, as it is wished in this pre-processing step.

\begin{proposition}
\label{prop:stability_step1}
Let $\mathcal{O}\subset\R^n$ be the orbit of a representation, potentially non-orthogonal, $\mu_\O$ its uniform measure, $\projbracket{\spn{\mathcal{O}}}$ the projection on its span, and $\sigma_\mathrm{max}^2,\sigma_\mathrm{min}^2$ the top and bottom nonzero eigenvalues of $\Sigma[\mu_\O]$.
Besides, let $\nu$ be a measure, $\Sigma[\nu]$ its covariance matrix, $\epsilon>0$ and $\Pi_{\Sigma[\nu]}^{>\epsilon}$ the projection on the subspace spanned by eigenvectors with eigenvalue at least $\epsilon$. 
Suppose
\begin{align*}
&
\frac{\W_2(\mu_\O,\nu)}{\sigma_\mathrm{min}}
< \bigg(\bigg(\frac{\Var\big[\|\mu_\O\|\big]}{\sigma_\mathrm{min}^2}+\frac{1}{2}\bigg)^2
\bigg)^{1/2}-\bigg(\frac{\Var\big[\|\mu_\O\|\big]}{\sigma_\mathrm{min}^2}\bigg)^{1/2}\bigg)
\bigg/ \bigg( 3(n+1) \bigg( \frac{\sigma_\mathrm{max}^2}{\sigma_\mathrm{min}^2} \bigg)\bigg)
\\\mathrm{and}~~~~~
&\bigg(2\bigg(\frac{\Var\big[\|\mu_\O\|\big]}{\sigma_\mathrm{min}^2}\bigg)^{1/2}  + \frac{\W_2(\mu_\O,\nu)}{\sigma_\mathrm{min}}\bigg)
\bigg(\frac{\W_2(\mu_\O,\nu)}{\sigma_\mathrm{min}}\bigg) < \frac{\epsilon}{\sigma_\mathrm{min}^2} \leq \frac{1}{2}.
\end{align*}
Then we have the following bound between the pushforward measures after \ref{item:step1}:
\begin{align*}
&\W_2\bigg(\sqrt{\Sigma[\mu_\O]^+}\projbracket{\spn{\mathcal{O}}}\mu_\O,~\sqrt{\Sigma[\nu]^+}\Pi_{\Sigma[\nu]}^{>\epsilon}\nu\bigg)\\
&\leq 
8(n+1)^{3/2}\bigg(\frac{\sigma_\mathrm{max}^3}{\sigma_\mathrm{min}^3}\bigg)
\bigg(\frac{\W_2(\mu_\O,\nu)}{\sigma_\mathrm{min}}\bigg)^{1/2}
\bigg(\bigg(\frac{\Var\big[\|\mu_\O\|\big]}{\sigma_\mathrm{min}^2}\bigg)^{1/2}  + \frac{\W_2(\mu_\O,\nu)}{\sigma_\mathrm{min}}\bigg)^{1/2}.
\end{align*}
\end{proposition}

\begin{proof}
This assumption on $\W_2(\mu_\O,\nu)$ is stronger than in Proposition \ref{prop:stability_PCA_orbit}. Thus, its result holds:
\[
\W_2\bigg(\projbracket{\spn{\mathcal{O}}}\mu_\O,~\Pi_{\Sigma[\nu]}^{>\epsilon}\nu\bigg)\bigg/\sigma_\mathrm{min}
\leq 
\underbrace{3(n+1)
\bigg(\frac{\sigma_\mathrm{max}^2}{\sigma_\mathrm{min}^2}\bigg)}_{\alpha}
 \bigg(\frac{\W_2(\mu_\O,\nu)}{\sigma_\mathrm{min}}\bigg).
\]
Still using the assumption on $\W_2(\mu_\O,\nu)$, we deduce that
\[
\W_2\bigg(\projbracket{\spn{\mathcal{O}}}\mu_\O,~\Pi_{\Sigma[\nu]}^{>\epsilon}\nu\bigg)\bigg/\sigma_\mathrm{min}
\leq \bigg(\frac{\Var\big[\|\mu_\O\|\big]}{\sigma_\mathrm{min}^2}+\frac{1}{2}\bigg)^2
\bigg)^{1/2}-\bigg(\frac{\Var\big[\|\mu_\O\|\big]}{\sigma_\mathrm{min}^2}\bigg)^{1/2}.
\]
Hence we can apply Proposition \ref{prop:stability_orthonormalization} on $\projbracket{\spn{\mathcal{O}}}\mu_\O$ and $\Pi_{\Sigma[\nu]}^{>\epsilon}\nu$.
They have been projected into the subspaces $\O$ and $X$ span, thus we can identify the inverse of the covariance matrices of the projected measures with the Moore-Penrose pseudo-inverse of the initial measures, therefore
$$
\Sigma[\projbracket{\spn{\mathcal{O}}}\mu_\O]^{-1/2} = \sqrt{\Sigma[\mu_\O]^+}\projbracket{\spn{\mathcal{O}}}
~~~~~\mathrm{and}~~~~~
\Sigma[\projbracket{\spn{\mathcal{O}}}\nu]^{-1/2} = \sqrt{\Sigma[\nu]^+}\Pi_{\Sigma[\nu]}^{>\epsilon}.
$$
Following these notations, the result of the proposition reads
\begin{align*}
&\W_2\bigg(\sqrt{\Sigma[\mu_\O]^+}\projbracket{\spn{\mathcal{O}}}\mu_\O,~\sqrt{\Sigma[\nu]^+}\Pi_{\Sigma[\nu]}^{>\epsilon}\nu\bigg)
\\
&\leq
\big(\sqrt{2n}+2\big)
\bigg(\frac{\sigma_\mathrm{max}}{\sigma_\mathrm{min}}\bigg)
\bigg(\alpha\frac{\W_2(\mu_\O,\nu)}{\sigma_\mathrm{min}}\bigg)^{1/2}
\bigg(2\bigg(\frac{\Var\big[\|\mu_\O\|\big]}{\sigma_\mathrm{min}^2}\bigg)^{1/2}  + \alpha\frac{\W_2(\mu_\O,\nu)}{\sigma_\mathrm{min}}\bigg)^{1/2}\\
&\leq
\alpha\big(\sqrt{2n}+2\big)
\bigg(\frac{\sigma_\mathrm{max}}{\sigma_\mathrm{min}}\bigg)
\bigg(\frac{\W_2(\mu_\O,\nu)}{\sigma_\mathrm{min}}\bigg)^{1/2}
\bigg(\bigg(\frac{\Var\big[\|\mu_\O\|\big]}{\sigma_\mathrm{min}^2}\bigg)^{1/2}  + \frac{\W_2(\mu_\O,\nu)}{\sigma_\mathrm{min}}\bigg)^{1/2},
\end{align*}
since $2\leq\alpha$.
To simplify this expression, we use the following inequality, valid for $n\geq0$:
$$
\alpha\big(\sqrt{2n}+2\big)
\bigg(\frac{\sigma_\mathrm{max}}{\sigma_\mathrm{min}}\bigg)
=
3(n+1)\big(\sqrt{2n}+2\big)
\bigg(\frac{\sigma_\mathrm{max}^3}{\sigma_\mathrm{min}^3}\bigg)
\leq 8(n+1)^{3/2}\bigg(\frac{\sigma_\mathrm{max}^3}{\sigma_\mathrm{min}^3}\bigg).
$$
Injected into the equation above, we obtain the result.
\end{proof}

\begin{remark}\label{rem:prop_orthonormalization_asymptotics}
The stability of \ref{item:step1}, expressed in Proposition \ref{prop:stability_step1}, shows that the output distance is of order $\W_2(\mu_\O,\nu)^{1/2}$.
The linear dependence in $\W_2(\mu_\O,\nu)$ has been lost in Proposition \ref{prop:stability_orthonormalization} because of the non-orthogonality of $\O$.
However, when $\O$ comes from an orthogonal representation, we have $\Var\big[\|\mu_\O\|\big]=0$, and the result of the proposition is a linear bound:
$$
\W_2\bigg(\sqrt{\Sigma[\mu_\O]^+}\projbracket{\spn{\mathcal{O}}}\mu_\O,~\sqrt{\Sigma[\nu]^+}\Pi_{\Sigma[\nu]}^{>\epsilon}\nu\bigg)
\leq
8(n+1)^{3/2}\bigg(\frac{\sigma_\mathrm{max}^3}{\sigma_\mathrm{min}^3}\bigg)
\bigg(\frac{\W_2(\mu_\O,\nu)}{\sigma_\mathrm{min}}\bigg).
$$
\end{remark}

\subsection{Analysis of \texttt{LiePCA}}\label{subsec:analysis_liepca}
This section is devoted to the study of \texttt{LiePCA}, \ref{item:step2} of our algorithm.
As presented in the original article \cite{DBLP:journals/corr/abs-2008-04278}, it has been designed to estimate the symmetry group of a manifold $\O \subset \R^n$, based on the observation of a point cloud $X$ sampled on it. 
As written in Equation \eqref{eq:operator_sigma}, \texttt{LiePCA} is based on the operator $\Lambda\colon \M_n(\R)\rightarrow \M_n(\R)$ defined as
\begin{equation*}
    \Lambda(A) = \sum_{1\leq i \leq N} \projbrackethat{\N_{x_i} X} \cdot A \cdot \projbracket{\spn{x_i}},        
\end{equation*}
where $X = \{x_i\}_{i=1}^N$ is the input point cloud, and the $\projbrackethat{\N_{x_i}X}$'s are $n\times n$ matrices, understood as estimators of the projection matrices on the normal spaces of the underlying manifold $\O$.
To study the operator $\Lambda$ through the lenses of the Wasserstein distance, we must define a version of it adapted to measures.
Given a measure $\mu$ on $\R^n$ with support satisfying $0 \notin \mathrm{supp}(\mu)$, and any map $\N[\mu] \colon \mathrm{supp}(\mu) \rightarrow \M_n(\R)$, we define the operator $\Lambda[\mu, \N[\mu]]\colon \M_n(\R)\rightarrow \M_n(\R)$ as
\begin{equation*}
    \Lambda[\mu, \N[\mu]](A) = \int \N[\mu](x) \cdot A \cdot \projbracket{\spn{x}} \d\mu(x).
\end{equation*}
In particular, if $\mu=\mu_\O$ is the uniform measure on the orbit $\O$, and $\N[\O] \colon \O \rightarrow \M_n(\R)$ is the map such that $\N[\O](x) = \projbracket{\N_x \O}$, with $\N_x \O$ the normal space of $\O$ at $x$, then $\Lambda[\mu_\O, \N[\O]]$ is to be understood as the \textit{ideal} \texttt{LiePCA} operator that we wish to estimate.
We denote it by $\Lambda_\O$:
\begin{equation}\label{eq:ideal_Lie_PCA_operator}
\Lambda_\O(A) = \int \projbracket{\N_x \O} \cdot A \cdot \projbracket{\spn{x}} \d\mu_\O(x).
\end{equation}
As we show below, its kernel is $\sym(\O)$, hence it allows us to estimate the symmetry group.

We start by studying the ideal \texttt{LiePCA} operator $\Lambda_\O$ in Section \ref{subsubsec:lie-pca_consistency}.
Next, we quantify in Section \ref{subusbsec:stability_lie_pca_operator} the stability of $\Lambda$ under small variations of $X$. 
Conjugated with stability results for the estimation of normal spaces, which we gather in Section \ref{subusbsec:stability_tangent_space_estimation}, we obtain the main result of this section, Proposition \ref{prop:Lie-PCA}, that guarantees the accuracy of \texttt{LiePCA} in practice.

We stress that such a study has already been initiated in the original article \cite[Sec.~3]{DBLP:journals/corr/abs-2008-04278}, but with a different focus: the authors computed, for several examples of submanifolds $\man\subset\R^n$, the minimal number of input points required for \texttt{LiePCA} to return an accurate estimation of $\sym(\man)$. 
Namely, their results include linear and affine subspaces, spheres, hyperboloids, and quadrics in $\R^n$.
However, they suppose that the estimation of normal spaces $\N_x\man$ is exact.
Our approach is rather the opposite: we are interested in \texttt{LiePCA}'s behavior when it is computed from the data, with the ultimate goal of quantifying the accuracy of our algorithm.

\subsubsection{Consistency of \texttt{LiePCA}}\label{subsubsec:lie-pca_consistency}
We first consider the idealized setting where, instead of a point cloud $X$, we have access to the whole orbit $\O$, as well as $\mu_\O$, the uniform measure on it, the normal spaces $\N_x \O$, and the ideal \texttt{LiePCA} operator $\Lambda_\O$, defined in Equation \eqref{eq:ideal_Lie_PCA_operator}.
We show in the following proposition that its kernel is $\sym(\O)$ and study its eigenvalues, subsequently enabling us to obtain precise estimation results.
From a statistical point of view, this result is a quantification of the confidence one has using \texttt{LiePCA}.

\begin{proposition}\label{prop:consistency_LiePCA_idealcase}
Let $\O$ be the orbit of a representation of a compact Lie group $G$ in $\R^n$ and $\Lambda_\O$ the ideal \texttt{LiePCA} operator.
Then its kernel is equal to $\sym(\O)$.
Moreover, $\Lambda_\O$ is equivariant with respect to the action of $G$ on $\M_n(\R)$ by conjugation.
Last, when $\O$ is the sphere $S^{n-1}$, its nonzero eigenvalues are exactly $\delta_n$ and $\delta'_n$, where 
$$
\delta_n = \frac{2(n-1)}{n(n(n+1)-2)}
~~~~~~\mathrm{and}~~~~~~
\delta'_n = \frac{1}{n},
$$
attained respectively by the symmetric matrices of trace zero and the identity matrix.
\end{proposition}

\begin{proof}
In this proof, we let $\Lambda$ denote the operator $\Lambda_\O$ and write $\phi\colon G \rightarrow \GL_n(\R)$ for the representation.
We divide the proof into three steps, following the three statements of the proposition.

\vspace{.25cm}\noindent
\underline{First step} (Kernel of $\Lambda$).
As shown in \cite[Lemma 2]{DBLP:journals/corr/abs-2008-04278}, the product $\projbracket{\N_x \O} \cdot A \cdot \projbracket{\spn{x}}$ is equal to the projection $\projbracket{(S_x\mathcal{O})^\bot}(A)$, where $S_x\mathcal{O} = \{A\in \M_n(\R) \mid Ax \in \mathrm{T}_x\mathcal{O}\}$. 
Combined with the definition of $\Lambda$ in Equation \eqref{eq:ideal_Lie_PCA_operator}, it follows that
\begin{equation}\label{eq:formulation_lambda_1}
\Lambda(A) = \int \projbracket{(S_x\O)^\bot}(A)\d\mu_\O(x).
\end{equation}
Besides, we observe that, by the linearity of the integral,
\begin{align}\label{eq:formulation_lambda_scalar}
\langle\Lambda(A),A\rangle
= \int \big\langle \projbracket{(S_x\O)^\bot}(A),A \big\rangle \d\mu_\O(x)
=\int \big\| \projbracket{(S_x\O)^\bot}(A) \big\|^2\d\mu_\O(x).
\end{align}
We now prove that the kernel of $\Lambda$ is $\sym(\O)$.
Let $A$ be an eigenvector of $\Lambda$ with eigenvalue $\lambda$. 
In particular, $\langle\Lambda(A),A\rangle = \lambda \|A\|^2$.
Inputting this to Equation \eqref{eq:formulation_lambda_scalar}, we get that $\lambda = 0$ if and only if $A$ is orthogonal to $(S_x\mathcal{O})^\bot$ for each $x\in \O$, that is, if and only if $A\in \bigcap_{x\in \O}S_x\mathcal{O}$.
Moreover, we have seen in Equation \eqref{eq:sym_algebra_formulation} that $\bigcap_{x\in \O}S_x\mathcal{O} = \sym(\O)$.
In conclusion, the kernel of $\Lambda$ is $\sym(\O)$.

\vspace{.25cm}\noindent
\underline{Second step} (Equivariance property).
To study $\Lambda$ further, we will derive another formulation.
Let $x_0\in\O$ be a point, fixed until the end of the proof.
Any other $x\in\O$ can be written $x = \phi(g)x_0$ for some $g\in G$.
By a direct inspection of the definition of $S_x \O$, we see that
\begin{align}\label{eq:formulation_lambda_2}
(S_{x}\O)^\bot &=\phi(g) \cdot (S_{x_0}\O)^\bot \cdot \phi(g)^{-1} \nonumber\\
\mathrm{and}~~~~~
\projbracket{(S_{x}\O)^\bot}(A) &= \phi(g) \cdot \projbracket{(S_{x_0}\O)^\bot}\bigg(\phi(g)^{-1} A \phi(g)\bigg) \cdot \phi(g)^{-1}.
\end{align}
Inserting Equation \eqref{eq:formulation_lambda_2} into Equation \eqref{eq:formulation_lambda_1} yields the expression:
\begin{align}\label{eq:formulation_lambda_3}
\Lambda(A) 
&= \int_G \phi(g) \cdot \projbracket{(S_{x_0}\O)^\bot}\bigg(\phi(g)^{-1} A \phi(g)\bigg) \cdot \phi(g)^{-1}\d(g),
\end{align}
with the integral taken over the Haar measure on $G$.
From this formula, we get that the relation
\begin{align*}
\phi(g)\Lambda(A)\phi(g)^{-1}=\Lambda\bigg(\phi(g)A\phi(g)^{-1}\bigg)
\end{align*}
holds for all $A\in\M_n(\R)$ and $g\in G$.
In other words, $\Lambda$ is an operator from $\M_n(\R)$ to itself that is equivariant with respect to the action by conjugation of $G$.
Consequently, by Schur's lemma, there exists a decomposition $\M_n(\R)=\bigoplus_{j=1}^p V_j$ of this representation into irreps such that, for all $j\in[1\isep p]$, $\Lambda$ restricted to $V_j$ is a homothety.
Their ratios are the eigenvalues of $\Lambda$.

Next, let $(B_i)_{i=1}^{n-d}$ denote an orthonormal basis of $(S_{x_0}\O)^\bot$.
It has dimension $n-d$ since, as we see from its definition, its complement has dimension $d + n(n-1)$.
For any $A\in\M_n(\R)$, the orthogonal projection on $(S_{x_0}\O)^\bot$ can be written as
\[
\projbracket{(S_{x_0}\O)^\bot}(A) = \sum_{i=1}^{n-d} \langle A,  B_i \rangle B_i.
\]
Injecting this into Equation \eqref{eq:formulation_lambda_3} and manipulating the matrix inner product, we see that
\begin{align*}
\langle \Lambda(A), A \rangle
= \sum_{i=1}^{n-d} \int \langle A, \phi(g) B_i \phi(g)^{-1} \rangle^2 \d(g).
\end{align*}
These integrals are matrix coefficients, thus, Schur orthogonality relations allow to decompose
\begin{align*}
\int \langle A, \phi(g) B_i \phi(g)^{-1} \rangle^2 \d(g)
&= \sum_{j=1}^p \int \langle \projbracket{V_j}(A), \phi(g) \projbracket{V_j}(B_i) \phi(g)^{-1} \rangle^2 \d(g)\\
&= \sum_{j=1}^p \|\projbracket{V_j}(A)\|^2 \|\projbracket{V_j}(B_i)\|^2 \dim(V_j)^{-1}.
\end{align*}
Taking the sum over the basis of $(S_{x_0}\O)^\bot$ yields
\begin{align*}
\langle \Lambda(A), A \rangle 
&= \sum_{i=1}^{n-d} \sum_{j=1}^p \|\projbracket{V_j}(A)\|^2 \|\projbracket{V_j}(B_i)\|^2 \dim(V_j)^{-1}
= \sum_{j=1}^p \|\projbracket{V_j}(A)\|^2 \beta_j \dim(V_j)^{-1},
\end{align*}
where $\beta_j = \sum_{i=1}^{n-d} \|\projbracket{V_j}(B_i)\|^2$.
We conclude that the eigenvalues of $\Lambda$ are exactly the
\begin{equation}\label{eq:eigenvalues_formula_beta}
\beta_j\dim(V_j)^{-1},
~j\in[1\isep p]. 
\end{equation}

\vspace{.25cm}\noindent
\underline{Third step} (Case $\O = S^{n-1}$).
We now study the particular case where $\O$ is the unit sphere of $\R^n$.
As we see from its definition in Equation \eqref{eq:formulation_lambda_1}, the \texttt{LiePCA} operator $\Lambda$ depends only on $\O$ and not on $G$.
Thus, without loss of generality, we can suppose that $G$ is $\SO(n)$, acting canonically on $\R^n$.
In this case, the action by conjugation on $\M_n(\R)$ is well understood: it decomposes into three invariant subspaces---the space of the skew-symmetric matrices, the space generated by the identity matrix $I$, and the space of the symmetric matrices of trace zero.
We will denote this decomposition as $\M_n(\R) = \so(n)\oplus V\oplus W$.
As shown in the previous step, $\Lambda$ is an equivariant operator, hence, its eigenspaces are precisely given by these three spaces.

The skew-symmetric matrices correspond to the eigenvalue $0$.
Indeed, the symmetry group of $S^{n-1}$ is $\SO(n)$.
Hence, its Lie algebra, $\so(n)$, forms the kernel of $\Lambda$, as proved in the first step.

To continue, let us observe that, for any $x_0\in S^{n-1}$, the normal space $\N_{x_0} \O$ is generated by $x_0$.
Consequently, the subspace $(S_{x_0}\O)^\bot$, of dimension $1 = n-(n-1)$, is generated by the projection matrix $\projbracket{\spn{x_0}}$.
By a direct computation, one sees that the projection of $\projbracket{\spn{x_0}}$ on $V$ is equal to $I/n$.
Applying Equation \eqref{eq:eigenvalues_formula_beta}, we obtain that $V$ is associated with the eigenvalue 
$$
\frac{\|I/n\|^2}{\dim(V)} = \frac{1}{n}.
$$

Last, we turn to $W$, the symmetric matrices of trace zero.
The projection of $\projbracket{\spn{x_0}}$ on it is equal to $\projbracket{\spn{x_0}}-I/n$.
Just as before, we deduce that $W$ is associated with the eigenvalue
$$
\frac{\|\projbracket{\spn{x_0}}-I/n\|^2}{\dim(W)} = \frac{1-1/n}{n(n+1)/2-1},
$$
which is equal to the quantity $\delta_n$ in the statement of the proposition.
\end{proof}

\begin{remark}\label{rem:universal_constant_Lie-PCA}    
In Proposition \ref{prop:consistency_LiePCA_idealcase}, it is not possible to obtain a `universal' lower bound on the nonzero eigenvalues of $\Lambda_\O$, i.e., a bound that would be independent of $\O$, even with $n$ or $G$ being fixed.
To illustrate this, let us consider the representation $\phi\colon \SO(2) \rightarrow \GL_4(\R)$ with distinct positive weights $(\omega_1,\omega_2)$.
Let $\epsilon>0$, $x_0=(1,0,\epsilon,0)$ a point, and $\O$ its orbit.
Explicitly, 
$$
\O = \big\{ \big(\cos(\omega_1 \theta),\sin(\omega_1 \theta),\epsilon\cos(\omega_2 \theta),\epsilon\sin(\omega_2 \theta)\big) \mid \theta \in \SO(2) \big\}.
$$
As we compute in Appendix \ref{subsubsec:computations_universal_constant}, the \texttt{LiePCA} operator $\Lambda_\O$ admits three nonzero eigenvalues when restricted to the skew-symmetric matrices:
\begin{align*}
\frac{\epsilon^2}{2(1+\epsilon^2)}\frac{w_1^2+w_2^2}{w_1^2+(\epsilon w_2)^2},~~~~~
\frac{1}{8}\bigg(1+\frac{1}{1+\epsilon^2}\frac{(w_1-\epsilon^2w_2)^2}{w_1^2+(\epsilon w_2)^2}\bigg),~~~~~
\frac{1}{8}\bigg(1+\frac{1}{1+\epsilon^2}\frac{(w_1+\epsilon^2w_2)^2}{w_1^2+(\epsilon w_2)^2}\bigg).
\end{align*}
When $\epsilon$ tends to zero, this first eigenvalue goes to zero.
This shows that there is no lower bound on the nonzero eigenvalues.
Nonetheless, an interesting observation can still be made in the case where $\epsilon=1$.
This corresponds to a homogeneous orbit, that is, an orbit whose covariance matrix is a homothety, as evoked in Section \ref{subsubsec:analysis_pca_orbit}.
In this case, the eigenvalues above become 
$$
\frac{1}{4},~~~~~
\frac{1}{8}\bigg(1+\frac{(w_1-w_2)^2}{2(w_1^2+w_2^2)}\bigg),
~~~~~
\frac{1}{8}\bigg(1+\frac{(w_1+w_2)^2}{2(w_1^2+w_2^2)}\bigg),
$$
all belonging to the interval $[1/8, 1/4]$.
More generally, we have observed experimentally that, for a homogeneous orbit $\O$ in $\R^n$, the nonzero eigenvalues of $\Lambda_\O$ belong to $[1/n^2,1/n]$, leaving open the possibility of a `universal' lower bound in this case.
On the one hand, such a bound is rather bad news: a spectral gap of the order of $1/n^2$ makes the identification of the kernel of \texttt{LiePCA} rather impractical in high dimensions. 
However, this will not be a problem for us in the applications of Section \ref{sec:applications}: although the datasets will be given in high dimension, we will manage to project them in lower dimension, while preserving the representation structure. 
Specifically, we will use the algorithm for $\SO(2)$ up to dimension 32, for $T^2$ up to dimension 14, and for $\SO(3)$ up to dimension 11.
\end{remark}

\subsubsection{Stability of \texttt{LiePCA}}\label{subusbsec:stability_lie_pca_operator}
We now study the stability of the \texttt{LiePCA} operator with respect to small variations in Wasserstein distance.
Let $\mu$ and $\nu$ be two measures on $\R^n$, $\N[\mu] \colon \mathrm{supp}(\mu) \rightarrow \M_n(\R)$ and $\N[\nu] \colon \mathrm{supp}(\nu) \rightarrow \M_n(\R)$ two maps, and consider the corresponding operators $\Lambda[\mu, \N[\mu]]$ and $\Lambda[\nu, \N[\nu]]$.
We wish to give an upper bound on the operator norm $\|\Lambda[\mu, \N[\mu]]-\Lambda[\nu, \N[\nu]]\|_\mathrm{op}$.
As it might already be expected, this bound will depend on the distance between $\mu$ and $\nu$, but also between $\N[\mu]$ and $\N[\nu]$.
To obtain explicit bounds, we choose an optimal transport plan $\pi$ for the Wasserstein distance $\W_2(\mu,\nu)$, and define
\begin{equation}\label{eq:omega}
\Omega = \int \| \N[\mu](x) - \N[\nu](y) \|\d\pi(x,y).    
\end{equation}
This quantity is a measure of the distance between the estimations of normal spaces.
We prove in this section a lemma that formulates the stability of the \texttt{LiePCA} operator in terms of $\W_2(\mu,\nu)$ and $\Omega$.
In the next section, we will show how $\Omega$ can be bounded by $\W_2(\mu,\nu)$, once a specific normal space estimator has been chosen, and provided regularity conditions on $\mu$.

\begin{lemma}
\label{lem:wasserstein_stability_lie_pca}
If the essential suprema $\sup(\|\N[\mu]\|)$ and $\sup(\|\mu\|^{-1})$ are finite then it holds that
$$
\| \Lambda[\mu, \N[\mu]]-\Lambda[\nu, \N[\nu]] \|_\mathrm{op}
\leq 2\sup(\|\N[\mu]\|)\sup(\|\mu\|^{-1}) \W_2(\mu,\nu) + \Omega.
$$
\end{lemma}

\begin{proof}
By definition of the operator norm, we have
\begin{align*}
    \| \Lambda[\mu, \N[\mu]]-\Lambda[\nu, \N[\nu]] \|_\mathrm{op} = \sup \| \Lambda[\mu, \N[\mu]](A)-\Lambda[\nu, \N[\nu]](A) \| 
\end{align*}
where the supremum is taken over all matrices $A$ such that $\|A\|=1$. Let $A$ be such a matrix.
We first use the inequality
\begin{align*}
\| \Lambda[\mu, \N[\mu]](A) -\Lambda[\nu, \N[\nu]] \|
&\leq \int \big\|  N[\mu](x) \cdot A \cdot \projbracket{\spn{x}} -  N[\nu](y) \cdot A \cdot \projbracket{\spn{y}} \big\| \d\pi(x,y).
\end{align*}
The term under the integral is upper bounded by
\begin{align}\label{terms:bounds_LiePCA}
\underbrace{\|\N[\mu](x) \cdot A \cdot (\projbracket{\spn{x}}-\projbracket{\spn{y}}) \|}_{(\mathrm{a})}
+
\underbrace{\|(\N[\mu](x) - \N[\nu](y))  \cdot A \cdot \projbracket{\spn{y}} \|}_{\mathrm{(b)}}.
\end{align}
To bound Term $\hyperref[terms:bounds_LiePCA]{(\mathrm{a})}$, we use the submultiplicativity of the Frobenius norm:
\begin{align}\label{eq:proof_LiePCA_terma}
\|\N[\mu](x) \cdot A \cdot (\projbracket{\spn{x}}-\projbracket{\spn{y}}) \|
\leq \| \N[\mu](x) \| \cdot \| A \| \cdot \| \projbracket{\spn{x}}-\projbracket{\spn{y}} \|.
\end{align}
By writing $\projbracket{\spn{x}} = (x/\|x\|) \cdot (x/\|x\|)^\top$ and $\projbracket{\spn{y}} = (y/\|y\|) \cdot (y/\|y\|)^\top$, we see that Lemma \ref{lem:stability_outer_product}, stated in Section \ref{subsubsec:analysis_pca_orbit}, yields $\|\projbracket{\spn{x}}-\projbracket{\spn{y}}\| 
\leq \|x/\|x\| - y/\|y\|\|$. Moreover, 
\begin{align*}
\bigg \|\frac{x}{\|x\|} - \frac{y}{\|y\|}\bigg\|
&= \frac{1}{\|x\|\|y\|} \bigg\| \|y\|x - \|x\|y \bigg\| \\
&\leq \frac{1}{\|x\|\|y\|} \bigg(\bigg\| \|y\|x - \|y\|y \bigg\| + \bigg\| \|y\|y - \|x\|y \bigg\|\bigg) 
= \frac{2}{\|x\|}\|x-y\|.
\end{align*}
We eventually deduce from Equation \eqref{eq:proof_LiePCA_terma} that Term $\hyperref[terms:bounds_LiePCA]{(\mathrm{a})}$ is upper bounded by 
$\sup(\|\N[\mu]\|) \cdot 1 \cdot 2 \sup(\|\mu\|^{-1}) \|x-y\|$.
Consequently, its integral over the transport plan $\pi$ is bounded by
\begin{equation}\label{eq:proof_LiePCA_terma_integral}
    2 \sup(\|\N[\mu]\|) \sup(\|\mu\|^{-1}) \int  \|x-y\| \d\pi(x,y)
    \leq 2\sup(\|\N[\mu]\|)\sup(\|\mu\|^{-1}) \W_2(\mu,\nu),
\end{equation}
where we applied Jensen's inequality.
To finish the proof, we turn to Term $\hyperref[terms:bounds_LiePCA]{(\mathrm{b})}$. The projection matrix $\projbracket{\spn{y}}$ has Frobenius norm equal to $1$.
Thus, the integral of Term $\hyperref[terms:bounds_LiePCA]{(\mathrm{b})}$ is upper bounded by $\int \|(\N[\mu](x) - \N[\nu](y))\| \d \pi(x,y)$, that is, $\Omega$.
Using this inequality and Equation \eqref{eq:proof_LiePCA_terma_integral} in Equation \eqref{terms:bounds_LiePCA} yields the result.
\end{proof}

\subsubsection{Stability of normal space estimation}\label{subusbsec:stability_tangent_space_estimation}
To turn Lemma \ref{lem:wasserstein_stability_lie_pca} into a result applicable in practice, we need to choose a method for estimating normal spaces and bound the term $\Omega$ defined in Equation \eqref{eq:omega}.
For any $x\in\O$, we denote by $\N_{x} \mathcal{O}$ the normal space of $\O$ at $x$.
Besides, let $X = \{x_i\}_{i=1}^N$ denote the observed point cloud.
As discussed in Section \ref{subsubsec:step2_estimation_normal_spaces}, for any integer $i\in[0\isep N]$, we chose as an estimator of $\projbracket{\N_{x_i} \mathcal{O}}$ the matrix
\begin{equation}\label{eq:definition_estimator_normalspaces}
\projbrackethat{\N_{x_i} X} = I - \Pi_{x_i}^{l,r}[X],    
\end{equation}
where $\Pi_{x_i}^{l,r}[X]$ is the projection matrix on any $l$ top eigenvectors of the \textit{local covariance matrix} $\Sigma_{x_i}^r[X]$ centered at $x_i$ and at scale $r$, itself defined as
$$
\Sigma_{x_i}^r[X] = \frac{1}{|Y|} \sum_{y \in Y} (y-x_i) (y-x_i)^\top,    
$$
where $Y = \{y \in X \mid \|y-x_i\|\leq r\}$, the set of input points at distance at most $r$ from $x_i$.
These notions directly translate to the measure-theoretical framework: if $\mu$ is a measure on $\R^n$, we define its \textit{local covariance matrix} centered at $x$ at scale $r$ as 
$$
\Sigma_x^r[\mu] = \int_{\mathcal{B}(x,r)}  (y-x) (y-x)^\top \frac{\d\mu(x)}{\mu(\mathcal{B}(x,r))},
$$
where $\mathcal{B}(x,r)$ is the closed ball of $\R^n$ centered at $x$ and of radius $r$.
Just as in the discrete case, we denote by $\Pi_{x}^{l,r}[\mu]$ the projection matrix on any $l$ top eigenvectors of $\Sigma_x^r[\mu]$.

It is direct to see that, if $\mu$ is the empirical measure on the finite point cloud $X$, then $\Sigma^r_x[X]$ and $\Sigma^r_x[\mu]$ coincide for every $x\in X$.
That is to say, one can study discrete covariance matrices via their measure-theoretic formulations.
Besides, if $\mu = \mu_\man$ is a `regular enough' measure on a submanifold $\man\subset\R^n$ of dimension $l$, and $x \in \man$, then it is known that $(l+2)/r^2\cdot\Sigma_x^r[\mu_\man]$ tends to $\projbracket{\mathrm{T}_x \man}$ when $r$ goes to zero, with $\projbracket{\mathrm{T}_x \man}$ the projection matrix on the tangent space.
To study how accurate this estimation is, it is natural to consider the following decomposition, which can be seen as a bias-variance trade-off.
For any $x,y\in\R^n$,
\begin{align*}
&\bigg\|\frac{1}{l+2}\projbracket{\mathrm{T}_x \man} - \frac{1}{r^2}\Sigma_y^r[\nu]\bigg\| \\
&\leq
\underbrace{\bigg\|
\frac{1}{l+2}\projbracket{\mathrm{T}_x \man}-\frac{1}{r^2}\Sigma_x^r[\mu_\man]\bigg\|}_{\mathrm{consistency}}
+
\underbrace{\bigg\|
\frac{1}{r^2}\Sigma_x^r[\mu_\man] - \frac{1}{r^2}\Sigma_y^r[\mu_\man]
\bigg\|}_{\mathrm{spatial~stability}}
+
\underbrace{\bigg\|
\frac{1}{r^2}\Sigma_y^r[\mu_\man] - \frac{1}{r^2}\Sigma_y^r[\nu]
\bigg\|}_{\mathrm{measure~stability}}.
\end{align*}
To bound these quantities, one needs additional assumptions, in particular regarding the geometry of $\man$.
We state below a result that involves the \textit{reach} of $\man$, denoted $\reach(\man)$. 
It is a quantity associated with submanifolds of $\R^n$, and is closely connected to the notion of curvature. 
We refer the reader to \cite{federer1959curvature} for an exposition of this notion.

\begin{lemma}[{\cite{tinarrage2023recovering}}]
\label{lem:consistency_loccovnorm}
Let $\man\subset \R^n$ be a compact $\mathcal{C}^2$-submanifold of dimension $l$ and $\mu_\man$ a measure, absolutely continuous with respect to the $l$-dimensional Hausdorff measure restricted to $\man$, and such that its density $f$ is $L$-Lipschitz (with respect to the geodesic distance) and bounded by $f_\mathrm{min}, f_\mathrm{max} > 0$. Moreover, let $\nu$ be any measure on $\R^n$, $p\geq1$ a real number, and choose $x\in \man$, $y\in \R^n$ and $r>0$ such that
\[
\|x-y\| \leq \bigg(\frac{\W_p(\mu,\nu)}{r^{l-1}}\bigg)^\frac{1}{2}
~~~~~~\mathrm{and}~~~~~~~
4 \bigg( \frac{\W_p(\mu,\nu)}{V_l f_\mathrm{min} (23/24)^l} \bigg)^\frac{1}{l+1} \leq r < \frac{1}{2}\mathrm{reach}(\man).
\]
where $V_l$ is the volume of the unit ball of $\R^l$.
Then it holds that
\begin{align*}
&\bigg\|\projbracket{\mathrm{T}_x \man} - \frac{l+2}{r^2}\Sigma_y^r[\nu]\bigg\|
\leq 
cr + c'\|x-y\| + c'' \bigg(\frac{\W_p(\mu,\nu)}{r^{l+1}}\bigg)^\frac{1}{2},
\end{align*}
where the constants are given in Equations \eqref{eq:loccovmatrix_consistency}, \eqref{eq:loccovmatrix_spatialstability}, and \eqref{eq:loccovmatrix_measurestability} below.
Moreover, if the left-hand term of the inequality is lower than $1/2$, then
\begin{align*}
&\big\|\projbracket{\mathrm{T}_x \man} - \Pi_{y}^{l,r}[\nu]\big\|
\leq \sqrt{2}\bigg(
cr + c'\|x-y\| + c'' \bigg(\frac{\W_p(\mu,\nu)}{r^{l+1}}\bigg)^\frac{1}{2}\bigg).
\end{align*}
\end{lemma}

\begin{proof}
The lemma is a consequence of the three following bounds, whose proofs are given in \cite{tinarrage2023recovering} under the names Proposition 4.1, Lemma 4.5, and Lemma 4.7.
The first states that, for any $r<\reach(\man)/2$ and $x\in \man$, one has
\begin{align}
&\bigg\| \frac{l+2}{r^2} \Sigma_x^r[\mu_\man] - \projbracket{\mathrm{T}_x \O}\bigg\|\leq c r\nonumber\\
\mathrm{with}~~~~~
&c = \frac{6(l+2)}{\reach(M)}+ \frac{l+2}{f_\mathrm{min}}\bigg(\frac{24}{23}\bigg)^l
\bigg(20L \bigg(\frac{5}{4}\bigg)^l + \frac{f_\mathrm{max}l}{2\reach(\man)}\bigg(5+2^l+2\bigg(\frac{5}{2}\bigg)^l\bigg)\bigg).
\label{eq:loccovmatrix_consistency}
\end{align}
Secondly, for $r<\reach(\man)/2$, $x\in \man$ and $y\in\R^n$ such that $\|x-y\|<r/4$, it holds that
\begin{equation}\label{eq:loccovmatrix_spatialstability}
\bigg\|\frac{l+2}{r^2}\Sigma_x^r[\mu_\man] - \frac{l+2}{r^2}\Sigma_y^r[\mu_\man]\bigg\|\leq c' \|x-y\|
~~~~\mathrm{with}~~~~
c' = l 2^{l+1}\bigg(1 + 4\frac{5^{l-1}}{3^l}\bigg) \bigg(\frac{30}{23}\bigg)^l \frac{f_\mathrm{max} }{f_\mathrm{min} }.
\end{equation}
Last, for any $y\in\R^n$ at distance at most  $(\W_p(\mu,\nu)/r^{l-1})^\frac{1}{2}$ from $\mathrm{supp}(\mu)$, and for $r$ such that $\W_p(\mu,\nu) \leq \min(1, V_l f_\mathrm{min} (23/24)^l) (r/4)^{l+1}$ and $\leq r < \frac{1}{2}\mathrm{reach}(\man)$, one has
\begin{align}
\bigg\|\frac{l+2}{r^2}\Sigma_y[\mu]-\frac{l+2}{r^2}\Sigma_y[\nu]\bigg\| \leq c'' \bigg(\frac{\W_p(\mu,\nu)}{r^{l+1}}\bigg)^\frac{1}{2} \nonumber\\  
\mathrm{with}~~~~~
c'' = l2^{l+1}(l+2)\bigg(2^{l-1}+2\frac{12\cdot5^{l-1}+1}{3^l} +2^{l+3}\bigg(\bigg(\frac{3}{2}\bigg)^{l-1}+1\bigg)\bigg),
\label{eq:loccovmatrix_measurestability}
\end{align}
and where $V_l$ denotes the volume of the unit ball of $\R^l$.
Summing these inequalities gives the bound on $\|\projbracket{\mathrm{T}_x \man} - \frac{l+2}{r^2}\Sigma_y^r[\nu]\|$.
The second statement is obtained by application of the Davis-Kahan theorem (Lemma \ref{lem:davis_kahan}) and by observing that the $l^\mathrm{th}$ eigengap of $\projbracket{\mathrm{T}_x \man}$ is~$1$.
\end{proof}

Based on this lemma, we deduce a bound on the quantity $\Omega$ introduced in Equation \eqref{eq:omega}.
In this context, the manifold $\man$ is the orbit $\O$ and the measure $\mu_\man$ is the uniform measure on $\O$, permitting a simplification of the constants $c$, $c'$, and $c''$.

\begin{lemma}
\label{lem:wasserstein_stability_tangent_space}
Let $\O$ be an orbit of representation of a compact Lie group in $\R^n$, $l$ its dimension, and $\mu_\O$ the uniform measure on it.
Let also $\nu$ be any measure on $\R^n$ and $\pi$ an optimal transport plan for $\W_2(\mu_\O,\nu)$.
Let $r>0$ be such that
$$\big(6\rho\big)^2\W_2(\mu_\O,\nu)^{\frac{1}{l+1}} \leq r \leq \big(6\rho\big)^{-1} ~~~~\mathrm{where}~~~~~ 
\rho = \bigg(16l(l+2)6^l\bigg)\frac{\max(\mathrm{vol}(\O),\mathrm{vol}(\O)^{-1})}{\min(1,\reach(\O))}.$$
Then it holds that
\begin{equation*}
\int \big\| \projbracket{\N_x \O} - \big(I - \Pi_{y}^{r,l}[\nu]\big) \big\|\d\pi(x,y)
\leq 
\sqrt{2}\rho \bigg(r + 3\bigg(\frac{\W_2(\mu_\O,\nu)}{r^{l+1}}\bigg)^{1/2}\bigg). 
\end{equation*}
\end{lemma}

\begin{proof}
We first observe that, as a consequence of the relation $I = \projbracket{\T_x \O} + \projbracket{\N_x \O}$, the integral is equal to $\int \| \projbracket{\T_x \O} - \Pi_{y}^{r,l}[\nu] \|\d\pi(x,y)$.
We split it into the set $\Delta$ and its complement $\Delta^\complement$, where we define $\Delta = \{(x,y)\in \R^n\times\R^n \mid \|x-y\|\leq \alpha\}$, and $\alpha = (\W_2(\mu_\O,\nu)/r^{l-1})^\frac{1}{2}$.
As it will be clearer later in the proof, the constant $\alpha$ has been chosen to calibrate the contributions of $\Delta$ and $\Delta^\complement$.
On the set $\Delta^\complement$, Markov's inequality yields
\begin{align*}
\pi\big(\Delta^\complement\big) 
&= \pi\big(\big\{(x,y)\in\R^n\times\R^n \mid \|x-y\|^2\geq \alpha^2\big\}\big)
\leq \bigg(\frac{\W_2(\mu_\O,\nu)}{\alpha}\bigg)^2.
\end{align*}
Since $(\W_2(\mu_\O,\nu)/\alpha)^2 = \W_2(\mu_\O,\nu) r^{l-1}$, we deduce that
\begin{align}\label{eq:proof_boundomega_deltacomplement}    
\int_{\Delta^\complement} \big\| \projbracket{\T_x \O} - \Pi_{y}^{r,l}[\nu] \|\d\pi(x,y)
&\leq \sup \big\| \projbracket{\T_x \O} - \Pi_{y}^{r,l}[\nu] \| \cdot \pi(\Delta^\complement)\nonumber\\ 
&\leq 2\sqrt{l} \W_2(\mu_\O,\nu) r^{l-1}.
\end{align}
We now turn to the set $\Delta$.
For $(x,y)\in\Delta$, the hypotheses of Lemma \ref{lem:consistency_loccovnorm} are satisfied, hence
\begin{align}\label{eq:proof_boundomega_applicationlemma}
\bigg\|\projbracket{\mathrm{T}_x \O} - \frac{l+2}{r^2}\Sigma_y^r[\nu]\bigg\|
\leq 
cr + c'\|x-y\| + c'' \bigg(\frac{\W_2(\mu_\O,\nu)}{r^{l+1}}\bigg)^\frac{1}{2}.
\end{align}
Moreover, in our context, we have $L=0$ and $f_\mathrm{min}=f_\mathrm{min}=\mathrm{vol}(\O)^{-1}$, where the volume of $\O$ refers to its $l$-dimensional Hausdorff measure. Simple manipulations show that we have 
$$
c\leq \max(\mathrm{vol}(\O),\mathrm{vol}(\O)^{-1}) \frac{16l(l+2)}{\reach(\O)}\bigg(\frac{60}{23}\bigg)^l, ~~~~~
c'\leq4l \bigg(\frac{100}{23}\bigg)^l 
~~~~~\mathrm{and}~~~~~
c''\leq16l(l+2)6^l.
$$
In particular, they are all lower than 
$$
\rho = \bigg(16l(l+2)6^l\bigg) \frac{\max(\mathrm{vol}(\O),\mathrm{vol}(\O)^{-1})}{\min(1,\reach(\O))}.
$$
Thus, we deduce from Equation \eqref{eq:proof_boundomega_applicationlemma} that
\begin{align*}
\bigg\|\projbracket{\mathrm{T}_x \O} - \frac{l+2}{r^2}\Sigma_y^r[\nu]\bigg\|
&\leq 
\rho r + \rho\|x-y\| + \rho\bigg(\frac{\W_2(\mu_\O,\nu)}{r^{l+1}}\bigg)^\frac{1}{2}
\leq \frac{1}{6} + \frac{1}{6} + \frac{1}{6} = \frac{1}{2}.
\end{align*}
The fact that each term is bounded by $1/6$ comes from the hypothesis of the lemma regarding $r$.
As a consequence, we can apply the second point of Lemma \ref{lem:consistency_loccovnorm} and get
\begin{align*}
\bigg\|\projbracket{\mathrm{T}_x \O} -  \Pi_{y}^{r,l}[\nu]\bigg\|
&\leq 
\sqrt{2}\rho\bigg(r + \|x-y\| + \bigg(\frac{\W_2(\mu_\O,\nu)}{r^{l+1}}\bigg)^\frac{1}{2}\bigg).
\end{align*}
Integrating this expression over $\Delta$ yields
\begin{align*}
\int_{\Delta} \big\| \projbracket{\T_x \O} - \Pi_{y}^{r,l}[\nu] \|\d\pi(x,y)
&\leq  
\sqrt{2}\rho \bigg(
r + \int_{\Delta}\|x-y\|\d\pi(x,y) + \bigg(\frac{\W_2(\mu_\O,\nu)}{r^{l+1}}\bigg)^\frac{1}{2}\bigg)\\
&\leq
\sqrt{2}\rho\bigg(
r + \W_2(\mu_\O,\nu) + \bigg(\frac{\W_2(\mu_\O,\nu)}{r^{l+1}}\bigg)^\frac{1}{2}\bigg),
\end{align*}
where we used Jensen's inequality on the second line.
Combined with Equation \eqref{eq:proof_boundomega_deltacomplement}, we get a bound on the full integral:
\begin{align*}
\int \big\| \projbracket{\T_x \O} - \Pi_{y}^{r,l}[\nu] \|\d\pi
\leq  2\sqrt{l} \W_2(\mu_\O,\nu) r^{l-1} + \sqrt{2}\rho\bigg(
r + \W_2(\mu_\O,\nu) + \bigg(\frac{\W_2(\mu_\O,\nu)}{r^{l+1}}\bigg)^\frac{1}{2}\bigg).
\end{align*}
Besides, using the hypothesis on $r$, we see that
\[
2\sqrt{l} \W_2(\mu_\O,\nu) r^{l-1} \leq \sqrt{2}\rho\bigg(\frac{\W_2(\mu_\O,\nu)}{r^{l+1}}\bigg)^\frac{1}{2}
~~~~~\mathrm{and}~~~~~
\W_2(\mu_\O,\nu) \leq \bigg(\frac{\W_2(\mu_\O,\nu)}{r^{l+1}}\bigg)^\frac{1}{2},
\]
and the result of the lemma follows.
\end{proof}

\begin{remark}\label{rem:estimation_tangentspace_bias_variance}
The inequality of Lemma \ref{lem:wasserstein_stability_tangent_space} can be seen as a bias-variance trade-off, the former being $r$, and the latter $(\W_2(\mu_\O,\nu)/r^{l+1})^{1/2}$.
In other words, the estimation of tangent and normal spaces is more accurate when $r$ is small, but becomes less stable with respect to variations.
Besides, a word should be said about the exponent $1/2$ on the Wasserstein distance.
In the general context of localization of measures, it can be shown that this exponent is optimal \cite[Remark 4.9]{tinarrage2023recovering}.
However, when it comes to the estimation of tangent spaces via local PCA, we do not know whether it can be improved.
We stress that any improvement of the bound in Lemma \ref{lem:consistency_loccovnorm} would result in a direct improvement of Lemma \ref{lem:wasserstein_stability_tangent_space}, and other results to come.
\end{remark}

A direct combination of Lemmas \ref{lem:wasserstein_stability_lie_pca} and \ref{lem:wasserstein_stability_tangent_space} yields our main result regarding \texttt{LiePCA}. 

\begin{proposition}\label{prop:Lie-PCA}
Let $\O$ be an orbit of representation of a compact Lie group in $\R^n$, $l$ its dimension, and its $\mu_\O$ the uniform measure.
Also, let $X\subset\R^n$ be a finite point cloud and $\mu_X$ be its empirical measure.
We consider the ideal \texttt{LiePCA} operator $\Lambda_\O$, and the operator $\Lambda$ computed from $X$ using the estimators of normal spaces defined in Equation \eqref{eq:definition_estimator_normalspaces}.
Let $r>0$ and $\rho$ be such as in Lemma \ref{lem:wasserstein_stability_tangent_space}.
Then it holds that
\begin{equation*}
\| \Lambda_\O-\Lambda \|_\mathrm{op}
\leq 
\sqrt{2}\rho \bigg(r + 4\bigg(\frac{\W_2(\mu_\O,\mu_X)}{r^{l+1}}\bigg)^{1/2}\bigg).
\end{equation*}
\end{proposition}

\subsection{Stability of minimization problems}\label{subsec:stability-minimization}

The execution of the algorithm involves a minimization program, either during \ref{item:step3}, formulated in Equation \eqref{eq:minimization_stiefel_2}, or during \ref{item:step3}', formulated in Equation \eqref{eq:minimization_grassmann}.
In both cases, it consists of several minimizations over $\Ort(n)$, solved in practice with gradient descent.
Studying the accuracy of such algorithms is, in general, a difficult task.
Indeed, their outputs depend crucially on the stopping criterion chosen by the user and are subject to numerical errors.
But most of all, minimization via gradient descent seeks only a \textit{local} minimum.
Instead of diving into these issues, which are out of the scope of this article, we will consider that the algorithm always returns a global minimum. 
Our objective with this section is to prove that the global minima of Equations \eqref{eq:minimization_stiefel_2} and \eqref{eq:minimization_grassmann} are close to the `ideal' object: a Lie algebra that underlies the data. 

At the core of this analysis lies the question of the rigidity of Lie subalgebras, which we introduce in Section \ref{subsubsec:rigidity_lie_subalgebras}. 
We then derive stability results for \ref{item:step3} and \ref{item:step3}' in Section \ref{susubsec:stability_step3}, enabling us to deduce in Section \ref{subsubsec:stability_minimization_Hausdorff} the consistency of \ref{item:step4}.

\subsubsection{Rigidity of Lie subalgebras}\label{subsubsec:rigidity_lie_subalgebras}
Let us recall the context of this article: $\O$ is an orbit of a representation of a compact Lie group in $\R^n$ and $X$ is a finite point cloud.
We write $\Lambda$ for the \texttt{LiePCA} operator obtained from $X$ and the estimation of its normal spaces via local PCA. 
Moreover, the ideal \texttt{LiePCA} operator, obtained by replacing $X$ with $\O$, is denoted by $\Lambda_\O$.
Let us consider \ref{item:step3}, which we perform with an input compact Lie group denoted by $G$, and whose Lie algebra is $\g$.
We consider the formulation given in Equation \eqref{eq:minimization_stiefel}:
\begin{equation*}
    \arg \min \sum_{i=1}^d  \| \Lambda(A_i) \|^2
    ~~~~\mathrm{s.t.}~~~~
    (A_1,\dots,A_d) \in \mathcal{V}(\g, \mathfrak{so}(n)).
\end{equation*}
As a first approximation to study this equation, we consider the ideal case $\Lambda_\O$.
This operator can be diagonalized as $\Lambda_\O = \sum_{j=1}^p \lambda_j \projbracket{V_j}$, where the eigenspaces are orthogonal and satisfy $\M_n(\R) = \bigoplus_{j=1}^p V_j$. 
As shown in Proposition \ref{prop:consistency_LiePCA_idealcase}, its kernel is equal to $\sym(\O)$, i.e., the eigenvalue $0$ is associated with the eigenspace $\sym(\O)$.
Denoting by $\lambda>0$ its smallest nonzero eigenvalue, it holds that, for every matrix $A\in\M_n(\R)$,
\begin{equation*}
\big\| \Lambda_\O(A)\big\|
~\geq~ 
\lambda \cdot\big\|\projbracket{\sym(\O)^\bot}(A)\big\|. 
\end{equation*}
The objective function $\sum_{i=1}^d  \| \Lambda_\O(A_i) \|^2$ is equal to zero precisely when the orthonormal frame $(A_1,\dots,A_d)$ is included in $\sym(\mathcal{O})$.
Moreover, by denoting by $\projbracket{\spn{A_i}_{i=1}^d}$ the projection on the space they span, we deduce
\begin{align}\label{eq:bounds_minimization_step3}
\sum_{i=1}^d  \| \Lambda_\O(A_i) \|^2
~&\geq~ 
\lambda^2\cdot
\sum_{i=1}^d\big\| \projbracket{\sym(\O)^\bot} (A_i) \big\|^2\nonumber\\
&=~ 
\lambda^2\cdot
\big\| \projbracket{\sym(\O)^\bot} \projbracket{\spn{A_i}_{i=1}^d} \big\|^2.
\end{align}
This equation will prove useful later, as it provides a nontrivial lower bound on the objective function.
Moreover, this Frobenius norm can be understood as a kind of distance between the spaces $\sym(\O)$ and $\spn{A_i}_{i=1}^d$, adapted to the case where they potentially are of different dimensions.
Indeed, when their dimensions coincide, basic manipulations show that
\begin{align}\label{eq:bounds_minimization_step3_equaldimension}
\big\| \projbracket{\sym(\O)^\bot}
\projbracket{\spn{A_i}_{i=1}^d} \big\|^2
~=~
\frac{1}{2}\big\|\projbracket{\sym(\O)}-\projbracket{\spn{A_i}_{i=1}^d}\big\|^2,
\end{align}
where we recognize the distance between subspaces defined in Section \ref{subsec:grassmann_lie_notgroup} (see Equation \eqref{eq:distance_grassmannian}).

Back to the original problem, and as a consequence of this discussion, we are compelled to study the minimum of the function 
$$
\h \mapsto \big\| \projbracket{\mathfrak{s}^\bot}
\projbracket{\h} \big\|^2
~~~~\mathrm{where}~~~~
\h \in \mathcal{G}(G, \mathfrak{so}(n)),
$$
and where $\mathfrak{s}$ is a Lie subalgebra of $\so(n)$ and $\mathcal{G}(G, \mathfrak{so}(n))$ the Grassmann variety of pushforward Lie algebras of $G$.
We have seen in Section \ref{subsec:grassmann_lie} that the connected components of $\mathcal{G}(G, \mathfrak{so}(n))$ are generated by the action of $\Ort(n)$ on it by conjugation, yielding the quotient space $\mathfrak{orb}(G,n)$.
Let us consider one such component, denoting by $(B^1,\dots,B^p) \in \mathfrak{orb}(G,n)$ an element, and $\h = \spn{\diag\big(B^k_j\big)_{k=1}^p }_{i=1}^d$ the corresponding Lie algebra.
On this component, the function reads
$$
O \mapsto \big\| \projbracket{\mathfrak{s}^\bot}
\projbracket{O \h O^\top} \big\|^2
~~~~~\mathrm{where}~~~~~
O\in\Ort(n).$$
If a certain conjugate Lie algebra $O \h O^\top$ belongs to $\mathfrak{s}$, then the minimum of this function is zero.
Otherwise, it is positive, and finding a lower bound has already been studied in the context of \textit{rigidity of Lie subalgebras}.
Namely, the $d$-dimensional Lie subalgebra $\h\subset \so(n)$ is called \textit{rigid} if its orbit $\{O \h O^\top \mid O \in \Ort(n)\}$, seen in the space of $d$-dimensional Lie subalgebra of $\so(n)$, admits a neighborhood consisting of only isomorphic Lie subalgebras.
It is, in general, a hard problem to determine which Lie subalgebras are rigid \cite{barrionuevo2023deformations}.
As a positive result, a sufficient condition for rigidity can be defined from the cohomology of $\h$ \cite{crainic2014survey}.
Another instance of the literature where rigidity is involved is that of \textit{switched linear systems}.
In this context, it has been studied the closely related problem of recovering the Lie subalgebra $\mathfrak{h}$, based on the observation of a sufficiently close vector space $\mathcal{A}\subset \so(n)$ \cite{agrachev2001lie,agrachev2010towards}.
Our problem is similar.

As it will be clearer in the next section, quantifying the rigidity of the Lie algebra $\h$ proves to be crucial to state explicit consistency results for our algorithm.
To this end, we define
\begin{equation}\label{eq:gamma_def}
\Rig(G,n) = \inf_{\h,\mathfrak{s}} 
\big\| \projbracket{\mathfrak{s}^\bot}
\projbracket{\h} \big\|^2,
\end{equation}
where the infimum is taken over all pairs $(\h,\mathfrak{s})$ such that $\h \in \mathcal{G}(G, \mathfrak{so}(n))$ and $\mathfrak{s}$ is a Lie subalgebra of $\so(n)$, pushforward of any compact Lie group (potentially $G$ itself), and that does not contain no conjugation $O \h O^\top$ for $O\in\Ort(n)$.
When there are no such pairs, we set $\Rig(G,n)=+\infty$.

When $G$ is $\SU(2)$ or $\SO(3)$, this infimum is greater than zero.
This comes from the fact that, being fixed an ambient dimension $n$, only a finite number of non-equivalent representations exist.
If the group is Abelian, however, the infimum may be zero. 
In this case, we restrict the definition of $\Rig(G,n)$ to the finite set of representations used in the algorithm (described in Section \ref{subsubsec:orbit_equivalence_so(2)}).
More precisely, given a positive integer $\omega_\mathrm{max}$, we define $\Rig(G,n,\omega_\mathrm{max})$ as in Equation \eqref{eq:gamma_def}, but imposing the additional assumption that the Lie subalgebras $\h$ and $\mathfrak{s}$ are spanned by matrices whose spectra come from primitive integral vectors of coordinates at most $\omega_\mathrm{max}$.
The following lemma, though not used in this article, provides an indicative lower bound.

\begin{lemma}\label{lem:lower_bound_gamma_abelian}
For any compact Lie group $G$, it holds that 
$$
\Rig(G,n,\omega_\mathrm{max}) 
\geq \Rig(\SO(2),n,\omega_\mathrm{max}) 
\geq 4/(n\omega_\mathrm{max}^2).
$$
\end{lemma}

\begin{proof}
We start with $G=\SO(2)$.
Without loss of generality, we suppose that $n$ is even and define $m = n/2$.
Let $\h$ and $\mathfrak{s}$ be as in Equation \eqref{eq:gamma_def}, and choose a unit skew-symmetric matrix $A$ that spans $\h$.
We have
\begin{equation}\label{eq:gamma_freqmin_proof}
\big\| \projbracket{\mathfrak{s}^\bot}
\projbracket{\h} \big\|^2
~=~
\big\| \projbracket{\mathfrak{s}^\bot}
(A) \big\|^2
~=~
\inf_{B \in \mathfrak{s}} \big\| B
- A \big\|^2.
\end{equation}
Let $B\in\mathfrak{s}$ be of unit norm.
The matrices $A$ and $B$ are skew-symmetric, thus there exist two $m$-tuples of non-negative real numbers $(\lambda_{i})_{i=1}^m$ and $(\lambda'_{i})_{i=1}^m$ such that their spectra takes the form
$$
\pm\iu\lambda_1/\delta , ~\dots, ~\pm\iu\lambda_m/\delta
~~~~\mathrm{and}~~~~
\pm\iu\lambda_1'/\delta', ~\dots, ~\pm\iu\lambda_m'/\delta'
$$
with $\delta = \sqrt{2\sum_{i=1}^m\lambda_i^2}$, $\delta' = \sqrt{2\sum_{i=1}^m(\lambda_i')^2}$, with $\iu$ the imaginary unit.
By assumption, the $\lambda_i$'s and $\lambda_i'$'s can be chosen as integers lower or equal to $\omega_\mathrm{max}$.
Hoffman-Wielandt inequality reads
$$
\|B-A\|^2 \geq 2\sum_{i=1}^m\bigg(\frac{\lambda_i}{\delta}- \frac{\lambda_i'}{\delta'}\bigg)^2.
$$
Note that, in Equation \eqref{eq:gamma_freqmin_proof}, $B$ does not have to be of unit norm.
The polynomial
$$
t \mapsto 2\sum_{i=1}^m\bigg(\frac{\lambda_i}{\delta}-t \frac{\lambda_i'}{\delta'}\bigg)^2,
$$
where $t\geq 0$, admits as a minimum the value 
$$
2\bigg(1-\bigg(\frac{1}{\delta\delta'}\sum_{i=1}^m\lambda_i\lambda_i'\bigg)^2\bigg).
$$
Let us suppose that the $m$-tuples are distinct.
That is, there exists $k\in[1\isep m]$ such that $\lambda_k \neq \lambda_k'$.
We compute:
\begin{equation*}
2\bigg(1-\bigg(\frac{1}{\delta\delta'}\sum_{i=1}^m\lambda_i\lambda_i'\bigg)^2\bigg)
\geq 
4\sum_{i=1}^m\frac{\lambda_i}{\delta}\bigg(\frac{\lambda_i}{\delta}-\frac{\lambda_i'}{\delta'}\bigg)
\geq
4\frac{\lambda_k}{\delta}\bigg(\frac{\lambda_k}{\delta}-\frac{\lambda_k'}{\delta'}\bigg).
\end{equation*}
Using that $|\lambda_k-\lambda_k'|\geq 1$ and $\delta,\delta' \leq \sqrt{2m\omega_\mathrm{max}^2}$, we finally obtain
$$
2\bigg(1-\bigg(\frac{1}{\delta\delta'}\sum_{i=1}^m\lambda_i\lambda_i'\bigg)^2\bigg)
\geq 
4\cdot\frac{1}{\sqrt{2m\omega_\mathrm{max}^2}}\cdot\frac{1}{\sqrt{2m\omega_\mathrm{max}^2}}
\geq 
\frac{2}{m\omega_\mathrm{max}^2}
=
\frac{4}{n\omega_\mathrm{max}^2}.
$$
We deduce the claim of the lemma.
More generally, when $G$ is a compact Lie group, the bound still holds, since $\big\| \projbracket{\mathfrak{s}^\bot}
\projbracket{\h} \big\|^2$ is lower bounded by $\big\| \projbracket{\mathfrak{s}^\bot}
(A) \big\|^2$ for any unit $A\in\h$.
\end{proof}

\subsubsection{Minimization of Equation \eqref{eq:minimization_stiefel_2}}\label{susubsec:stability_step3}
We now deduce a stability result for \ref{item:step3}.
Namely, and as a direct consequence of the rigidity of Lie subalgebras, we obtain that Equation \eqref{eq:minimization_stiefel_2} is minimized, with respect to the first variable in $\mathfrak{orb}(G,n)$, by a correct orbit equivalence class.
However, the minimizer $\widehat{\h}$, after optimization on the second variable in $\Ort(n)$, may not be close to the underlying algebra.
This latter problem, known as the \textit{stability} of Lie subalgebras, is considered more subtle than its rigidity equivalent \cite[Remark 5.13]{crainic2014survey}.
We circumvent this issue by working with the quantity $\Rig(\SO(2),n,\omega_\mathrm{max})$.

\begin{lemma}\label{lem:stability_minimization_liealgebra}
Let $\O$ be an orbit of a compact Lie group in $\R^n$, $X$ a finite point cloud, and $G$ a compact Lie group of dimension $d$.
We suppose that, for a certain parameter $\omega_\mathrm{max}>0$, the Lie algebra $\sym(\O)$ is spanned by matrices whose spectra come from a primitive integral vectors of coordinates at most $\omega_\mathrm{max}$.
Let $\Rig(G,n,\omega_\mathrm{max})$ be the corresponding rigidity constant.
In addition, we denote by $\lambda$ the smallest nonzero eigenvalue of the ideal \texttt{LiePCA} operator $\Lambda_\O$.
We suppose that the distance between the empirical and ideal operators is bounded above by
$$
\|\Lambda-\Lambda_\O\|_\op \leq \frac{\lambda^2}{16d} \Rig(G,n,\omega_\mathrm{max}).
$$
Consider an orbit-equivalence class $(B^1,\dots,B^p) \in \mathfrak{orb}(G,n)$.
This class being fixed, we suppose that the minimum of Equation~\eqref{eq:minimization_stiefel_2} with respect to $O\in\Ort(n)$ is upper bounded as follows:
$$
\min_{O\in\Ort(n)} \sum_{i=1}^d \bigg\| \Lambda \bigg(O \diag\big(B^k_j\big)_{k=1}^p O^\top\bigg) \bigg\|^2
\leq \frac{\lambda^2}{8}\Rig(G,n,\omega_\mathrm{max}).
$$
Let $\widehat{\h}$ be a minimizer.
Then $\widehat{\h}$ is conjugate to a subset of $\sym(\O)$, and satifies the bound
$$
\big\|\projbracket{\sym(\O)^\bot}  \projbracket{\widehat{\h}}\big\|
\leq \beta 
~~~~~~\text{where}~~~~~~
\beta = \frac{2\sqrt{d\|\Lambda-\Lambda_\O\|_\op}}{\lambda}.
$$
In addition,
\begin{itemize}
\item \textbf{General case:}
Suppose that the hypotheses hold for $\Rig(\SO(2),n,\omega_\mathrm{max})$ instead of $\Rig(G,n,\omega_\mathrm{max})$. 
Then for every $A \in \widehat{\h}$, there exists $O\in\Ort(n)$ such that $B=OAO^\top$ belongs to  $\sym(\O)$ and
$$
\|A-B\| \leq 2 \beta.
$$
\item \textbf{Transitive case:}
Suppose that $\sym(\O)$ and $G$ have equal dimensions. 
Then there exists $O\in\Ort(n)$ such for every $A \in \widehat{\h}$, $B=OAO^\top$ belongs to $\sym(\O)$ and
$$
\|A-B\| \leq 2\sqrt{2} \beta.
$$
Moreover, this inequality still holds if one replaces the cost function in Equation~\eqref{eq:minimization_stiefel_2} (\ref{item:step3}) by that of Equation~\eqref{eq:minimization_grassmann} (\ref{item:step3}').
\end{itemize}
\end{lemma}

\begin{proof}
We split the proof into four parts, following the statements of the lemma.

\vspace{.25cm}\noindent
\underline{A conjugation of $\widehat{\h}$ lies in $\sym(\O)$.}
Let $\phi = \bigoplus_{k=1}^p\phi_k$ be the almost-faithful representation of $G$ in $\R^n$, with $B^1,\dots,B^p$ the corresponding Lie algebras of the irreps, and denote $C_j=\diag\big(B^k_j\big)_{k=1}^p$ for all $j\in [1\isep d]$. 
Let $\widehat{\h}_0\subset\so(n)$ be the subalgebra it spans.
In particular, $\widehat{\h}_0$ is conjugate to $\widehat{\h}$.
We are dealing with Equation \eqref{eq:minimization_stiefel_2}, which consists in minimizing $f\colon\Ort(n)\rightarrow\R$ defined below, together with its ideal counterpart, denoted $f_\O$:
\[
f(O) = \sum_{j=1}^d \big\|\Lambda\big(O C_j O^\top\big)\big\|^2
~~~~~\mathrm{and}~~~~~
f_\O(O) = \sum_{j=1}^d \big\|\Lambda_\O\big(O C_j O^\top\big)\big\|^2.
\]
Using the notation $\alpha = \|\Lambda-\Lambda_\O\|_\op$, we see that, for all $O\in\Ort(n)$,
\begin{align}
|f_\O(O) - f(O)|
&\leq \sum_{j=1}^d \big(\underbrace{\big\|\Lambda_\O\big(O C_j O^\top\big)\big\|}_{\leq 1}+\underbrace{\big\|\Lambda\big(O C_j O^\top\big)\big\|}_{\leq 1}\big)
\underbrace{\big(\big\|\Lambda_\O\big(O C_j O^\top\big)\big\|-\big\|\Lambda\big(O C_j O^\top\big)\big\|\big)}_{\leq\alpha}\nonumber\\
&\leq 2d\alpha\label{eq:proof_upperbound_differenceobjectives}.
\end{align}
That is to say, $\|f_\O - f\|_\infty \leq 2d\alpha$.
Now, let $O_*\in\Ort(n)$ be such that $\widehat{\h} = O_*\widehat{\h}_0 O_*^\top$ minimizes $f$.
Together with the assumption $\min f \leq \lambda^2\Rig(G,n,\omega_\mathrm{max})/8$, we deduce that
$$
f_\O(O_*) \leq f(O_*) + \|f_\O - f\|_\infty
\leq \frac{\lambda^2}{8}\Rig(G,n,\omega_\mathrm{max}) + 2d\alpha.
$$
By injecting the other assumption $\alpha=\|\Lambda-\Lambda_\O\|_\op \leq \lambda^2\Rig(G,n,\omega_\mathrm{max})/16d$, we obtain
$$
f_\O(O_*) \leq \frac{\lambda^2}{8}\Rig(G,n,\omega_\mathrm{max}) + 2d\frac{\lambda^2}{16d}\Rig(G,n,\omega_\mathrm{max}) 
= \frac{\lambda^2}{4}\Rig(G,n,\omega_\mathrm{max}).
$$
Besides, we have derived in Equation \eqref{eq:bounds_minimization_step3} the lower bound
\begin{equation}\label{eq:proof_lowerbound_objective}
f_\O(O) \geq \lambda^2 \cdot \big\| \projbracket{\sym(\O)^\bot}  \projbracket{O \widehat{\h}_0 O^\top}  \big\|^2.    
\end{equation}
Combining these two inequalities yields
\begin{equation}\label{eq:proof_upperbound_objective}
\big\| \projbracket{\sym(\O)^\bot}  \projbracket{O_* \widehat{\h}_0 O_*^\top}  \big\|^2
\leq \frac{1}{4} \Rig(G,n,\omega_\mathrm{max}).
\end{equation}
Now, by the definition of $\Rig(G,n,\omega_\mathrm{max})$ (in Equation \eqref{eq:gamma_def}), we conclude that $\widehat{\h}=O_* \widehat{\h}_0 O_*^\top$ is conjugate to a subalgebra of $\sym(\O)$.

\vspace{.25cm}\noindent
\underline{Inequality on $\big\|\projbracket{\sym(\O)^\bot}  \projbracket{\widehat{\h}}\big\|$.}
We deduce from the preceding fact that $f_\O$ admits zero as a minimum.
Consequently, Equation~\eqref{eq:proof_upperbound_differenceobjectives} yields
$$
\min f \leq \underbrace{\min f_\O}_{0} + \|f_\O - f\|_\infty \leq 2d\alpha.
$$
We still consider $O_*\in\Ort(n)$ such that $\widehat{\h} = O_*\widehat{\h}_0 O_*^\top$ minimizes $f$.
Let us observe here that, if $\widehat{\h}$ were a minimizer of $f_\O$, then the conclusions of the lemma would follow directly.
However, $\widehat{\h}$ is only a minimizer of $f$, which slightly complicates the proof.
The value of $f_\O$ is bounded by
$$
f_\O(O_*) \leq \underbrace{f(O_*)}_{\min f} + \|f_\O - f\|_\infty
\leq 2d\alpha+2d\alpha,
$$
Combined with Equation \eqref{eq:proof_lowerbound_objective}, we obtain
\begin{equation}\label{eq:proof_upperbound_norm_subalgebra}
\big\|\projbracket{\sym(\O)^\bot}  \projbracket{\widehat{\h}}\big\|^2
\leq \frac{1}{\lambda^2}\cdot f_\O(O_*)
\leq \frac{4d\alpha}{\lambda^2},
\end{equation}
which is the constant $\beta^2$ in the lemma.

\vspace{.25cm}\noindent
\underline{General case.}
Consider the rigidity constant of $\SO(2)$.
In this context, Equation \eqref{eq:proof_upperbound_objective} reads
$$
\big\|\projbracket{\sym(\O)^\bot}  \projbracket{\widehat{\h}}\big\|^2
\leq
\frac{1}{4}\Rig(\SO(2),n,\omega_\mathrm{max}).
$$
In particular, for any unit norm matrix $A \in \widehat{\h}$, there exists a matrix $B \in \sym(\O)$ such that 
$$
\|A-B\|^2 \leq 
\frac{1}{4}\Rig(\SO(2),n,\omega_\mathrm{max}).
$$
After normalizing the matrix $B$, we obtain
$$
\bigg\|A-\frac{B}{\|B\|}\bigg\|^2 \leq 4 \|A-B\|^2
\leq \Rig(\SO(2),n,\omega_\mathrm{max}).
$$
We deduce, by the definition of $\Rig(\SO(2),n,\omega_\mathrm{max})$, that $A$ and $B/\|B\|$ admit the same spectrum.
Hence, there exists $O\in\Ort(n)$ such that $B/\|B\| = OAO^\top$, as announced in the lemma.
Moreover, since $B$ is the projection of $A$ on $\sym(\O)$, one has
$$
\bigg\|A-\frac{B}{\|B\|}\bigg\|^2 
\leq 4 \|A-B\|^2
\leq 4 \big\|\projbracket{\sym(\O)^\bot}  \projbracket{\widehat{\h}}\big\|^2.
$$

\vspace{.25cm}\noindent
\underline{Transitive case.}
The result when $\dim (G) = \dim (\sym(\O))$ is based on Equation~\eqref{eq:bounds_minimization_step3_equaldimension}, which reads
$$
\big\|\projbracket{\sym(\O)}-\projbracket{\widehat{\h}}\big\|^2
=2\big\| \projbracket{\sym(\O)^\bot}\projbracket{\widehat{\h}}\big\|^2.
$$
Using the characterization of the norm via principal angles, given in Equation~\eqref{eq:norm_proj_principalangles}, one shows that a conjugation matrix $O$ such that $\widehat{\h} = O \sym(\O) O^\top$ can be chosen to satisfy
$$
\|I-O\|\leq \sqrt{2} \big\|\projbracket{\sym(\O)}-\projbracket{\widehat{\h}}\big\|.
$$
Now, consider $A \in \widehat{\h}$ of unit norm, and define $B=OAO^\top$.
By submultiplicativity of the norm, we see that
\begin{align*}
\|A-B\| =
\|(I-O)A+OA(I-O^\top)\|
\leq 2 \|I-O\|.
\end{align*}
Consequently, $\|A-B\|\leq 2\sqrt{2}\beta$, as claimed.

Last, in the case where one considers Equation~\eqref{eq:minimization_grassmann} instead of \eqref{eq:minimization_stiefel_2}, the output subspace of \texttt{LiePCA} is close to the full algebra $\sym(\O)$ and the proof carries over verbatim.
However, in the general case, we point out that the output need not lie close to a pushforward algebra of $G$.
\end{proof}

\begin{remark}
Naturally, one may wish to promote the general case of Lemma \ref{lem:stability_minimization_liealgebra} to a single (i.e., uniform) conjugating matrix $O$, as in the transitive case.
However, we have not been able to prove such a result.
\end{remark}

\subsubsection{Stability of \ref{item:step4}}\label{subsubsec:stability_minimization_Hausdorff}
We now study the last step of our algorithm.
Let $\mu_\O$ and $\mu_X$ denote, respectively, the uniform measure on $\O$ and the empirical measure on $X$.
Moreover, let $\widehat{\h}$ be the Lie algebra output by \ref{item:step3} or \ref{item:step3}'. 
The execution of \ref{item:step4} consists in either choosing a point $x\in X$, forming the estimated orbit $\widehat{\O}_x$ and computing the Hausdorff distance $\HDmid{X}{\widehat{\O}_x}$, or forming the estimated measure $\mu_{\widehat{\O}}$ and computing the Wasserstein distance $\W_2\big(\mu_X, \mu_{\widehat{\O}} \big)$.
We remind the reader that these objects are defined as
\begin{align*}
\widehat{\O}_x &= \big\{\exp(A)x \mid A \in \widehat{\h}\big\}
~~~~~~~\mathrm{and}~~~~~~~
\mu_{\widehat{\O}} = \frac{1}{N}\sum_{i=1}^N \mu_{\widehat{\O}_{x_i}}.
\end{align*}
As one sees in this definition, and because of the matrix exponential, little inaccuracies in $\widehat{\h}$ may result in huge discrepancies between $\widehat{\O}_x$ and $\O$.
Indeed, arbitrarily close Lie algebras may have distant exponentials.
This problem is avoided, however, when they are close \textit{and} conjugate. Indeed, we recall that, in the direction of conjugate skew-symmetric matrices, the exponential map has bounded variation.
The proof is elementary and deferred to Section \ref{sec:supplementary_preliminaries_symmetrygroup}.

\begin{lemma}\label{lem:bound_exp_so(n)}
For any of pair skew-symmetric matrices $A,B\in\so(n)$ with integer spectra, conjugate through an orthogonal matrix, and for all $t\in\R$, it holds that
$$\|\exp(tA) - \exp(tB)\| \leq \|A - B\|.$$
\end{lemma}

We now finally deduce a stability result for \ref{item:step4}.
We draw the reader's attention to the fact that, in the general case stated below, no bound on the Wasserstein distance is obtained.
This is because the distance $\W_2\big(\mu_\O, \mu_{\widehat{\O}} \big)$ compares $\mu_\O$ and $\mu_{\widehat{\O}}$ symmetrically.
However, when the action is not transitive, the generated orbit $\widehat{\O}_x$ is only a subset of $\O$.

\begin{proposition}\label{prop:stability_step4}
Let $\O\subset \R^n$ be an orbit, $X$ a point cloud, and $\widehat{\h}\subset \so(n)$ a Lie subalgebra.
Let $\delta$ be the maximal distance from $\O$ to the origin.
\begin{itemize}
\item \textbf{General case:}
Suppose that there exists $\beta\geq0$ such that all $A\in \widehat{\h}$ is conjugate to a $B \in \sym(\O)$ for which $\|A-B\| \leq \beta$.
Then for all $x\in X$, we have 
$$
\HDmid{\widehat{\O}_x}{\O} \leq 
(\beta+1)\d(x,\O) + \beta \delta.
$$
\item \textbf{Transitive case:}
Suppose that there exists $\beta\geq0$ and an orthogonal matrix that conjugates $\widehat{\h}$ to $\sym(\O)$, and for which $\|A-B\| \leq \beta$ for all conjugate pairs.
Then it holds
\begin{align*}
\HD{\widehat{\O}_x}{\O} &\leq 
(\beta+1)\d(x,\O) + \beta \delta, \\
\W_2\big(\mu_{\widehat{\O}}, \mu_\O \big)&\leq
\W_2(\mu_\O,\mu_X)+\beta\delta.    
\end{align*}
\end{itemize}
\end{proposition}

\begin{proof}
We start with the general case.
Given $x\in X$, let $y \in \O$ be such that $\|x-y\|= \d(x,\O)$.
We consider the intermediate sets
\begin{align*}
\widehat{\O}_x &= \{\exp(A)x \mid A \in \widehat{\h}\},~~~
\O_x = \{\exp(B)x \mid B \in \sym(\O)\},~~~
\O_y = \{\exp(B)y \mid B \in \sym(\O)\}.
\end{align*}
Note that we have $\O_y = \O$, since $y\in \O$, and the exponential of $\sym(\O)$ yields the group $\Sym(\O)$, which acts transitively on $\O$.
The triangle inequality for the Hausdorff distance reads
$$
\HDmid{\widehat{\O}_x}{\O_y}
\leq \HDmid{\widehat{\O}_x}{\O_x} + \HDmid{\O_x}{\O_y}.
$$
The second term is upper bounded by $\|x-y\|$, hence by $\d(x,\O)$.
Regarding the first term, let us first consider an element $A \in \widehat{\h}$ of norm $1$, and, by hypothesis, $B\in\sym(\O)$ and $O\in\Ort(n)$ such that $B = O A O^\top$ and $\|B-A\| \leq \beta$.
As stated in Lemma \ref{lem:bound_exp_so(n)}, for any $t \in \R$, we have 
\begin{equation*}
\|\exp(tA) - \exp(tB)\| \leq \beta. 
\end{equation*}
In particular, the Hausdorff distance between $\{\exp(tA)x\mid t\in\R\}$ and $\{\exp(tB)x\mid t\in\R\}$ is upper bounded by $\beta \|x\|$.
Taking the union over all $A \in \widehat{\h}$, we deduce that $\HDmid{\widehat{\O}_x}{\O_x}\leq \beta \|x\|$.
Gathering the two inequalities yields 
$$
\HDmid{\widehat{\O}_x}{\O_y} \leq \beta \|x\| + \d(x,\O).
$$
To conclude, we inject $\|x\| \leq \|y\| + \|x-y\| \leq \delta + \d(x,\O)$.

Next, we consider the second scenario of the proposition: there exists $O\in\Ort(n)$ such that $\widehat{\h} = O\sym(\O)O^\top$.
In this case, the inequality on the Hausdorff distance follows from the same argument as above.
Regarding the Wasserstein distance, we build a transport plan between $\mu_\O$ and $\mu_{\widehat{\O}}$ as follows.
Let $\pi$ be a transport plan for $\W_2(\mu_\O,\mu_X)$.
For any point $x\in X$ and $y\in\O$, let $\pi_{(x,y)}$ denote the deterministic transport plan between $\mu_{\widehat{\O}_{x}}$ and $\mu_\O$ obtained from
\begin{align*}
\{\exp(A)x \mid A \in \widehat{\h}\} = \widehat{\O}_{x} &\rightarrow \O\\
\exp(A)x &\mapsto O\exp(A)O^\top y.
\end{align*}
This map is surjective as a consequence of the hypothesis $\widehat{\h} = O\sym(\O)O^\top$.
Moreover, one has
\begin{align*}
\|\exp(A)x-O\exp(A)O^\top y\|
&\leq 
\|\exp(A)x-\exp(A) y\| + \|\exp(A) y - O\exp(A)O^\top y\|\\
&\leq \|x-y\| + \beta \|y\|\leq \|x-y\| + \beta \delta,
\end{align*}
where we used Lemma \ref{lem:bound_exp_so(n)} with $\beta$.
Finally, let $\pi^*$ denote the measure obtained by integrating the $\pi_{(x,y)}$'s over $\pi$.
More precisely, $\pi^*$ is defined, for any test function $\psi\colon \R^n\times \R^n\rightarrow \R$, as
$$
\int \psi(x,y)\d\pi^*(x,y) = \int \bigg( \int \psi(x',y')\d\pi_{(x,y)}(x',y') \bigg) \d\pi(x,y).
$$
It is a transport plan between $\mu_\O$ and $\mu_{\widehat{\O}}$.
We compute:
\begin{align*}
\W_2^2(\mu_\O,\mu_{\widehat{\O}}) &\leq \int \|x-y\|^2 \d\pi^*(x,y)\\
&= \int \bigg( \int \|x'-y'\|^2\d\pi_{(x,y)}(x',y') \bigg) \d\pi(x,y)\\
&\leq \int \bigg(\|x-y\| + \beta \delta\bigg)^2 \d\pi(x,y)
\leq \bigg(\W_2(\mu_\O,\mu_X)+\beta\delta \bigg)^2,
\end{align*}
the last inequality coming from the subadditivity of the $L^2$ norm.
The proposition follows.
\end{proof}

\begin{remark}\label{rem:projection_mindimspace}
As pointed out in Section \ref{subsec:PCA_preprocessing}, the dimension reduction implemented in \ref{item:step1} has a significant importance on the algorithm.
Indeed, by projecting $\O$ onto $\spn{\O}$, the subspace it spans, we ensure that elements $O \in \Sym(\O)\setminus\{I\}$ act non-trivially on $\O$.
Otherwise, and as evoked in Section \ref{subsec:symmetry_group}, there would exist symmetries $O$ satisfying $Ox = x$ for all $x\in \O$.
If such a matrix had been obtained via \ref{item:step3}, then the orbit $\widehat{\O}_x$ would be a singleton.
Although satisfying the statement of Proposition \ref{prop:stability_step4}, it has little interest from a data analysis perspective.
\end{remark}

\subsection{Consistency of the algorithm}\label{subsec:robustness_algorithm}

We now arrive at our main theorem by combining the results obtained so far.
It states that, as long as the point cloud $X$ is sampled sufficiently close to the underlying orbit $\O$, Algorithm~\ref{alg: 1} returns an accurate estimation of $\O$, either in the form of an orbit $\widehat{\O}_x$ close to $\O$ (or close to being contained in $\O$), or of a probability measure $\mu_{\widehat{\O}}$ close to its uniform measure.
The closeness between $X$ and $\O$ is quantified by both the Hausdorff and Wasserstein distances in the former case, and only by the Wasserstein distance in the latter.
Recall that the algorithm, in addition to $X$, takes as an input a compact Lie group $G$, and parameters $\epsilon$, $l$, $r$ and $\omega_\mathrm{max}$, corresponding respectively to the threshold for dimension reduction (\ref{item:step1}), the dimension and radius of local PCA (\ref{item:step2}), and a bound on the spectra of Lie algebras (\ref{item:step3} and \ref{item:step3}').
Among these parameters, our result requires that $G$ is the underlying Lie group, that $l$ equal the dimension of $\O$, and that $\omega_\mathrm{max}$ is sufficiently large to contain $\sym(\O)$.

Moreover, we remind the reader that our analysis involves constants of four different natures.
Namely, the orthogonality and homogeneity of $\O$ are quantified through $\Var\big[\|\mu_\O\|\big]/\sigma_\mathrm{min}^2$ and  $\sigma_\mathrm{max}^2/\sigma_\mathrm{min}^2$ (see Section \ref{subsubsec:analysis_pca_orbit}), the robustness of \texttt{LiePCA} via its bottom nonzero eigenvalue $\lambda$ (see Section \ref{subsubsec:lie-pca_consistency}), the geometry of the orbit with $\mathrm{vol}(\O)$ and $\reach(\O)$ (see Section \ref{subusbsec:stability_tangent_space_estimation}), and finally the rigidity of Lie subalgebras with $\Rig(G,n)$ and $\Rig(G,n,\omega_\mathrm{max})$ (see Section \ref{subsubsec:rigidity_lie_subalgebras}).
More accurately, the geometric quantities---volume and reach---are computed from the orthonormalized orbit $\widetilde{\O}$, obtained after \ref{item:step1}.
Although they could be expressed in terms of $\O$ only, we did not pursue that question further.

\begin{theorem}\label{th:robustness_algorithm}
Let $G$ be a compact Lie group of dimension $d$, $\mathcal{O}$ an orbit of an almost-faithful representation $\phi$ in $\R^n$, potentially non-orthogonal, and $l$ its dimension.
Let $\mu_\O$ be the uniform measure on $\O$, $\sigma_\mathrm{max}^2$ and $\sigma_\mathrm{min}^2$ the top and bottom nonzero eigenvalues of the covariance matrix $\Sigma[\mu_\O]$, and $\lambda$ the bottom nonzero eigenvalue of the ideal \texttt{LiePCA} operator $\Lambda_\O$.
Denote by $\widetilde{\O} = \sqrt{\Sigma[\mu_\O]^+} \O$ the orthonormalized orbit.
Besides, let $X\subset\R^n$ be a finite point cloud and $\mu_X$ its empirical measure.
Lastly, choose a positive integer $\omega_\mathrm{max}$ such that $\sym(\O)$ is spanned by matrices whose spectra come from primitive integral vectors of coordinates at most $\omega_\mathrm{max}$.
Define
\begin{alignat*}{2}
&\omega = \frac{\W_2(\mu_\O,\mu_X)}{\sigma_\mathrm{min}},
~~~~~~~~~
&&\upsilon = \bigg(\frac{\Var\big[\|\mu_\O\|\big]}{\sigma_\mathrm{min}^2}\bigg)^{1/2},
\\
&\widetilde{\omega} = 4(n+1)^{3/2}\bigg(\frac{\sigma_\mathrm{max}^3}{\sigma_\mathrm{min}^3}\bigg)\bigg(\omega(\upsilon + \omega)\bigg)^{1/2},
~~~~~~~~~
&&\rho = \bigg(16l(l+2)6^l\bigg)\frac{\max(\mathrm{vol}(\widetilde{\O}),\mathrm{vol}(\widetilde{\O})^{-1})}{\min(1,\reach(\widetilde{\O}))},
\end{alignat*}
\begin{equation*}
\gamma = 
\begin{cases}
\big(2^{11/2}d\rho\big)^{-1}\cdot\lambda^2\cdot\Rig(\SO(2),n,\omega_\mathrm{max}) ~~\text{if \ref{item:step3} in general case},\\
\big(2^{11/2}d\rho\big)^{-1}\cdot\lambda^2\cdot\Rig(G,n,\omega_\mathrm{max}) ~~~~~~~\text{if \ref{item:step3} in transitive case or \ref{item:step3}'}.
\end{cases}
~~~~~~
\end{equation*}
Suppose that $\omega$ is sufficiently small to satisfy
\begin{align*}
&\omega < \bigg(\bigg(\upsilon^2+\frac{1}{2}\bigg)^{1/2} -\upsilon\bigg)
\bigg/\bigg(3(n+1)\frac{\sigma_\mathrm{max}^2}{\sigma_\mathrm{min}^2}  \bigg)
,~~~~~~
\widetilde{\omega} \leq \min\bigg\{
\bigg(\frac{1}{6\rho}\bigg)^{3(l+1)},
\frac{\gamma^{l+3}}{16},
\bigg(\frac{\gamma}{(6\rho)^2}\bigg)^{l+1}
\bigg\}.
\end{align*}
Choose two parameters $\epsilon$ and $r$ in the following nonempty sets:
\begin{align*}
\epsilon &\in \bigg( (2\upsilon+\omega)\omega\sigma_\mathrm{min}^2, 
~\frac{1}{2}\sigma_\mathrm{min}^2\bigg],
~~~~~
r \in \bigg[\big(6\rho\big)^{2}\widetilde{\omega}^{1/(l+1)}, ~\big(6\rho\big)^{-1}\bigg]\cap\bigg[ \big(4/\gamma\big)^{2/(l+1)}\widetilde{\omega}^{1/(l+1)}, ~\gamma\bigg].
\end{align*}
Let $\widehat{\h}$, $\widehat{\phi}$, $\widehat{\O}_x$ and $\mu_{\widehat{\O}}$ be the output of Algorithm \ref{alg: 1} performed on $X$ with parameters $G$, $\epsilon$, $r$, $l$, $\omega_\mathrm{max}$ and an arbitrary $x\in X$.
We suppose the minimization problems are computed exactly.
\begin{itemize}
\item \textbf{General case:}
It holds that
\begin{equation}\label{eq:theorem_hausdorff}
\HDmid{\widehat{\O}_x}{\widetilde{\O}}
\leq
2^{1/2} \frac{\HDmid{X}{\O}}{\sigma_\mathrm{min}}
+
2^{9/4}\frac{(nd\rho)^{1/2}}{\lambda}
\bigg(r + 4 \bigg(\frac{\widetilde{\omega}}{r^{l+1}}\bigg)^{1/2}\bigg)^{1/2}.
\end{equation}
Moreover, the output Lie algebra $\widehat{\h}$ is conjugate to a subalgebra of $\sym(\O)$.
\item \textbf{Transitive case:}
Suppose that $\sym(\O)$ and $G$ have equal dimensions. Then 
\begin{align}
\HD{\widehat{\O}_x}{\widetilde{\O}}
&\leq
(1+2^{-1/2}) \frac{\HDmid{X}{\O}}{\sigma_\mathrm{min}}
+
2^{11/4}\frac{(nd\rho)^{1/2}}{\lambda}
\bigg(r + 4 \bigg(\frac{\widetilde{\omega}}{r^{l+1}}\bigg)^{1/2}\bigg)^{1/2},
\label{eq:theorem_hausdorff_transitive}\\
\W_2\big(\mu_{\widehat{\O}}, \mu_{\widetilde{\O}} \big)
&\leq 2^{-1/2} \frac{\W_2(\mu_X,\mu_\O)}{\sigma_\mathrm{min}} 
+
2^{11/4}\frac{(nd\rho)^{1/2}}{\lambda}
\bigg(r + 4 \bigg(\frac{\widetilde{\omega}}{r^{l+1}}\bigg)^{1/2}\bigg)^{1/2}.\label{eq:theorem_wasserstein}
\end{align}
Moreover, the output representation $\widehat{\phi}$ is orbit-equivalent to $\phi$.
\end{itemize}
\end{theorem}

\begin{proof}
We will verify each step of the algorithm, first considering an execution with \ref{item:step3} without futher assumption (general case), and then with \ref{item:step3} or \ref{item:step3}' under the assumption $\dim (G) = \dim (\sym(\O))$ (transitive case).
Before entering into the proof, we verify that if $\omega$ satisfies the assumptions of the theorem, then the intervals given for $\epsilon$ and $r$ are non-empty.

\vspace{.2cm}\noindent
\underline{Study of the intervals.}
The first inequality on $\omega$ implies $\omega < (\upsilon^2+1/2)^{1/2} -\upsilon)$. 
As we have observed in Equation \eqref{eq:bound_W2_with_sigma}, this inequality is equivalent to $(2\upsilon+\omega)\omega < 1/2$.
In particular, the interval for $\epsilon$ is non-empty.
Next, we study the set for $r$. It reads
$$
\bigg[\big(6\rho\big)^{2}\cdot\widetilde{\omega}^{1/(l+1)}, ~\big(6\rho\big)^{-1}\bigg]
\cap
\bigg[ \big(4/\gamma\big)^{2/(l+1)}\cdot\widetilde{\omega}^{1/(l+1)}, ~\gamma\bigg].
$$
As a direct consequence of the assumptions $\widetilde{\omega} \leq \big(6\rho\big)^{-3(l+1)}$ and $\widetilde{\omega} \leq \gamma^{l+3}/4^2$, we see that both the intervals are non-empty.
Let us prove that they intersect.
It is enough to show the following relations on their opposite endpoints:
$$
\big(6\rho\big)^{-1} \nless \big(4/\gamma\big)^{2/(l+1)}\cdot\widetilde{\omega}^{1/(l+1)}
~~~~\mathrm{and}~~~~
\gamma \nless\big(6\rho\big)^{2}\cdot\widetilde{\omega}^{1/(l+1)},
$$
or, equivalently,
$$
\big(6\rho\big)^{-1} \cdot \big(4/\gamma\big)^{-2/(l+1)}\geq \widetilde{\omega}^{1/(l+1)}
~~~~\mathrm{and}~~~~
\gamma \cdot \big(6\rho\big)^{-2} \geq \widetilde{\omega}^{1/(l+1)},
$$
This second inequality is already an assumption of the theorem.
Concerning the first one, we remark that the two assumptions above on $\widetilde{\omega}$ can be rewritten as $\widetilde{\omega}^{1/(l+1)} \leq \big(6\rho\big)^{-3}$ and $\widetilde{\omega}^{1/(l+1)} \leq \gamma \big(\gamma/4\big)^{2/(l+1)}$.
We decompose: 
\begin{align*}
\widetilde{\omega}^{1/(l+1)} 
=\big(\widetilde{\omega}^{1/(l+1)}\big)^{1/3}\big(\widetilde{\omega}^{1/(l+1)}\big)^{2/3}
&\leq\big(\big(6\rho\big)^{-3}\big)^{1/3}\cdot\big(\gamma \big(\gamma/4\big)^{2/(l+1)}\big)^{2/3}\\
&=\big(6\rho\big)^{-1}\big(4/\gamma\big)^{-2/(l+1)}\big(\gamma^{l+2}4^2\big)^{1/3(l+1)}.
\end{align*}
Besides, we have $\gamma \leq \big( 4(2d+1)\sqrt{2}\big)^{-1} < 4^{-2}$.
Therefore, the term $\big(\gamma^{l+2}4^2\big)^{1/3(l+1)}$ is lower than $1$ and we deduce the wanted inequality:
\begin{align*}
\widetilde{\omega}^{1/(l+1)} 
\leq \big(6\rho\big)^{-1}\big(4/\gamma\big)^{-2/(l+1)}.
\end{align*}

\vspace{.2cm}\noindent
\underline{\ref{item:step1}.}
We now study the algorithm step by step.
To start, let us define the orthonormalized sets and the pushforward measures 
\begin{alignat*}{2}
\widetilde{\O} &= \sqrt{\Sigma[\mu_\O]^+}\projbracket{\spn{\mathcal{O}}}\O,
&&\widetilde{X} = \sqrt{\Sigma[\nu]^+}\Pi_{\Sigma[\nu]}^{>\epsilon} X,\\
\mu_{\widetilde{\O}}&=\sqrt{\Sigma[\mu_\O]^+}\projbracket{\spn{\mathcal{O}}}\mu_\O,
~~~~
&&\mu_{\widetilde{X}}=\sqrt{\Sigma[\nu]^+}\Pi_{\Sigma[\nu]}^{>\epsilon}\mu_X.
\end{alignat*}
We note that $\mu_{\widetilde{\O}}$ and $\mu_{\widetilde{X}}$ are respectively the uniform measure on $\widetilde{\O}$ and the empirical measure on $\widetilde{X}$.
Moreover, $\mu_{\widetilde{X}}$ is the output of \ref{item:step1}.
Besides, we have seen that the first assumption of the theorem regarding $\omega$ is equivalent to the first assumption of Proposition \ref{prop:stability_step1}.
Together with the hypothesis on $\epsilon$, which is equivalent to its second assumption, we deduce that this proposition can be used.
It reads as follows, with the quantity $\widetilde{\omega}$ appearing explicitly:
$$
\W_2(\mu_{\widetilde{\O}},\mu_{\widetilde{X}})
\leq
4(n+1)^{3/2}\bigg(\frac{\sigma_\mathrm{max}^3}{\sigma_\mathrm{min}^3}\bigg)\bigg(\omega(\upsilon + \omega)\bigg)^{1/2}.
$$

\vspace{.2cm}\noindent
\underline{\ref{item:step2}.}
Let us denote by $\Lambda_{\widetilde{\O}}$ the ideal \texttt{LiePCA} operator on $\widetilde{\O}$ and by $\Lambda$ that computed from $\widetilde{X}$.
The hypothesis on $r$, given by the first interval, yields
$$
\big(6\rho\big)^2\W_2(\mu_{\widetilde{\O}},\mu_{\widetilde{X}})^{1/(l+1)} \leq r \leq \big(6\rho\big)^{-1}.
$$
Hence, the assumptions of Proposition \ref{prop:Lie-PCA} are satisfied. 
We deduce that the bound
\begin{align*}
\| \Lambda_{\widetilde{\O}}-\Lambda \|_\mathrm{op}
&\leq 
\sqrt{2}\rho \bigg(r + 4\bigg(\frac{\W_2(\mu_{\widetilde{\O}},\mu_{\widetilde{X}})}{r^{l+1}}\bigg)^{1/2}\bigg)
\leq     
\sqrt{2}\rho \bigg(r + 4\bigg(\frac{\widetilde{\omega}}{r^{l+1}}\bigg)^{1/2}\bigg).
\end{align*}

\vspace{.2cm}\noindent
\underline{\ref{item:step3}.}
Let $(A_1^*,\dots,A_d^*)$ be a minimizer of Equation \eqref{eq:minimization_stiefel_2} and $\widehat{\h}\subset\so(n)$ be the space it spans.
From the hypothesis $r \geq \big(\widetilde{\omega}4^2/\gamma^2\big)^{1/(l+1)}$ we deduce that
$$
4\bigg(\frac{\widetilde{\omega}}{r^{l+1}}\bigg)^{1/2} \leq \gamma.
$$
Together with the other bound $r \leq \gamma$, we obtain
\begin{equation*}
\| \Lambda_{\widetilde{\O}}-\Lambda \|_\mathrm{op}
\leq     
\sqrt{2}\rho \bigg(r + 4\bigg(\frac{\widetilde{\omega}}{r^{l+1}}\bigg)^{1/2}\bigg)
\leq     
\sqrt{2}\rho(\gamma+\gamma)
= 2\sqrt{2}\rho\gamma.
\end{equation*}
Besides, by definition of $\gamma$, we have
$$
2\sqrt{2}\rho\gamma = \frac{\lambda^2}{16d}\Rig(\SO(2),n,\omega_\mathrm{max}).
$$
Combined, these inequalities yield
\begin{equation}\label{eq:theorem_proof_operatornorm_liepca}
\| \Lambda_{\widetilde{\O}}-\Lambda \|_\mathrm{op}
\leq     
\frac{\lambda^2}{16d}\Rig(\SO(2),n,\omega_\mathrm{max}).
\end{equation}
Thus, the first assumption of Lemma \ref{lem:stability_minimization_liealgebra} is satisfied.
Besides, the second assumption holds precisely when $\widehat{\h}$ is conjugate to a subalgebra of $\sym(\O)$, as a consequence of Equation~\eqref{eq:proof_upperbound_differenceobjectives}:
$$
2d\|\Lambda_{\widetilde{\O}}-\Lambda \|_\mathrm{op} \leq 2d \frac{\lambda^2}{16d}\Rig(\SO(2),n,\omega_\mathrm{max})
= \frac{\lambda^2}{8}\Rig(\SO(2),n,\omega_\mathrm{max}).
$$
In this case, the lemma states that for all $A\in\widehat{\h}$, there exists a $B \in \sym(\O)$ and an $O\in\Ort(n)$ such that $A = O B O^\top$ and $\|A-B\| \leq 2\beta$, where
\begin{align*}
\beta^2
&=\frac{4d\|\Lambda_{\widetilde{\O}}-\Lambda \|_\mathrm{op}}{\lambda^2}
\leq 4d\sqrt{2}\frac{\rho}{\lambda^2}\bigg(r + 4\bigg(\frac{\widetilde{\omega}}{r^{l+1}}\bigg)^{1/2}\bigg).
\end{align*}
Moreover, we deduce from this first equality that $2\beta \leq 1$.

\vspace{.2cm}\noindent
\underline{\ref{item:step4}.}
The existence of the constant $2\beta$ allows us to apply Proposition \ref{prop:stability_step4}.
Since the orbit $\widetilde{\O}$ is orthonormalized, the covariance matrix $\Sigma[\mu_{\widetilde{\O}}]$ is the identity and the norm $\|\mu_{\widetilde{\O}}\|$ is constant.
Besides, as we have seen in Equation \eqref{eq:ideal_cov_variance_trace}, $\E[\|\mu_{\widetilde{\O}}\|^2] = \Tr\big( \Sigma[\mu_{\widetilde{\O}}] \big)$.
We deduce the radius of $\widetilde{\O}$ to be $\sqrt{n}$.
This is the constant $\delta$ appearing in the proposition.
We obtain
\begin{align}
\HDmid{\widehat{\O}_x}{\widetilde{\O}}
&\leq 
(2\beta+1)\d(x,\widetilde{\O}) + 2\beta \delta \nonumber\\ 
&\leq 
2\d(x,\widetilde{\O}) + 2\beta\sqrt{n} \nonumber\\   
&\leq 
2\d(x,\widetilde{\O})
+
2\bigg(4d\sqrt{2}\frac{\rho}{\lambda^2}\bigg)^{1/2}
\bigg(r + 4 \bigg(\frac{\widetilde{\omega}}{r^{l+1}}\bigg)^{1/2}\bigg)^{1/2}\cdot\sqrt{n} \nonumber\\   
&= 
2\d(x,\widetilde{\O})
+
2^{9/4}\frac{(nd\rho)^{1/2}}{\lambda}
\bigg(r + 4 \bigg(\frac{\widetilde{\omega}}{r^{l+1}}\bigg)^{1/2}\bigg)^{1/2}.
\label{eq:theorem_hausdorff_proof}
\end{align}
Last, we wish to replace $\d(x,\widetilde{\O})$ with a term involving $\HDmid{X}{\O}$.
To do so, we use the obvious upper bound $\d(x,\widetilde{\O})\leq\HDmid{\widetilde{X}}{\widetilde{\O}}$, as well as the definitions of $\widetilde{\O}$ and $\widetilde{X}$:
\begin{align*}
\HDmid{\widetilde{X}}{\widetilde{\O}} &=\HDmid{\sqrt{\Sigma[\nu]^+}\Pi_{\Sigma[\nu]}^{>\epsilon} X~}{~ \sqrt{\Sigma[\mu_\O]^+}\projbracket{\spn{\mathcal{O}}}\O}  \\   
&\leq \big\|\sqrt{\Sigma[\nu]^+} - \sqrt{\Sigma[\mu_\O]^+}  \big\|_\mathrm{op} \cdot\HDmid{\Pi_{\Sigma[\nu]}^{>\epsilon} X ~}{~ \projbracket{\spn{\mathcal{O}}}\O} \\
&\leq \big\|\sqrt{\Sigma[\nu]^+} - \sqrt{\Sigma[\mu_\O]^+}  \big\|_\mathrm{op}
\cdot\big\|\Pi_{\Sigma[\nu]}^{>\epsilon} - \projbracket{\spn{\mathcal{O}}} \big\|_\mathrm{op}\cdot \HDmid{ X}{\O}.
\end{align*}
We recall that the operator norm is not greater than the Frobenius one.
Moreover, bounds for $\big\|\sqrt{\Sigma[\nu]^+} - \sqrt{\Sigma[\mu_\O]^+}  \big\|$ and $\big\|\Pi_{\Sigma[\nu]}^{>\epsilon} - \projbracket{\spn{\mathcal{O}}} \big\|_\mathrm{op}$ have been obtained respectively in Equations \eqref{eq:majoration_frobenius_pca} and \eqref{eq:proof_orthonormalization_3} of Propositions \ref{prop:stability_PCA_orbit} and \ref{prop:stability_orthonormalization}.
We obtain that the product is upper bounded by
\begin{align*}
&\big\|\sqrt{\Sigma[\nu]^+} - \sqrt{\Sigma[\mu_\O]^+}  \big\|_\mathrm{op}
\cdot\big\|\Pi_{\Sigma[\nu]}^{>\epsilon} - \projbracket{\spn{\mathcal{O}}} \big\|_\mathrm{op}\\
\leq
&\frac{2}{\sigma_\mathrm{min}^4}\bigg(2\Var\big[\|\mu_\O\|\big]^{1/2}  + \W_2(\mu_\O,\mu_X)\bigg)^{3/2}\W_2(\mu_\O,\mu_X)^{3/2}.
\end{align*}
Besides, we have seen that the first assumption on $\omega$ is equivalent to
$$
\bigg(2\Var\big[\|\mu_\O\|\big]^{1/2}+\W_2(\mu_\O,\nu)\bigg) \W_2(\mu_\O,\nu) \leq \frac{1}{2}\sigma_\mathrm{min}^2.
$$
We obtain the wanted inequality:
\begin{equation}\label{eq:theorem_proof_bound_productnorms}    
\HDmid{\widetilde{X} }{ \widetilde{\O}}
\leq 
\frac{2}{\sigma_\mathrm{min}^4} \bigg( \frac{1}{2}\sigma_\mathrm{min}^2\bigg)^{3/2}
\HDmid{X}{\O}
= \frac{1}{\sqrt{2}\sigma_\mathrm{min}} \HDmid{X}{\O}.
\end{equation}
Together with Equation \eqref{eq:theorem_hausdorff_proof}, we eventually deduce Equation \eqref{eq:theorem_hausdorff} of the theorem.

\vspace{.2cm}\noindent
\underline{Transitive case.}
We now consider an execution of the algorithm under the assumption $\dim (G) = \dim (\sym(\O))$.
The proof is similar.
Let $\widehat{\h}\subset\so(n)$ denote the Lie algebra spanned by a minimizer of Equations \eqref{eq:minimization_stiefel_2} or \eqref{eq:minimization_grassmann}.
Just as we have obtained in Equation \eqref{eq:theorem_proof_operatornorm_liepca}, we have
$$
\| \Lambda_{\widetilde{\O}}-\Lambda \|_\mathrm{op}
\leq     
\frac{\lambda^2}{16d}\Rig(G,n,\omega_\mathrm{max}).
$$
Therefore we can apply Lemma \ref{lem:stability_minimization_liealgebra}, which states that $\widehat{\h}$ is conjugate to $\sym(\O)$, hence that $\widehat{\phi}$ is orbit-equivalent to $\phi$.
Moreover, we can apply Proposition \ref{prop:stability_step4} with $2\sqrt{2}\beta$:
\begin{align*}
\HD{\widehat{\O}_x}{\widetilde{\O}} &\leq 
(2\sqrt{2}\beta+1)\d(x,\widetilde{\O}) + 2\sqrt{2}\beta \delta, \\
\W_2\big(\mu_{\widehat{\O}}, \mu_{\widetilde{\O}} \big)&\leq
\W_2(\mu_{\widetilde{\O}},\mu_{\widetilde{X}})+2\sqrt{2}\beta\delta.    
\end{align*}
As previously, and since the orbit $\widetilde{\O}$ is orthonormal, we have $\delta=\sqrt{n}$.
Following the same proof as in the general case, we obtain the bound on the Hausdorff distance stated in Equation~\eqref{eq:theorem_hausdorff_transitive} of the theorem.
Last, to simplify the bound on the Wasserstein distance, we express $\W_2(\mu_{\widetilde{\O}},\mu_{\widetilde{X}})$ in terms of $\W_2(\mu_\O,\mu_X)$ via the inequalities
\begin{align*}
\W_2\big( \mu_{\widetilde{X}}, \mu_{\widetilde{\O}}\big) &=\W_2\bigg(\sqrt{\Sigma[\nu]^+}\Pi_{\Sigma[\nu]}^{>\epsilon} \mu_X, \sqrt{\Sigma[\mu_\O]^+}\projbracket{\spn{\mathcal{O}}}\mu_{\O}\bigg)  \\   
&\leq \big\|\sqrt{\Sigma[\nu]^+} - \sqrt{\Sigma[\mu_\O]^+}  \big\|_\mathrm{op} 
\cdot\W_2\bigg(\Pi_{\Sigma[\nu]}^{>\epsilon} \mu_X, \projbracket{\spn{\mathcal{O}}}\mu_{\O}\bigg) \\
&\leq \big\|\sqrt{\Sigma[\nu]^+} - \sqrt{\Sigma[\mu_\O]^+}  \big\|_\mathrm{op}
\cdot\big\|\Pi_{\Sigma[\nu]}^{>\epsilon} - \projbracket{\spn{\mathcal{O}}} \big\|_\mathrm{op}\cdot \W_2\big( \mu_X, \mu_{\O}\big).
\end{align*}
We have seen in Equation \eqref{eq:theorem_proof_bound_productnorms} that this product of norms is upper bounded by $1/\big(\sqrt{2}\sigma_\mathrm{min}\big)$.
We eventually deduce Equation \eqref{eq:theorem_wasserstein} of the theorem.
\end{proof}

\begin{remark}\label{rem:reformulation_theorem}
Equations \eqref{eq:theorem_hausdorff}, \eqref{eq:theorem_hausdorff_transitive} and \eqref{eq:theorem_wasserstein} justify the correctness of Algorithm \ref{alg: 1}, in the sense that, if the parameters are correctly chosen and are small, then the output of the algorithm is close to the unknown underlying orbit $\O$.
Namely, one reads from these equations that the significant quantities to obtain a correct output are $\HDmid{X}{\O}$ or $\W_2\big(\mu_{X},\mu_\O\big)$, the initial distance between the data and the underlying object, but also $r$ and $\big(\widetilde{\omega}/r^{l+1}\big)^{1/2}$, the quantities reflecting the bias-variance trade-off when estimating tangent spaces, as discussed in Remark \ref{rem:estimation_tangentspace_bias_variance}.
Let us comment on this point further. 
We see that the trade-off is optimized when $r$ is equal to $4\big(\widetilde{\omega}/r^{l+1}\big)^{1/2}$, that is, when $r$ is of order $\widetilde{\omega}^{1/(l+3)}$.
Moreover, as we have observed in Remark \ref{rem:prop_orthonormalization_asymptotics}, in the case where the input orbit is not orthogonal---that is, when $\upsilon>0$---, then $\widetilde{\omega}$ has order $\W_2\big(\mu_{X},\mu_\O\big)^{1/2}$.
We conclude, as a rule of thumb, that an `optimal' parameter $r$ must be chosen to be of order $\W_2\big(\mu_{X},\mu_\O\big)^{1/2(l+3)}$.
For such a choice---albeit unknown to the user---, the asymptotic behavior of Equation \eqref{eq:theorem_wasserstein}, using the big-$\Theta$ notation, is
\begin{align}\label{eq:theorem_optimal_choice}
\W_2\big(\mu_{\widehat{\O}}, \mu_{\widetilde{\O}} \big)
= \Theta\bigg( \W_2\big(\mu_{X},\mu_\O\big)^{1/4(l+3)}\bigg).
\end{align}
Compared to the first term of Equation \eqref{eq:theorem_wasserstein}, of order $\Theta\big(\W_2\big(\mu_{X},\mu_\O\big)\big)$, this latter term, being raised to the power of $1/4(l+3)$, is crucially slower.
This shows that the `slowest' step of the algorithm is that of estimating normal spaces.
\end{remark}

\begin{remark}
We have chosen to state Theorem \ref{th:robustness_algorithm} in terms of $\W_2\big(\mu_{\O},\mu_X \big)$, the Wasserstein distance between the input and underlying measures. We consider it to be a general theoretical framework.
Alternatively, one can deduce a formulation in terms of the number of data points.
Namely, if $N$ is the cardinality of $X$, then is is known that the distance $\W_2\big(\mu_{\O}, \mu_X\big)$ between the uniform measure on a $l$-manifold and the empirical measure on a $N$-sample has rate $(1/N)^{(1/l)}$ \cite{singh2018minimax,divol2021short}.
Thus, the result of the theorem, with `optimal' parameter $r$ of Equation \eqref{eq:theorem_optimal_choice}, reads
\begin{align*}
\W_2\big(\mu_{\widehat{\O}}, \mu_{\widetilde{\O}} \big)
= \Theta\big( 1/N^{1/4l(l+3)}\big).
\end{align*}
\end{remark}

\section{Applications}\label{sec:applications}
We now apply Algorithm \ref{alg: 1} to concrete problems, briefly presented in the introduction.
More than an extensive list of tasks involving Lie groups, this section aims to identify and justify the conditions that guarantee the effectiveness of the methods developed throughout this work.

\subsection{Image analysis}\label{subsec:pixel_permutations}

Since an image (greyscale or RGB) is nothing more than a collection of real values on a pixel grid (or voxel grid, in 3D), we can treat it as a point in an Euclidean space, and a set of images as a point cloud. 
We aim to show how Lie group representations appear naturally in this context, and how this information can be exploited in conjunction with {\texttt{LieDetect}}.
The simplest case is that of pixel permutations: if an Abelian finite group acts on an image (e.g., through translations), then we show that the orbit of this image lies along a compact Lie group orbit. Non-Abelian permutations are considered in a second subsection, focused on 3D-image datasets. We then turn to regression problems, which we study through the lens of harmonic analysis.

\subsubsection{Pixel permutations}\label{subsec: ppt}
Let us model a one-channel $(m_x\times m_y)$-pixels image as a signal $f\colon[1\isep m_x]\times [1\isep m_y]\to \R$, and define its lift $f^\uparrow$ as the signal's natural embedding in $\R^n$, for $n=m_x\times m_y$. 
Consider a permutation $\sigma$ applied to the pixel grid.
It generates, from $f$, the new image $f \circ \sigma$. 
Visually, such a transformation $\sigma$ only rearranges the pixels, leaving quantities such as total brightness conserved (see the top of Figure \ref{fig: ppt}). If we denote the permutation matrix associated with $\sigma$ by the same variable, then the lifts are related by $(f \circ \sigma)^\uparrow = \sigma\cdot f^\uparrow$.

\begin{figure}[ht]
    \centering
    \includegraphics[width=0.8\textwidth]
    {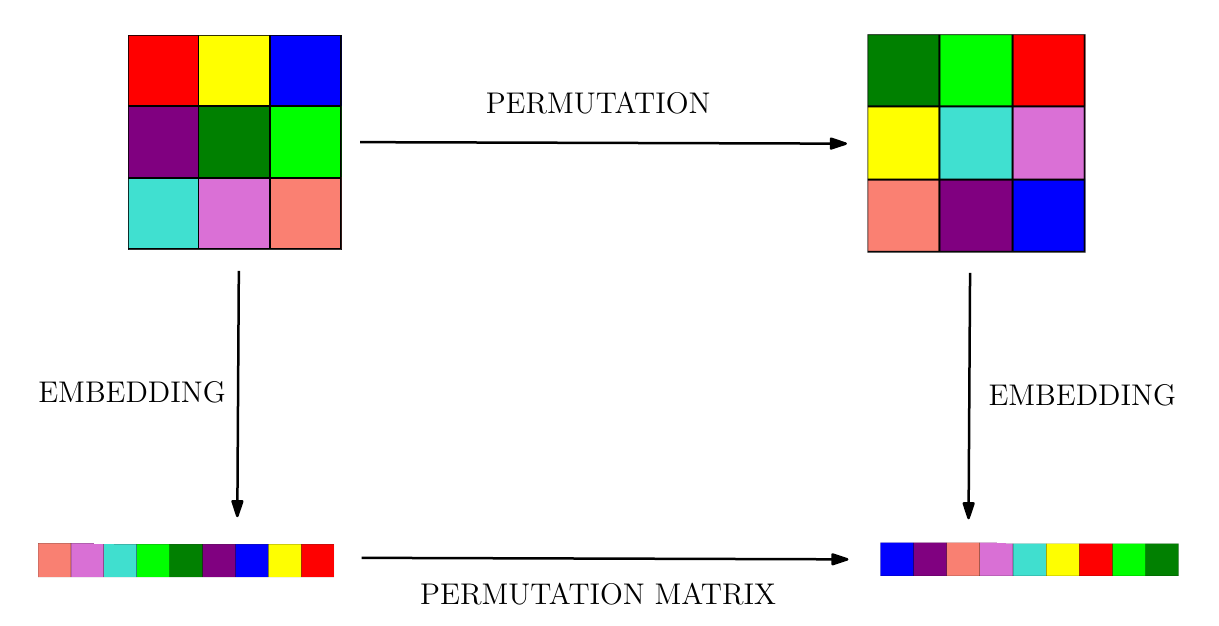}
    \caption{Diagram identifying two signals permutation-apart from each other on a $3\times 3$ grid, where each color represents the signals' value at the pixels. Notice that the diagram commutes.}\label{fig: ppt}
    \centering
\end{figure}

We shall recall here some important properties of permutation matrices. For example, by construction, it is easy to see that if $\sigma$ is a permutation matrix, then it is an orthogonal transformation on $\R^{n}$. Besides, if $\sigma$ and $\sigma'$ are two permutations on the pixel grid, then the composition $\sigma \circ \sigma'$ also is a permutation, whose associated permutation matrix is given by the product $\sigma \cdot \sigma'$. Consequently, if $\Sigma$ is a group of grid permutations, we may define an orthogonal representation by considering its permutation matrices. This shows that every group of permutations is a subgroup of $\Ort(n)$.

We are, hoewever, interested in determining whether these groups are also contained in smaller Lie groups. Explicitly, if $\Sigma$ is a permutation group and $f$ an initial signal, then its orbit $\{\sigma\cdot f^\uparrow \mid \sigma \in \Sigma\}$ is a subset of $\R^n$ and is included in the orbit of a representation of a compact Lie group.
An instructive case is that of a cyclic group, say of order $p$, generated by a permutation $\sigma$. 
Supposing the permutation is even, the corresponding matrix is special orthogonal, hence it can be written as $\exp(2\pi L/p)$ for a certain skew-symmetric matrix $L$.
The orbit thus reads
$$
\big\{ \exp(2\pi m L/p)\cdot f^\uparrow \mid m \in [0\isep p-1] \big\}.
$$
In particular, it is a subset of the orbit of $f^\uparrow$ under the representation of $\SO(2)$ in $\R^n$ given by $\theta\mapsto \exp(2\pi \theta L)$.
The lemma below shows that this observation holds for any Abelian group $\Sigma$ and that, in addition, the orbit can be projected in lower-dimensional subspaces, provided that they are chosen in accordance with the covariance matrix of the orbit.
To keep the focus of this section on applications, we defer its proof to Appendix \ref{sec:appendix_applications}.

\begin{lemma}\label{th: Tn-ppt}
Consider point cloud $X = \{x_i\}_{i=1}^N$ of $\R^n$ generated by the application of an Abelian permutation group $\Sigma$ of rank $d$ to an initial image that has been centered, i.e., $\sum_{i=1}^N x_i = 0$. 
Then $X$ lies on an orbit of a representation of the torus $T^d$.
More precisely, there exists a representation $\phi\colon T^d \rightarrow \SO(n)$, a point $x\in\R^n$, and $p_1,\dots,p_d\in \Z^d$ such that 
$$
X = \big\{ \phi\big(2\pi m_1 / p_1, \dots, 2\pi m_d/ p_d \big)\cdot x \mid m_1 \in [0\isep p_1-1], ~\dots, ~m_d \in [0\isep p_d-1] \big\}.
$$
Moreover, suppose the point cloud is projected into a subspace that is a sum of eigenspaces of the data's covariance matrix (all of which have even dimension). 
Then it lies on an orbit of $T^d$.
\end{lemma}

\begin{example}[Translated gorillas on circle]\label{ex:gorillas_circle}
We consider a dataset formed of 130 RGB gorilla images with $120\times 130$ pixels each. As illustrated in the left-side of Figure \ref{fig: gorilla circle}, these images are formed by subsequent applications of the translation group $\Sigma = \langle (0,1) \rangle \simeq \Z/130\Z$ to an initial image, meaning that the respective embeddings lie exactly on an orbit of this $\Sigma$ in $\R^{120\times 130\times 3}$. Alternatively, by Lemma \ref{th: Tn-ppt}, these embeddings must also be samples of an orbit of $T^1=\SO(2)$ on this space, whose specific representation type can be given by {\texttt{LieDetect}}. 
The middle image of Figure \ref{fig: gorilla circle} indicates the bottom eigenvalues of \ref{item:step2}'s \texttt{LiePCA} operator when applied to the embedded gorilla images, with dimension previously reduced to 8 through PCA. 
One particularly small eigenvalue is observed, as predicted by Lemma \ref{th: Tn-ppt}. The estimated representation type for this problem has weights $(1,2,3,4)$. Good convergence of {\texttt{LieDetect}} in this case can be confirmed by the right-most image of the figure, where the reconstructed orbit (magenta) is visually identified to be very close to the embedded images (black)---in fact, the estimated Hausdorff distance is of only $0.0086$.

Naturally, we might ask what influence projecting the images into $\R^8$ has had on the outputs. In Table~\ref{table:gorillas_translation_SO(2)}, we repeat the experiment for all even dimensions between 2 and 26. As can be seen, the estimated orbit always approaches the point cloud correctly, except for dimension 26, where the algorithm fails. We attribute this to the fact that the estimation of tangent spaces, required by \texttt{LiePCA}, is no longer good enough for this dimension.
In this regard, we point out that the estimation of tangent spaces was performed with only two neighbors---knowing that the data live on a circle, this is sufficient.
If we were to use, say, the ten nearest neighbors of each point, then the algorithm would fail as early as dimension 10.

\begin{figure}[H]
\centering
\includegraphics[width=0.28\textwidth]
{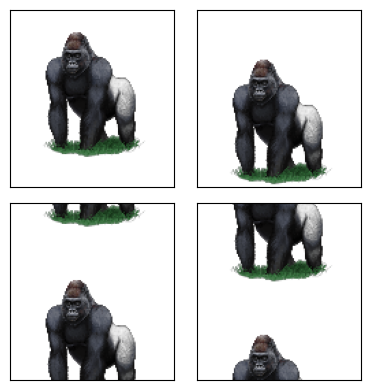}
~~
\includegraphics[width=0.31\textwidth]
{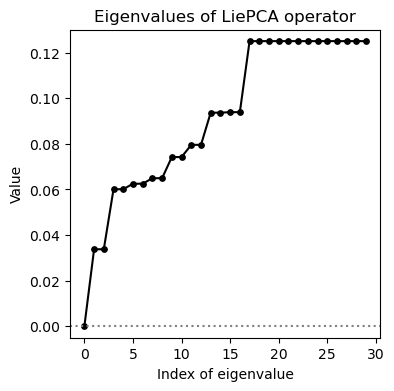}
~~
\includegraphics[width=0.31\textwidth]
{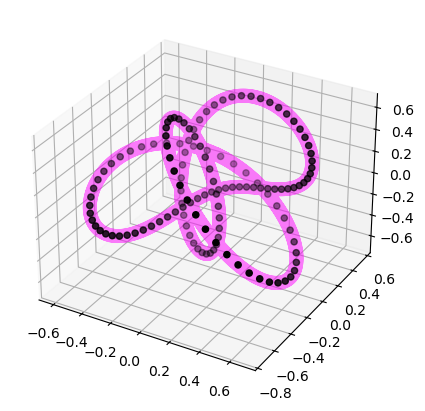}
\caption{Illustrations of Example \ref{ex:gorillas_circle}.
\textbf{Left:} sample of the dataset of translated gorilla images in a $\SO(2)$-like permutation group. \textbf{Middle:} eigenvalues of the \texttt{LiePCA} operator $\Lambda$, with one particularly small. \textbf{Right:} estimated orbit (magenta) for the point cloud (black) of the embedded images in dimension 8 (reduced to dimension 3 via PCA for visualization).}
\label{fig: gorilla circle}
\end{figure}

\begin{longtable}{||c|l|c||}\hline\begin{tabular}{@{}c@{}}Ambient\\dimension\end{tabular}&Representation found&\begin{tabular}{@{}c@{}}Haudorff\\distance\end{tabular}\\*\hline2&(1,)&0.0031\\4&(1, 2)&0.0050\\6&(1, 2, 3)&0.0068\\8&(1, 2, 3, 4)&0.0086\\10&(1, 2, 3, 4, 5)&0.0104\\12&(1, 2, 3, 4, 5, 6)&0.0122\\14&(1, 2, 3, 4, 5, 6, 8)&0.0148\\16&(1, 2, 3, 4, 5, 6, 7, 8)&0.0159\\18&(1, 2, 3, 4, 5, 6, 7, 8, 10)&0.0183\\20&(1, 2, 3, 4, 5, 6, 7, 8, 9, 10)&0.0206\\22&(1, 2, 3, 4, 5, 6, 7, 8, 9, 10, 11)&0.0581\\24&(1, 2, 3, 4, 5, 6, 7, 8, 9, 10, 11, 12)&0.0513\\26&(1, 2, 3, 4, 5, 6, 7, 8, 9, 10, 11, 12, 13)&0.5475\\*\hline\caption{Results of {\texttt{LieDetect}} on gorillas translated in one direction, as a function of the dimension reduction performed in \ref{item:step1}.}
\label{table:gorillas_translation_SO(2)}
\end{longtable}
\end{example}

\begin{example}[Translated gorillas on torus]\label{ex:gorillas_torus}
Let us generalize the prior example to higher dimensions. 
Translating the images both horizontally and vertically would, by Lemma \ref{th: Tn-ppt}, yield a representation of the torus $T^2$.
More precisely, we let the group $\Z/60\Z\times\Z/65\Z$ act on the images, where the element $(1,0)$ corresponds to a horizontal translation by 2 pixels, and $(0,1)$ a vertical translation by 2 pixels (see Figure \ref{fig: gorilla torus}).
After flattening, we are faced with a set of $60\times65$ points in $\R^{120\times 130\times 3}$, that we reduce to dimension 8 with PCA, and apply \ref{item:step1}'s orthornormalization.
The middle plot of Figure \ref{fig: gorilla torus} shows the result of \ref{item:step2}, where a solution space of dimension 2 can be clearly identified. Because the skew-symmetrized outputs of Step 2 commute up to $10^{-13}$ in matrix Frobenius norm, we apply the reformulation of the algorithm in Section \ref{sec: algorithm torus simp} to torus data, giving the representation type of $((0, 1, 1, 1), (-1, 0, 1, 2))$. The Hausdorff distance with the generated orbit was only $0.015$, which we consider small. 
Next, we repeat the whole experiment, changing the initial dimension reduction
(see Table \ref{table:gorillas_translation_T2}). From dimension 14 onwards, the estimated orbit becomes incorrect, and we should resort, as in the previous example, to a better estimation of tangent spaces.

\begin{figure}[H]
\centering
\includegraphics[width=0.28\textwidth]
{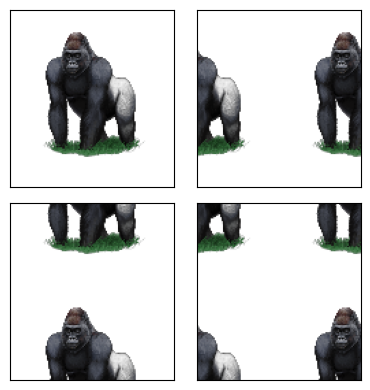}
~~
\includegraphics[width=0.31\textwidth]
{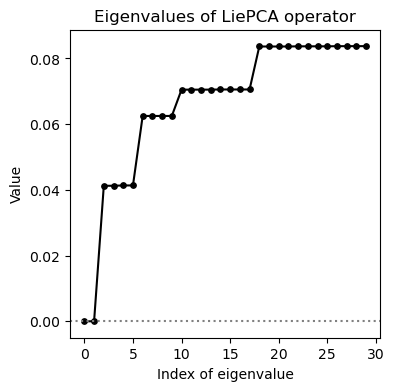}
~~
\includegraphics[width=0.31\textwidth]
{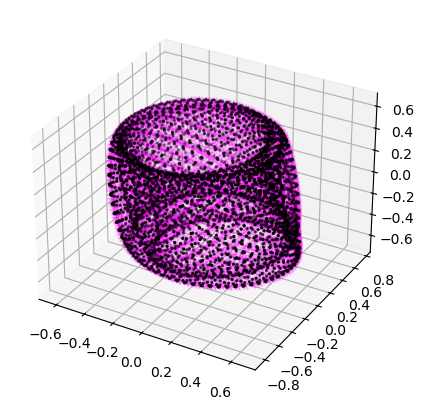}
\caption{Illustrations of Example \ref{ex:gorillas_torus}. \textbf{Left:} sample of the dataset of translated gorilla images in a $T^2$-like permutation group. \textbf{Middle:} eigenvalues of the matrix $\Lambda$, after dimension reduction into $\R^8$. Two small values stands out. \textbf{Right:} orbit output by our algorithm.}
\label{fig: gorilla torus}
\end{figure}
\vspace{-0.35em}

\begin{longtable}{||c|l|c||}\hline\begin{tabular}{@{}c@{}}Ambient\\dimension\end{tabular}&Representation found&\begin{tabular}{@{}c@{}}Haudorff\\distance\end{tabular}\\*\hline4&((0, 1), (1, 0))&0.0383\\6&((-2, 1, 1), (-1, 0, 1))&0.0506\\8&((0, 1, 1, 1), (-1, 0, 1, 2))&0.0624\\10&((-2, 0, 2, 1, 1), (-2, 1, 1, 0, 1))&0.0768\\12&((-2, 2, 1, 0, 1, -1), (-1, 2, 2, -1, 1, 0))&0.0353\\14&((-1, 1, 2, -2, 1, 2, 0), (0, 2, 1, -2, 1, 0, -1))&0.8207\\*\hline
\caption{Results of {\texttt{LieDetect}} on gorillas translated in both directions, as a function of the dimension reduction performed in \ref{item:step1}.}
\label{table:gorillas_translation_T2}
\end{longtable}
\end{example}

\begin{example}[Rotated gorillas on circle]\label{ex:gorillas_circle_rotations}
Because of interpolation, rotations of images are not exactly pixel permutations and do not fall within the assumptions of Lemma \ref{th: Tn-ppt}. Nevertheless, we still expect {\texttt{LieDetect}} to detect orbits of $\SO(2)$ with some noise.
For this experiment, we use \texttt{scipy}'s implementation of 2D rotation, with default spline interpolation of order 3.
We generate 360 images from the initial $120\times 130$ gorilla and project the point cloud in dimension 32 via PCA.
As shown in Figure \ref{fig: gorilla rotations}, we estimate a correct orbit, with (non-symmetric) Hausdorff distance to the initial point cloud of approximately 0.0334.
Conversely, the Hausdorff distance from the estimated orbit to the data is only 0.0866.
This is remarkably small: the maximal distance between two consecutive images (after dimension reduction and orthonormalization) is 0.1718, and the Hausdorff distance cannot be greater than half of it, that is, 0.0859.

\begin{figure}[H]
\centering
\includegraphics[width=0.28\textwidth]
{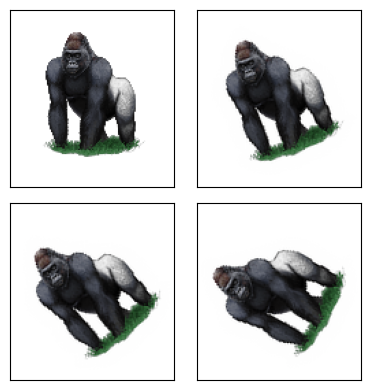}
~~
\includegraphics[width=0.31\textwidth]
{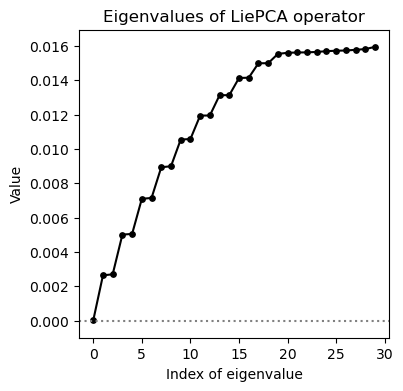}
~~
\includegraphics[width=0.31\textwidth]
{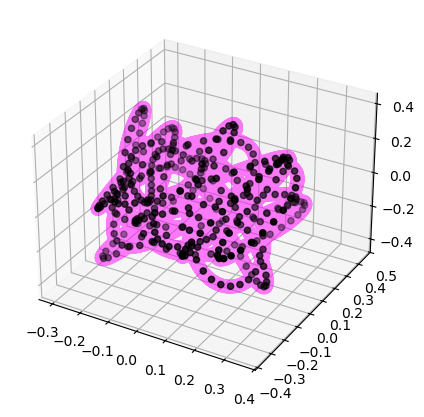}
\caption{Illustrations of Example \ref{ex:gorillas_circle_rotations}.
\textbf{Left:} sample of the dataset of rotated gorillas. \textbf{Middle:} bottom eigenvalues of the \texttt{LiePCA} operator $\Lambda$, where one value is clearly distinguishable. \textbf{Right:} estimated orbit (magenta) for the point cloud (black) of the embedded images in dimension 32.}
\label{fig: gorilla rotations}
\end{figure}
\end{example}

\subsubsection{Rotation of 3D bodies}\label{subsubsec:rot_3D_bodies}
As shown in Example \ref{ex:gorillas_circle_rotations} above, applying rotations to a 2D image produces a point cloud that can be modeled as a representation of $\SO(2)$. In the case of 3D images, rotated in all directions, we expect an $\SO(3)$-orbit. However, the results do not generalize so easily, for two reasons. Firstly, the orbit forms a manifold of dimension 3, which complicates the estimation of tangent spaces. In addition, 3D objects are often implemented as meshes, not images, for which several notions of rotation are used.

More precisely, we will consider three rotation processes, illustrated in Figure \ref{fig:rotation_3D_methods}.
Starting from an initial mesh of a 3D object (we use here the armadillo from the library \texttt{open3d}) and an initial grid of voxels (of size $20\times20\times20$ in our experiments), 
\begin{itemize}
\itemsep0.cm
\item[From mesh:] we rotate the initial mesh, and identify which voxels intersect it. This is done through \texttt{open3d}'s native function \texttt{create\_from\_triangle\_mesh}, which yields a binary 3D image.
\item[From image:] we build an initial 3D image by sampling many points on the mesh and counting how many points fall in each voxel, and we rotate directly this image, with \texttt{scipy}.
\item[From points:] we generate an initial sample of points on the mesh, rotate these points, build a kernel density estimator on them, and record the values of it on the voxel grid.
In practice, we use a Gaussian kernel with a bandwidth of $0.25$.
\end{itemize}
While the first method more accurately captures the rotation of the object than the second, they both yield noisy point clouds, since, by localizing points in voxels, we lose information.
The third one, in which voxels integrate information from all the points, is arguably more robust.

\begin{figure}[ht]\centering
\begin{minipage}{0.31\linewidth}\centering
\includegraphics[width=0.9\linewidth]
{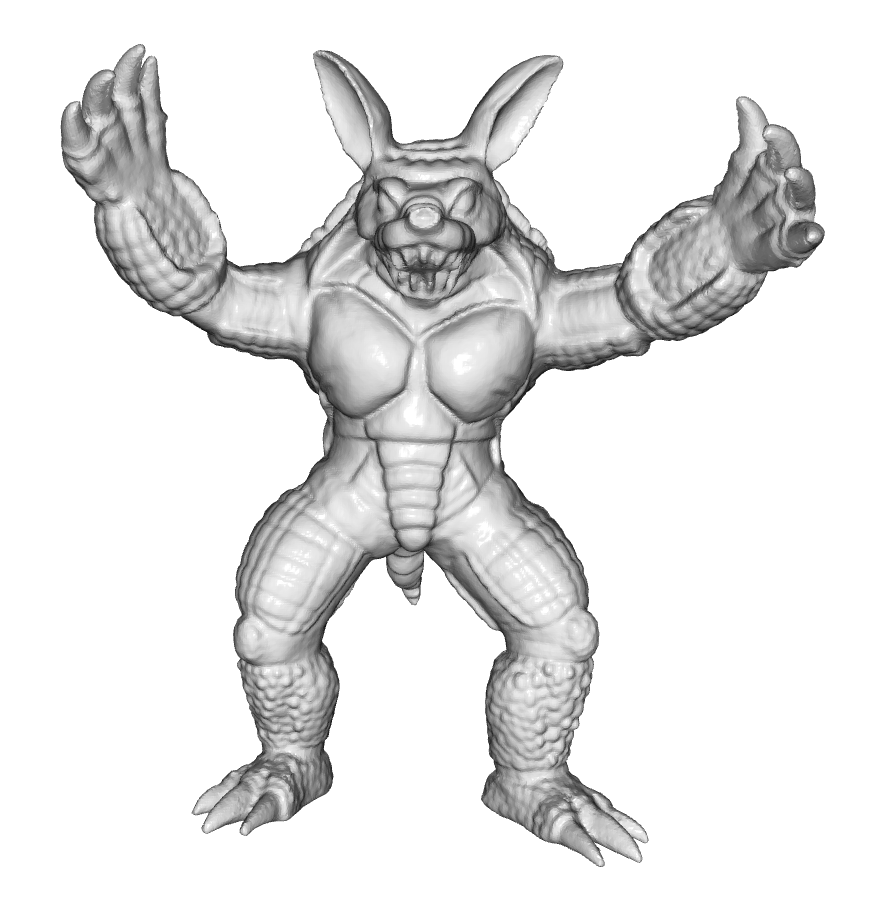}
\\\vspace{-.5cm}
\includegraphics[width=0.99\linewidth]
{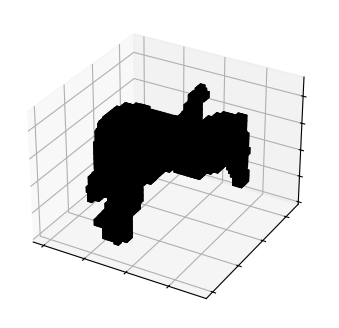}
\end{minipage}
~
\begin{minipage}{0.31\linewidth}\centering
\includegraphics[width=0.99\linewidth]
{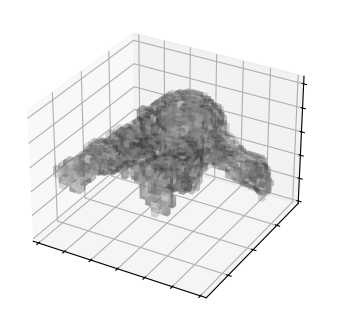}
\\\vspace{-.5cm}
\includegraphics[width=0.99\linewidth]
{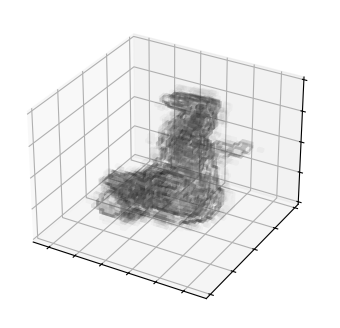}
\end{minipage}
~
\begin{minipage}{0.31\linewidth}\centering
\includegraphics[width=0.99\linewidth]
{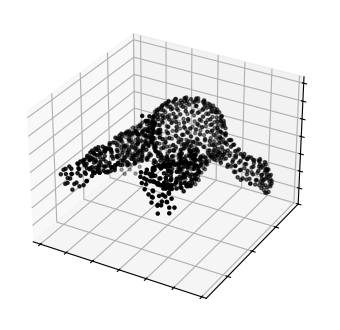}
\\\vspace{-.5cm}
\includegraphics[width=0.99\linewidth]
{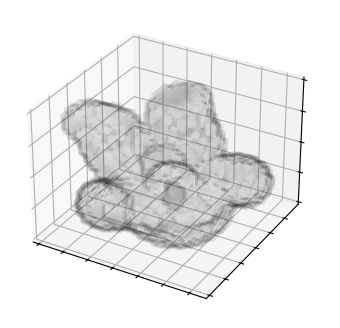}
\vspace{-.5cm}
\end{minipage}
\caption{To generate 3D images of a rotated armadillo, three methods are considered.
\textbf{Left:} a high-quality mesh (top) is rotated and then voxelized (bottom).
\textbf{Middle:} the mesh is first voxelized (top) and the image is rotated (bottom).
\textbf{Right:} a point cloud is sampled on the mesh (top), then rotated and used as the input of a kernel density estimation (bottom).}
\label{fig:rotation_3D_methods}
\end{figure}

For each method, we generated a set of 5000 images of shape $20\times20\times20$, seen as point clouds in $\R^{20\times 20\times 20}$.
The top eigenvalues of the covariance matrix of the first point cloud are
$$
0.141, 
~~0.103, 
~~0.100,
~~0.069, 
~~0.068, 
~~0.067, 
~~0.059, 
~~0.058, 
~~0.034, 
~~0.033, 
~~0.033, 
~~0.032.
$$
One identifies three significant gaps between consecutive eigenvalues: after the first, third, and eighth one. 
Let us choose dimension reduction into $\R^8$, since the two other values would yield uninteresting point clouds. 
For the point clouds obtained via the three methods, we apply {\texttt{LieDetect}} for the group $\SO(3)$, each time detecting the optimal representation $(3,5)$, with Hausdorff distances (computed in \ref{item:step4}) given in Table \ref{table:armadillos_rotated} and orbits represented in Figure \ref{fig:rotation_3D_methods_orbits}.

\begin{table}[ht]\centering
\begin{tabular}{||ll||} 
\hline
Rotation method &  Distance $\HDmid{X}{\widehat{\O}_x}$ \\ 
\hline
From mesh   & $1.1184$  \\ 
From image  & $0.3521$  \\ 
From points & $0.1909$ \\
\hline
\end{tabular}
\caption{Results of {\texttt{LieDetect}} on rotated armadillos.}
\label{table:armadillos_rotated}
\end{table}

\begin{figure}[ht]\centering
\begin{minipage}{0.32\linewidth}\centering
\includegraphics[width=0.99\linewidth]
{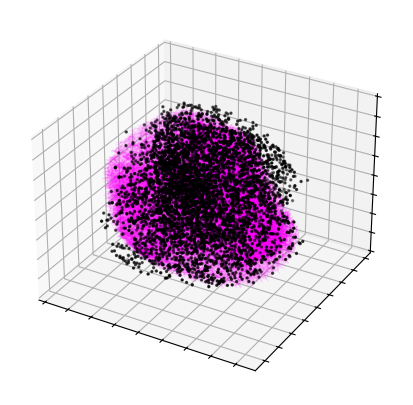}
\end{minipage}
\begin{minipage}{0.32\linewidth}\centering
\includegraphics[width=0.99\linewidth]
{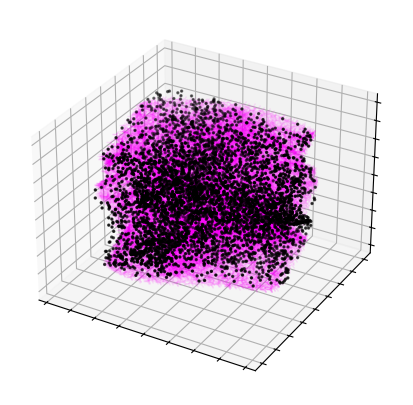}
\end{minipage}
\begin{minipage}{0.32\linewidth}\centering
\includegraphics[width=0.99\linewidth]
{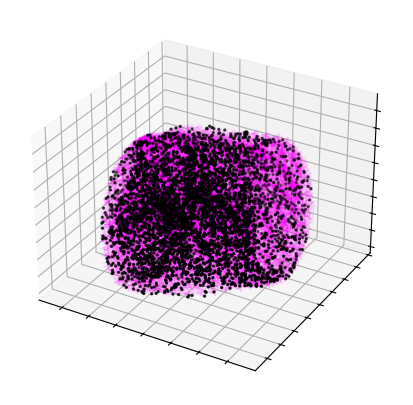}
\end{minipage}
\caption{Orbits estimated by {\texttt{LieDetect}} on the rotated armadillo images, generated from the mesh (\textbf{left}), the image (\textbf{middle}), or the points (\textbf{right}).}
\label{fig:rotation_3D_methods_orbits}
\end{figure}

It is worth noting that all our analysis relies on the choice of a specific embedding dimension, chosen as 8 in the example above.
More generally, given a point cloud $X$ in $\R^n$ obtained from the transformation of images, one may wonder in which dimension $X$ should be projected to observe a representation orbit.
In the case of Abelian pixel permutations, as stated in Lemma \ref{th: Tn-ppt}, one is free to choose any eigenspace.
For actions of $\SO(3)$, however, we did not study this problem further.
From an empirical point of view, we propose to tackle this question in two ways.
The first idea consists of identifying the significant gaps in consecutive eigenvalues of the covariance matrix $\Sigma[X]$ of $X$.
Indeed, as studied in Section \ref{subsubsec:analysis_pca_orbit}, when $X$ is the orbit of an orthogonal representation, the eigenvalues of $\Sigma[X]$ come as tuples of equal values, each tuple representing an invariant subspace.
For the second idea we consider, for each dimension, the projected point cloud, and compute the ratio between the third and fourth bottom eigenvalues of its \texttt{LiePCA} operator.
According to Proposition \ref{prop:consistency_LiePCA_idealcase}, if $X$ supports an action of $\SO(3)$, then only the first three eigenvalues must be close to zero, hence the ratio must be large.
We represent these quantities in Figure \ref{fig:rotation_3D_methods_eigengaps} for the three rotation methods and dimensions up to 19.
As expected, all the curves show a peak for dimension 8.
Interestingly, another peak is seen at dimension 15, for the third rotation method.
Although we have not taken this analysis further, we mention that, after $\R^3$ and $\R^5$, the next irreducible representation of $\SO(3)$ happens precisely in $\R^7$, leaving open the possibility that the data lies on the representation $(3,5,7)$ in $\R^{15}$.

\begin{figure}[ht]\centering
\begin{minipage}{0.32\linewidth}\centering
\includegraphics[width=0.99\linewidth]
{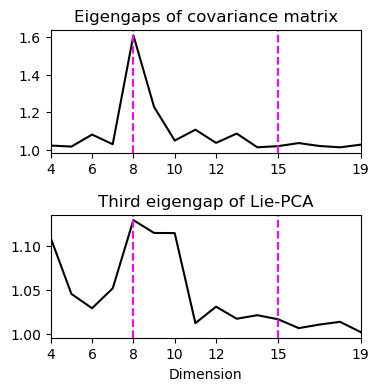}
\end{minipage}
\begin{minipage}{0.32\linewidth}\centering
\includegraphics[width=0.99\linewidth]
{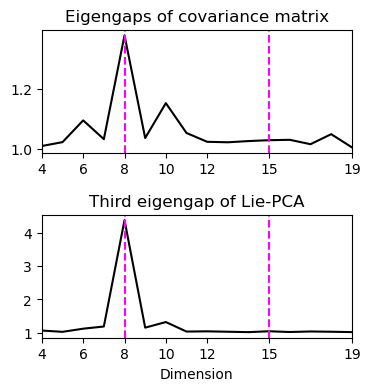}
\end{minipage}
\begin{minipage}{0.32\linewidth}\centering
\includegraphics[width=0.99\linewidth]
{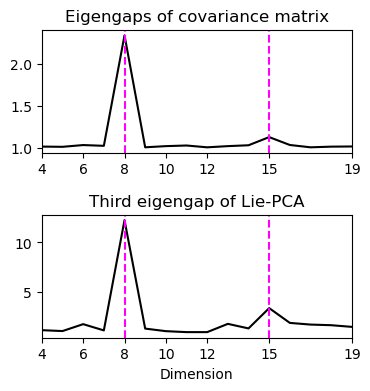}
\end{minipage}
\caption{Ratio of consecutive eigenvalues of the covariance matrix (\textbf{top}), and ratio between third and fourth eigenvalue of the \texttt{LiePCA} operator, as a function of the embedding dimension (\textbf{bottom}), for the three armadillo datasets. Images generated 
from the mesh (\textbf{left}), from the image (\textbf{middle}) or the points (\textbf{right}).
}
\label{fig:rotation_3D_methods_eigengaps}
\end{figure}

\subsubsection{Group-aware feature extraction}\label{subsec:harmonic_analysis}
Below, we suggest two applications of our algorithm in the context of regression where the data lie within orbits of representations.
In both cases, the datasets consist of collections of images that have been transformed by toric translations (Example \ref{ex:harmonic_analysis_T2}) and spatial rotations (Example \ref{ex:harmonic_analysis_SO3}), forming orbits of representations of the groups $T^2$ and $\SO(3)$. Using the representation estimated by \texttt{LieDetect}, we obtain a new, lower-dimensional set of coordinates for the data, in which the objective function $F$ assumes the form of the series described in Equation \eqref{eq: harmonic transform}. This suggests an application of \texttt{LieDetect} for group-aware feature extraction, which allows for more efficient machine learning models compared to estimating the objective function directly from the natural embedding of images in the Euclidean space.

\begin{example}[Regression via harmonic analysis on the torus]\label{ex:harmonic_analysis_T2}
Let us refer back to the gorillas of Example \ref{ex:gorillas_torus}, observed to form a point cloud $X=\{x_i\}_{i=1}^N$ that lies on an orbit of a representation of the torus (see Figure \ref{fig: gorilla torus}). Any complex function $F$ with domain the set of images may be then seen as the restriction $\widetilde{F}|_X$ of a map $\widetilde{F}: T^2\to \C$, which is well-behaved enough to be expressed as 
\begin{equation}\label{eq:linear reg ha}
\widetilde{F}(\theta_1, \theta_2) = \sum_{n_1, n_2 \in \Z} \widehat{f}(n_1, n_2) \cdot\exp(2\pi i\theta_1 n_1)\cdot \exp(2\pi i\theta_2 n_2)
\end{equation}
for the usual coordination $(\theta_1, \theta_2)$ of $T^2$ and for the Fourier coefficients $\widehat{f}$ defined as in Example \ref{ex: fourier}.
Notice that, since $\widehat{f}(n_x, n_y)$ are just coefficients, they may be learned through linear regression from a training set of values $F(\theta_1,\theta_2)$ for $(\theta_1,\theta_2)\in T^2$.
For instance, in \cite{miller1973complex}, the authors give a Maximum Likelihood Estimation of these complex coefficients as
\begin{equation*}
\widehat{{\beta}}(n_1,n_2) = (A^* A)^{-1}A^* \vec{F},
\end{equation*}
where $\vec{F}$ is the vector of all outputs of $F$, and $A$ is a regressor matrix.
In addition, this procedure relies on a hyperparameter $\nu_\mathrm{max}$, the maximal frequency, which, the bigger it becomes, the finer an approximation of $F$ we achieve. Of course, the best result happens in the infinite limit; nevertheless, in practice, taking values of $\nu_\mathrm{max}$ of order $10$ is often enough.

In our context, however, the input is given as a collection of images $x_i\in\R^n$, which are not explicitly mapped to a point $(\theta_1,\theta_2)$ of the torus.
This is where \texttt{LieDetect} comes into play, by estimating a representation of $T^2$ in $\R^n$ that describes the data.
More precisely, the algorithm gives a representation $\phi\colon T^2\rightarrow \SO(n)$ and an initial point $x_0$ such that the orbit 
$$\big\{\phi(\theta_1,\theta_2)\cdot x_0 \mid (\theta_1,\theta_2)\in T^2 \big\}$$ 
lies close to $X$. 
Through the derived representation $\d \phi$, we can describe this orbit as
$$\big\{\exp(\d\phi(t_1,t_2))\cdot x_0 \mid (t_1,t_2)\in \mathfrak{t}^2 \big\}.$$
It is even enough to consider a subset of the Lie algebra $\mathfrak{t}^2$, namely, a fundamental domain for the exponential map, as described in Section~\ref{sec:step4_hasdorff}.
Now, any image $x_i\in X$ can be pulled back to $\mathfrak{t}^2$, by identifying which pair minimizes the cost
\begin{equation}\label{eq:toroidal_coordinates}
\min_{(t_1,t_2)\in \t^2} \big\|x_i - \exp(\d\phi(t_1,t_2))\cdot x_0 \big\|.    
\end{equation}
Through this procedure, one obtains a map $X\rightarrow \t^2$, that we shall call \textit{toroidal coordinates}. Therefore, these smaller dimensional coordinates ultimately work as a feature extraction step, which simplifies the original Euclidean embedded data before the application of machine learning models. In particular, using these coordinates, we know the objective function to be approximately given by Equation \eqref{eq:linear reg ha}, which, again, might be computed by the very simple Maximum Likelihood Estimation described above.

Figure \ref{fig: filter gorilla} defines a function that illustrates a regression problem: we consider the Gaussian $130\times 120$ pixels filter shown on the left.
Then, we take the convolution of the filter with the gorilla pictures, that is, multiply, pixel-wise, the filter and the image, and sum the values of the resulting pixels. We see that, by fixing the filter and applying this convolution process to each of the gorillas' images, we define a function $F\colon X\to \R$, whose final value is a measurement of the total brightness of the image \emph{after} multiplication by the filter.
On the right of Figure \ref{fig: filter gorilla}, we represent the function $F$ as a color plot on the $65\times 60$ grid, where the color of each point corresponds to the value of an image $x_i\in X$.
We remind the reader that the set of images has been obtained by translating an initial gorilla 65 times in the vertical and 60 times in the horizontal direction, hence they can be arranged in a grid of shape $65\times60$.
Notice that, for the points closer to the center---that is, for the least translated images---$F$ achieves smaller values since less of the white part of the gorilla's image is exposed to the white part of the filter. 

\begin{figure}[ht]
\centering
\includegraphics[width=\textwidth]
{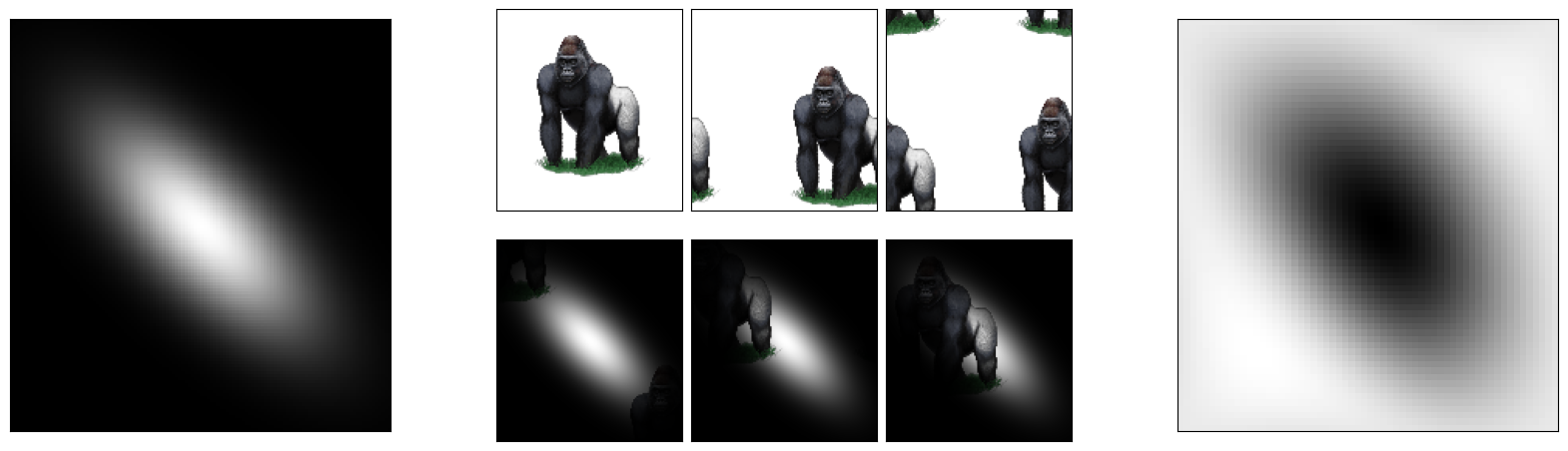}
\caption{Illustrations of Example \ref{ex:harmonic_analysis_T2}.
\textbf{Left:} our Gaussian filter is a $130\times 120$ image.
\textbf{Middle:} each gorilla image is multiplied coordinate-wise with the filter, and the sum of the pixels is then computed. This yields a real value $F(x_i)$ for each image $x_i$.
\textbf{Right:} representation of the values $\{F(x_i)\}_{i=1}^N$, where the inputs are arranged in a grid of shape $65\times60$. 
}
\label{fig: filter gorilla}
\end{figure}

To transform this into an actual regression problem, we randomly split the data, here seen as the vector embedding of the images $X=\{x_i\}_{i=1}^N$ (inputs) and their respective values $\{F(x_i)\}_{i=1}^N$ (outputs), into a training set of $90\%$ the original size, and a test set with the remaining $10\%$. 
To speed up computations, the models are not fed directly with the $130\times120$ pixels images, but with the projection of the point cloud in $\R^{50}$.
Table \ref{tb: models} shows the Mean Squared Error (MSE) achieved on test data by classic regression methods.
The random forest used \texttt{sklearn}'s implementation default parameters. For the SVM, we used RBF kernel and grid search on the sets $\{0.1, 1, 10, 100, 1000\}$ and $\{0.1, 0.2, 0.5, 1\}$ for the parameters $C$ and $\epsilon$, respectively. Finally, different feedforward neural network architectures with further $10\%$ data dedicated to validation were tested, and the reported result corresponds to the prediction of the one that achieved the smallest MSE on test data.

Lastly, the toroidal coordinates model has been trained following the procedure described above.
Using only the training data, a PCA projection into $\R^4$ was performed, and \texttt{LieDetect} was used to estimate a representation of $T^2$ that fits the data, just as we have done in Example \ref{ex:gorillas_torus}.
Toroidal coordinates $c\colon X\rightarrow \t^2$ have been found by sampling many points on $\t^2$ and identifying those minimizing Equation \eqref{eq:toroidal_coordinates}.
Finally, Maximum Likelihood Estimation using \cite{miller1973complex} on Equation \eqref{eq:linear reg ha} has been performed, taking for input data the coordinates $c(x_i)$ and their respective value $F(x_i)$.
This model can then be used to predict the value $F(x)$ of an image $x$ in the test set by simply projecting it into $\R^4$, following the PCA projection performed on the test data, identifying toroidal coordinates $c(x)$, and computing the value of the regression function on it.
This procedure is shown in Figure \ref{fig:toroidal_coordinates}, where each input $x$ is represented via its coordinates $x(c)\in[0,1]^2$, and colored with the corresponding value $F(x)$.
The initial task has been reduced to a simple regression problem on the square.

Although Table \ref{tb: models} indicates that the harmonic analysis performs better on this task, we stress that this list of models is far from exhaustive and it is very likely that fine-tuning extra hyperparameters on the SVM and neural network implementations would allow for models that improve over the Lie group-oriented technique. Moreover, the embedded gorilla images have a close-to-zero deviation from the underlying torus manifold, which is unlikely to happen in other applications. The point, however, is that through the geometric knowledge of data, competitive solutions to machine learning tasks are possible, indicating that the methods described in this article represent not only gains in terms of explainability but also in accuracy.

\begin{table}[ht]
\centering
\begin{tabular}{||c|c||}\hline
Model&MSE on test data\\\hline
Random forest&5.9953$\times 10^{-5}$\\Support Vector Machine&8.6959$\times 10^{-5}$\\Neural Network&6.0687$\times 10^{-5}$\\
Toroidal coordinates&\textbf{4.0333}$\times 10^{-5}$\\\hline
\end{tabular}
\vspace{.25cm}
\caption{MSE of the models for regression with abstract harmonic analysis in Example \ref{ex:harmonic_analysis_T2}. 
The models were trained on $90\%$ of the gorilla image dataset. The best score is shown in bold.}
\label{tb: models}
\end{table}

\begin{figure}[ht]
\centering
\includegraphics[width=0.7\textwidth]{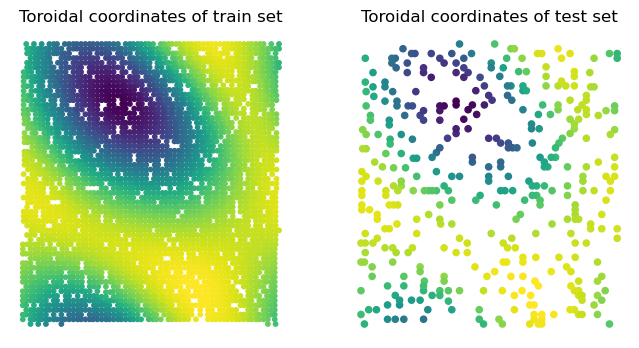}
\caption{\textbf{Left:} \texttt{LieDetect} allows to associate each image $x$ of the train set to toroidal coordinates $c(x)\in[0,1]^2$. The points are colored by the value of the objective function $F(x)\in\R$. \textbf{Right:} harmonic analysis allows for a regression of $F$, subsequently computed on test images.}
\label{fig:toroidal_coordinates}
\end{figure}
\end{example}

\begin{example}[Prediction of rotation matrices]\label{ex:harmonic_analysis_SO3}
Another task that has received attention in Machine Learning is that of detecting, based on an initial 3D image and a rotated version of it, the orthogonal matrix underlying the transformation \cite{saxena2009learning}.
We aim to show in this example how \texttt{LieDetect} can be used for such a task.
More precisely, we shall consider the set $\{x_i\}_{i=1}^N$ of $N=5000$ images of armadillos defined in Section \ref{subsubsec:rot_3D_bodies}, obtained by rotating an initial 3D image (input).
Our regression goal, $\{F(x_i)\}_{i=1}^N$, is the unit vector of $\R^3$ that points to the head of the armadillo (output).
Just as before, we build a training set of $90\%$ of the original size and test with the remaining $10\%$.
Figure \ref{fig:rotation_3D_regression} shows the values of $F(x_i)$.

We remember from Section \ref{subsubsec:rot_3D_bodies} that the images support an action of $\SO(3)$ when projected in $\R^8$.
We thus consider this particular dimension and estimate the underlying orbit with \texttt{LieDetect}.
The output is an orthogonal representation $\phi\colon \SO(3)\rightarrow \SO(8)$, whose orbit $\{\phi(g)\cdot x_0\mid g\in \SO(3) \}$, with $x_0\in\R^8$ an arbitrary initial point, lies close to the point cloud (as in Table \ref{table:armadillos_rotated}).
In a second step, following the idea of harmonic analysis, we use the representation $\phi$ to pull the points $x_i$ back to the algebra $\so(3)$ via the derived representation $\d\phi$. 
This is done by sampling many points on $\so(3)$, and identifying, for each $x_i$, a $c_i\in\so(3)$ which minimizes 
$$
\min_{c\in \so(3)} \big\|x_i - \exp(\d\phi(c))\cdot x_0 \big\|.
$$
This procedure yields a new representation of the data $\{x_i\}_{i=1}^N$, as a set of points $\{c_i\}_{i=1}^N$ in $[0,1]^3$. As already observed in Figure \ref{fig:rotation_3D_regression}, the objective function $F$, now seen with domain $\so(3)$, looks simpler. We shall call $\{c_i\}_{i=1}^N$ the \textit{orthogonal coordinates} for the point cloud $X$.

Although we could proceed with a harmonic decomposition of the objective function $F$ (using spherical harmonics) and solve the regression problem again with a Maximum Likelihood Estimation---similarly to what we did in the previous example---we want here to stress the feature extraction property of using orthogonal coordinates, which improves the ability of estimation independently of the exact model we apply after transforming $F$ into a well-behaved infinite sum. We train an SVM classifier with input $\{c_i\}_{i=1}^N$ and output $\{F(x_i)\}_{i=1}^N$. More precisely, three SVM regressions are computed (one per coordinate of the output), using a standard radial basis function kernel and optimizing the parameters on a large grid.
As seen in Table \ref{tab:rotation_3D_regression_scores}, a very good score is attained: the mean squared error is only $0.0066$.
In comparison, we train a collection of SVM classifiers, now on the full input $\{x_i\}_{i=1}^N$, after applying dimension reduction, with dimensions ranging from 3 to 10. 
It is striking that the classifier trained on orthogonal coordinates, which are only 3-dimensional, returns a better score than those trained on the initial point cloud, even when reduced to dimension 10 (the error is $0.0122$).
As suggested by Figure \ref{fig:rotation_3D_regression}, taking the points back into the Lie algebra $\so(3)$ acts as a kind of signal decorrelation, extracting features relevant to the action of $\SO(3)$.

\begin{figure}[ht]\centering
\includegraphics[width=0.99\linewidth]
{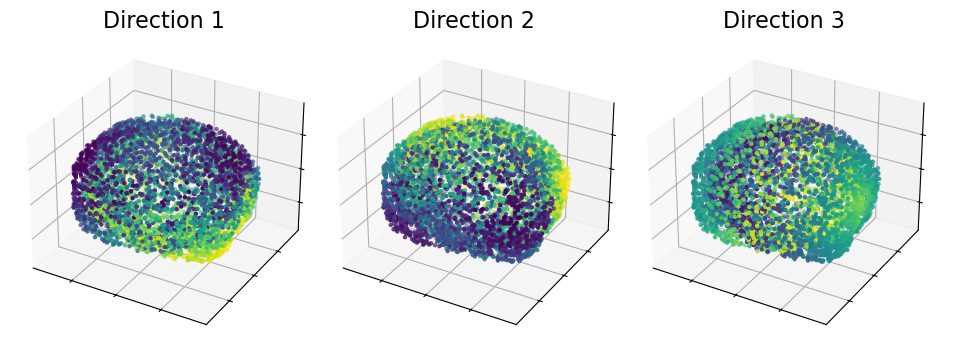}\\
\includegraphics[width=0.99\linewidth]
{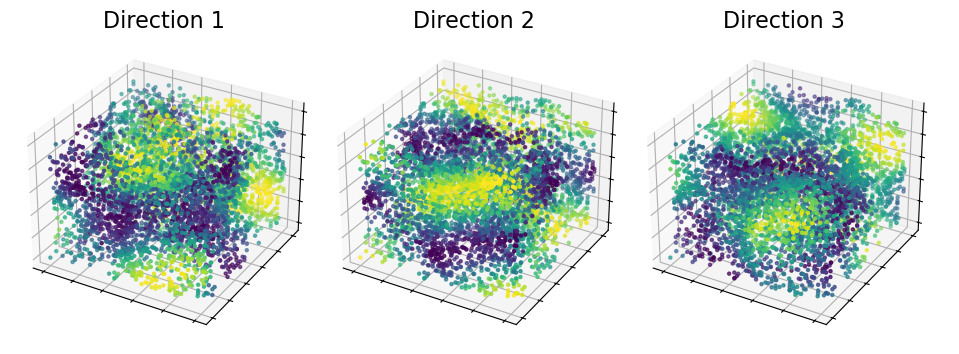}
\caption{Illustrations of Example \ref{ex:harmonic_analysis_SO3}.
\textbf{Top:} The regression task associates to a 3D image $x_i\in\R^n$ the unit vector $F(x_i)\in\R^3$ pointing to the head of the armadillo. This data is represented as scatter plots, for each coordinate of $F(x_i)$ (Direction 1, 2, and 3), with the color indicating the magnitude, and with the base point being the point cloud $\{x_i\}_{i=1}^N$ projected in $\R^3$ via PCA for visualization purposes.
\textbf{Bottom:} The magnitudes of $F(x_i)$ are also represented, now based on the orthogonal coordinates $\{c_i\}_{i=1}^N$, obtained by pulling back the 3D images $x_i$ to the Lie algebra $\so(3)$. This results in a point cloud in $[0,1]^3$.
}
\label{fig:rotation_3D_regression}
\end{figure}

\begin{table}[H]
\centering
\begin{tabular}{||c|c||}\hline
\hline Model&MSE on test data\\*\hline
SVM in dimension 3&0.4003\\SVM in dimension 4&0.2496\\SVM in dimension 5&0.1295\\SVM in dimension 6&0.0380\\SVM in dimension 7&0.0148\\SVM in dimension 8&0.0119\\SVM in dimension 9&0.0114\\SVM in dimension 10&0.0122\\SVM on orthogonal coordinates&\textbf{0.0066}\\*\hline
\end{tabular}
\vspace{.2cm}
\caption{Models trained on $90\%$ of the armadillo 3D image dataset (\textbf{left column}) and their test Mean Squared Error (\textbf{right column}). 
The set of 5000 images is projected in dimension $n$ ranging from $3$ to $10$ via PCA, yielding a point cloud $x_i\in\R^n$, and an SVM classifier is fitted to estimate the unit vector $F(x_i)$ of $\R^3$ that points towards the head of the armadillo.
In addition, the same SVM is trained on the orthogonal coordinates $\{c_i\}_{i=1}^N\subset \R^3$ instead of $\{x_i\}_{i=1}^N\subset\R^n$. It shows a better score (shown in bold), even though this encoding of the data is only of dimension 3.}
\label{tab:rotation_3D_regression_scores}
\end{table}
\end{example}

\subsection{Machine learning}

In this section, we extend the use of \texttt{LieDetect} to two additional data science tasks: sampling on orbits of representations and equivariant neural networks.

\subsubsection{Density estimation}\label{ex:sampling}
\texttt{LiePCA} (\ref{item:step2}) was originally proposed in \cite{DBLP:journals/corr/abs-2008-04278} for density estimation and sampling problems. In their setting, $X=\{x_i\}_{i=1}^N$ is a finite sample of an unknown distribution $\mu$ supported on a manifold $\man$ embedded in a real vector space $\R^n$, and with symmetry group $G=\text{Sym}(\man)$ acting transitively on it. 
The problem consists of sampling, from the knowledge of $X$ only, new points from the distribution.
The authors proceed to draw the new points $Y=\{y_j\}_{j=1}^{N'}$ using the output $\{A_k\}_{k=1}^{\dim G}$ of \ref{item:step2} through 
\begin{equation}\label{eq:samples}
   y_j = \exp\bigg[\sum_{k=1}^{\dim G} t_k A_k\bigg]\cdot x_i
\end{equation}
where for each $j$, a point $x_i$, with $1\leq i \leq n$, called $y_j$'s \textit{source}, is randomly selected from the original dataset, and $t_k$ are drawn from a multivariate Gaussian distribution in $\R^{\dim G}$. 
In particular, the Gaussian distribution is taken to have zero mean and diagonal covariance, set according to the Silverman's rule of thumb\footnote{The Silverman's factor of a subset of $n$ points of a $l$-dimensional submanifold is $\big(\frac{n(l+2)}{4}\big)^{-\frac{1}{l+4}}$.} \cite{silverman2018density}. This technique, originally from kernel density estimation, scales the variances according to what is the most natural to the implicit dimension of $\man$, which can be estimated through manifold learning techniques, but we will assume given in our experiments.

Some care is, nonetheless, important here: because $\{A_k\}_{k=1}^{\dim G}$ are estimations of (a basis for) the \textit{vector space} $\mathfrak{sym}(\mathcal{\man})$, with no regards to its group structure, directly sampling according to Equation \eqref{eq:samples} may cause deviations from $\man$, especially for high values of $t_k$. This is illustrated on the left side of Figure \ref{fig:samples}, where the result of \texttt{LiePCA}, $A$, applied to samples $\{x_i\}_{i=1}^N$ of an orbit of a $\SO(2)$-representation in $\R^4$ is used to generate new samples through the law
\begin{equation}\label{eq:sample}
    y_j = \exp(t_j A)\cdot x_1
\end{equation}
where we now artificially fix a single source $x_1$ and $t_j$ is drawn uniformly on a fixed real interval $[0,T]$, with $T$ chosen so that a full period of the representation is performed. Because $A$ is not guaranteed to be an element of $\sym(\man)$, its exponential is not guaranteed to be periodic, explaining why the generated orbit (in blue) does not `close'. The authors of \cite{DBLP:journals/corr/abs-2008-04278} remedy this situation by resampling $y_j$ if its distance to $X$ is larger than a fixed threshold $\tau$. When this happens, one draws again $y_j$ for the same choice of source, but new values of $t_k$. Their choice for $\tau$ is not made explicit in the article, but it can be neither too small (the sampled distribution would not have connected support) nor too big (there might be a lot of noise due to the lack of periodicity).
In what follows, we will take it to be the maximum distance between two closest neighbors in $X$. With this correction, the authors empirically show that using \texttt{LiePCA} for sampling outperforms other techniques such as local PCA or kernel density estimation.

\begin{figure}[ht]
\centering
\begin{minipage}{0.49\linewidth}\center
\includegraphics[width=0.8\textwidth]{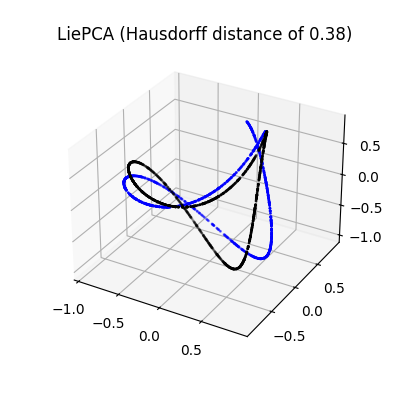}
\end{minipage}
\begin{minipage}{0.49\linewidth}\center
\includegraphics[width=0.8\textwidth]{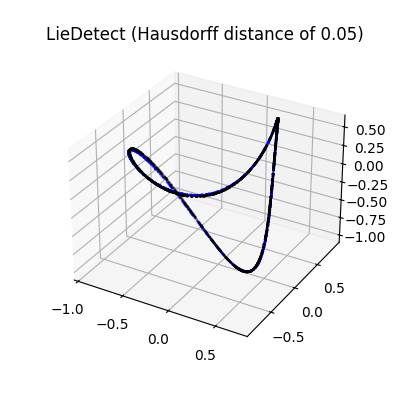}
\end{minipage}
\caption{New samples (blue) generated from a single source as in the law of Equation~\eqref{eq:sample}, where $A$ is estimated using only \texttt{LiePCA} (\textbf{left}) or using the full \texttt{LieDetect} algorithm (\textbf{right}) applied to a $\SO(2)$ orbit in $\R^4$. Black samples correspond to a ground-truth distribution (uniform distribution) on the orbit. 
Corresponding non-symmetric Hausdorff distances from the generated points and points from the ground-truth distribution are shown above the figures.}
\label{fig:samples}
\end{figure}

It is natural to wonder whether the sampling technique of the authors of \texttt{LiePCA} is affected if we substitute $\{A_k\}_{k=1}^{\dim G}$ by the outputs of \texttt{LieDetect}'s \ref{item:step3}, which we denote by $\{\widehat{A}_k\}_{k=1}^{\dim G}$. 
Because $A_k$ are estimated through tangent spaces only, while $\widehat{A}_k$ takes into account the Lie algebra structure, sampling using the latter is likely more robust to noise or small values of $n$ than the former. We compare these two strategies with a third, perhaps more natural, in which we rely on the periodicity of $\widehat{A}_k$ to arbitrarily fix a \textit{single source} $x_1$ and sample according to 
\begin{equation*}
    y_j=  \exp\bigg[\sum_{k=1}^{\dim G}t_k \widehat{A}_k \bigg]\cdot x_1
\end{equation*}
where the $t_k$'s now form a regularly-spaced finite subset of a fundamental domain of the Lie algebra, as explained in Section \ref{sec:step4_hasdorff} (see the right-hand side of Figure \ref{fig:samples}). 
In summary, we consider the following three sampling paradigms:
\begin{itemize}
    \item Multi-source \texttt{LiePCA}: sample according to 
    \begin{equation*}
        y_j=  \exp\bigg[\sum_{k=1}^{\dim G}t_k A_k \bigg]\cdot x_i
    \end{equation*}
    with $\{A_k\}_{k=1}^{\dim G}$ being the output of \texttt{LiePCA} (\ref{item:step2}), $\{t_k\}_{k=1}^{\dim G}$ drawn from Gaussian distributions using Silverman's rule of thumb, $x_i$ drawn uniformly in $X$ (for each $y_j$), and resampling whenever the distance from $y_j$ to $X$ is greater than $\tau$.
    
    \item Multi-source \texttt{LieDetect}: sample according to 
    \begin{equation*}
        y_j=  \exp\bigg[\sum_{k=1}^{\dim G}t_k \widehat{A}_k \bigg]\cdot x_i
    \end{equation*}
    with $\{\widehat{A}_k\}_{k=1}^{\dim G}$ being the output of \ref{item:step3}, $\{t_k\}_{k=1}^{\dim G}$ drawn from Gaussian distributions using Silverman's rule of thumb, $x_i$ drawn uniformly in $X$ (for each $y_j$), and resampling whenever the distance from $y_j$ to $X$ is greater than $\tau$.
    
    \item Single-source \texttt{LieDetect}: sample according to
    \begin{equation*}
        y_j=  \exp\bigg[\sum_{k=1}^{\dim G}t_k\widehat{A}_k \bigg]\cdot x_1
    \end{equation*}
    with $\{\widehat{A}_k\}_{k=1}^{\dim G}$ being the output of \ref{item:step3}, $\{t_k\}_{k=1}^{\dim G}$ chosen on a regularly-spaced finite subset of the Lie algebra $\langle\widehat{A}_k\rangle_{k=1}^{\dim G}$, and $x_1$ a fixed random point in $X$. 
\end{itemize}
We note that while multi-source sampling should approximately reproduce any distribution $\mu$ supported in $\man$, single-source \texttt{LieDetect} is, in principle, only useful when $\mu$ is uniform. 
Moreover, we point out that there is a lot of room to improve on the choice of $x_1$, although we have not further explored this direction here.

We compare, in Table \ref{table:experiment_density_estimation}, the three sampling paradigms for different configurations, assuming the distribution $\mu$ uniform on orbits of different Lie groups $G$, where we also vary the embedding space $\R^n$, the number of sampled points $N$ and the level of noise added to the data. 
To isolate the performance of each technique, we manually chose representation types that do not require an application of PCA in \ref{item:step1}. 
Following \texttt{LiePCA}'s original article, we use the \textit{symmetric} Hausdorff distance as a metric of quality of sampling, which we computed between the estimated samples (respectively 500, 5000 and 8000 for groups of dimension 1, 2, and 3) and the same number of ground-truth samples (chosen as evenly spaced on the orbit).
Some clear patterns can be observed, depending on the initial sample size and the presence of noise.

First, single-source \texttt{LieDetect} consistently outperforms the other methods in the regime of many points and no noise (second row of each representation), except in one case.
This is no surprise: in this setting, \texttt{LieDetect} is expected to estimate the representation type correctly, and the orbit is reconstructible from any point.
The only example where \texttt{LiePCA} obtains a better score is the representation $(3,4)$ of $\SU(2)$ in $\R^7$. 
As a matter of fact, this representation is particularly more challenging to estimate than the others. 
Indeed, we saw in Example~\ref{ex:rep_SU2_R7} that the symmetry group of the corresponding orbit has dimension 4 and not 3 (see Figure \ref{fig:repsR7_LiePCA}), which makes optimization of \ref{item:step3} more difficult, sometimes failing to identify it. 
In this case, the orbit generated by single-source \texttt{LieDetect} is not close to the underlying theoretical orbit.
Nevertheless, it is interesting to observe that, whereas multi-source \texttt{LiePCA} obtains an average distance of 0.25 with a very low variance, single-source \texttt{LieDetect} obtains a distance either around 0.2 (60\% of the time), better than \texttt{LiePCA}, or around 1, when the optimization fails.

With fewer input points, however, \texttt{LieDetect} more often fails to identify the representation, leading to a single-source-generated orbit far away from the underlying orbit (first row of each representation). Nonetheless, even in this case, \texttt{LieDetect}'s correction has some benefit over \texttt{LiePCA}, yielding better scores in the multi-source case.
It should be noted, however, that the difference between both is less significant, with a representation ($T^2$ in $\R^6$) for which multi-source \texttt{LiePCA} shows a lower Hausdorff distance.
This phenomenon is also particularly apparent in the case of many sampled points (and still no noise): multi-source \texttt{LiePCA} and multi-source \texttt{LieDetect} have similar scores, sometimes even equal.

In the presence of noise, similar patterns can be observed, depending on the Lie group.
When the group is $\SO(2)$, multi-source \texttt{LieDetect} is superior to the other methods, both with few and many initial points (third and fourth row of each representation).
This observation, however, is undermined by groups of dimension 2 and 3: in the many points regime, single-source \texttt{LieDetect} performs better; while with few points, it is observed that any of the three methods can yield a better score.
Such a phenomenon can be attributed to the `complexity' of the representations: while certain groups only admit one almost-faithful representation in $\R^n$ (such as $\SU(2)$ in $\R^5$), for which \texttt{LieDetect} is assured to output the correct representation type, others present many such representation (such as $T^2$ in $\R^6$, admitting 56 almost-faithful representations with frequencies at most $2$).
In the latter case, \texttt{LieDetect} may not estimate the correct representation, hence single-source \texttt{LieDetect} is bound to fail.

In conclusion, adding our \ref{item:step3} to the pipeline of \texttt{LiePCA}'s density estimation allows us to consistently improve the method, or at least perform as well.
An even more striking improvement is provided by single-source \texttt{LieDetect} (the method employed in our other applications), although it only succeeds if enough initial points are provided and $\mu$ is known to be uniform.

\begin{table}[ht!]\center
\begin{tabular}{|ccc|rr|ccc||}
\hline
\multicolumn{3}{||c|}{Orbit} & \multicolumn{2}{c|}{Parameters} & \multicolumn{3}{c||}{Method} \\ \hline
\multicolumn{1}{||c|}{\begin{tabular}{@{}c@{}}Lie \\ group\end{tabular}}   & \multicolumn{1}{c|}{\begin{tabular}{@{}c@{}}$\R^n$ \\  dim.\end{tabular}} & \multicolumn{1}{c|}{\begin{tabular}{@{}c@{}}Rep. \\  type\end{tabular}} & \multicolumn{1}{c|}{\begin{tabular}{@{}c@{}}Number \\ samples\end{tabular}} & \begin{tabular}{@{}c@{}}Noise\\ STD\end{tabular} & \multicolumn{1}{c|}{\begin{tabular}{@{}c@{}}Multi-source \\ \texttt{LiePCA}\end{tabular}} & \multicolumn{1}{c|}{\begin{tabular}{@{}c@{}}Multi-source \\ \texttt{LieDetect}\end{tabular}} & \multicolumn{1}{c||}{\begin{tabular}{@{}c@{}}Single-source \\ \texttt{LieDetect}\end{tabular}} \\ \hline\hline


\multicolumn{1}{||c|}{\multirow{4}{*}{$\SO(2)$}} &         \multicolumn{1}{c|}{\multirow{4}{*}{4}} &         \multirow{4}{*}{(1, 2)} & \multicolumn{1}{r|}{30} &         \multirow{2}{*}{0} & \multicolumn{1}{c|}{0.64$\pm$0.12} &         \multicolumn{1}{c|}{\textbf{0.50$\pm$0.16}} & 0.97$\pm$0.21 \\ \cline{4-4} \cline{6-8}         \multicolumn{1}{||c|}{} & \multicolumn{1}{c|}{} & &         \multicolumn{1}{r|}{100} & & \multicolumn{1}{c|}{0.10$\pm$0.04} &         \multicolumn{1}{c|}{0.08$\pm$0.02} & \textbf{0.05$\pm$0.03} \\ \cline{4-8}         \multicolumn{1}{||c|}{} & \multicolumn{1}{c|}{} & &         \multicolumn{1}{r|}{30} & \multirow{2}{*}{0.1} &         \multicolumn{1}{c|}{0.77$\pm$0.12} & \multicolumn{1}{c|}{\textbf{0.68$\pm$0.17}} &         0.99$\pm$0.24 \\ \cline{4-4} \cline{6-8}         \multicolumn{1}{||c|}{} & \multicolumn{1}{l|}{} & &         \multicolumn{1}{r|}{100} & & \multicolumn{1}{c|}{0.41$\pm$0.05} &         \multicolumn{1}{c|}{\textbf{0.38$\pm$0.04}} & 0.44$\pm$0.17 \\ \hline\hline 

\multicolumn{1}{||c|}{\multirow{4}{*}{$\SO(2)$}} &         \multicolumn{1}{c|}{\multirow{4}{*}{6}} &         \multirow{4}{*}{(1, 2, 3)} & \multicolumn{1}{r|}{30} &         \multirow{2}{*}{0} & \multicolumn{1}{c|}{0.88$\pm$0.13} &         \multicolumn{1}{c|}{\textbf{0.84$\pm$0.10}} & 1.13$\pm$0.08 \\ \cline{4-4} \cline{6-8}         \multicolumn{1}{||c|}{} & \multicolumn{1}{c|}{} & &         \multicolumn{1}{r|}{200} & & \multicolumn{1}{c|}{0.11$\pm$0.03} &         \multicolumn{1}{c|}{0.11$\pm$0.03} & \textbf{0.03$\pm$0.01} \\ \cline{4-8}         \multicolumn{1}{||c|}{} & \multicolumn{1}{c|}{} & &         \multicolumn{1}{r|}{30} & \multirow{2}{*}{0.1} &         \multicolumn{1}{c|}{1.00$\pm$0.13} & \multicolumn{1}{c|}{\textbf{0.88$\pm$0.09}} &         1.15$\pm$0.11 \\ \cline{4-4} \cline{6-8}         \multicolumn{1}{||c|}{} & \multicolumn{1}{l|}{} & &         \multicolumn{1}{r|}{200} & & \multicolumn{1}{c|}{0.53$\pm$0.05} &         \multicolumn{1}{c|}{\textbf{0.49$\pm$0.06}} & 0.90$\pm$0.18 \\ \hline\hline 

\multicolumn{1}{||c|}{\multirow{4}{*}{$\SO(2)$}} &         \multicolumn{1}{c|}{\multirow{4}{*}{6}} &         \multirow{4}{*}{(1, 3, 10)} & \multicolumn{1}{r|}{50} &         \multirow{2}{*}{0} & \multicolumn{1}{c|}{0.98$\pm$0.02} &         \multicolumn{1}{c|}{\textbf{0.89$\pm$0.06}} & 1.06$\pm$0.11 \\ \cline{4-4} \cline{6-8}         \multicolumn{1}{||c|}{} & \multicolumn{1}{c|}{} & &         \multicolumn{1}{r|}{500} & & \multicolumn{1}{c|}{0.31$\pm$0.07} &         \multicolumn{1}{c|}{0.31$\pm$0.07} & \textbf{0.06$\pm$0.11} \\ \cline{4-8}         \multicolumn{1}{||c|}{} & \multicolumn{1}{c|}{} & &         \multicolumn{1}{r|}{50} & \multirow{2}{*}{0.1} &         \multicolumn{1}{c|}{1.07$\pm$0.09} & \multicolumn{1}{c|}{\textbf{0.89$\pm$0.05}} &         1.07$\pm$0.13 \\ \cline{4-4} \cline{6-8}         \multicolumn{1}{||c|}{} & \multicolumn{1}{l|}{} & &         \multicolumn{1}{r|}{500} & & \multicolumn{1}{c|}{0.49$\pm$0.04} &         \multicolumn{1}{c|}{\textbf{0.46$\pm$0.04}} & 0.65$\pm$0.18 \\ \hline\hline 


\multicolumn{1}{||c|}{\multirow{4}{*}{$T^2$}} &         \multicolumn{1}{c|}{\multirow{4}{*}{6}} &         \multirow{4}{*}{$\begin{pmatrix}1&2&1\\2&1&1\end{pmatrix}$} & \multicolumn{1}{r|}{100} &         \multirow{2}{*}{0} & \multicolumn{1}{c|}{\textbf{0.84$\pm$0.06}} &         \multicolumn{1}{c|}{0.86$\pm$0.03} & 1.07$\pm$0.13 \\ \cline{4-4} \cline{6-8}         \multicolumn{1}{||c|}{} & \multicolumn{1}{c|}{} & &         \multicolumn{1}{r|}{1000} & & \multicolumn{1}{c|}{0.21$\pm$0.03} &         \multicolumn{1}{c|}{0.17$\pm$0.01} & \textbf{0.15$\pm$0.00} \\ \cline{4-8}         \multicolumn{1}{||c|}{} & \multicolumn{1}{c|}{} & &         \multicolumn{1}{r|}{100} & \multirow{2}{*}{0.1} &         \multicolumn{1}{c|}{1.04$\pm$0.08} & \multicolumn{1}{c|}{\textbf{0.93$\pm$0.03}} &         1.09$\pm$0.13 \\ \cline{4-4} \cline{6-8}         \multicolumn{1}{||c|}{} & \multicolumn{1}{l|}{} & &         \multicolumn{1}{r|}{1000} & & \multicolumn{1}{c|}{0.50$\pm$0.04} &         \multicolumn{1}{c|}{0.49$\pm$0.04} & \textbf{0.32$\pm$0.07} \\ \hline\hline 


\multicolumn{1}{||c|}{\multirow{4}{*}{$T^3$}} &         \multicolumn{1}{c|}{\multirow{4}{*}{8}} &         \multirow{4}{*}{\begin{tabular}{@{}c@{}} ((0, 1, -1, 0),\\(1, 0, -1, -1),\\(1, 0, -1, -1))\end{tabular}} & \multicolumn{1}{r|}{500} &         \multirow{2}{*}{0} & \multicolumn{1}{c|}{0.46$\pm$0.07} &         \multicolumn{1}{c|}{\textbf{0.31$\pm$0.03}} & 0.34$\pm$0.02 \\ \cline{4-4} \cline{6-8}         \multicolumn{1}{||c|}{} & \multicolumn{1}{c|}{} & &         \multicolumn{1}{r|}{2000} & & \multicolumn{1}{c|}{0.43$\pm$0.11} &         \multicolumn{1}{c|}{0.33$\pm$0.06} & \textbf{0.32$\pm$0.03} \\ \cline{4-8}         \multicolumn{1}{||c|}{} & \multicolumn{1}{c|}{} & &         \multicolumn{1}{r|}{500} & \multirow{2}{*}{0.1} &         \multicolumn{1}{c|}{\textbf{0.87$\pm$0.04}} & \multicolumn{1}{c|}{0.93$\pm$0.13} &         0.91$\pm$0.18 \\ \cline{4-4} \cline{6-8}         \multicolumn{1}{||c|}{} & \multicolumn{1}{l|}{} & &         \multicolumn{1}{r|}{2000} & & \multicolumn{1}{c|}{0.53$\pm$0.03} &         \multicolumn{1}{c|}{0.51$\pm$0.03} & \textbf{0.42$\pm$0.06} \\ \hline\hline 


\multicolumn{1}{||c|}{\multirow{4}{*}{$\SU(2)$}} &         \multicolumn{1}{c|}{\multirow{4}{*}{5}} &         \multirow{4}{*}{(5)} & \multicolumn{1}{r|}{500} &         \multirow{2}{*}{0} & \multicolumn{1}{c|}{0.29$\pm$0.04} &         \multicolumn{1}{c|}{0.26$\pm$0.01} & \textbf{0.22$\pm$0.02} \\ \cline{4-4} \cline{6-8}         \multicolumn{1}{||c|}{} & \multicolumn{1}{c|}{} & &         \multicolumn{1}{r|}{2000} & & \multicolumn{1}{c|}{0.23$\pm$0.01} &         \multicolumn{1}{c|}{0.23$\pm$0.01} & \textbf{0.21$\pm$0.02} \\ \cline{4-8}         \multicolumn{1}{||c|}{} & \multicolumn{1}{c|}{} & &         \multicolumn{1}{r|}{500} & \multirow{2}{*}{0.1} &         \multicolumn{1}{c|}{0.52$\pm$0.05} & \multicolumn{1}{c|}{0.41$\pm$0.02} &         \textbf{0.27$\pm$0.05} \\ \cline{4-4} \cline{6-8}         \multicolumn{1}{||c|}{} & \multicolumn{1}{l|}{} & &         \multicolumn{1}{r|}{2000} & & \multicolumn{1}{c|}{0.48$\pm$0.04} &         \multicolumn{1}{c|}{0.42$\pm$0.02} & \textbf{0.25$\pm$0.03} \\ \hline\hline 

\multicolumn{1}{||c|}{\multirow{4}{*}{$\SU(2)$}} &         \multicolumn{1}{c|}{\multirow{4}{*}{7}} &         \multirow{4}{*}{(7)} & \multicolumn{1}{r|}{500} &         \multirow{2}{*}{0} & \multicolumn{1}{c|}{0.74$\pm$0.05} &         \multicolumn{1}{c|}{\textbf{0.54$\pm$0.01}} & 0.56$\pm$0.09 \\ \cline{4-4} \cline{6-8}         \multicolumn{1}{||c|}{} & \multicolumn{1}{c|}{} & &         \multicolumn{1}{r|}{2000} & & \multicolumn{1}{c|}{0.49$\pm$0.01} &         \multicolumn{1}{c|}{0.49$\pm$0.01} & \textbf{0.40$\pm$0.02} \\ \cline{4-8}         \multicolumn{1}{||c|}{} & \multicolumn{1}{c|}{} & &         \multicolumn{1}{r|}{500} & \multirow{2}{*}{0.1} &         \multicolumn{1}{c|}{1.01$\pm$0.04} & \multicolumn{1}{c|}{\textbf{0.59$\pm$0.02}} &         0.65$\pm$0.09 \\ \cline{4-4} \cline{6-8}         \multicolumn{1}{||c|}{} & \multicolumn{1}{l|}{} & &         \multicolumn{1}{r|}{2000} & & \multicolumn{1}{c|}{0.81$\pm$0.06} &         \multicolumn{1}{c|}{0.52$\pm$0.02} & \textbf{0.45$\pm$0.04} \\ \hline\hline 

\multicolumn{1}{||c|}{\multirow{4}{*}{$\SU(2)$}} &         \multicolumn{1}{c|}{\multirow{4}{*}{7}} &         \multirow{4}{*}{(3, 4)} & \multicolumn{1}{r|}{500} &         \multirow{2}{*}{0} & \multicolumn{1}{c|}{\textbf{0.31$\pm$0.03}} &         \multicolumn{1}{c|}{0.53$\pm$0.23} & 0.69$\pm$0.41 \\ \cline{4-4} \cline{6-8}         \multicolumn{1}{||c|}{} & \multicolumn{1}{c|}{} & &         \multicolumn{1}{r|}{2000} & & \multicolumn{1}{c|}{\textbf{0.25$\pm$0.01}} &         \multicolumn{1}{c|}{0.38$\pm$0.18} & 0.51$\pm$0.39 \\ \cline{4-8}         \multicolumn{1}{||c|}{} & \multicolumn{1}{c|}{} & &         \multicolumn{1}{r|}{500} & \multirow{2}{*}{0.1} &         \multicolumn{1}{c|}{0.55$\pm$0.03} & \multicolumn{1}{c|}{0.56$\pm$0.16} &         \textbf{0.52$\pm$0.31} \\ \cline{4-4} \cline{6-8}         \multicolumn{1}{||c|}{} & \multicolumn{1}{l|}{} & &         \multicolumn{1}{r|}{2000} & & \multicolumn{1}{c|}{0.51$\pm$0.03} &         \multicolumn{1}{c|}{0.55$\pm$0.12} & \textbf{0.45$\pm$0.25} \\ \hline\hline 

\multicolumn{1}{||c|}{\multirow{4}{*}{$\SU(2)$}} &         \multicolumn{1}{c|}{\multirow{4}{*}{11}} &         \multirow{4}{*}{(4, 7)} & \multicolumn{1}{r|}{500} &         \multirow{2}{*}{0} & \multicolumn{1}{c|}{0.91$\pm$0.02} &         \multicolumn{1}{c|}{0.76$\pm$0.02} & \textbf{0.64$\pm$0.07} \\ \cline{4-4} \cline{6-8}         \multicolumn{1}{||c|}{} & \multicolumn{1}{c|}{} & &         \multicolumn{1}{r|}{5000} & & \multicolumn{1}{c|}{0.57$\pm$0.02} &         \multicolumn{1}{c|}{0.57$\pm$0.02} & \textbf{0.44$\pm$0.02} \\ \cline{4-8}         \multicolumn{1}{||c|}{} & \multicolumn{1}{c|}{} & &         \multicolumn{1}{r|}{500} & \multirow{2}{*}{0.1} &         \multicolumn{1}{c|}{1.27$\pm$0.10} & \multicolumn{1}{c|}{\textbf{0.79$\pm$0.03}} &         0.89$\pm$0.08 \\ \cline{4-4} \cline{6-8}         \multicolumn{1}{||c|}{} & \multicolumn{1}{l|}{} & &         \multicolumn{1}{r|}{5000} & & \multicolumn{1}{c|}{0.95$\pm$0.14} &         \multicolumn{1}{c|}{0.66$\pm$0.03} & \textbf{0.58$\pm$0.12} \\ \hline
\end{tabular}
\vspace{-.25cm}
\caption{Results of the experiment described in Section \ref{ex:sampling}.
The table displays the symmetric Hausdorff distance between samples generated through Equation~\eqref{eq:samples} with multi-source \texttt{LiePCA} (sixth column), multi-source \texttt{LieDetect} (seventh column), and single-source \texttt{LieDetect} (eighth column) for different Lie groups (first column) embedded in $\R^n$ (second column).
For each set of parameters and each method, several runs were considered (100 for groups of dimension 1 and 2, and 10 for dimension 3), and we report the mean distance and its standard deviation.
In each row, the best score is indicated in bold font.
The input of the methods consists of a certain number of input points (fourth column) sampled uniformly on a representation (third column), with potentially the addition of Gaussian noise (standard deviation in the fifth column).
}
\label{table:experiment_density_estimation}
\end{table}

\subsubsection{Equivariant neural networks}\label{sec: equivariant neural networks}
Equivariant networks were born from the generalization of the $\Z^2$-translational equivariance of convolutional neural networks (CNNs), proved by \cite{lecun1995convolutional}. These extensions---known as group equivariant neural networks---introduce architectures that preserve actions of a group $G$: if $X_{\text{in}}$, the vector space where the input data lives, receives a representation $\rho_{\text{in}}$ of $G$ and $X_{\text{out}}$, the network's output space, is related to $X_{\text{in}}$ through $\mathcal{F}:X_{\text{in}}\to X_{\text{out}}$, then there is also a representation of $G$ in $X_{\text{out}}$ such that
\begin{equation*}
    \mathcal{F}(\rho_{\text{in}}(g)(x))=\rho_{\text{out}}(g)[\mathcal{F}(x)],
\end{equation*}
for all $x\in X_{\text{in}}$ and $g\in G$. Whenever this equation holds, we call both the representations, $\rho_{\text{in}}$ and $\rho_{\text{out}}$, and the spaces, $X_{\text{in}}$ and $X_{\text{out}}$, equivariant. Equivariance becomes important when data possess some symmetries and the network $\mathcal{F}$ is expected to take them into account. The now-long literature on the topic \citep{cohen2016group, jenner2021steerable, cohen2016steerable, deng2021vector, weiler2018learning} presents several tasks, both artificial and on natural data, that benefit from equivariance, as well as theoretical arguments that justify their use \citep{kondor2018generalization, jenner2021steerable}.

Actual implementations of equivariant networks are divided into two groups. The first, and perhaps most natural choice, the G-CNNs, directly expands the CNN architecture to general symmetry groups and has been extended to work even with non-compact Lie groups through Monte Carlo techniques \cite{finzi2020generalizing}. Although these implementations are easier to generalize for arbitrary groups, they sometimes perform slightly worse than the other set of implementations, known as steerable CNNs. In this second approach \citep{jenner2021steerable, cohen2016steerable, weiler2018learning}, inspired by quantum field
theory, features are modeled as vector fields, that is, (smooth) fiber bundles $f_i: \R^n\to \R^{c_i}$, where $\R^n$ is the ambient space---often assumed to be $\R^2$ or $\R^3$---and $\R^c$ are the fibers, where $c_i$ are the channels' hyperparameters of the network. At each layer $i$ of the network, the steerable architectures ensure that there exists a representation $\rho_i$ of the group $G$ on $\R^{c_i}$, which defines an action $\rho^{\text{reg}}_i$ on $f_i$ by
\begin{equation*}
    [\rho^{\text{reg}}_i(g) f_i](x)=\rho_i(g)\cdot f_i(g^{-1}\cdot x)
\end{equation*}
for all $x\in \R^n$, all $g\in G$, and where $g^{-1}\cdot x$ is some fixed $G$-representation on $\R^n$. Notice that two representations are showing up in the equation: $\rho_i$, which acts on the finite-dimensional vector space $\R^{c_i}$, and the representation $\rho^{\text{reg}}_i$, known as the corresponding regular representation, which acts on the infinite-dimensional vector space of all $\R^{c_i}$-fields over $\R^n$, $C^\infty(\R^n, \R^{c_i})$. In practice, because the ambient space $\R^n$ is often discretized---e.g., with pixels when $n=2$ or voxels when $n=3$---$\rho^{\text{reg}}_i$ becomes a representation on the finite vector space $\R^{c_i\times d}$, where $d$ is the number of discrete components. Finally, the equivariance of the whole network is guaranteed through the equivariance of any two consecutive representations, $\rho_i^{\text{reg}}$ and $\rho_{i+1}^{\text{reg}}$. In applications, some non-linear functions are also included in-between layers, but these are shown not to spoil the general equivariance property of the full network \citep{jenner2021steerable}.

Several computations have been proposed to check the equivariance of steerable networks (see, for example, \cite{e2cnn} and \cite{e3nn_paper, e3nn}) and here, we propose using \texttt{LieDetect} to check $\SO(2)$ equivariance of architectures that perform well in the rotated MNIST problem. Rotated MNIST corresponds to an augmentation of the famous MNIST
dataset, in which the handwritten digits are also rotated by their centers. As in vanilla MNIST, the task consists of predicting the digit from its image, but, because of the extra rotations, the problem is significantly harder. In \cite{e2cnn}, the authors propose using a steerable architecture for this task, which, even though implements equivariance with respect to the cardinality 8 finite rotational group (i.e., to rotations only up to $\pi/4$), achieves competitive performance for this task---about 98 \% accuracy in only 30 epochs of training. In principle, however, the choice of a finite group of rotations might imply that the network is not equivariant under the \textit{full} symmetry group $\SO(2)$, a property that would be impossible to verify using just the tools available by the \texttt{e2cnn} package \citep{e2cnn}. Moreover, because the implementation implies an application of all rotations to some square grids (known as filters), some numerical error due to interpolation might arise whenever rotations by angles that are not multiple of $\pi/2$ are applied, which means that the network may not even be exactly equivariant for the finite group $R_8$. To assess the question of whether the finite-group equivariant architectures are also equivariant under $\SO(2)$ translations, we retrain the same architecture for different cardinalities of finite groups of rotation, $R_n$, and compare each layer's representation $\rho^{\text{reg}}_i$ to representations of $\SO(2)$. All the implementations consisted of 6-steerable block networks, with some non-linearities between layers, with each block corresponding to direct sums of regular representations. Notice that even the inputs can be understood as $G$-steerable layers when seen as features $f:\R^2\to\R$, with respect to the trivial representation of $G$. Therefore, at each layer $i$, if the network has no errors, we can assume that $\rho_i^{\text{reg}}$ is a representation of $G$ on $\R^{c_i\times d_i}$, where $d_i$ is the number of grids of the output of $i$, and $i=0,\dots,6$.

With the networks trained, we evaluate them on a dataset consisting of the same image (the digit 6), rotated about its center by 500 equally spaced angles from $0$ to $2\pi$. If the networks were exactly $\SO(2)$-equivariant, the output of each of their layers evaluated at the dataset would be samples from the orbits of $\rho_i^{\text{reg}}$. We then apply \texttt{LieDetect} to each of these samples to reconstruct $\SO(2)$ representations. Figure \ref{fig: nn} shows the non-symmetric Hausdorff distance from the reconstructed orbits and the samples for the 6 different layers and each equivariant model. In general, we notice that the bigger the group's cardinality, the more likely the representations are to be representations of $\SO(2)$---which is expected, as larger $n$ should imply a better approximation of the full group. Implementations of $R_4$ equivariance, however, seem to better preserve the $\SO(2)$ actions---at least in the first layers---than implementations of $R_8$ equivariance. This behavior is most likely caused by the fact that, in the former's implementation, no interpolation of the filters is necessary, as only rotations by multiples of $\pi/2$ are considered. Nevertheless, for all cases other than $R_2$, a good approximation of $\SO(2)$-equivariance is observed, which helps to justify the success of finite cardinality rotation groups' architectures in the rotated MNIST task. Finally, we notice that the deviation from $\SO(2)$ representations tends to increase at later layers, which might be caused by the accumulation of numerical error due to interpolation. Moreover, we could not satisfactorily explain the reason for the peak in error at the second layer, a trend that was common among all cardinalities.

\begin{figure}[ht]
\centering
\includegraphics[width=0.8\textwidth]
{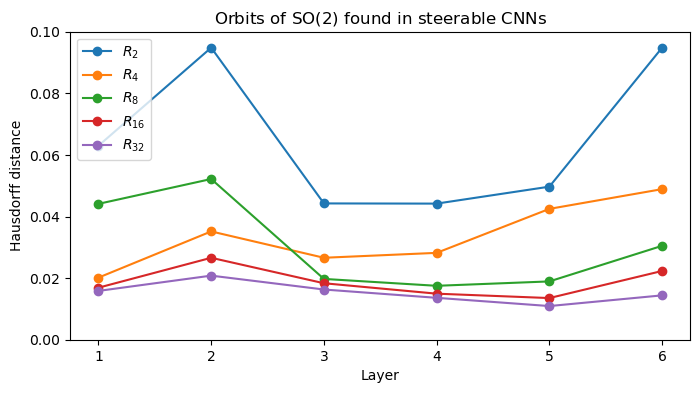}
\caption{Non-symmetric Hausdorff distance between the 500 samples used for estimation and the orbits of $\SO(2)$ representations estimated with \texttt{LieDetect} at each layer of the common $R_n$ -steerable architecture for different cardinalities of the finite rotation group $R_n$.}\label{fig: nn}
\centering
\end{figure}

\subsection{Physics}
In this section, we will exemplify how \texttt{LieDetect} may be employed in two different settings in Physics: the study of conformational spaces of molecules and the search for analytic solutions of orbits of the three-body problem.
We point out that, although intuition guided us to consider these datasets, we had no clear guarantee of the algorithm's success. The approximate linear orbits found in the cyclooctane dataset and certain three-body orbits were somewhat surprising.

\subsubsection{Conformational space of cyclooctane}\label{subsec:conformational_space}
In the study of molecular structures, the term \textit{conformation} refers to the different spatial arrangements of atoms within a molecule that can be achieved through rotation around single bonds. 
These conformations can vary due to the flexibility of these rotations, leading to different shapes and spatial orientations of the molecule.
The collection of all possible conformations a molecule can adopt constitutes its \textit{conformational space}. 
An interesting example is that of the cyclooctane molecule ($\mathrm{C}_8\mathrm{H}_{16}$): it was found by Martin \emph{et al.}~\cite{brown2008algorithmic,martin2010topology} that it has conformational space formed by the Klein bottle and a 2-sphere, intersecting in two circles.
This result has inspired the development of new Topological Data Analysis methods to analyze the conformational space of other molecules \cite{membrillo2019topology,steinberg2019topological} and to aid visualization and dimension reduction \cite{stolz2020geometric,scoccola2022fibered,lim2023hades}.

To go further in the exploration of the conformational space of cyclooctane, we investigate, in this section, whether it admits a Lie group action.
To this end, we will apply \texttt{LieDetect} on a collection of conformers generated with the popular library \texttt{RDkit} \cite{landrum2013rdkit}, used in some of the articles already cited \cite{membrillo2019topology,scoccola2022fibered}.
It is important to note that the study of conformation space depends on the metric we give it. 
If we see cyclooctane conformations as points in $\R^{72}$ (24 atoms times 3 coordinates) or $\R^{24}$ (8 atoms of carbon times 3 coordinates), then we can compare them with the Euclidean distance. 
However, another distance is commonly used: the root-mean-square deviation (RMSD), obtained by first aligning the molecules through translations, rotations, and permutations of
the atoms \cite{sadeghi2013metrics}.
In our case, we will use the first construction, since our algorithm searches for representation orbits immersed in some ambient Euclidean space.
Therefore, aligning or not the conformers at the start has an important influence on the results, as it will become clear in the following examples.

\paragraph{Conformers not aligned}
For our first experiment, we generate 10,000 conformers of cyclooctane, without aligning them.
More precisely, we generate 50,000 of them with \texttt{RDkit}, from which we remove the hydrogen atoms, to consider only the 8 carbon atoms (as is commonly done), which we then represent in $\R^{24}$ (8 atoms times 3 coordinates).
Subsequently, we select a subset of 10,000 conformers from this point cloud, starting with a random conformer, and iteratively adding the furthest point.
The resulting point cloud, visualized in Figure \ref{fig:cyclooctane_1}, seems to be made up of two parts: one tubular and one that goes all the way around. 
We expect to find a structure analogous to the union of a Klein bottle and a sphere.

Next, to apply dimension reduction, we inspect the first eigenvalues of the covariance matrix:
\begin{center}
0.474, ~0.454, ~0.012, ~0.012, ~0.008, ~0.007, ~0.006, ~0.006, ~0.006, ~0.002.
\end{center}
We choose to project the point cloud in $\R^4$ via PCA.
We then orthonormalize it, via \ref{item:step1} of our algorithm.
However, contrary to what is expected, the points, after orthonormalization, do not all have norm 1. This indicates that the whole dataset does not lie within a single representation orbit of a compact Lie group, which corroborates the expectation that it is composed of two orbits. Thus, in addition, we project them onto the unit sphere.
The resulting point cloud is visualized in Figure \ref{fig:cyclooctane_2}.
We point out that this projection is not necessary for the algorithm to work properly. Apart from slightly improving the final score, it mainly results in a more readable figure, reminiscent of that found in \cite{martin2010topology} (albeit, here, in the case of non-aligned conformers).

As seen on Figure \ref{fig:cyclooctane_2}, the point cloud exhibits even more clearly the tubular and circular components.
We extract them through the following simple procedure: the 15\% most anomalous points are discarded (using the Local Outlier Factor algorithm implemented in \texttt{sklearn}), and the remaining set is clustered into two via single-linkage clustering.
We focus on the `toroidal component' (red on the figure), made of 7520 points.

\texttt{LiePCA}, computed in \ref{item:step2}, shows two small eigenvalues, suggesting an action of the 2-torus (see Figure \ref{fig:cyclooctane_2}).
We eventually apply \ref{item:step3} and \ref{item:step4} with the torus group $T^2$.
The output is an orbit $\widehat{\O}_x$ of $T^2$, generated by an initial point $x$ of the point cloud.
The Hausdorff distances are relatively small: $\HDmid{X}{\widehat{\O}_x} \approx 0.2047$ and $\HDmid{\widehat{\O}_x}{X} \approx 0.1047$.
We conclude that the data is well approximated by a representation of $T^2$.
We point out that, unlike all the other experiments in this article, and because the data studied here are somewhat noisy, the initial point $x$ has not been arbitrarily chosen as the first point of $X$. Instead, all points of $X$ were tested, and the one giving the best result was chosen.

\begin{figure}[ht]
\begin{subfigure}[t]{.49\textwidth}\center
\includegraphics[width=.85\linewidth]{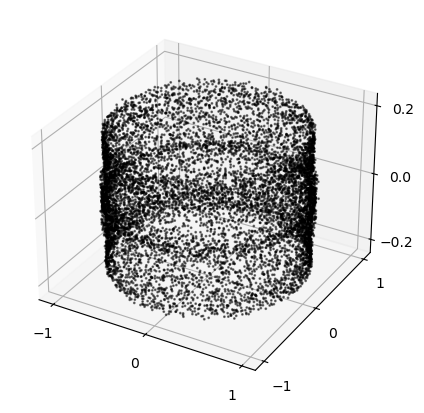}
\caption{A set of 10,000 cyclooctane conformers is generated and embedded in $\R^{24}$. 
Projected in dimension 3, one sees a cylinder surrounded by a circle.}
\label{fig:cyclooctane_1}
\end{subfigure}
\hfill
\begin{subfigure}[t]{.49\textwidth}\center
\includegraphics[width=.85\linewidth]{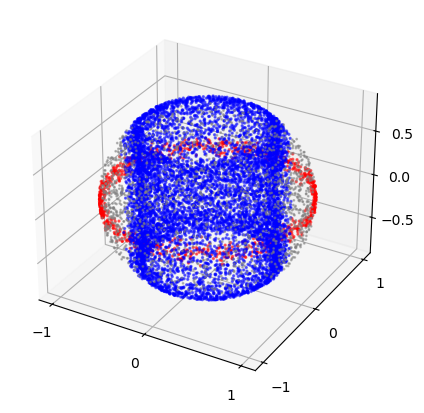}
\caption{The point cloud is projected in $\R^4$ and orthonormalized. After discarding 15\% outliers (in grey), two clusters appear clearly. We keep the red one.}
\label{fig:cyclooctane_2}
\end{subfigure}
\begin{subfigure}[t]{.49\textwidth}\center
\includegraphics[width=.7\linewidth]{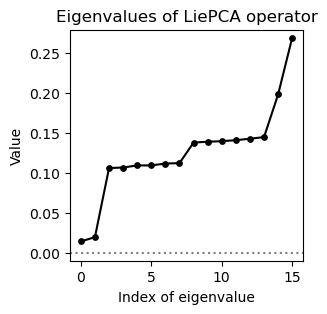}
\caption{\texttt{LiePCA} exhibits two significantly small eigenvalues, suggesting a symmetry group of dimension 2, and an action of the torus $T^2$.}
\label{fig:cyclooctane_3}
\end{subfigure}
\hfill
\begin{subfigure}[t]{.49\textwidth}\center
\includegraphics[width=.85\linewidth]{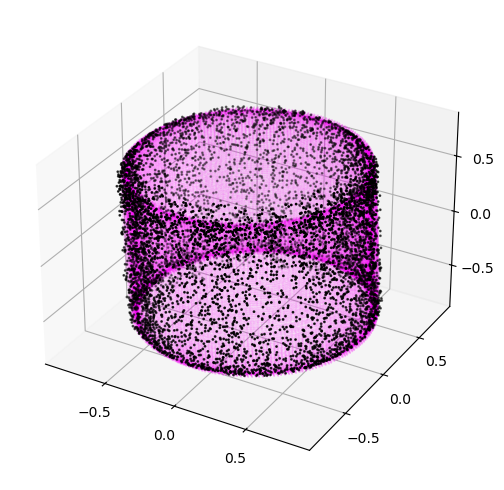}
\caption{\texttt{LieDetect} finds an orbit of a representation of $T^2$ in $\R^4$ that fits the data, at Hausdorff distance $\HDmid{X}{\widehat{\O}_x} \approx 0.2047$ and $\HDmid{\widehat{\O}_x}{X} \approx 0.1047$.}
\label{fig:cyclooctane_4}
\end{subfigure}
\caption{A component of the conformational space of cyclooctane is close to a linear orbit of the torus. In this experiment, we do not align the conformers.}
\end{figure}

\paragraph{Conformers aligned}
In the previous experiment, the conformers were not aligned with a reference conformer, potentially leading to an inaccurate conformation space.
As a matter of fact, the torus we observed does not coincide with the results of the literature, which report a Klein bottle.
To represent the molecule more accurately, we repeat the experiment, now aligning the conformers after they are generated (through \texttt{AlignMolConformers} of \texttt{RDkit}).
The resulting point cloud $X\subset\R^{24}$ is visualized on Figure \ref{fig:cyclooctane_1_align}.
In contrast to the previous dataset, the point cloud already exhibits a well-separated structure: one strip and two accumulation points.

Just as before, we extract the `strip component' by discarding the 10\% most anomalous points and classifying the resulting subset into three classes with single linkage clustering.
We keep the largest cluster, shown in red in Figure \ref{fig:cyclooctane_2_align}, that we project into $\R^4$ and orthonormalize.

The \texttt{LiePCA} operator now exhibits only one small eigenvalue (see Figure \ref{fig:cyclooctane_3_align}). 
This prompts us to apply our algorithm with the group $\SO(2)$.
\texttt{LieDetect} finds an orbit $\widehat{\O}_x$, but it does not accurately describe the data: the non-symmetric Hausdorff distance is $\HDmid{X}{\widehat{\O}_x} \approx 1.0034$. 
On the other hand, the distance from the estimated orbit to the data is small: $\HDmid{\widehat{\O}_x}{X} \approx 0.2065$.
This suggests that we have found a non-transitive action of $\SO(2)$ on $X$ (see Figure \ref{fig:cyclooctane_4_align}).

Indeed, the estimated representation of $\SO(2)$ yields a family of 1-dimensional orbits $\widehat{\O}_x$, $x\in X$, that fit the point cloud well.
Namely, the distance $\HDmid{\widehat{\O}_x}{X}$ has mean value $0.1225$, standard deviation $0.0419$, and attains the maximal value $0.3277$.
The fact that the point cloud supports a non-transitive action of $\SO(2)$ is consistent with the topology of the Klein bottle.
We point out that a similar case has been encountered in Example \ref{ex: Mobius}, with the Möbius strip.

There are two conclusions to be drawn from our analysis. Firstly, by a rather direct application of \texttt{LieDetect}, we obtain a structure compatible with the Klein bottle, as reported in the literature.
We also observe, in the unaligned data, the structure of a torus, a space that gives, by quotient, the Klein bottle.
Secondly, and perhaps more remarkably, this geometry is obtained from the initial data in $\R^{24}$ only using linear transformations. This contrasts with all the articles cited above, where non-linear dimension reduction algorithms are employed. 

\begin{figure}[H]
\begin{subfigure}[t]{.49\textwidth}\center
\includegraphics[width=.85\linewidth]{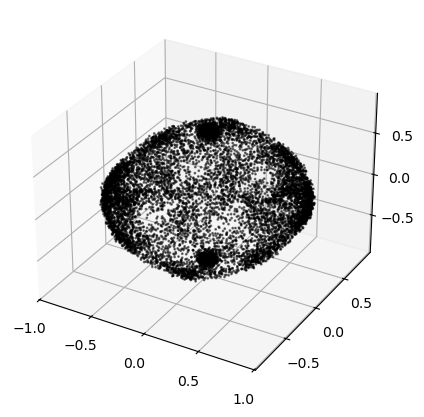}
\caption{A set of 10,000 cyclooctane conformers is generated, aligned, and embedded in $\R^{24}$. One identifies three components: a surface and two clusters.}
\label{fig:cyclooctane_1_align}
\end{subfigure}
\hfill
\begin{subfigure}[t]{.49\textwidth}\center
\includegraphics[width=.85\linewidth]{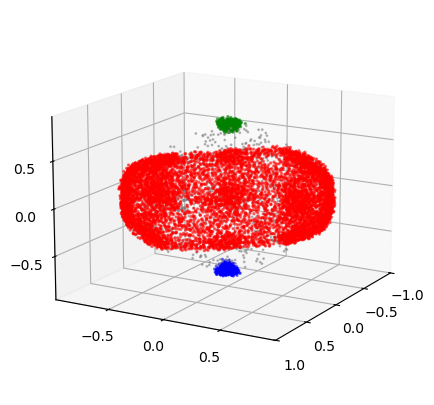}
\caption{10\% of the data is discarded (grey), and the remaining points are clustered in three classes via single linkage clustering. We keep the red cluster.}
\label{fig:cyclooctane_2_align}
\end{subfigure}
\begin{subfigure}[t]{.49\textwidth}\center
\includegraphics[width=.7\linewidth]{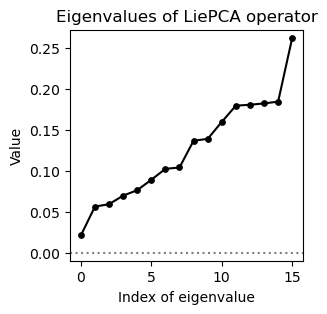}
\caption{The point cloud is projected into $\R^4$ and orthonormalized. \texttt{LiePCA} exhibits one significantly small eigenvalue, suggesting an action of $\SO(2)$.}
\label{fig:cyclooctane_3_align}
\end{subfigure}
\hfill
\begin{subfigure}[t]{.49\textwidth}\center
\includegraphics[width=.85\linewidth]{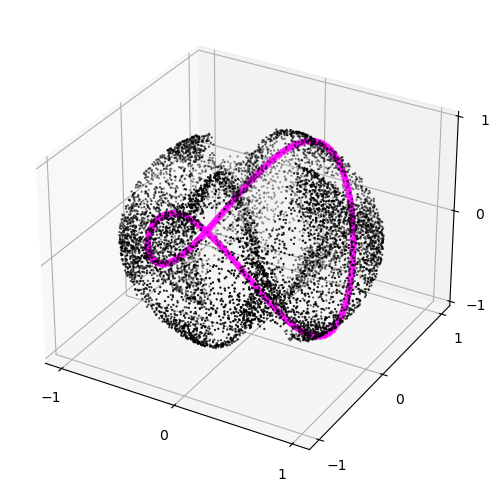}
\caption{\texttt{LieDetect} finds an action of $\SO(2)$ that stabilizes the point cloud. The mean value of $\HDmid{\widehat{\O}_x}{X}$, where $x$ ranges over the point cloud, is $0.1225$}
\label{fig:cyclooctane_4_align}
\end{subfigure}
\caption{A component of the conformational space of cyclooctane admits a linear action of $\SO(2)$. In this experiment, we align the conformers.}
\end{figure}

\subsubsection{Approximate closed-form solutions to the three-body problem}\label{subsec:3body}
The problem of describing the orbits of three massive objects, subjected only to the influence of their mutual gravitational force, has interested physicists for both its practical and theoretical particularities \cite{musielak2014three}. Different from its two-body counterpart, in which orbits were fully described from the beginning of Newtonian mechanics \citep{Newton_2016}, closed-form (analytic) solutions to the three-body problem were only discovered by Euler in 1740 and Lagrange in 1772 with some very constrained parameter choices \citep{liao2022three}. The difficulty in finding a general closed-form solution was transformed into an impossibility thanks to the seminal work of Henri Poincaré \citep{poincare1889problem}, who proved that the associated dynamical system is non-integrable, with most solutions not periodic. Even though, Poincaré's result does not rule out the possibility of \textit{specific} closed-form solutions under some configurations, but it makes the task of finding them especially complicated. Although several periodic solutions were found in the last decades, mostly through numerical techniques \citep{broucke1975periodic,broucke1975relative,broucke1979isosceles,moore1993braids,vsuvakov2013three,liao2022three}, no analytic orbits were identified.

Back to Lie groups, we notice that if $z(t)$ is the vector in $\R^{6}$ which registers the position of each of the bodies in the plane, and if the solution is periodic, then $z(t)$ is within an orbit of a $\SO(2)$ \textit{action} in $\R^{6}$, even if it has no analytic solution. If, however, $z(t)$, for some reason, lies within an orbit of a \textit{representation} of $\SO(2)$, then it admits an analytic form, which can be easily derived from the knowledge of the representation type. 
More precisely, it is a linear combination of cosines of at most three distinct frequencies.
For this, we propose using \texttt{LieDetect} to identify whether $z(t)$ lies within representation orbits of $\SO(2)$. Of course, because of potential uncertainties in the algorithm, this should not be seen as a way of proving the existence of alternative closed-form solutions, but only an indication that some of the well-known periodic dynamics are \textit{close} to analytic solutions, which might be useful in demonstrating some theoretical results.

We will be focusing on the periodic orbits discovered by Broucke in 1975 \citep{broucke1975periodic} for three bodies of equal mass in the plane, and whose (approximate) initial conditions are found on the website\footnote{The \text{Three-Body Gallery} gathers the initial conditions of several periodic orbits \url{http://three-body.ipb.ac.rs/}. In particular, Broucke's orbits are found at \url{http://three-body.ipb.ac.rs/broucke.php}.} of the authors of \citep{vsuvakov2013three}, within ten decimal places (\url{http://three-body.ipb.ac.rs/}).
Using \texttt{scipy}'s integration of ordinary differential equations---namely, \texttt{scipy.integrate.solve\_ivp} with integration method \texttt{Radau} (a Runge-Kutta method)---, we generated the orbits of nine systems, represented on Table \ref{tab:threebody} (second column).
The bodies' positions are concatenated to form a subset $X$ of $10,000$ points in $\R^6$ (third column).
In each case, we observed, through PCA, that $X$ is (approximately) contained in a subspace of dimension $4$, on which we projected it.   
Next, we applied \texttt{LieDetect}, performing first the orthonormalization step, and then estimating the representation type through \ref{item:step3}' for $\SO(2)$ (described in Section \ref{sec: algorithm so(2) simp}).
The optimal type is given in the fourth column of Table \ref{tab:threebody}.
In the verification step, we not only calculate the Hausdorff distance between $X$ and the estimated orbit (in magenta, in the last column), but also go back to the plane (reversing the orthonormalization step), and compute, for each body, the Hausdorff distance between its initial and its estimated orbit (in green, red and blue).
In contrast to the rest of the article, we here calculate the \textit{symmetrical} Hausdorff distance, to more precisely observe the quality of our estimation.
We note that the typical distance between two consecutive points of $X$ (after orthonormalization) is of the order of $0.01$ (value obtained numerically), thus a Hausdorff distance of this magnitude indicates an excellent estimate.

In Table \ref{tab:threebody}, two systems stand out: A2 and A3, exhibiting very small Hausdorff distances: $0.01176$ and $0.01323$.
One also observes, visually, a convincing fit.
This is an argument in favor of the existence of approximate analytic formulae for these orbits, as linear combinations of $\cos(t)$ and $\cos(6t)$ for A2, and $\cos(t)$ and $\cos(5t)$ for A3, as suggested by the estimated representation type $(1,6)$ and $(1,5)$.
Knowing that these orbits are notorious for not admitting a simple formula, this result is rather surprising.
Regarding the other systems, one observes that, while some can be ruled out easily (A1, A7, and R1), the others (A11, R2, R9, and R12) appear, visually, to be well approximated by the estimated orbit.
To extend our study and find out whether the discrepancy is only due to numerical errors or not, it would be necessary to consider more efficient integration methods and to increase the number of points generated.

\clearpage
\begin{longtable}{||ccccl||}
\hline
Name & \begin{tabular}{@{}c@{}}Orbits of the\\three bodies in $\R^2$\end{tabular} & \begin{tabular}{@{}c@{}}Orbit of the\\embedding in $\R^6$\end{tabular} & \begin{tabular}{@{}c@{}}\texttt{LieDetect}\\output\end{tabular} & \begin{tabular}{@{}c@{}}Hausdorff\\distances\end{tabular}\\*\hline
A1 & \begin{minipage}{.3\linewidth}\vspace{1mm}\includegraphics[width=1\linewidth]{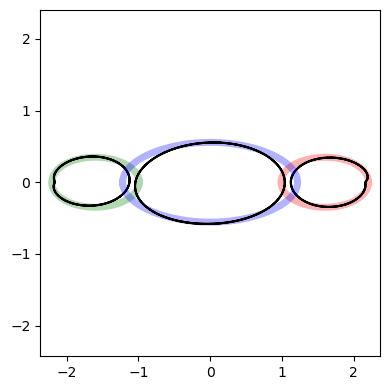}\vspace{1mm}\end{minipage} & \begin{minipage}{.3\linewidth}\vspace{1mm}\includegraphics[width=1\linewidth]{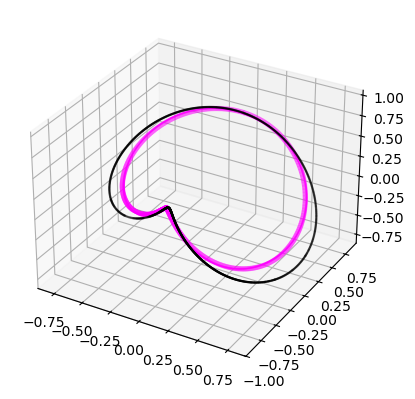}\vspace{1mm}\end{minipage} & $(1, 2)$ & \begin{tabular}{@{}c@{}}\textcolor{dartmouthgreen}{0.13668} \\ \textcolor{red}{0.13681}\\ \textcolor{blue}{0.17907} \\ \textcolor{deepmagenta}{0.22507}\end{tabular}\\\hline
A2 & \begin{minipage}{.3\linewidth}\vspace{1mm}\includegraphics[width=1\linewidth]{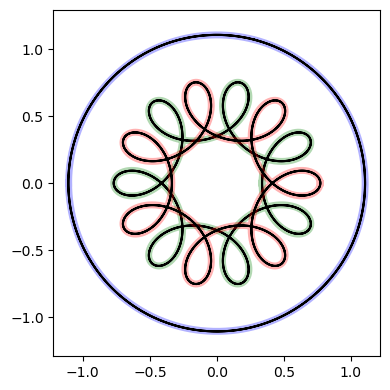}\vspace{1mm}\end{minipage} & \begin{minipage}{.3\linewidth}\vspace{1mm}\includegraphics[width=1\linewidth]{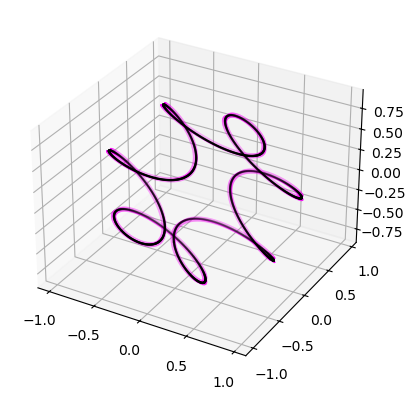}\vspace{1mm}\end{minipage} & $(1, 6)$ & \begin{tabular}{@{}c@{}}\textcolor{dartmouthgreen}{0.00405} \\ \textcolor{red}{0.00404}\\ \textcolor{blue}{0.00327} \\ \textcolor{deepmagenta}{0.01176}\end{tabular}\\\hline
A3 & \begin{minipage}{.3\linewidth}\vspace{1mm}\includegraphics[width=1\linewidth]{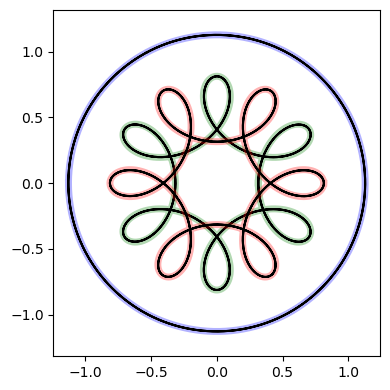}\vspace{1mm}\end{minipage} & \begin{minipage}{.3\linewidth}\vspace{1mm}\includegraphics[width=1\linewidth]{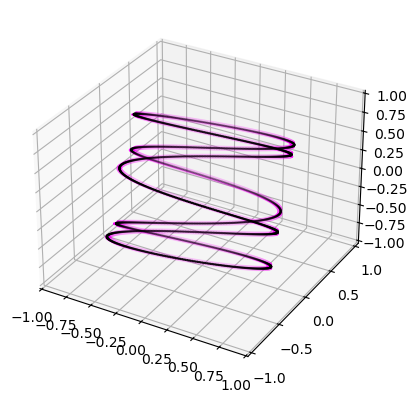}\vspace{1mm}\end{minipage} & $(1, 5)$ & \begin{tabular}{@{}c@{}}\textcolor{dartmouthgreen}{0.00515} \\ \textcolor{red}{0.00518}\\ \textcolor{blue}{0.00273} \\ \textcolor{deepmagenta}{0.01323}\end{tabular}\\\hline
A7 & \begin{minipage}{.3\linewidth}\vspace{1mm}\includegraphics[width=1\linewidth]{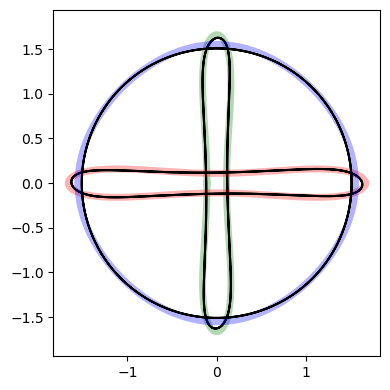}\vspace{1mm}\end{minipage} & \begin{minipage}{.3\linewidth}\vspace{1mm}\includegraphics[width=1\linewidth]{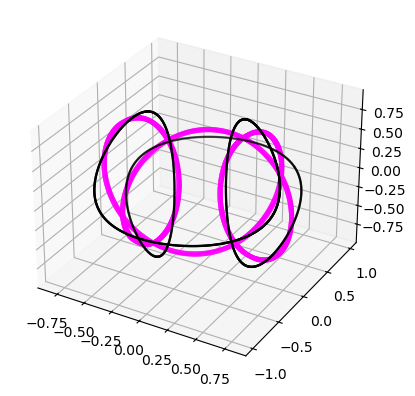}\vspace{1mm}\end{minipage} & $(1, 3)$ & \begin{tabular}{@{}c@{}}\textcolor{dartmouthgreen}{0.03642} \\ \textcolor{red}{0.03532}\\ \textcolor{blue}{0.04316} \\ \textcolor{deepmagenta}{0.1767}\end{tabular}\\\hline
A11 & \begin{minipage}{.3\linewidth}\vspace{1mm}\includegraphics[width=1\linewidth]{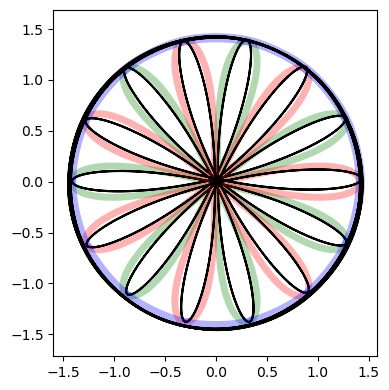}\vspace{1mm}\end{minipage} & \begin{minipage}{.3\linewidth}\vspace{1mm}\includegraphics[width=1\linewidth]{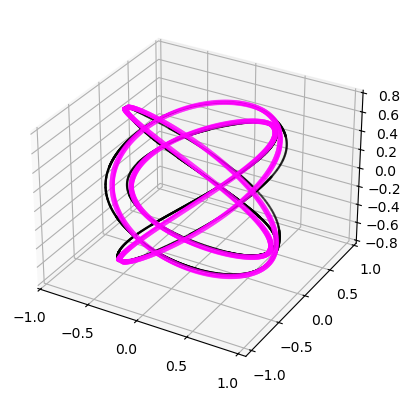}\vspace{1mm}\end{minipage} & $(3, 4)$ & \begin{tabular}{@{}c@{}}\textcolor{dartmouthgreen}{0.06831} \\ \textcolor{red}{0.07886}\\ \textcolor{blue}{0.04532} \\ \textcolor{deepmagenta}{0.06893}\end{tabular}\\\hline
R1 & \begin{minipage}{.3\linewidth}\vspace{1mm}\includegraphics[width=1\linewidth]{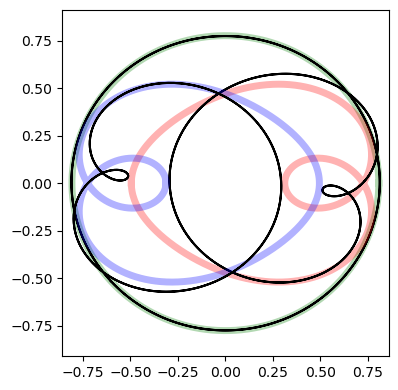}\vspace{1mm}\end{minipage} & \begin{minipage}{.3\linewidth}\vspace{1mm}\includegraphics[width=1\linewidth]{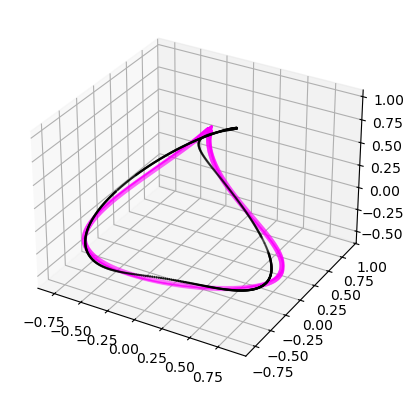}\vspace{1mm}\end{minipage} & $(1, 2)$ & \begin{tabular}{@{}c@{}}\textcolor{dartmouthgreen}{0.00696} \\ \textcolor{red}{0.20459}\\ \textcolor{blue}{0.20541} \\ \textcolor{deepmagenta}{0.45210}\end{tabular}\\\hline
R2 & \begin{minipage}{.3\linewidth}\vspace{1mm}\includegraphics[width=1\linewidth]{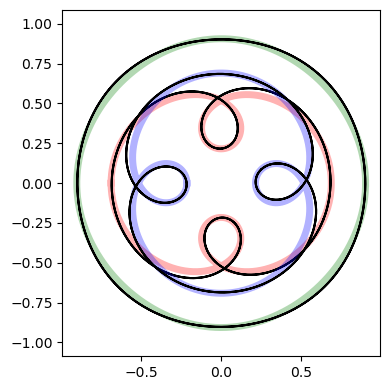}\vspace{1mm}\end{minipage} & \begin{minipage}{.3\linewidth}\vspace{1mm}\includegraphics[width=1\linewidth]{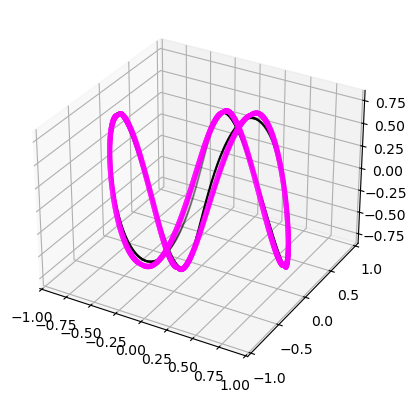}\vspace{1mm}\end{minipage} & $(1, 3)$ & \begin{tabular}{@{}c@{}}\textcolor{dartmouthgreen}{0.02645} \\ \textcolor{red}{0.04213}\\ \textcolor{blue}{0.04209} \\ \textcolor{deepmagenta}{0.05618}\end{tabular}\\\hline
R9 & \begin{minipage}{.3\linewidth}\vspace{1mm}\includegraphics[width=1\linewidth]{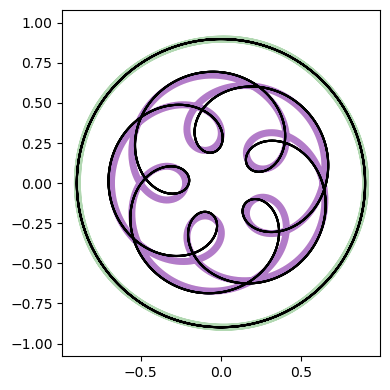}\vspace{1mm}\end{minipage} & \begin{minipage}{.3\linewidth}\vspace{1mm}\includegraphics[width=1\linewidth]{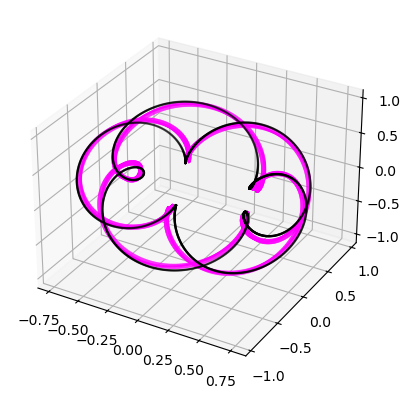}\vspace{1mm}\end{minipage} & $(2, 7)$ & \begin{tabular}{@{}c@{}}\textcolor{dartmouthgreen}{0.00489} \\ \textcolor{red}{0.05039}\\ \textcolor{blue}{0.05038} \\ \textcolor{deepmagenta}{0.07829}\end{tabular}\\\hline
R12 & \begin{minipage}{.3\linewidth}\vspace{1mm}\includegraphics[width=1\linewidth]{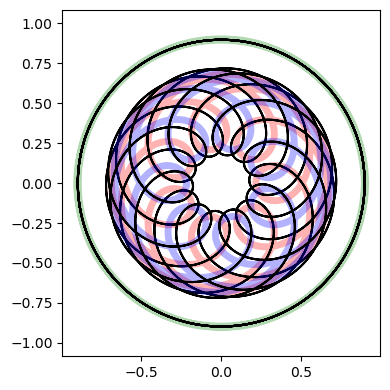}\vspace{1mm}\end{minipage} & \begin{minipage}{.3\linewidth}\vspace{1mm}\includegraphics[width=1\linewidth]{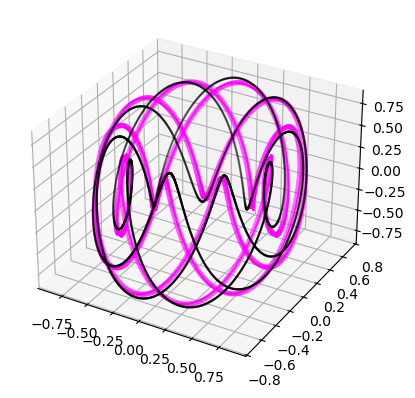}\vspace{1mm}\end{minipage} & $(3, 10)$ & \begin{tabular}{@{}c@{}}\textcolor{dartmouthgreen}{0.00639} \\ \textcolor{red}{0.05133}\\ \textcolor{blue}{0.05137} \\ \textcolor{deepmagenta}{0.10782}\end{tabular}\\\hline
\caption{
Nine periodic orbits of the three-body problem are considered.
\textbf{First column:} name of the orbit following \citep{vsuvakov2013three}.
\textbf{Second column:} positions of the three bodies (green, red, and blue) during a whole period. In addition, the black curves represent the orbits estimated by our method.
\textbf{Third column:} embedding of the three bodies' positions in $\R^6$ (in magenta), projected in dimension 3 through PCA for visualization purposes, and orbit estimated by our method (black).
\textbf{Fourth column:} estimated representation type.
\textbf{Fifth column:} symmetric Hausdorff distance between the first body and the first two coordinates of the estimated orbit (green), between the second body and the second two coordinates of the estimated orbit (red), between the third body and the third two coordinates of the estimated orbit (blue), and between the embedding in $\R^6$ and the whole estimated orbit.
}
\label{tab:threebody}
\end{longtable}

\section{Conclusion}
We presented in Algorithm \ref{alg: 1} what we believe to be the first method in the literature to identify, in an unsupervised fashion, 
Lie group representation orbits in point clouds. 
Moreover, because we estimate the equivalence class of the representation in the process, the algorithm can be used to reconstruct the orbit, eventually allowing us to determine which compact Lie group generate it. 
There are several constraints assumed in the application of the algorithm that we have made explicit in Theorem \ref{th:robustness_algorithm} and which we aim to sequentially loosen in future works.

The first line of expansion consists of lifting the \textit{representation}  constraint. 
This does not only mean expanding Algorithm \ref{alg: 1} to non-linear actions on point clouds, although it definitely is a direction to further explore, but consists in considering more inventive maps from Lie groups. 
For example, there is in Physics literature a genuine interest in projective representations, which are maps $\phi: G\to \GL(m, \C)/\C^*I$, where $\C^*I$ is the set of nonzero multiples of the identity. These non-conventional representations are the most well-known objects used in the traditional theory of classification of $\Spin(n)$ complex representations, giving rise to the definition of the geometric objects known as `spinors'. Similar derivations of spinors can be achieved by means of Clifford algebras, as briefly introduced in \cite{belinfante1989survey}. Besides, infinite-dimensional representations, because of their natural translation to vector field theory and their important role in quantum mechanics, are also worth exploring.

Expanding the methods to more general groups also deserves to be investigated. Starting with finite groups, these have already received a fair share of literary interest in machine learning.
Many, if not all, of the properties used in deriving Algorithm \ref{alg: 1}, such as orthogonality relations for matrix coefficients and classification of irreps, also work for this class of groups. For more general Lie groups, though, some of the properties here assumed fail to hold. In fact, although some of the non-compact cases such as Poincaré and Lorentz's groups have their representations well-studied for decades, this does not mean that they are easy to work with. For both of these cases, for example, no non-trivial finite-dimensional representation is unitary \cite{woit2017quantum}, which deeply impacts our computations. On the other hand, given a general manifold $\man$ embedded in some $\R^n$, the description of its symmetry group, $\Sym(\man)$, may easily skip the compact case, justifying the interest in these less-well-behaved scenarios.
In this scenario, an important tool would be Weyl's unitarian trick, which allows us to study representations of certain non-compact Lie groups from compact ones.

In addition, the methods explored here seem to present a fruitful field of applications to machine learning. As reviewed in Section \ref{sec:introduction}, most existing architectures, including equivariant neural networks, do not require precise knowledge of representation types nor, sometimes, of the acting Lie group to properly work.
We believe, on the other hand, that this important knowledge, which can be attained through Algorithm \ref{alg: 1}, may be used to suggest different inference architectures, which have the potential to overcome many of the challenges faced by current Lie group-aware techniques. Notice that these challenges do not necessarily restrict themselves to improvements in prediction capacity: in the many distinct areas discussed here where Lie groups play a significant role---i.e., image processing, harmonic analysis, and classical mechanics---the existing techniques have already shown significant success, and the extra knowledge of the representation will probably signify better accuracy. 
Yet, there are many fields, such as particle physics, where the knowledge of the precise representation type becomes crucial, and machine learning algorithms that incorporate the precise representation equivalence class in their architecture may prove to be the correct form of addressing some inference tasks.

Finally, the data presented in this article offer a fruitful avenue for statistical analysis, which we reserve for future work. Although the key geometric quantities in Theorem \ref{th:robustness_algorithm}---for example, the reach of the orbit or its covariance---are identified explicitly, they remain unknown to the user. A deeper investigation of these quantities would make it possible to develop rigorous statistical tests of the linear orbit hypothesis, and Lie-action-aware dimension reduction.

\paragraph{Acknowledgements.}
The original work behind this article was developed for HE's master's thesis, supervised by RT. We are mostly in debt to César Camacho, who was HE's co-advisor, as well as the members of the thesis jury, Clément Maria, Eduardo Mendes, and Jameson Cahill, not only for agreeing to evaluate the original work but also for many valuable inputs. 
Finally, we are indebted to the anonymous reviewers for their important feedback and suggestions. 

\appendix
\section{Additional comments}
\subsection{Notations}\label{sec:notations}
Notation conventions are summarized in the table below.

{\renewcommand{\arraystretch}{1.2}
\begin{longtable}{||r l ||} 
\hline Symbol &  Meaning \\*\hline\hline
$\spn{v_1, \dots, v_N}$ & Vector subspace spanned by the vectors $v_1, \dots, v_N \in \R^n$ \\\hline

$[i \isep j]$ & The set of integers between $i$ and $j$ included \\\hline

$X$ & A finite subset of $\R^n$ \\\hline

$S^{n-1}, B^n$ & Unit sphere and unit ball of $\R^n$ \\\hline

$\diag(A_1,\dots,A_n)$  & Block-diagonal matrix whose blocks are the matrices $A_1,\dots,A_n$ \\\hline

$\langle A,B\rangle$, $\|A\|$& Euclidean inner product and Euclidean norm of vectors $A,B\in\R^n$
\\
& or Frobenius inner product and Frobenius norm of matrices $A,B\in\M_n(\R)$\\\hline

$\|A\|_\mathrm{op}$ & Operator norm of a matrix $A\in\M_n(\R)$ \\\hline

$\projbracket{V}$ or $\proj_V$ & Orthogonal projection matrix on a subspace $V\subset \R^n$ or $V\subset \so(n)$ \\\hline

$\mathrm{Isom}(\man)$, & Isometry (resp. symmetry) group of a Riemannian manifold $\man$
\\$\Sym(\man)$& isometrically embedded in $\R^n$ (see Section \ref{subsec:symmetry_group})\\\hline

$\T_x\man$, $\N_x\man$  & The tangent (resp. normal) space of a submanifold $\man\subset\R^n$ at $x\in\man$ \\\hline

$G$, $\g$, $\phi$, $\O$ & A Lie group, a Lie algebra, a representation, an orbit of representation\\\hline

$\widehat{\h}$, $\widehat{\phi}$, $\widehat{\O}_x$, $\mu_{\widehat{\O}}$ 
&  Outputs of Algorithm \ref{alg: 1} (defined in Section \ref{subsec:step4}) \\\hline

$\text{Rep}(G, \R^n)$,  & The equivalence classes (resp. orbit-equivalence classes) of almost- \\$\text{OrbRep}(G, \R^n)$&faithful representations of the Lie group $G$ in $\R^n$ 
(see Section \ref{subsec:structure_orbits})\\\hline

$\text{Irr}(G)$ & A set consisting of a choice of an orthogonal representative \\
& for each equivalence class of irreps of the Lie group $G$\\\hline

$\mathfrak{orb}(G,n)$ & Pushforward algebras $\big\{ \big(B^{\phi_1},\dots,B^{\phi_p}\big) \mid (\phi_1,\dots,\phi_p) \in \mathrm{OrbRep}(G,n) \big\}$
\\& for an arbitrary choice of orthogonal representatives (see Section \ref{subsec:grassmann_lie})\\\hline

$\G^\mathrm{Lie}(d, \mathfrak{so}(n))$, & The Grassmannian (resp. Stiefel) variety of $d$-dimensional\\$\mathcal{V}^\mathrm{Lie}(d, \mathfrak{so}(n))$& Lie subalgebras of $\so(n)$ (see Section \ref{subsec:grassmann_lie_notgroup})\\\hline

$\G(G, \mathfrak{so}(n))$, & The Grassmannian (resp. Stiefel) variety of $d$-dimensional\\$\mathcal{V}(G, \mathfrak{so}(n))$& Lie subalgebras of $\so(n)$ pushforward from $G$ (see Section \ref{subsec:grassmann_lie})\\\hline

$\Lambda$~~~&  \texttt{LiePCA} operator (defined in Equation \eqref{eq:operator_sigma})\\\hline

$\Lambda_\O$  &  Ideal \texttt{LiePCA} operator (defined in Equation \eqref{eq:ideal_Lie_PCA_operator})\\\hline

$\HDmid{X}{Y}$  &  Non-symmetric Hausdorff  distance (resp. Hausdorff distance) \\$\HD{X}{Y}$& between two compact subsets $X,Y\subset \R^n$ (see Section \ref{sec:step4_hasdorff}) \\\hline

$\W_p(\mu,\nu)$  &  Wasserstein distance between $p$-integrable measures (see Section \ref{subsubsec:step4_wasserstein}.)\\\hline

$\mu_X$, $\mu_{G}$, $\mu_{\mathcal{O}}$  & Empirical measure on a point cloud $X$, Haar measure on \\& a compact Lie group $G$, uniform measure on an orbit $\O$\\\hline

$\Sigma[\mu]$ &Covariance matrix of the measure $\mu$ (defined in Equation \eqref{eq:covariance_matrix})\\\hline

$\Rig(G,n)$ & Rigidity of pushforward algebras of $G$ in $\R^n$ (resp. with \\$\Rig(G,n, \omega_{\max})$& weights at most $\omega_{\max}$) (defined in Equation \eqref{eq:gamma_def}) \\\hline

\end{longtable}
}

\subsection{Supplementary results}\label{sec:supplementary_results}
\subsubsection{Matrix exponential}\label{sec:supplementary_preliminaries_symmetrygroup}
\begin{proof}[Proof of Equation \eqref{eq:sym_algebra_formulation}]
As it is the case for any Lie subgroup of $\GL_n(\R)$, we have the equality 
\begin{equation}\label{eq:sym_algebra_formulation2}
\sym(\man) = \big\{A \in \M_n(\R)\mid \forall t \in \R, ~\exp(tA) \in \Sym(\man)\big\}.
\end{equation}
It can be found for instance in \cite[Th.~4.1]{sepanski2007compact}.
Let us first show that the set defined in Equation \eqref{eq:sym_algebra_formulation2} is equal to that of Equation \eqref{eq:sym_algebra_formulation}.
If $A$ is a matrix such that $\exp(tA) \in \Sym(\man)$ for all $t\in \R$, then for any $x\in \man$ the curve $t\mapsto \exp(tA)x$ lies in $\man$, hence its derivative at $0$, equal to $Ax$, lies in $\T_x \man$.
Conversely, suppose that $A$ satisfies $A x \in \T_x \man$ for all $x\in \man$.
Then for any $x_0 \in \man$, the curve $\gamma\colon t\mapsto \exp(tA)x_0$ has derivative $\gamma'(t) = A \gamma(t)$, hence $\gamma'(t) \in \T_{\gamma(t)}\man$, provided that $\gamma(t) \in \man$.
Consequently, the curve $\gamma$ satisfies an ordinary partial differential equation on the submanifold $\man$, showing that $\gamma(t)$ lies in $\man$. 
\end{proof}

\begin{proof}[Proof of Lemma \ref{lem:bound_exp_so(n)}]
Without loss of generality, we can suppose that $A$ is block diagonal.
There exists an integer $m\leq n/2$ and $m$ integers $\omega_1,\dots,\omega_m$ such that we can write
\begin{align*}
A &= \diag\big(L(\omega_1),\dots,L(\omega_m)\big)\\
\mathrm{and}~~~~~
\exp(tA) &= \diag\big(R(t\omega_1),\dots,R(t\omega_m)\big),
\end{align*}
with potentially zeros at the end, and where $L(\cdot)$ and $R(\cdot)$ are the $2\times2$ matrices defined in Example \ref{ex:normal_form_SO(2)}.
We have the identity $R(t\omega_1)=\cos(t\omega_i)I + (\sin(t\omega_i)/\omega_i)L(\omega_i)$.
That is, we can write 
$$
\exp(tA) = \diag\big(\cos(t\omega_1)I,\dots,\cos(t\omega_m)I\big) +\diag\big((\sin(t\omega_1)/\omega_1)L(\omega_1),\dots,(\sin(t\omega_m)/\omega_m)L(\omega_m)\big).
$$
Now, let us write $B = OAO^\top$ with $O\in \Ort(n)$.
Let also $e_1,\dots,e_n$ denote the canonical basis of $\R^n$.
For any integer $i\in[0,n]$ and $j$ the corresponding index in $[0,m]$, we deduce that
$$
\|\exp(tA) e_{i} - O\exp(tA)O^\top e_{i} \| = (\sin(t\omega_j)/\omega_j) \|Ae_{i}-Be_{i}\| \leq \|Ae_{i}-Be_{i}\|.  
$$
Summing this expression over all the $e_i$'s yields $\|\exp(tA) - \exp(t O A O^\top) \|\leq \|A-O A O^\top\|$.
\end{proof}

\subsubsection{Irreps of \texorpdfstring{$\SU(2)$}{SU(2)} and \texorpdfstring{$\SO(3)$}{SO(3)}}\label{app: irreducible reps su(2)}
As mentioned in Example \ref{ex:irreps_su(2)}, using the double covering $\SU(2)\rightarrow\SO(3)$, any representation $\SO(3)$ yields a representation of $\SU(2)$, although the converse is not true. 
This may be expressed by parametrizing the complex irreducible representations of $\SU(2)$ through $m\in \mathbb{Z}^+$ (i.e., $\widehat{\SU}(2) = \mathbb{Z}^+$) or, following Physics literature, by $j= m/2\in \frac{1}{2}\mathbb{N}^+$. 
It is well known that, while the representations with integer $j$ may be derived from (complex) representations of $\SO(3)$, the ones for half-integer $j$ may only be derived from the double cover $\SU(2)$ \cite{woit2017quantum}. Moreover, the ones for integer $j$ are said to be of the real type, whereas the others are of quaternionic type \cite{iwahori1959real}. Although all representations $\phi_j$ of real type may be transformed in real representations of $\su(2)$ by a simple change of coordinates, defining representations in $\R^{2j+1}$, for representations of quaternionic type, this is only achievable if $2j \equiv 1 \: (\text{mod } 4)$, in which case the representation dimension is of $4j+2$ \cite{itzkowitz1991note}. Therefore, if $\R^n$ receives a real irreducible representation of $\SU(2)$, then either $n$ is odd, in which case the representation is also a representation of $\SO(3)$, or $n$ is a multiple of 4. Moreover, for each choice of $n$, at most one representation of these groups can exist up to equivalence.

In \cite{campoamor2015elementary}, Campoamor-Stursberg derives the matrix elements of the real irreducible representations of $\SO(3)$ and $\SU(2)$. For the usual basis of $\so(3)$, the commutators are given by
\begin{equation*}
    [L_i,L_j] = \epsilon_{ijk}L_k
\end{equation*}
for the Levi-Civita symbol $\epsilon_{ijk}$.
Denoting the weights as $j\in \frac{1}{2}\mathbb{N}$, when $j$ is an integer, the corresponding representation of $\su(2)$ is of dimension $2j+1$ and have standard basis
\begin{equation*}
\begin{split} 
[\d\phi(L_1)]_{k,l}
   &=\bigg(\frac{1+(-1)^k}{2}\bigg)\bigg(\delta^l_{k+1}a_{\big\lfloor\frac{k}{2}\big\rfloor}+\delta^{l+3}_k a_{\big\lfloor\frac{k-2}{2}\big \rfloor}\bigg)\\
   &\;\;\; -\bigg(a_j+\sqrt{\frac{j^2+j}{2}}\bigg)(\delta^l_{2j+1}\delta_k^{2j}-\delta^{l}_{2j}\delta^{2j+1}_k)\\
   &\;\;\;-\bigg(\frac{1+(-1)^{k-1}}{2}\bigg)\bigg(\delta^l_{k+3}a_{\big\lfloor\frac{k+1}{2}\big\rfloor}+\delta^{l+1}_ka_{\big\lfloor \frac{k-1}{2}\big\rfloor}\bigg)\\
   [\d\phi(L_2)]_{k,l} &= 
   \bigg(a_j+\sqrt{\frac{j^2+j}{2}}\bigg)(\delta^l_{2j+1}\delta^{2j-1}_k-\delta^l_{2j-1}\delta_k^{2j+1})\\
   &\;\;\;-\delta^l_{k+2}a_{\big\lfloor\frac{k+1}{2}\big\rfloor}-\delta_k^{l+2}a_{\big\lfloor\frac{k-1}{2}\big\rfloor}\\
   [\d\phi(L_3)]_{k,l} &= \frac{1}{4}\times[(1+(-1)^k)\delta_{l+1}^{k}(2j+2-k)+((-1)^k-1)\delta_{k+1}^{l}(2j+1-k)]
\\\mathrm{where }~~~~~~~~~~~~~~~~~ a_l&=\sqrt{\frac{2jl-l(l-1)}{4}},1\leq l\leq j,
~~~~~~~~~~~~~~~~~~~~~~~~~~~~~~~
\end{split}
\end{equation*}

and $\lfloor \cdot \rfloor$ is the usual floor function. For half-integer $j$ such that $2j \equiv 1 \pmod{4}$, the representation has dimension $4j+2$ and has standard basis
\begin{equation*}
\begin{split}
   [\d\phi_j(L_1)]_{k,l}&= \bigg(\frac{1+(-1)^{k-1}}{2}\bigg)\bigg(\delta^l_{k+3}a_{\big\lfloor\frac{k+1}{2}\big\rfloor}+\delta^{l+1}_ka_{\big\lfloor \frac{k-1}{2}\big\rfloor}\bigg)\\
   &\;\;\;-\bigg(\frac{1+(-1)^{k}}{2}\bigg)\bigg(\delta^l_{k+1}a_{\big\lfloor \frac{k}{2}\big\rfloor}+\delta^{l+3}_ka_{\big\lfloor \frac{k-2}{2}\big\rfloor}\bigg)\\
   [\d\phi(L_2)]_{k,l} &= \delta^l_{k+2}a_{\big\lfloor\frac{k+1}{2}\big\rfloor}-\delta_k^{l+2}a_{\big\lfloor\frac{k-1}{2}\big\rfloor}\\
   [\d\phi(L_3)]_{k,l} &=\frac{1}{4}\times [1+(-1)^k]\delta^{l+1}_k(2j+2-k)+[(-1)^k-1]\delta^{k+1}_l(2j+1-k).
\end{split}
\end{equation*}
For other values of $j$, no real irreducible representations of $\su(2)$ are possible. 

\subsubsection{Computations for Remark \ref{rem:universal_constant_Lie-PCA}}\label{subsubsec:computations_universal_constant}
Consider the representation $\phi\colon \SO(2) \rightarrow \GL_4(\R)$ with distinct positive weights $(\omega_1,\omega_2)$, a real number $\epsilon>0$, $x_0=(1,0,\epsilon,0) \in \R^4$, and its orbit
$$
\O = \big\{ \big(\cos(\omega_1 \theta),\sin(\omega_1 \theta),\epsilon\cos(\omega_2 \theta),\epsilon\sin(\omega_2 \theta)\big) \mid \theta \in \SO(2) \big\}.
$$
The tangent vector at $x_0$ is equal to $(0,\omega_1,0,\epsilon\omega_2)$.
We then deduce that $(S_{x_0}\O)^\bot$ admits $(B_1,B_2,B_3)$ as an orthonormal basis, where
$$
B_1 = \begin{pmatrix}1&0&\epsilon&0\\0&0&0&0\\0&0&0&0\\0&0&0&0\end{pmatrix},
~~~~~B_2 = \begin{pmatrix}0&0&0&0\\0&0&0&0\\1&0&\epsilon&0\\0&0&0&0\end{pmatrix},
~~~~~B_3 =\begin{pmatrix}0&0&0&0\\\epsilon\omega_2&0&\epsilon^2\omega_2&0\\0&0&0&0\\-\omega_1&0&-\epsilon\omega_1&0\end{pmatrix}.
$$
Besides, the representation $\phi$ on $\R^4$ induces a representation on $\M_4(\R)$ by conjugation. 
Its restriction to $\so(4)$ decomposes into irreps as $\so(4) = \langle A_1 \rangle
\oplus \langle A_2 \rangle
\oplus \langle A_3,A_4 \rangle
\oplus \langle A_5,A_6 \rangle$, where
\begin{align*}
&A_1=\begin{pmatrix}0&-w_1&0&0\\w_1&0&0&0\\0&0&0&-w_2\\0&0&w_2&0\end{pmatrix},
~A_2=\begin{pmatrix}0&-w_2&0&0\\w_2&0&0&0\\0&0&0&w_1\\0&0&-w_1&0\end{pmatrix},
~A_3=\begin{pmatrix}0&0&0&-1\\0&0&1&0\\0&-1&0&0\\1&0&0&0\end{pmatrix},\\
&A_4=\begin{pmatrix}0&0&1&0\\0&0&0&1\\-1&0&0&0\\0&-1&0&0\end{pmatrix},
~~~~~~~~~A_5=\begin{pmatrix}0&0&0&-1\\0&0&-1&0\\0&1&0&0\\1&0&0&0\end{pmatrix},
~~~~~~~~~A_6=\begin{pmatrix}0&0&1&0\\0&0&0&-1\\-1&0&0&0\\0&1&0&0\end{pmatrix}.
\end{align*}
As we have seen in the proof of Proposition \ref{prop:consistency_LiePCA_idealcase}, these subspaces are stabilized by the action of $\SO(2)$ by conjugation.
Moreover, the first one, $\langle A_1 \rangle$, is the kernel of $\Lambda_\O$.
We can compute the eigenvalues of the other subspaces with the formula in Equation \eqref{eq:eigenvalues_formula_beta}. 
They are, respectively:
\begin{align*}
\sum_{i=1}^3 \big\langle A_2/\|A_2\|, B_i/\|B_i\| \big\rangle^2 
&= \frac{\epsilon^2}{2(1+\epsilon^2)}\frac{w_1^2+w_2^2}{w_1^2+(\epsilon w_2)^2},\\
\frac{1}{2}\bigg(\sum_{i=1}^3 \big\langle A_3/\|A_3\|, B_i/\|B_i\| \big\rangle^2+\sum_{i=1}^3 \big\langle A_4/\|A_4\|, B_i/\|B_i\| \big\rangle^2\bigg) 
&= \frac{1}{8}\bigg(1+\frac{1}{1+\epsilon^2}\frac{(w_1-\epsilon^2w_2)^2}{w_1^2+(\epsilon w_2)^2}\bigg),\\
\frac{1}{2}\bigg(\sum_{i=1}^3 \big\langle A_5/\|A_5\|, B_i/\|B_i\| \big\rangle^2+\sum_{i=1}^3 \big\langle A_6/\|A_6\|, B_i/\|B_i\| \big\rangle^2\bigg) 
&= \frac{1}{8}\bigg(1+\frac{1}{1+\epsilon^2}\frac{(w_1+\epsilon^2w_2)^2}{w_1^2+(\epsilon w_2)^2}\bigg).
\end{align*}

\subsubsection{Computations for Section \ref{subsec:pixel_permutations}}\label{sec:appendix_applications}

\begin{proof}[Proof of Lemma \ref{th: Tn-ppt}]
We first suppose that $\Sigma$ has rank one, that is to say, it is generated by a permutation matrix $\sigma$, and denote its order by $p$. If $x$ is any point of $X$, we can write
$$
X = \big\{ \sigma^m x \mid m\in[0\isep p-1]   \big\}.
$$
We distinguish two cases.
If $\sigma$ is an even permutation, then it has matrix determinant equal to $1$. 
That is, it belongs to $\SO(n)$, hence there exists a skew-symmetric matrix $L$ such that $\sigma = \exp(2\pi L/p)$.
The powers of $\sigma$ takes the form $\exp(2\pi m L/p)$ for $m \in [0\isep p-1]$ and we deduce that $X$ is included in the orbit of $x$ under the representation $\theta \mapsto \exp(2\pi \theta L/p)$ of $\SO(2)$.

On the other hand, if $\sigma$ is an odd permutation, then it does not admit a real logarithm.
In this case, we remember that $\sigma$ has eigenvalues $1$ and $-1$. 
The first comes from the fact that the vector $w=(1,\dots,1)\in\R^n$ is stabilized by $\sigma$, and the second is obtained from the decomposition of $\sigma$ into disjoint cycles.
Let $w'$ denote an eigenvector associated to $-1$.
We have an orthogonal decomposition $\R^n = V \oplus \spn{w}\oplus \spn{w'}$.
The matrix $\sigma$ stabilizes $V$. 
Consequently, in this decomposition of $\R^n$, it can be written as
$$
\sigma = \diag \big( O', 1, -1 \big)
$$
for a certain $O'\in \Ort(n-1)$.
The submatrix $\diag(1, -1)$ has determinant $-1$, hence that of $O'$ is $1$, meaning that it admits a real logarithm: $O' = \exp(2\pi L/p)$.
Besides, we note that the action of $\sigma$ on $X$ is the same as that of 
$$
\sigma' = \diag \big( O', -1, -1 \big).
$$
Indeed, the data has been centered, hence the point $x$ is orthogonal to the penultimate axis $\spn{w}$. Moreover, $\sigma'$ is now a matrix of determinant $1$, reducing to the first case.

More generally, let us suppose that $\Sigma$ has rank $d$, generated by a set of permutations $\{\sigma_1,\dots, \sigma_d\}$ of order $p_1,\dots,p_d$.
As we have seen above, each $\sigma_i$ can be written as $\exp(2\pi L_i/p_i)$ for a certain skew-symmetric matrix $L_i$.
Moreover, since the matrices $\sigma_i$'s commute by assumption, we can choose the $L_i$'s to commute.
We conclude that $X$ is included in the orbit of $x$ under the well-defined representation of $T^d$:
$$
(\theta_1,\dots,\theta_d) \longmapsto \prod_{j=1}^d \exp(2\pi \theta_j L_j/p_j).
$$
 
Next, we prove the second part of the lemma, starting with the case $d=1$ and supposing, without loss of generality, that $n$ is even.
The point cloud $X$ can be written as $\{ \phi(2\pi m/ p)\cdot x \mid m\in [0\isep p-1]\}$ for some representation $\phi$ of $\SO(2)$.
By decomposition into irreps, there exists $\omega \in \Z^{n/2}$ and a basis of $\R^n$ where $\phi$ is equal to 
$$
\theta \longmapsto \diag\big( R(\omega_1),\dots,R(\omega_{n/2}) \big).
$$
and where $R$ is the $2\times2$ rotation matrix defined in Example \ref{ex:normal_form_SO(2)}.
Moreover, using the explicit expression given in Lemma \ref{lm: covariance} below, the covariance matrix of $X$ takes the form
\begin{equation*}
\frac{1}{2} \begin{pmatrix}
x_1^2+x_2^2&&&&&\\
&x_1^2+x_2^2&&&&\\
&&\ddots&&&\\
&&&&x_{n-1}^2+x_{n}^2&\\
&&&&&x_{n-1}^2+x_{n}^2
\end{pmatrix}.
\end{equation*}
In particular, if $\lambda$ is any eigenvalue, then the eigenspace associated with $\lambda$ is a sum of stable planes of the representation. We deduce the statement of the lemma.
More generally, the result still holds if $\Sigma$ has rank $d$, because the covariance matrix can still be written as a diagonal matrix, with repeated eigenvalues on the stable planes.
\end{proof}

\begin{lemma}\label{lm: covariance}
Let $n$ be even, $\omega \in \Z^{n/2}$, and consider the corresponding representation of $\SO(2)$:
$$
\phi\colon \theta \longmapsto \diag\big( R(2\pi\omega_1\theta),\dots,R(2\pi\omega_{n/2}\theta) \big).
$$
Let $p$ be an integer that is a multiple of all the $\omega_i$, $x\in\R^n$ a nonzero point, and define the set $X = \{\phi(2\pi m/p)x \mid m\in[0\isep p-1]\}$.
Then the covariance matrix of $X$ is
\begin{equation*}
\Sigma[X] =
\frac{1}{2} \begin{pmatrix}
x_1^2+x_2^2&&&&&\\
&x_1^2+x_2^2&&&&\\
&&\ddots&&&\\
&&&&x_{n-1}^2+x_{n}^2&\\
&&&&&x_{n-1}^2+x_{n}^2
\end{pmatrix}.
\end{equation*}
\end{lemma}

\begin{proof}
Let us denote $d=n/2$. 
Explicitly, $X$ is the set of the points
\begin{equation}\label{eq: normal vec}
x_m = 
\begin{pmatrix}
\cos(2\pi m\omega_1/p)x_1-\sin(2\pi m \omega_1/p)x_2\\
\sin(2\pi m\omega_1/p)x_1+\cos(2\pi m \omega_1/p)x_2\\
\vdots\\
\cos(2\pi m\omega_1/p)x_{n-1}-\sin(2\pi m \omega_1/p)x_n\\
\sin(2\pi m\omega_1/p)x_{n-1}+\cos(2\pi m \omega_1/p)x_n
\end{pmatrix},
\end{equation}
for $m\in[0\isep p-1]$.
Denote the $ij$-th entry of $\Sigma[X]$ by $a_{ij}$. 
We wish to show that
\begin{enumerate}
\item\label{enum:proof_PCA_PPT:1} $\frac{1}{p}\sum_i a_{2i+1}^2=\frac{1}{p}\sum_i a_{2i+2}^2=\frac{1}{2}(x_{2i+1}^2+x_{2i+2}^2)$ for $0\leq i < d-1$,
\item\label{enum:proof_PCA_PPT:2} $\frac{1}{p}\sum_i a_{2i+1}a_{2i+2}=0$ for $0\leq i < d-1$,
\item\label{enum:proof_PCA_PPT:3} $\frac{1}{p}\sum_i a_{2i+1}a_{2j+1}=\frac{1}{|\sigma|}\sum_i a_{2i+1}a_{2j+2}=0$ for $0\leq i\neq j<d-1$.
\end{enumerate}
In order to prove Item \ref{enum:proof_PCA_PPT:1}, we expand the product $a_{2i+1}^2$ and apply Equations (\ref{eq:trigonometry:1}) and (\ref{eq:trigonometry:3}) stated in the elementary Lemma \ref{lm: trigonometric equations} below, whose proof is omitted.
Similarly, Item \ref{enum:proof_PCA_PPT:2} is proven by expanding $a_{2i+1}a_{2i+2}$ and using Equations (\ref{eq:trigonometry:1}) and (\ref{eq:trigonometry:3}).
In the same vein, we obtain Item \ref{enum:proof_PCA_PPT:3} from $a_{2i+1}a_{2j+1}$ with Equation (\ref{eq:trigonometry:2}).
\end{proof}

\begin{lemma}\label{lm: trigonometric equations}
For every integers $0<\omega,\nu<n$, one has
\begin{align}\label{eq:trigonometry:1}
\begin{split}
\sum_{k=0}^{n-1} \cos(2\pi\omega k/n)\sin(2\pi\nu k/n) = 0.
\end{split}\end{align}
Moreover, when $0<\omega,\nu<n$ are distinct, it holds
\begin{align}\label{eq:trigonometry:2}\begin{split}  
\sum_{k=0}^{n-1} \cos(2\pi\omega k/n)\cos(2\pi\nu k/n) = \sum_{k=0}^{n-1} \sin(2\pi\omega k/n)\sin(2\pi\nu k/n) = 0.
\end{split}\end{align}
Last, for every integer $0<\omega<n/2$,
\begin{align}\label{eq:trigonometry:3}\begin{split}
\sum_{k=0}^{n-1} \cos(2\pi\omega k/n)^2 = \sum_{k=0}^{n-1} \sin(2\pi\omega k/n)^2 = \frac{n}{2}
\end{split}\end{align}
\end{lemma}

\bibliographystyle{plain}
\bibliography{main.bib}

\begin{thebibliography}{100}

\bibitem{10.1214/18-AOS1685}
Eddie Aamari and Cl{\'e}ment Levrard.
\newblock {Nonasymptotic rates for manifold, tangent space and curvature
  estimation}.
\newblock {\em The Annals of Statistics}, 47(1):177 -- 204, 2019.

\bibitem{absil2009optimization}
P-A Absil, Robert Mahony, and Rodolphe Sepulchre.
\newblock Optimization algorithms on matrix manifolds.
\newblock In {\em Optimization Algorithms on Matrix Manifolds}. Princeton
  University Press, 2009.

\bibitem{adem2007commuting}
Alejandro Adem and Frederick~R Cohen.
\newblock Commuting elements and spaces of homomorphisms.
\newblock {\em Mathematische Annalen}, 338(3):587--626, 2007.

\bibitem{agrachev2010towards}
Andrei~A Agrachev, Yuliy Baryshnikov, and Daniel Liberzon.
\newblock Towards robust {L}ie-algebraic stability conditions for switched
  linear systems.
\newblock In {\em 49th IEEE Conference on Decision and Control (CDC)}, pages
  408--413. IEEE, 2010.

\bibitem{agrachev2001lie}
Andrei~A Agrachev and Daniel Liberzon.
\newblock Lie-algebraic stability criteria for switched systems.
\newblock {\em SIAM Journal on Control and Optimization}, 40(1):253--269, 2001.

\bibitem{arias2017spectral}
Ery Arias-Castro, Gilad Lerman, and Teng Zhang.
\newblock Spectral clustering based on local {PCA}.
\newblock {\em Journal of Machine Learning Research}, 18(9):1--57, 2017.

\bibitem{baird2007cohomology}
Thomas~John Baird.
\newblock Cohomology of the space of commuting n--tuples in a compact {L}ie
  group.
\newblock {\em Algebraic \& Geometric Topology}, 7(2):737--754, 2007.

\bibitem{barrionuevo2023deformations}
Josefina Barrionuevo, Paulo Tirao, and Diego Sulca.
\newblock Deformations and rigidity in varieties of {L}ie algebras.
\newblock {\em Journal of Pure and Applied Algebra}, 227(3):107217, 2023.

\bibitem{belinfante1989survey}
Johan~G Belinfante and Bernard Kolman.
\newblock {\em A survey of {L}ie groups and {L}ie algebras with applications
  and computational methods}.
\newblock SIAM, 1989.

\bibitem{boissonnat2019reach}
Jean-Daniel Boissonnat, Andr{\'e} Lieutier, and Mathijs Wintraecken.
\newblock The reach, metric distortion, geodesic convexity and the variation of
  tangent spaces.
\newblock {\em Journal of applied and computational topology}, 3:29--58, 2019.

\bibitem{brocker2013representations}
Theodor Br{\"o}cker and Tammo Tom~Dieck.
\newblock {\em Representations of compact {L}ie groups}, volume~98.
\newblock Springer Science \& Business Media, 2013.

\bibitem{broucke1975relative}
R~Broucke.
\newblock On relative periodic solutions of the planar general three-body
  problem.
\newblock {\em Celestial mechanics}, 12(4):439--462, 1975.

\bibitem{broucke1975periodic}
R~Broucke and D~Boggs.
\newblock Periodic orbits in the planar general three-body problem.
\newblock {\em Celestial mechanics}, 11(1):13--38, 1975.

\bibitem{broucke1979isosceles}
Roger Broucke.
\newblock On the isosceles triangle configuration in the planar general
  three-body problem.
\newblock {\em Astronomy and Astrophysics, vol. 73, no. 3, Mar. 1979, p.
  303-313. Research supported by the University of Texas.}, 73:303--313, 1979.

\bibitem{brown2008algorithmic}
W~Michael Brown, Shawn Martin, Sara~N Pollock, Evangelos~A Coutsias, and
  Jean-Paul Watson.
\newblock Algorithmic dimensionality reduction for molecular structure
  analysis.
\newblock {\em The Journal of chemical physics}, 129(6), 2008.

\bibitem{buet2017varifold}
Blanche Buet, Gian~Paolo Leonardi, and Simon Masnou.
\newblock A varifold approach to surface approximation.
\newblock {\em Archive for Rational Mechanics and Analysis}, 226:639--694,
  2017.

\bibitem{DBLP:journals/corr/abs-2008-04278}
Jameson Cahill, Dustin~G Mixon, and Hans Parshall.
\newblock Lie {PCA}: {D}ensity estimation for symmetric manifolds.
\newblock {\em Applied and Computational Harmonic Analysis}, 65:279--295, 2023.

\bibitem{campoamor2015elementary}
Rutwig Campoamor-Stursberg.
\newblock An elementary derivation of the matrix elements of real irreducible
  representations of {SO}(3).
\newblock {\em Symmetry}, 7(3):1655--1669, 2015.

\bibitem{cartan1952theorie}
{\'E}lie Cartan.
\newblock {\em La th{\'e}orie des groupes finis et continus et l'analysis
  situs}.
\newblock Gauthier-Villars, 1952.

\bibitem{chirikjian2000engineering}
Gregory~S Chirikjian.
\newblock {\em Engineering applications of noncommutative harmonic analysis:
  with emphasis on rotation and motion groups}.
\newblock CRC press, 2000.

\bibitem{cohen2014learning}
Taco Cohen and Max Welling.
\newblock Learning the irreducible representations of commutative {L}ie groups.
\newblock In {\em International Conference on Machine Learning}, pages
  1755--1763. PMLR, 2014.

\bibitem{cohen2016group}
Taco Cohen and Max Welling.
\newblock Group equivariant convolutional networks.
\newblock In {\em International conference on machine learning}, pages
  2990--2999. PMLR, 2016.

\bibitem{cohen2016steerable}
Taco~S Cohen and Max Welling.
\newblock Steerable cnns.
\newblock {\em arXiv preprint arXiv:1612.08498}, 2016.

\bibitem{crainic2014survey}
Marius Crainic, Florian Sch{\"a}tz, and Ivan Struchiner.
\newblock A survey on stability and rigidity results for {L}ie algebras.
\newblock {\em Indagationes Mathematicae}, 25(5):957--976, 2014.

\bibitem{NIPS2013_af21d0c9}
Marco Cuturi.
\newblock Sinkhorn distances: Lightspeed computation of optimal transport.
\newblock In C.J. Burges, L.~Bottou, M.~Welling, Z.~Ghahramani, and K.Q.
  Weinberger, editors, {\em Advances in Neural Information Processing Systems},
  volume~26. Curran Associates, Inc., 2013.

\bibitem{davis1970rotation}
Chandler Davis and William~Morton Kahan.
\newblock The rotation of eigenvectors by a perturbation. {I}{I}{I}.
\newblock {\em SIAM Journal on Numerical Analysis}, 7(1):1--46, 1970.

\bibitem{dehmamy2021automatic}
Nima Dehmamy, Robin Walters, Yanchen Liu, Dashun Wang, and Rose Yu.
\newblock Automatic symmetry discovery with {L}ie algebra convolutional
  network.
\newblock {\em Advances in Neural Information Processing Systems},
  34:2503--2515, 2021.

\bibitem{deng2021vector}
Congyue Deng, Or~Litany, Yueqi Duan, Adrien Poulenard, Andrea Tagliasacchi, and
  Leonidas~J Guibas.
\newblock Vector neurons: A general framework for so (3)-equivariant networks.
\newblock In {\em Proceedings of the IEEE/CVF International Conference on
  Computer Vision}, pages 12200--12209, 2021.

\bibitem{divol2021short}
Vincent Divol.
\newblock A short proof on the rate of convergence of the empirical measure for
  the {W}asserstein distance.
\newblock {\em arXiv preprint arXiv:2101.08126}, 2021.

\bibitem{duke2007introduction}
W~Duke, A~Granville, and Z~Rudnick.
\newblock An introduction to the {L}innik problems.
\newblock {\em Equidistribution in number theory, an introduction},
  237:197--216, 2007.

\bibitem{duke1988hyperbolic}
William Duke.
\newblock Hyperbolic distribution problems and half-integral weight {M}aass
  forms.
\newblock {\em Inventiones mathematicae}, 92(1):73--90, 1988.

\bibitem{fecko2006differential}
Mari{\'a}n Fecko.
\newblock {\em Differential geometry and {L}ie groups for physicists}.
\newblock Cambridge university press, 2006.

\bibitem{federer1959curvature}
Herbert Federer.
\newblock Curvature measures.
\newblock {\em Transactions of the American Mathematical Society},
  93(3):418--491, 1959.

\bibitem{finzi2020generalizing}
Marc Finzi, Samuel Stanton, Pavel Izmailov, and Andrew~Gordon Wilson.
\newblock Generalizing convolutional neural networks for equivariance to lie
  groups on arbitrary continuous data.
\newblock In {\em International Conference on Machine Learning}, pages
  3165--3176. PMLR, 2020.

\bibitem{flamary2021pot}
R{\'e}mi Flamary, Nicolas Courty, Alexandre Gramfort, Mokhtar~Z. Alaya,
  Aur{\'e}lie Boisbunon, Stanislas Chambon, Laetitia Chapel, Adrien Corenflos,
  Kilian Fatras, Nemo Fournier, L{\'e}o Gautheron, Nathalie~T.H. Gayraud,
  Hicham Janati, Alain Rakotomamonjy, Ievgen Redko, Antoine Rolet, Antony
  Schutz, Vivien Seguy, Danica~J. Sutherland, Romain Tavenard, Alexander Tong,
  and Titouan Vayer.
\newblock {POT}: {P}ython {O}ptimal {T}ransport.
\newblock {\em Journal of Machine Learning Research}, 22(78):1--8, 2021.

\bibitem{folland2016course}
Gerald~B Folland.
\newblock {\em A course in abstract harmonic analysis}, volume~29.
\newblock CRC press, 2016.

\bibitem{friedland1983simultaneous}
Shmuel Friedland.
\newblock Simultaneous similarity of matrices.
\newblock {\em Advances in mathematics}, 50(3):189--265, 1983.

\bibitem{gaal2018certain}
Marcell Ga{\'a}l.
\newblock On certain generalized isometries of the special orthogonal group.
\newblock {\em Archiv der Mathematik}, 110:61--70, 2018.

\bibitem{gallian2021contemporary}
Joseph~A Gallian.
\newblock {\em Contemporary abstract algebra}.
\newblock Chapman and Hall/CRC, 2021.

\bibitem{e3nn_paper}
Mario Geiger and Tess Smidt.
\newblock e3nn: Euclidean neural networks, 2022.

\bibitem{e3nn}
Mario Geiger, Tess Smidt, Alby M., Benjamin~Kurt Miller, Wouter Boomsma,
  Bradley Dice, Kostiantyn Lapchevskyi, Maurice Weiler, Michał Tyszkiewicz,
  Simon Batzner, Dylan Madisetti, Martin Uhrin, Jes Frellsen, Nuri Jung, Sophia
  Sanborn, Mingjian Wen, Josh Rackers, Marcel Rød, and Michael Bailey.
\newblock Euclidean neural networks: e3nn, April 2022.

\bibitem{griffiths2020introduction}
David Griffiths.
\newblock {\em Introduction to elementary particles}.
\newblock John Wiley \& Sons, 2020.

\bibitem{harris2020array}
Charles~R. Harris, K.~Jarrod Millman, St{\'{e}}fan~J. van~der Walt, Ralf
  Gommers, Pauli Virtanen, David Cournapeau, Eric Wieser, Julian Taylor,
  Sebastian Berg, Nathaniel~J. Smith, Robert Kern, Matti Picus, Stephan Hoyer,
  Marten~H. van Kerkwijk, Matthew Brett, Allan Haldane, Jaime~Fern{\'{a}}ndez
  del R{\'{i}}o, Mark Wiebe, Pearu Peterson, Pierre G{\'{e}}rard-Marchant,
  Kevin Sheppard, Tyler Reddy, Warren Weckesser, Hameer Abbasi, Christoph
  Gohlke, and Travis~E. Oliphant.
\newblock Array programming with {NumPy}.
\newblock {\em Nature}, 585(7825):357--362, sep 2020.

\bibitem{higgins2018definition}
Irina Higgins, David Amos, David Pfau, Sebastien Racaniere, Loic Matthey,
  Danilo Rezende, and Alexander Lerchner.
\newblock Towards a definition of disentangled representations.
\newblock {\em arXiv preprint arXiv:1812.02230}, 2018.

\bibitem{highnam1986optimal}
Peter~T Highnam.
\newblock Optimal algorithms for finding the symmetries of a planar point set.
\newblock {\em Information Processing Letters}, 22:219--222, 1986.

\bibitem{hilgert2011structure}
Joachim Hilgert and Karl-Hermann Neeb.
\newblock {\em Structure and geometry of {L}ie groups}.
\newblock Springer Science \& Business Media, 2011.

\bibitem{hoffman2003variation}
Alan~J Hoffman and Helmut~W Wielandt.
\newblock The variation of the spectrum of a normal matrix.
\newblock In {\em Selected Papers Of Alan J Hoffman: With Commentary}, pages
  118--120. World Scientific, 2003.

\bibitem{hoffman1966lie}
William~C Hoffman.
\newblock The {L}ie algebra of visual perception.
\newblock {\em Journal of mathematical Psychology}, 3(1):65--98, 1966.

\bibitem{hoffman1968neuron}
William~C Hoffman.
\newblock The neuron as a {L}ie group germ and a {L}ie product.
\newblock {\em Quarterly of Applied Mathematics}, 25(4):423--440, 1968.

\bibitem{horesh2023equidistribution}
Tal Horesh and Yakov Karasik.
\newblock Equidistribution of primitive lattices in {R}n.
\newblock {\em The Quarterly Journal of Mathematics}, 74(4):1253--1294, 2023.

\bibitem{itzkowitz1991note}
Gerald Itzkowitz, Sheldon Rothman, and Helen Strassberg.
\newblock A note on the real representations of {SU}(2,{C}).
\newblock {\em Journal of Pure and Applied Algebra}, 69(3):285--294, 1991.

\bibitem{iwahori1959real}
Nagayoshi Iwahori.
\newblock On real irreducible representations of {L}ie algebras.
\newblock {\em Nagoya Mathematical Journal}, 14:59--83, 1959.

\bibitem{jenner2021steerable}
Erik Jenner and Maurice Weiler.
\newblock Steerable partial differential operators for equivariant neural
  networks.
\newblock {\em arXiv preprint arXiv:2106.10163}, 2021.

\bibitem{kaslovsky2014non}
Daniel~N Kaslovsky and Fran{\c{c}}ois~G Meyer.
\newblock Non-asymptotic analysis of tangent space perturbation.
\newblock {\em Information and Inference: a Journal of the IMA}, 3(2):134--187,
  2014.

\bibitem{kelleher2009generating}
Jerome Kelleher and Barry O'Sullivan.
\newblock Generating all partitions: a comparison of two encodings.
\newblock {\em arXiv preprint arXiv:0909.2331}, 2009.

\bibitem{kondor2018generalization}
Risi Kondor and Shubhendu Trivedi.
\newblock On the generalization of equivariance and convolution in neural
  networks to the action of compact groups.
\newblock In {\em International conference on machine learning}, pages
  2747--2755. PMLR, 2018.

\bibitem{landrum2013rdkit}
Greg Landrum et~al.
\newblock Rdkit: A software suite for cheminformatics, computational chemistry,
  and predictive modeling.
\newblock {\em Greg Landrum}, 8(31.10):5281, 2013.

\bibitem{lecun1995convolutional}
Yann LeCun, Yoshua Bengio, et~al.
\newblock Convolutional networks for images, speech, and time series.
\newblock {\em The handbook of brain theory and neural networks},
  3361(10):1995, 1995.

\bibitem{lee2013smooth}
John~M Lee.
\newblock {\em Introduction to smooth manifolds}.
\newblock Springer, 2013.

\bibitem{liao2022three}
Shijun Liao, Xiaoming Li, and Yu~Yang.
\newblock Three-body problem—from newton to supercomputer plus machine
  learning.
\newblock {\em New Astronomy}, 96:101850, 2022.

\bibitem{lim2023hades}
Uzu Lim, Harald Oberhauser, and Vidit Nanda.
\newblock Hades: Fast singularity detection with local measure comparison.
\newblock {\em arXiv preprint arXiv:2311.04171}, 2023.

\bibitem{lim2021tangent}
Uzu Lim, Harald Oberhauser, and Vidit Nanda.
\newblock Tangent space and dimension estimation with the {W}asserstein
  distance.
\newblock {\em SIAM Journal on Applied Algebra and Geometry}, 8(3):650--685,
  2024.

\bibitem{linnik2ergodic}
Y.V. Linnik.
\newblock Ergodic properties of algebraic fields.
\newblock {\em Ergebnisse der Mathematik und ihrer Grenzgebiete}, 1968.

\bibitem{lopatin2011orthogonal}
AA~Lopatin.
\newblock Orthogonal invariants of skew-symmetric matrices.
\newblock {\em Linear and Multilinear Algebra}, 59(8):851--862, 2011.

\bibitem{martin2010topology}
Shawn Martin, Aidan Thompson, Evangelos~A Coutsias, and Jean-Paul Watson.
\newblock Topology of cyclo-octane energy landscape.
\newblock {\em The journal of chemical physics}, 132(23), 2010.

\bibitem{meinecke2012simultaneous}
Frank~C Meinecke.
\newblock Simultaneous diagonalization of skew-symmetric matrices in the
  symplectic group.
\newblock In {\em Latent Variable Analysis and Signal Separation: 10th
  International Conference, LVA/ICA 2012, Tel Aviv, Israel, March 12-15, 2012.
  Proceedings 10}, pages 147--154. Springer, 2012.

\bibitem{membrillo2019topology}
Ingrid Membrillo-Solis, Mariam Pirashvili, Lee Steinberg, Jacek Brodzki, and
  Jeremy~G Frey.
\newblock Topology and geometry of molecular conformational spaces and energy
  landscapes.
\newblock {\em arXiv preprint arXiv:1907.07770}, 2019.

\bibitem{miller1973complex}
Kenneth~S Miller.
\newblock Complex linear least squares.
\newblock {\em Siam Review}, 15(4):706--726, 1973.

\bibitem{MILNOR1976293}
John Milnor.
\newblock Curvatures of left invariant metrics on {L}ie groups.
\newblock {\em Advances in Mathematics}, 21(3):293--329, 1976.

\bibitem{mitra2013symmetry}
Niloy~J Mitra, Mark Pauly, Michael Wand, and Duygu Ceylan.
\newblock Symmetry in 3{D} geometry: Extraction and applications.
\newblock In {\em Computer Graphics Forum}, volume~32, pages 1--23. Wiley
  Online Library, 2013.

\bibitem{moakher2002means}
Maher Moakher.
\newblock Means and averaging in the group of rotations.
\newblock {\em SIAM journal on matrix analysis and applications}, 24(1):1--16,
  2002.

\bibitem{moore1993braids}
Cristopher Moore.
\newblock Braids in classical dynamics.
\newblock {\em Physical Review Letters}, 70(24):3675, 1993.

\bibitem{musielak2014three}
Zdzislaw~E Musielak and Billy Quarles.
\newblock The three-body problem.
\newblock {\em Reports on Progress in Physics}, 77(6):065901, 2014.

\bibitem{myers1939group}
Sumner~B Myers and Norman~Earl Steenrod.
\newblock The group of isometries of a {R}iemannian manifold.
\newblock {\em Annals of Mathematics}, pages 400--416, 1939.

\bibitem{Newton_2016}
Isaac Newton.
\newblock {\em The Principia: Mathematical principles of natural philosophy}.
\newblock University of California Press, 2016.

\bibitem{niyogi2008finding}
Partha Niyogi, Stephen Smale, and Shmuel Weinberger.
\newblock Finding the homology of submanifolds with high confidence from random
  samples.
\newblock {\em Discrete \& Computational Geometry}, 39:419--441, 2008.

\bibitem{nolte2006identifying}
Guido Nolte, Frank~C Meinecke, Andreas Ziehe, and Klaus-Robert M{\"u}ller.
\newblock Identifying interactions in mixed and noisy complex systems.
\newblock {\em Physical Review E}, 73(5):051913, 2006.

\bibitem{peter1927vollstandigkeit}
Fritz Peter and Hermann Weyl.
\newblock Die vollst{\"a}ndigkeit der primitiven darstellungen einer
  geschlossenen kontinuierlichen gruppe.
\newblock {\em Mathematische Annalen}, 97(1):737--755, 1927.

\bibitem{pfau2020disentangling}
David Pfau, Irina Higgins, Alex Botev, and S{\'e}bastien Racani{\`e}re.
\newblock Disentangling by subspace diffusion.
\newblock {\em Advances in Neural Information Processing Systems},
  33:17403--17415, 2020.

\bibitem{poincare1889problem}
Henri Poincar{\'e}.
\newblock {\em On the three-body problem and the {\'e}quations of dynamics}.
\newblock Memoir crowned with the prize of HM King Oscar II, 1889.

\bibitem{procesi2007lie}
Claudio Procesi et~al.
\newblock {\em Lie groups: an approach through invariants and representations},
  volume 115.
\newblock Springer, 2007.

\bibitem{richardson1988conjugacy}
R.~W. Richardson.
\newblock {Conjugacy classes of $n$-tuples in {L}ie algebras and algebraic
  groups}.
\newblock {\em Duke Mathematical Journal}, 57(1):1 -- 35, 1988.

\bibitem{sadeghi2013metrics}
Ali Sadeghi, S~Alireza Ghasemi, Bastian Schaefer, Stephan Mohr, Markus~A Lill,
  and Stefan Goedecker.
\newblock Metrics for measuring distances in configuration spaces.
\newblock {\em The Journal of chemical physics}, 139(18), 2013.

\bibitem{saxena2009learning}
Ashutosh Saxena, Justin Driemeyer, and Andrew~Y Ng.
\newblock Learning 3-d object orientation from images.
\newblock In {\em 2009 IEEE International conference on robotics and
  automation}, pages 794--800. IEEE, 2009.

\bibitem{schmidt1998distribution}
Wolfgang~M Schmidt.
\newblock The distribution of sublattices of {Z}m.
\newblock {\em Monatshefte f{\"u}r Mathematik}, 125:37--81, 1998.

\bibitem{Schnemann1968OnTO}
Peter~H. Sch{\"o}nemann.
\newblock On two-sided orthogonal {P}rocrustes problems.
\newblock {\em Psychometrika}, 33:19--33, 1968.

\bibitem{scoccola2022fibered}
Luis Scoccola and Jose~A. Perea.
\newblock {FibeRed: Fiberwise Dimensionality Reduction of Topologically Complex
  Data with Vector Bundles}.
\newblock In Erin~W. Chambers and Joachim Gudmundsson, editors, {\em 39th
  International Symposium on Computational Geometry (SoCG 2023)}, volume 258 of
  {\em Leibniz International Proceedings in Informatics (LIPIcs)}, pages
  56:1--56:18, Dagstuhl, Germany, 2023. Schloss Dagstuhl -- Leibniz-Zentrum
  f{\"u}r Informatik.

\bibitem{sepanski2007compact}
Mark~R Sepanski.
\newblock {\em Compact {L}ie groups}.
\newblock Springer, 2007.

\bibitem{shi2016symmetry}
Zeyun Shi, Pierre Alliez, Mathieu Desbrun, Hujun Bao, and Jin Huang.
\newblock Symmetry and orbit detection via {L}ie-algebra voting.
\newblock In {\em Computer Graphics Forum}, volume~35, pages 217--227. Wiley
  Online Library, 2016.

\bibitem{silverman2018density}
Bernard~W Silverman.
\newblock {\em Density estimation for statistics and data analysis}.
\newblock Routledge, 2018.

\bibitem{singer2012vector}
Amit Singer and H-T Wu.
\newblock Vector diffusion maps and the connection {L}aplacian.
\newblock {\em Communications on pure and applied mathematics},
  65(8):1067--1144, 2012.

\bibitem{singh2018minimax}
Shashank Singh and Barnab{\'a}s P{\'o}czos.
\newblock Minimax distribution estimation in {W}asserstein distance.
\newblock {\em arXiv preprint arXiv:1802.08855}, 2018.

\bibitem{sohldickstein2017unsupervised}
Jascha Sohl-Dickstein, Ching~Ming Wang, and Bruno~A Olshausen.
\newblock An unsupervised algorithm for learning {L}ie group transformations.
\newblock {\em arXiv preprint arXiv:1001.1027}, 2010.

\bibitem{steinberg2019topological}
Lee Steinberg.
\newblock {\em Topological Data Analysis and its Application to Chemical
  Systems}.
\newblock PhD thesis, University of Southampton, 2019.

\bibitem{stewart1990matrix}
Gilbert~W Stewart and Ji-guang Sun.
\newblock {\em Matrix perturbation theory}.
\newblock Academic press, 1990.

\bibitem{stolz2020geometric}
Bernadette~J Stolz, Jared Tanner, Heather~A Harrington, and Vidit Nanda.
\newblock Geometric anomaly detection in data.
\newblock {\em Proceedings of the national academy of sciences},
  117(33):19664--19669, 2020.

\bibitem{vsuvakov2013three}
Milovan {\v{S}}uvakov and Veljko Dmitra{\v{s}}inovi{\'c}.
\newblock Three classes of newtonian three-body planar periodic orbits.
\newblock {\em Physical review letters}, 110(11):114301, 2013.

\bibitem{taha2015efficient}
Abdel~Aziz Taha and Allan Hanbury.
\newblock An efficient algorithm for calculating the exact {H}ausdorff
  distance.
\newblock {\em IEEE transactions on pattern analysis and machine intelligence},
  37(11):2153--2163, 2015.

\bibitem{tinarrage2023recovering}
Rapha{\"e}l Tinarrage.
\newblock Recovering the homology of immersed manifolds.
\newblock {\em Discrete \& Computational Geometry}, pages 1--86, 2023.

\bibitem{torres2008fundamental}
Enrique Torres~Giese and Denis Sjerve.
\newblock Fundamental groups of commuting elements in {L}ie groups.
\newblock {\em Bulletin of the London Mathematical Society}, 40(1):65--76,
  2008.

\bibitem{JMLR:v17:16-177}
James Townsend, Niklas Koep, and Sebastian Weichwald.
\newblock Pymanopt: A {P}ython toolbox for optimization on manifolds using
  automatic differentiation.
\newblock {\em Journal of Machine Learning Research}, 17(137):1–5, 2016.

\bibitem{tyagi2013tangent}
Hemant Tyagi, El{\i}f Vural, and Pascal Frossard.
\newblock Tangent space estimation for smooth embeddings of {R}iemannian
  manifolds.
\newblock {\em Information and Inference: A Journal of the IMA}, 2(1):69--114,
  2013.

\bibitem{umeyama1988eigendecomposition}
Shinji Umeyama.
\newblock An eigendecomposition approach to weighted graph matching problems.
\newblock {\em IEEE transactions on pattern analysis and machine intelligence},
  10(5):695--703, 1988.

\bibitem{e2cnn}
Maurice Weiler and Gabriele Cesa.
\newblock {General E(2)-Equivariant Steerable CNNs}.
\newblock In {\em Conference on Neural Information Processing Systems
  (NeurIPS)}, 2019.

\bibitem{weiler2018learning}
Maurice Weiler, Fred~A Hamprecht, and Martin Storath.
\newblock Learning steerable filters for rotation equivariant cnns.
\newblock In {\em Proceedings of the IEEE Conference on Computer Vision and
  Pattern Recognition}, pages 849--858, 2018.

\bibitem{woit2017quantum}
Peter Woit, Woit, and Bartolini.
\newblock {\em Quantum theory, groups and representations}.
\newblock Springer, 2017.

\bibitem{wolter1985optimal}
Jan~D Wolter, Tony~C Woo, and Richard~A Volz.
\newblock Optimal algorithms for symmetry detection in two and three
  dimensions.
\newblock {\em The Visual Computer}, 1:37--48, 1985.

\end{thebibliography}
\end{document}